\documentclass[draft]{elsarticle}
\biboptions{sort&compress}
\def\revision#1{#1}

\usepackage[T1]{fontenc}
\usepackage{amsmath,amssymb,ifthen}
\usepackage{amsthm}
\usepackage{fullpage}
\usepackage{longtable}
\usepackage{graphicx,psfrag,subfigure}
\usepackage[dvips]{color}
\usepackage{lmodern,pst-node}
\usepackage{moreverb}
\SpecialCoor

\DeclareMathOperator{\mmid}{mid}
\DeclareMathOperator{\ddiv}{div}
\DeclareMathOperator{\curl}{curl}
\DeclareMathOperator{\dev}{dev}
\DeclareMathOperator{\tr}{tr}

\def\A{{\mathbb A}}
\def\B{{\mathbb B}}

\def\H{\widetilde{H}}

\def\R{{\mathbb R}}
\def\N{{\mathbb N}}

\def\O{{\mathbb O}}

\def\EE{{\mathcal E}}

\def\II{{\mathcal I}}
\def\KK{{\mathcal K}}
\def\MM{{\mathcal M}}

\def\OO{{\mathcal O}}
\def\PP{{\mathcal P}}
\def\RR{{\mathcal R}}
\def\SS{{\mathcal S}}
\def\TT{{\mathcal T}}

\def\XX{{\mathcal X}}
\def\YY{{\mathcal Y}}

\def\res{{r}}

\newcommand{\dist}[3][]{{\rm d\!l}[\ifthenelse{\equal{#1}{}}{}{#1;}#2,#3]}
\newcommand{\eff}[3][]{{\rm osc}\ifthenelse{\equal{#1}{}}{}{_{#1}}(#2;#3)}
\def\diam{{\rm diam}}
\def\diam{{\rm diam}}

\def\bform#1#2{b(#1\,,\,#2)}
\def\Bform#1#2#3{B(#3;\,#1\,,\,#2)}
\def\norm#1#2{\|#1\|_{#2}}

\def\set#1#2{\big\{#1\,:\,#2\big\}}

\def\eps{\varepsilon}

\def\slp{\mathfrak{V}} 
\def\dlp{\mathfrak{K}} 
\def\hyp{\mathfrak{W}} 

\newcommand{\normLtwo}[3][]{#1\|#2#1\|_{L^2(#3)}}
\newcommand{\normHeh}[3][]{#1\|#2#1\|_{H^{1/2}(#3)}}
\newcommand{\enorm}[1]{\|#1\|}
\newcommand{\quasinorm}[2]{|\!|\!|#1|\!|\!|_{#2}}

\def\normL2#1#2{\|#1\|_{L^2(#2)}}
\def\normHme#1#2{\|#1\|_{\widetilde{H}^{-1}(#2)}}
\newcommand{\dual}[3][]{#1\langle#2\,,\,#3#1\rangle}

\newcounter{constantsnumber}
\def\namec#1#2{%
 \def\flag{0}
 \ifthenelse{\equal{#1}{triangle}}{C_{\mbox{\scriptsize$\Delta$}}}{%
 \ifthenelse{\equal{#1}{rel}}{C_{\rm rel}}{%
  \ifthenelse{\equal{#1}{mesh}}{C_{\rm mesh}}{%
  \ifthenelse{\equal{#1}{sz}}{C_{\rm sz}}{%
  \ifthenelse{\equal{#1}{dislocrel}}{C_{\rm dlr}}{%
  \ifthenelse{\equal{#1}{eff}}{C_{\rm eff}}{%
  \ifthenelse{\equal{#1}{main}}{C_{\rm V}}{%
  \ifthenelse{\equal{#1}{opt}}{C_{\rm opt}}{%
  \ifthenelse{\equal{#1}{normequiv}}{C_{\rm norm}}{%
  \ifthenelse{\equal{#1}{reliable}}{C_{\rm rel}}{%
  \ifthenelse{\equal{#1}{efficient}}{C_{\rm eff}}{%
  \ifthenelse{\equal{#1}{dlr}}{C_{\rm drel}}{%
  \ifthenelse{\equal{#1}{stable}}{C_{\rm stab}}{%
  \ifthenelse{\equal{#1}{stab}}{\widetilde{C}_{\rm stab}}{%
  \ifthenelse{\equal{#1}{stab2}}{\overline{C}_{\rm stab}}{%
  \ifthenelse{\equal{#1}{lbb}}{C_{\rm LBB}}{%
  \ifthenelse{\equal{#1}{reduction}}{C_{\rm red}}{%
   \ifthenelse{\equal{#1}{unibound}}{C_{\rm hot}}{%
    \ifthenelse{\equal{#1}{hotConst}}{C_{\rm osc}}{%
   \ifthenelse{\equal{#1}{inverseK}}{C_{\rm K}}{%
  \ifthenelse{\equal{#1}{refined}}{C_{\rm ref}}{%
  \ifthenelse{\equal{#1}{estconv}}{C_{\rm est}}{%
  \ifthenelse{\equal{#1}{optimal}}{C_{\rm opt}}{%
  \ifthenelse{\equal{#1}{conv}}{C_{\rm conv}}{%
    \ifthenelse{\equal{#1}{ieRconv}}{C_{\rm conv}}{%
    \ifthenelse{\equal{#1}{equivRconv}}{C_{\rm conv}}{%
  \ifthenelse{\equal{#1}{mon}}{C_{\rm mon}}{%
  \ifthenelse{\equal{#1}{cea}}{C_{\mbox{\scriptsize\rm C\'ea}}}{%
  \ifthenelse{\equal{#1}{Bcont}}{C_{\rm cont}}{%
  \ifthenelse{\equal{#1}{qosum}}{{C_{\rm qo}}}{%
\ifthenelse{\equal{#1}{lemopthelp}}{2}{%
  \ifthenelse{\equal{#2}{newcounter}}{\refstepcounter{constantsnumber}\label{const#1}}{}C_{\ref{const#1}}}%
}}}}}}}}}}}}}}}}}}}}}}}}}}}}}}}
\def\setc#1{\namec{#1}{newcounter}}
\def\c#1{\namec{#1}{reference}}
\def\definec#1{\refstepcounter{constantsnumber}\label{const#1}}%

\newcounter{contractionnumber}
\def\nameq#1#2{%
  \ifthenelse{\equal{#1}{reduction}}{\rho_{\rm red}}{%
  \ifthenelse{\equal{#1}{estconv}}{\rho_{\rm est}}{%
  \ifthenelse{\equal{#1}{cea}}{\rho_{\mbox{\scriptsize C\'ea}}}{%
  \ifthenelse{\equal{#1}{conv}}{q_{\rm conv}}{%
 \ifthenelse{\equal{#1}{htc}}{\rho_{\rm h}}{%
  \ifthenelse{\equal{#2}{newcounter}}{\refstepcounter{contractionnumber}\label{contraction#1}}{}q_{\ref{contraction#1}}}%
}}}}}
\def\setq#1{\nameq{#1}{newcounter}}
\def\q#1{\nameq{#1}{reference}}
\def\namer#1#2{%
  \ifthenelse{\equal{#1}{reduction}}{\rho_{\rm red}}{%
  \ifthenelse{\equal{#1}{estconv}}{\rho_{\rm est}}{%
  \ifthenelse{\equal{#1}{cea}}{\rho_{\mbox{\scriptsize C\'ea}}}{%
  \ifthenelse{\equal{#1}{qo}}{\boldsymbol{\color{blue}{\rho_{\rm qo}}}}{%
  \ifthenelse{\equal{#1}{conv}}{\rho_{\rm conv}}{%
    \ifthenelse{\equal{#1}{ieRconv}}{\rho_{\rm conv}}{%
    \ifthenelse{\equal{#1}{equivRconv}}{\rho_{\rm conv}}{%
        \ifthenelse{\equal{#1}{refinereduction}}{\rho_{\rm refine}}{%
  \ifthenelse{\equal{#2}{newcounter}}{\refstepcounter{contractionnumber}\label{contraction#1}}{}\rho_{\ref{contraction#1}}}%
}}}}}}}}
\def\setr#1{\namer{#1}{newcounter}}
\def\r#1{\namer{#1}{reference}}

\def\osc{{\rm osc}}

\def\hot{{\rm hot}}

\newtheorem{theorem}{Theorem}
\newtheorem{proposition}[theorem]{Proposition}
\newtheorem{lemma}[theorem]{Lemma}
\newtheorem{corollary}[theorem]{Corollary}
\newtheorem{algorithm}[theorem]{Algorithm}

\newtheorem{remark}[theorem]{Remark}
\newtheorem{consequence}[theorem]{Consequence}
\numberwithin{theorem}{section}

\def\epsqo{{\eps_{\rm qo}}}

\def\T{\mathbb T}

\numberwithin{equation}{section}
\begin{document}

\title{Axioms of Adaptivity}
\author[b]{C.~Carstensen}
\ead{cc@math.hu-berlin.de}
\author[v]{M.~Feischl\corref{cor1}}
\ead{michael.feischl@tuwien.ac.at}
\author[v]{M.~Page}
\ead{marcus.page@tuwien.ac.at}
\author[v]{D.~Praetorius}
\ead{dirk.praetorius@tuwien.ac.at}

\cortext[cor1]{Corresponding author}

\address[b]{Institut f\"ur Mathematik, Humboldt Universit\"at zu Berlin, Unter den Linden 6, 10099 Berlin, Germany; Department of Computational Science and Engeneering, Yonsei University, 120-749 Seoul, Korea}%
\address[v]{Institute for Analysis and Scientific Computing,
      Vienna University of Technology,
      Wiedner Hauptstra\ss{}e 8-10,
      A-1040 Wien, Austria}


      


\begin{keyword}
  finite element method\sep boundary element method\sep a~posteriori error estimators\sep adaptive algorithm\sep local mesh-refinement \sep convergence\sep optimality\sep iterative solvers 
 \end{keyword}


%
%
%

\begin{abstract}
This paper aims first at a simultaneous axiomatic presentation of the proof of optimal convergence rates for adaptive finite element methods
and second at some refinements of particular questions like 
the avoidance of (discrete) lower bounds,
inexact solvers, 
inhomogeneous boundary data, or the use of equivalent error estimators. Solely four axioms guarantee the optimality in terms of the error estimators. 

Compared to the state of the art in the temporary literature, the improvements of this article can be 
summarized as follows: 
First, a general framework is presented which covers the existing literature 
on optimality of adaptive schemes.
The abstract analysis covers linear as 
well as nonlinear problems and is independent of the underlying finite element or
boundary element method.
Second, efficiency of the error estimator is neither needed to prove
convergence nor quasi-optimal convergence behavior of the error estimator. In this paper, efficiency exclusively
characterizes the approximation classes involved in terms of the best-approximation error and data resolution and so the upper bound on the optimal marking parameters does not depend on the
efficiency constant.
Third, some general quasi-Galerkin orthogonality is not only 
sufficient, but also necessary for the $R$-linear convergence of 
the error estimator, which is a fundamental ingredient in the current
quasi-optimality analysis due to Stevenson 2007.
Finally, the general analysis allows for equivalent error estimators 
and inexact solvers as well as different non-homogeneous and mixed boundary conditions.
\end{abstract}
\maketitle

\section{Introduction \& Outline}

\subsection{State of the art}
%
The impact of adaptive mesh-refinement in computational partial differential 
equations (PDEs) cannot be overestimated. Several books in the area provide 
sufficient evidence of the success in many practical applications in the 
computational sciences and engineering. Related books from the mathematical
literature, e.g., ~\cite{ao00,v96, babs,repin,repinalleinzuhaus} provide many a posteriori error estimators 
which compete in~\cite{CCMer,CBK}, and overview 
articles~\cite{Carstensen:2005:Unifying,CCJH,CCEigHop} outline an abstract framework for their derivation.

This article contributes to the theory of optimality of adaptive algorithms in 
the spirit of~\cite{d1996,mns,bdd,stevenson07,ckns,ks,cn,nonsymm}
for conforming finite element methods (FEMs) and exploits the overall 
mathematics for nonstandard FEMs like nonconforming methods~\cite{ch06b,rabus10,BeMao10,bms09,cpr13,HX07,mzs10}
and mixed formulations~\cite{ch06a,LCMHJX,CR2012,HuangXu} as well as boundary element methods (BEMs)~\cite{fkmp,gantumur,affkp,ffkmp:part1,ffkmp:part2} and possibly non-homogeneous or mixed boundary 
conditions~\cite{mns03,dirichlet2d,dirichlet3d}.

Four main arguments compose the set of axioms and \revision{identify} necessary conditions for optimal convergence of adaptive mesh-refining algorithms.
This abstract framework answers questions like: What is the state-of-the-art technique for the design of an optimal adaptive mesh-refining strategy,
and which ingredients are really necessary to guarantee quasi-optimal rates? The overall format of the adaptive algorithm follows the standard loop
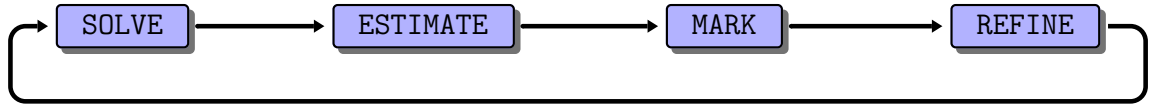
\begin{center}
\newcommand\psBox[4][white]{\rput(#2){\rnode[rc]{#4}{%
  \psframebox[fillcolor=#1]{\tiny{\makebox{\tabular{c}#3\endtabular}}}}}}%
\newcommand\psbox[4][white]{\rput(#2){\rnode{#4}{%
  \psframebox[fillcolor=#1]{\tiny{\makebox{\tabular{c}#3\endtabular}}}}}}%
\newcommand\psBoxT[4][white]{\rput(#2){\rnode[lc]{#4}{%
  \psframebox[fillcolor=#1]{\tiny{\makebox{\tabular{c}#3\endtabular}}}}}}%
\psset{framearc=0.2,shadow=true,xunit=1mm,yunit=1mm,fillstyle=solid,shadowcolor=black!55}%
\begin{pspicture}(-10,0)(150,20)%
\psBoxT[blue!30]{10,15}{\large\texttt{SOLVE}}{a1}%
\psBox[blue!30]{50,15}{\large\texttt{ESTIMATE}}{a2}%
\psBox[blue!30]{90,15}{\large\texttt{MARK}}{a3}%
\psBox[blue!30]{130,15}{\large\texttt{REFINE}}{a4}%
\psset{nodesep=3pt,shadow=false,linewidth=0.061,fillstyle=none,arrowinset=0,arrowlength=0.8}%
\ncloop[angleA=180,loopsize=1,arm=.5,linearc=.2]{<-}{a1}{a4}
\psset{nodesep=3pt,arrows=c->,shortput=nab,shadow=false,linewidth=0.061,arrowinset=0,arrowlength=0.8}%
\ncline{a1}{a2}%
\ncline{a2}{a3}%
 \ncline{a3}{a4}%
\end{pspicture}
\end{center}
in the spirit of the pioneering works~\cite{bv,bm87}. This is the most popular version of 
adaptive FEM and BEM in practice. 
While earlier works~\cite{msv,siebert,nc2010} which faced an abstract framework for adaptivity were only concerned with convergence of adaptive conforming FEM, the present article provides a problem and discretization independent framework for convergence and quasi-optimal rates of adaptive algorithms. In particular, this includes adaptive FEM and BEM with conforming, nonconforming,
as well as mixed methods.

\subsection{Contributions of this work}
The contributions
in this paper have the flavour of a survey and a general 
description in the first half comprising Sections~\ref{sec:setting}--\ref{section:examples2}, although the strategy is different from the  
main stream of, e.g.,~\cite{stevenson07,ckns,ks,cn} and the overview articles like~\cite{primer,afem}: The efficiency is not 
used and data approximation terms do no enter in the beginning. 
Instead, the optimality is firstly  proved in terms of the a~posteriori error
estimators. This approach of~\cite{dirichlet3d,nonsymm} appears natural as the algorithm only concerns the estimator rather than 
the unknown error. Efficiency solely enters in a second step, where this first 
notion of optimality is shown to be equivalent to optimality in terms of nonlinear approximation 
classes which include best approximation error plus data approximation 
terms~\cite{ckns}. In our opinion, this strategy enjoys the following advantages (a)--(b): 

(a) Unlike~\cite{stevenson07,ckns,ks,cn}, the upper bound for adaptivity 
parameters which guarantee quasi-optimal convergence rates, is independent 
of the efficiency constant. Such an observation might be a first step to the 
mathematical understanding of the empirical observation that each adaptivity parameter $0<\theta<1$ yields optimal convergence
rates in the asymptotic regime. 

(b) Besides boundary element methods, see e.g.~\cite{cs96,cms,fkmp}, there might be other (nonlinear) problems, where an optimal efficiency 
estimate is unknown or cannot even be expected.
Then, our approach guarantees at least
that the adaptive strategy will lead to the best possible convergence 
behaviour
with respect to the 
computationally available a~posteriori error estimator.

The first half of this paper discusses a small set of rather general 
axioms~\eqref{A:stable}--\eqref{A:dlr} and therefore involves 
several simplifying restrictions such as an exact solver. 
Although the axioms are motivated from the literature on adaptive FEM for linear problems and \revision{constitute the} main ingredients for any optimality proof in literature so far, we are able to show that this minimal set of four axioms is sufficient and, in some sense, even necessary to \revision{prove} optimality.
Unlike the overview articles~\cite{primer,afem}, the analysis is not bound to a particular model problem, but applies to any problem within the framework of Section~\ref{sec:setting} and therefore sheds new light onto the theory of adaptive algorithms.
In Section~\ref{section:examples1}, 
these axioms are met for different formulations of the Poisson
model problem and allow to reproduce and even improve the state-of-the-art
results from the literature for conforming AFEM~\cite{stevenson07,ckns},
nonconforming AFEM~\cite{ch06a,rabus10, bms09, mzs10}, mixed AFEM~\cite{ch06b,LCMHJX,HuangXu},
and ABEM for weakly-singular~\cite{fkmp,ffkmp:part1,gantumur} and 
hyper-singular integral equations~\cite{ffkmp:part2,gantumur}. 
Moreover,
further examples from Section~\ref{section:examples2} show that our frame
also covers conforming AFEM for non-symmetric problems~\cite{cn,nonsymm,mn2005},
linear elasticity~\cite{CR2012,braess,SR} and different formulations of the
Stokes problem~\cite{braess,SR,brezzi-fortin,Gi_Ra,CR73,CKP11}. We thus provide a general framework of 
four axioms that unifies the diversity of the quasi-optimality analysis from the literature. Given any adaptive scheme that fits into the above frame, the validity of those four axioms \revision{guarantee} optimal convergence behaviour independently of the concrete setup.

To illustrate the extensions and applicability of our axioms of 
adaptivity~\eqref{A:stable}--\eqref{A:dlr},
the second half of this paper treats further advanced topics and contributes
with new mathematical insight in the striking performance of adaptive schemes.

First, Section~\ref{section:inexact} generalizes~\cite{BeMao10} and analyzes the influence of inexact
  solvers, which are important for iterative solvers, especially for nonlinear 
problems. This also gives a mathematically satisfactory explanation of
the stability of adaptive schemes against computational noise as e.g.\
rounding errors in computer arithmetics.

Second, the historic development of adaptive algorithms focused on 
residual-based a~posteriori error estimators, but all kinds of locally 
equivalent a~posteriori error estimators can be exploited as refinement
indicators as well. Section~\ref{section:locequiv} provides the means to show optimal 
convergence behaviour even in this case and
extends~\cite{ks} which is restricted to a patch-wise marking strategy with unnecessary refinements. The refined analysis \revision{in} this paper is essentially based on a novel equivalent mesh-size
function. It provides a mathematical background for the standard AFEM 
algorithm with facet-based and/or non-residual error estimators. To illustrate the analysis
from Section~\ref{section:locequiv}, Section~\ref{section:examples3} provides
several examples with facet-based formulations of the residual estimators as well as non-residual 
error estimators like the ZZ-estimator in the frame of the Poisson model problem.

Third, only few is known about optimal convergence behaviour of adaptive
FEM in the frame of nonlinear problems. \revision{To the authors' best} knowledge, the
following works provide all results available and analyze adaptive lowest-order 
Courant finite elements for three
particular situations: The work~\cite{bdk} considers the $p$-Laplacian,
while~\cite{nonsymm,gmz12}
consider model problems in the frame of strongly-monotone operators. In Section~\ref{section:nonlinear}, the abstract
framework of Section~\ref{section:optimality} and Section~\ref{section:locequiv}
is used to reproduce these results. As for the linear problems considered and unlike~\cite{bdk}, efficiency is only used to characterize the approximation class, but avoided for the quasi-optimality analysis.

Finally, the development of adaptive algorithms focused on
homogeneous Dirichlet problems. Section~\ref{section:boundary} considers 
inhomogeneous boundary conditions of mixed Dirichlet-Neumann-Robin type.
In particular, the issue of inhomogeneous Dirichlet data, at a first glance regarded 
as a minor technical detail, introduces severe technical difficulties arising from 
the additional approximation of the non-homogeneous Dirichlet data in
the fractional-order trace space $H^{1/2}$. While a first convergence result
for 2D AFEM is already found in~\cite{mns03}, quasi-optimal convergence rates
have been derived only recently in~\cite{dirichlet2d} for lowest-order elements
in 2D and more general in~\cite{dirichlet3d}. The last work, however, 
proposes an artificial two-step D\"orfler marking, while the present 
refined analysis now provides optimal convergence behavior even in case of 
the common adaptive loop and standard D\"orfler marking. We refer to Section~\ref{section:boundary} for details.

\subsection{Brief discussion of axioms}
The abstract framework is independent of the precise application and its respective discretization. Let $\XX$ be a vector space,
where $u\in\XX$ denotes the target to be approximated. 
This general assumption includes the cases where $u$ is some (possibly non-unique) solution of a variational equality or inequality.
For any shape-regular triangulation $\TT$ from some mesh-refining algorithm,
let $\XX(\TT)$ be a discrete space, which may be nonconforming 
in the sense that $\XX(\TT)$ is not necessarily a subspace of $\XX$. Let $U(\TT)\in\XX(\TT)$ denote some discrete approximation returned by the
numerical solver at hand. 
Finally, assume that $\XX\cup\XX(\TT)$ is
equipped with some quasi-metric $\dist[\TT]\cdot\cdot$. In most applications this will either be a norm or a quasi-norm in some suitable Banach space. Notice that uniqueness of continuous and discrete 
solution $u$ resp.\ $U(\TT)$ is \emph{not}
explicitly assumed or required.

In this rather general setting, the local contributions
\begin{align*}
 \eta_T(\TT;\cdot)\,:\,\XX(\TT )\to[0,\infty)
 \quad\text{for all } T\in\TT
\end{align*}
of an a~posteriori error estimator
\begin{align*}
 \eta(\TT;V) = \Big(\sum_{T\in\TT}\eta_T(\TT;V)^2\Big)^{1/2}
 \quad\text{for all }V\in\XX(\TT)
\end{align*}
serve as refinement indicators in the module \texttt{MARK} of the 
adaptive scheme. 
To single out the elements $T$ for refinement of the adaptively generated meshes $\TT_\ell$, the D\"orfler
marking strategy~\cite{d1996} determines a set 
$\MM_\ell\subseteq\TT_\ell$ of minimal cardinality such that
\begin{align}
 \theta\,\eta(\TT_\ell;U(\TT_\ell))^2
 \le \sum_{T\in\MM_\ell}\eta_T(\TT_\ell;U(\TT_\ell))^2
\end{align}
for some fixed bulk parameter $0<\theta<1$.
The following four axioms are sufficient for optimal convergence.
They are formally defined in Section~\ref{section:axioms} below and outlined here for a convenient reading and overview. 

The first axiom~\eqref{A:stable} asserts {\em stability on non-refined elements} 
in the sense that 
\begin{align}\tag{\ref{A:stable}}
 \Big|\Big(\sum_{T\in\mathcal{S}}\eta_T(\widehat\TT;\widehat V)^2\Big)^{1/2}
 -\Big(\sum_{T\in\mathcal{S}}\eta_T(\TT;V)^2\Big)^{1/2}\Big|
 \leq \c{stable}\dist[\widehat\TT]{\widehat V}{V}
\end{align}
holds for any subset $\mathcal{S}\subseteq\TT\cap\widehat\TT$ of non-refined element domains, for all admissible refinements $\widehat\TT$ of a 
triangulation $\TT$, and for all \revision{corresponding} discrete functions 
$V\in\XX(\TT)$ and  $\widehat V\in\XX(\widehat\TT)$. In practice, this
axiom is easily verified by the triangle inequality and appropriate
inverse estimates.

The second axiom~\eqref{A:reduction} asserts a
{\em reduction property on refined elements} in the sense that 
\begin{align}\tag{\ref{A:reduction}}
 \sum_{T\in\widehat\TT\setminus\TT}\eta_T(\widehat\TT;U(\widehat\TT))^2
 \leq \q{reduction} \sum_{T\in\TT\setminus\widehat\TT}\eta_T(\TT;U(\TT))^2 
 + \c{reduction}\dist[\widehat\TT]{U(\widehat\TT)}{U(\TT)}^2
\end{align}
holds for any admissible refinement $\widehat\TT$ of a triangulation 
$\TT$ and their \revision{corresponding} discrete approximations $ U(\widehat\TT)$ and $U(\TT)$.
Such an estimate is the contribution of~\cite{ckns} and follows 
from the observation that the contributions of the error estimators
are weighted by the local mesh-size \revision{which uniformly decreases} on each
refined element. Together with the triangle inequality, an appropriate 
inverse estimate then proves~\eqref{A:reduction}. 

The third axiom~\eqref{A:qosum} asserts an appropriate 
{\em quasi-orthogonality} which generalizes the Pythagoras theorem
\begin{align}\tag{{A3$^\star$}}\label{dpr:a3}
 \dist{u}{U(\TT_{\ell+1})}^2
 + \dist{U(\TT_{\ell+1})}{U(\TT_\ell)}^2
 =  \dist{u}{U(\TT_{\ell})}^2
\end{align}
met for conforming methods in a Hilbert space setting, 
where $\dist{u}{v} = \norm{u-v}\XX$ stems from the Hilbert space norm
and $\dist\cdot\cdot = \dist[\TT_{\ell}]\cdot\cdot = \dist[\TT_{\ell+1}]\cdot\cdot$. 
The Pythagoras theorem~\eqref{dpr:a3} implies the
quasi-orthogonality axiom~\eqref{A:qosum}.
Our
formulation generalizes the quasi-orthogonalities 
found in the literature~\cite{ch06a,ch06b,cn,mn2005},
see Section~\ref{section:quasiorth}. Moreover, Proposition~\ref{prop:Rconv} 
below shows that~\eqref{A:qosum} is essentially equivalent to linear 
convergence of the adaptive algorithm. 
In particular, we shall see below that our quasi-orthogonality axiom~\eqref{A:qosum} cannot be weakened further if one aims to follow the state-of-the-art proofs of quasi-optimal convergence rates which go back to~\cite{stevenson07,ckns}.

A common property of error estimators is reliability~\eqref{intro:reliable}
\begin{align}\label{intro:reliable}
 \dist[\TT]{u}{U(\TT)}
 \leq \c{reliable} \eta(\TT;U(\TT))
\end{align}
for all triangulations $\TT$ and the corresponding discrete 
solution \revision{$U(\TT_\ell)$.} As stated below, reliability is implied by the fourth axiom~\eqref{A:dlr} and is therefore not an axiom itself.

With those three axioms~\eqref{A:stable}--\eqref{A:qosum} and reliability~\eqref{intro:reliable}, the adaptive algorithm leads to linear convergence in the sense of
\begin{align}\label{intro:convergence}
 \eta(\TT_{\ell+k};U(\TT_{\ell+k}))^2
 \le \c{conv}\,\r{conv}^k\,\eta(\TT_\ell;U(\TT_\ell))^2
 \quad\text{for all }k,\ell\in\N_0:=\{0,1,2,\ldots\}
\end{align}
and some constants $0<\r{conv}<1$ and $\c{conv}>0$, 
cf.\ Theorem~\ref{thm:main}~(i) below. Some short remarks are in order
to stress the main differences to the nowadays main stream 
literature. Unlike~\cite{ckns,ks}, we do \emph{not} consider the
\emph{quasi-error} which is the weighted sum of error plus error estimator. Unlike~\cite{mn2005,stevenson07,cn}, we do \emph{not} consider the
\emph{total error} which is the weighted sum of error plus oscillations. The analysis of this paper avoids the use of any lower bound of the error,
while~\cite{mn2005,stevenson07,cn} build on some (even stronger) discrete 
local lower bound.
Instead, we generalize and extend the approach of the recent
 work~\cite{nonsymm} and only rely on the error estimator and the reliability 
estimate~\eqref{intro:reliable}.

The final axiom~\eqref{A:dlr} asserts \emph{discrete reliability} of
the error estimator: For any admissible refinement $\widehat\TT$ of a 
triangulation $\TT$ and their respective discrete approximation $U(\widehat\TT)$
and $U(\TT)$, we assume that
\begin{align}\tag{\ref{A:dlr}}
 \dist[\widehat\TT]{U(\widehat\TT)}{U(\TT)}
 \le\c{dlr}\,\Big(\sum_{T\in\RR(\TT,\widehat\TT)}\eta_T(\TT;U(\TT))^2\Big)^{1/2},
\end{align}
where $\RR(\TT,\widehat\TT) \subseteq \TT$ is a ``small'' superset
of the set of refined elements, i.e.\ $\TT\backslash\widehat\TT\subseteq\RR(\TT,\widehat\TT)$
and $\RR(\TT,\widehat\TT)$ contains up to a fixed multiplicative constant the same
number of elements as $\TT\backslash\widehat\TT$. Such a property has first been
shown in~\cite{stevenson07}, where $\RR(\TT,\widehat\TT)$ denotes $\TT\backslash\widehat\TT$
plus one additional layer of elements. By means of this property, it is shown that
the D\"orfler marking strategy used to single out the elements for refinement,
is not only sufficient for linear convergence~\eqref{intro:convergence}, but in
some sense even necessary. We refer to Proposition~\ref{prop:doerfler} below for
a precise statement of this ``equivalence''. Lemma~\ref{lemma:dlr2rel} shows that discrete reliability~\eqref{A:dlr} implies reliability~\eqref{intro:reliable}.

With the axioms~\eqref{A:stable}--\eqref{A:dlr}, we prove in Theorem~\ref{thm:main}~(ii) that the
adaptive algorithm leads to the best possible algebraic convergence order for the
error estimator in the sense that
\begin{align*}
\norm{(\eta(\cdot),U(\cdot))}{\B_s}
&:=\sup_{N\in\N_0}\inf_{|\TT|-|\TT_0|\leq N}\eta(\TT;U(\TT))\,(N+1)^s
\\&
\simeq \sup_{\ell\in\N_0}\eta(\TT_\ell;U(\TT_\ell))
\,(|\TT_\ell|-|\TT_0|+1)^{s}
\end{align*}%
for all $s>0$. The use of $N+1$ instead of $N$ above is just a minor detail which avoids division by zero.
By definition, $\norm{(\eta(\cdot),U(\cdot))}{\B_s}<\infty$ means that an algebraic convergence
$\eta(\TT;U(\TT))=\OO(N^{-s})$ is theoretically possible if the optimal meshes
$\TT$ with $N$ elements are chosen.
In explicit terms, this means that the adaptive algorithm will asymptotically regain
the best convergence rate and hence quasi-optimal meshes
with respect to the error estimator $\eta(\cdot)$.

To relate quasi-optimal estimator convergence with convergence rates of the error, we consider efficiency of the error estimator in the sense that
\begin{align}\label{intro:efficient}
 \c{efficient}^{-1}\,\eta(\TT;U(\TT)) 
 \le \dist[\TT]{u}{U(\TT)} + \eff{\TT}{U(\TT)}
\end{align}
for all triangulations $\TT$ and the corresponding discrete 
solution \revision{$U(\TT)$.} Here, $\eff{\TT}{U(\TT)}$ denotes certain data oscillation terms which 
are ---for simple model problems--- of higher order. By use 
of~\eqref{intro:efficient}, the approximability  $\norm{\cdot}{\B_s}$ can equivalently be
formulated in the form of nonlinear approximation classes found e.g.\ in~\cite{ckns,cn,ks}. Details are given in Section~\ref{section:approxclass}. Moreover, 
if $\eff{\TT}{U(\TT)}$ satisfies 
\begin{align*}
\norm{\osc(\cdot)}{\O_s}:=\sup_{N\in\N}\inf_{|\TT|-|\TT_0|\leq N}\eff{\TT}{U(\TT)}(N+1)^s<\infty
\end{align*}
and the error is quasi-monotone, the approximability $\norm{\cdot}{\B_s}$ can be related to
\begin{align*}
 \norm{(u,U(\cdot))}{\A_s}:=\sup_{N\in\N}\inf_{|\TT|-|\TT_0|\leq N}\dist[\TT]{u}{U(\cdot)}(N+1)^s,
\end{align*}
which characterizes the discretization error only. Theorem~\ref{thm:mainerr} then states 
\begin{align}
\norm{(u,U(\cdot))}{\A_s} +\norm{\osc(\cdot)}{\O_s}\simeq \sup_{\ell\in\N_0}\frac{\dist[\TT_\ell]{u}{U(\TT_\ell)}}{(|\TT_\ell|-|\TT_0|+1)^{-s}}+\norm{\osc(\cdot)}{\O_s}.
\end{align}%
and proves that the adaptive algorithm will asymptotically
recover the best optimal convergence rate and hence quasi-optimal meshes
with respect to the discretization error.
In particular, the adaptive scheme then performs as good 
as or 
even better than any other adaptive mesh-refining scheme based on the same
mesh-refinement.

\subsection{Outline}
The remaining parts of this paper are organized as follows. 

{\Large\gray$\bullet$} Section~\ref{sec:setting} 
(page~\pageref{sec:setting}ff.)
introduces the abstract setting and specifies the assumptions posed 
on the continuous space $\XX$ and on the discrete space
$\XX(\TT)$. Moreover, the adaptive algorithm is formally stated, and 
admissible mesh-refinement strategies are discussed. 

{\Large\gray$\bullet$} 
Section~\ref{section:axioms} 
(page~\pageref{section:axioms}ff.) 
starts with the precise statement of the 
four axioms~\eqref{A:stable}--\eqref{A:dlr} required and analyzes
relations between those.
\revision{The} short historical overview \revision{emphasises} where the respective axioms have appeared
first in the literature. 

{\Large\gray$\bullet$} 
Section~\ref{section:optimality} (page~\pageref{section:optimality}ff.) 
 states and proves the main theorem on 
convergence and quasi-optimal rates of the adaptive algorithm in the abstract 
framework.

{\Large\gray$\bullet$} 
Section~\ref{section:examples1}
(page~\pageref{section:examples1}ff.)
exemplifies
the abstract theory for 
different discretizations of the Laplace model problem. We consider 
conforming FEM (Section~\ref{section:ex:poisson}), nonconforming FEM
(Section~\ref{section:ex:nonconforming}), and mixed FEM 
(Section~\ref{section:ex:mixed}), as well as conforming BEM for 
weakly-singular integral equations (Section~\ref{section:ex:symm})
and hyper-singular integral equations (Section~\ref{section:ex:hypsing}).

{\Large\gray$\bullet$} 
Section~\ref{section:examples2}
(page~\pageref{section:examples2}ff.)
considers further examples 
from the frame of second-order elliptic PDEs. 
Besides conforming FEM for non-symmetric PDEs 
(Section~\ref{section:ex:nonsymm}), we consider nonconforming and mixed FEM
for the Stokes system (Section~\ref{example:stokes} and 
Section~\ref{example:mixedstokes})
as well as mixed FEM for the Lam\'e system \revision{for} linear elasticity (Section~\ref{example:lame}).

{\Large\gray$\bullet$} 
Section~\ref{section:inexact} (page~\pageref{section:inexact}ff.) extends
 the abstract framework to include
inexact solvers into the analysis. 

{\Large\gray$\bullet$} 
Section~\ref{section:locequiv} (page~\pageref{section:locequiv}ff.) 
further extends the analysis to cover 
a~posteriori error estimators which are not weighted by the local
mesh-size $h$, but 
are locally equivalent to an error estimator which 
satisfies~\eqref{A:stable}--\eqref{A:dlr}. A prominent example
of this estimator class are recovery-based error estimators (Section~\ref{section:est:zz}) for FEM
which are occasionally also called ZZ-estimators after Zienkiewicz and Zhu~\cite{zz}.
For these estimators, the reduction property~\eqref{A:reduction} can hardly
be proved. Still, one can prove convergence even with quasi-optimal convergence
rates. The technical heart of the matter is a novel mesh-width function which
is pointwise equivalent to the usual local mesh-width, but contractive on the entire
patch of a refined element (Proposition~\ref{prop:htilde}).

{\Large\gray$\bullet$} 
Section~\ref{section:examples3} (page~\pageref{section:examples3}ff.)
provides several examples for locally equivalent FEM error estimators 
for the Poisson model problem. This includes facet-based formulations of the residual error estimator (Section~\ref{section:dp:facet-based}) as well as recovery-based error estimators (Section~\ref{section:est:zz}).

{\Large\gray$\bullet$} 
Section~\ref{section:nonlinear} (page~\pageref{section:nonlinear}ff) applies the abstract analysis to nonlinear FEM model problems. We consider adaptive FEM for strongly monotone operators (Section~\ref{section:nonlinear}), the $p$-Laplace problem (Section~\ref{section:ex:plaplace}), and an elliptic eigenvalue problem (Section~\ref{section:eigenvalue}).

{\Large\gray$\bullet$} 
The final Section~\ref{section:boundary}
(page~\pageref{section:boundary}ff.)
 aims to analyze non-homogeneous
boundary conditions in adaptive FEM computations. As model problem serves
the Laplace equation with mixed Dirichlet-Neumann-Robin boundary conditions. 
\revision{Emphasis is on} inhomogeneous Dirichlet conditions, where an additional 
discretization is required, since discrete functions cannot satisfy continuous
Dirichlet conditions. Our analysis generalizes and improves the recent 
works~\cite{dirichlet2d,dirichlet3d}.

\subsection{Notation}
Some practical guide to the notation concludes this introduction. Lower case 
letters denote quantities on the continuous level like the solution $u$, 
while upper case letters denote their discrete 
counterparts usually labelled with respect to the triangulation at hand like 
the discrete approximation $U(\TT)$.

The symbol $|\cdot|$ has multiple meanings which, however, cannot lead to
ambiguities. For vectors and scalars, $|x|$ denotes the Euclidean length.
For finite sets $M$, $|M|$ denotes the number of elements. 
Finally, for subsets and elements $T\subset\R^d$, $|T|$ denotes either 
the $d$-dimensional Lebesgue measure or the $(d-1)$-dimensional 
surface measure. This will be clear from the context.

Throughout all statements, all constants as well as their dependencies are 
explicitly given. In proofs, we may abbreviate notation by use of the
symbol $\lesssim$ which indicates $\le$ up to some multiplicative constant
which is clear from the context. Moreover, the symbol $\simeq$ states that
both estimates $\lesssim$ as well as $\gtrsim$ hold.

\begin{table}[t]
\begin{tabular}{p{3cm}|p{4cm}||p{3cm}|p{4cm}}
\textbf{name} & \textbf{first appearance}&\textbf{name} & \textbf{first appearance}\\
\hline\hline
$ \c{triangle} $ & Section~\ref{section:setting:continuous} & $ \c{patchbound} $ & Equation~\eqref{eq:patchbound}\\
$ C_{\rm{}min} $ & Remark~\ref{rem:adaptivealgorithm} & $ C_{\rm{}ie} $ & Theorem~\ref{thm:mainerr2}\\
$ \c{mesh},C_{\rm{}son},\gamma $ & Section~\ref{section:refinement} & $ \widetilde{C}_{\rm{}rel} $ & Equation~\eqref{eq:iereliable}\\
$ \c{stable},\c{reduction} $ & Section~\ref{subsection:axioms} & $ \c{lbb} $ & Equation~\eqref{eq:inexact:lbb}\\
$ \c{hotConst},\c{dlr} $ & Section~\ref{subsection:axioms} & $ \c{dp1},\c{dp2} $ & Section~\ref{section:index}\\
$ \c{refined},\c{qosum} $ & Section~\ref{subsection:axioms} & $ \c{dp3} $ & Equation~\eqref{eq:glob}\\
$ \c{mon} $ & Equation~\eqref{eq:mon} & $ \c{dp4} $ & Equation~\eqref{dp:doerfler}\\
$ \c{cea} $ & Equation~\eqref{eq:cea} & $ \c{loc} $ & Equation~\eqref{eq:loc}\\
$ \c{qo},\c{muell} $ & Section~\ref{section:quasiorth} & $ \c{stab2} $ & Equation~\eqref{A:kernelstab}\\
$ \c{optimal},\c{conv} $ & Theorem~\ref{thm:main},~\ref{thm:main2},~\ref{thm:main3} & $ \c{htildeequiv} $ & Proposition~\eqref{prop:htilde}\\
$ C_{\rm{}emon} $ & Theorem~\ref{thm:mainerr} & $ \c{meshbound} $ & Equation~\eqref{eq:meshbound}\\
$ C_{\rm{}apx} $ & Proposition~\ref{prop:charAprox} & $ \c{errestequiv} $ & Lemma~\ref{lem:errestequiv}\\
$ \c{estconv} $ & Lemma~\ref{lem:estconv} & $ C_{\rm{}Taylor} $ & Equation~\eqref{eq:taylordef}\\
$ \c{estsum},\c{invsum},\c{Rconv} $ & Lemma~\ref{lem:Rconv} & $ C_{\rm{}pyth} $ & Equation~\eqref{eq:dirichlet3d}\\
$ \c{lemopthelp2} $ & Lemma~\ref{lem:optimality} &  & \\
\hline
\end{tabular}
\medskip
\caption{Important constants $C>0$ and their first appearance in the manuscript.}
\label{table:constants:C}
\end{table}

\begin{table}[t]
\begin{tabular}{p{3cm}|p{4cm}}
\textbf{name} & \textbf{first appearance}\\
\hline\hline
$ \q{reduction},\epsqo $ & Section~\ref{subsection:axioms} \\
$ \q{estconv} $ & Lemma~\ref{lem:estconv} \\
$ \r{Rconv} $ & Lemma~\ref{lem:Rconv} \\
$ \r{conv} $ & Theorem~\ref{thm:main},~\ref{thm:main2},~\ref{thm:main3} \\
$ \q{htc} $ & Proposition~\ref{prop:htilde} \\
\hline
\end{tabular}
\medskip
\caption{Important contraction constants $0<\rho<1$ and their first appearance in the manuscript.}
\label{table:constants:rho}
\end{table}

Finally, the symbols $C>0$ and $\gamma>0$ denote positive constants, while
$0<\rho<1$ denote contraction constants. To improve readability, the most important constants as well as their respective first appearances are collected 
in Table~\ref{table:constants:C} and Table~\ref{table:constants:rho}.

\section{Abstract Setting}\label{sec:setting}
\noindent
This section is devoted to the definition of the problem
and the precise statement of the adaptive algorithm.

\subsection{Adaptive approximation problem}\label{section:setting:continuous}
Suppose that $\XX$ is a vector space. Based on some initial triangulation $\TT_0$, let $\T$ denote the set of all admissible 
refinements of $\TT_0$ as described  in Section~\ref{section:refinement} below. Each $\TT\in\T$ induces a finite dimensional space $\XX(\TT)$.
Suppose the existence of a numerical solver 
\begin{align}
 U(\cdot):\T\to\XX(\cdot)
\end{align}
which provides 
some discrete approximation $U(\TT)\in\XX(\TT)$ of some (unknown) limit
\begin{align}
 u\in\XX.
\end{align}
For adaptive error estimation, each element domain $T\in\TT$ admits a computable \emph{refinement indicator}
\begin{align*}
\eta_T(\TT;\cdot)\,:\,\XX(\TT)\to[0,\infty)
\end{align*}
which specifies the global \emph{error estimator}
\begin{align}\label{def:estimator}
\eta(\TT;V)^2:=\sum_{T\in\TT} \eta_T(\TT;V)^2\quad\text{for all }V\in \XX(\TT).
\end{align}
\subsection{Error measure and further approximation property}
We assume that $\XX\cup\XX(\TT)$ is equipped with some
error measure $\dist[\TT]\cdot\cdot$ which satisfies
the following properties for all $v,w,y\in\XX\cup\XX(\TT)$ and some universal constant $\setc{triangle}>0$, namely
\begin{itemize}
\item ({\bf non-negativity}) $\dist[\TT]{v}{w}\ge0$;
\item ({\bf quasi-symmetry}) $\dist[\TT]{v}{w} \le \c{triangle}\,\dist[\TT]{w}{v}$;
\item ({\bf quasi-triangle inequality}) $\c{triangle}^{-1}\,\dist[\TT]{v}{y}\le \dist[\TT]{v}{w}+\dist[\TT]{w}{y}$.
\end{itemize}
Suppose the following \emph{compatibility condition}: 
For any refinement $\widehat\TT$ of $\TT$, $\dist[\widehat\TT]\cdot\cdot$ is even
well-defined on $\XX\cup\XX(\TT)\cup\XX(\widehat\TT)$ with
$\dist[\widehat\TT]{v}{V} = \dist[\TT]{v}{V}$ for all $v\in\XX$ and $V\in\XX(\TT)$.
Suppose that each mesh $\TT\in\T$ allows for  the {\em further approximation property} of $u\in\XX$ in the sense that
for all $\eps>0$, there exists a refinement $\widehat\TT\in\T$ of $\TT$ such that
\begin{align}\label{eq:unifconv}
\dist[\widehat\TT]{u}{U(\widehat\TT)} \leq \eps.
\end{align}

\begin{remark} In many applications,~\eqref{eq:unifconv} holds for a sufficiently fine uniform refinement $\widehat\TT$ of $\TT$ and follows from
a~priori estimates for smooth functions and density arguments.
\end{remark}



\subsection{Adaptive algorithm}\label{section:algorithm}
Under the assumptions of Section~\ref{section:setting:continuous}, the general adaptive algorithm reads as follows.
\begin{algorithm}\label{algorithm}
\textsc{Input:} Initial triangulation $\TT_0$ and bulk parameter $0<\theta\leq 1$.\\
\textbf{Loop: }For $\ell=0,1,2,\ldots$ do ${\rm (i)}-{\rm(iv)}$.
\begin{itemize}
\item[\rm(i)] Compute discrete approximation $U(\TT_\ell)$.
\item[\rm(ii)] Compute refinement indicators $\eta_T(\TT_\ell;U(\TT_\ell))$ for all $T\in\TT_\ell$.
\item[\rm(iii)] Determine set $\MM_\ell\subseteq\TT_\ell$ of (almost) minimal cardinality such that
\begin{align}\label{eq:doerfler}
 \theta\,\eta(\TT_\ell;U(\TT_\ell))^2 \le \sum_{T\in\MM_\ell}\eta_T(\TT_\ell;U(\TT_\ell))^2.
\end{align}
\item[\rm(iv)] Refine (at least) the marked elements $T\in\MM_\ell$ to generate triangulation $\TT_{\ell+1}$.
\end{itemize}
\textsc{Output:} Discrete approximations $U(\TT_\ell)$ and error estimators
$\eta(\TT_\ell;U(\TT_\ell))$ for all $\ell\in\N_0$.
\end{algorithm}

\begin{remark}\label{rem:adaptivealgorithm}
Suppose that $\SS_\ell\subseteq\TT_\ell$ is some (not necessarily unique) set
of minimal cardinality which satisfies the D\"orfler marking criterion~\eqref{eq:doerfler}.
In step (iii) the phrase {\em almost minimal cardinality} means that 
$|\MM_\ell| \le C_{\rm min} \, |\SS_\ell|$ with some $\ell$-independent constant $C_{\rm min}\ge1$.
\end{remark}

\begin{remark}
A greedy algorithm for~\eqref{eq:doerfler}, sorts the elements $\TT_\ell=\{T_1,\ldots,T_N\}$ such that
$\eta_{T_1}(\TT_\ell;U(\TT_\ell))\geq \eta_{T_2}(\TT_\ell;U(\TT_\ell))\geq \ldots\geq \eta_{T_N}(\TT_\ell;U(\TT_\ell))$ and takes the minimal $1\leq J\leq N$ such that $\theta\eta_\ell(\TT_\ell;U(\TT_\ell))^2\leq \sum_{ j=1}^J\eta_{T_j}(\TT_\ell;U(\TT_\ell))^2$. This results in logarithmic-linear growth of the complexity. The relaxation to almost minimal cardinality of $\MM_\ell$ allows to employ a sorting algorithm based on binning so that $\MM_\ell$ in~\eqref{eq:doerfler} can be determined in linear complexity~\cite[Section~5]{stevenson07}.
\end{remark}

\begin{remark}
Small adaptivity parameters $0<\theta\ll1$ lead to only few marked elements and so to possibly very local mesh-refinements. The other extreme, $\theta=1$ basically leads to uniform refinement, where (almost) all elements are refined.
\end{remark}

\subsection{Mesh-refinement}\label{section:refinement}
For adaptive mesh-refinement, any strategy may be used if it fulfils the properties~\eqref{refinement:sons}--\eqref{refinement:overlay} specified below. 
From now on, we use an arbitrary, but fixed mesh-refinement strategy. Possible examples are found in Section~\ref{section:refinement:examples}.
Given an initial triangulation $\TT_0$, the set of admissible triangulations reads
\begin{align}\label{eq:triangulations}
 \T := \set{\TT}{\TT\text{ is an admissible refinement of }\TT_0}.
\end{align}
Moreover, the subset of all admissible triangulations in $\T$ which have at most $N\in\N$ elements more than the initial mesh $\TT_0$ reads
\begin{align*}
 \T(N) := \set{\TT\in\T}{|\TT|-|\TT_0|\le N},
\end{align*}
where $|\cdot|=\text{card}(\cdot)$ is the counting measure.
Each refined element $T\in\TT$ is split into at least two and at most into $C_{\rm son}\geq 2$ sons. This implies the estimate
\begin{align}\label{refinement:sons}
|\TT\setminus\widehat\TT |\leq |\widehat\TT|-|\TT|
\end{align}
for all refinements $\widehat\TT\in\T$ of $\TT\in\T$ and for one-level refinements $\TT_{\ell+1}$ of $\TT_\ell$ 
\begin{align}\label{refinement:sons2}
 |\TT_{\ell+1}|-|\TT_\ell|\leq (C_{\rm son}-1)|\TT_\ell| 
\end{align}
%
The refinement strategy allows
for the closure estimate for triangulations generated by Algorithm~\ref{algorithm} in the sense that
\begin{align}\label{refinement:closure}
 |\TT_\ell| - |\TT_0|
 \le\c{mesh}\,\sum_{k=0}^{\ell-1}|\MM_{k}|\quad\text{for all }\ell\in\N
\end{align}
with some constant $\setc{mesh}>0$ which depends only on $\T$.
Finally, assume that for any two meshes $\TT,\TT'\in\T$ there is a coarsest common refinement
$\TT\oplus\TT'\in\T$ which satisfies
\begin{align}\label{refinement:overlay}
 |\TT\oplus\TT'|\le |\TT| + |\TT'| - |\TT_0|.
\end{align}

\begin{remark}
 The linear convergence~\eqref{eq:thm:main:conv1} of $\eta(\cdot)$ which is stated in Theorem~\ref{thm:main}~(i), is independent of~\eqref{refinement:sons2}--\eqref{refinement:overlay}. The optimal convergence rate of $\eta(\cdot)$ from~\eqref{eq:thm:main:opt1} which is stated in Theorem~\ref{thm:main}~(ii) requires the validity of~\eqref{refinement:sons} and~\eqref{refinement:closure}--\eqref{refinement:overlay} for the upper bound, while the lower bound relies only on~\eqref{refinement:sons2}.
\end{remark}

\subsection{Examples for admissible mesh-refinement strategies}\label{section:refinement:examples}
This short section, comments on admissible mesh-refinement strategies 
with properties~\eqref{refinement:sons}--\eqref{refinement:overlay}.

\begin{subequations}\label{refinement:shapereg}
For $d=1$, simple bisection satisfies~\eqref{refinement:sons}--\eqref{refinement:overlay}.
Since usual error estimates, however, rely on the boundedness of the $\gamma$-shape regularity in the
sense of
\begin{align}\label{eq:gamma:shaperegular}
 \max\set{|T|/|T^\prime|&}{T,T^\prime\in\TT,\,T\cap T^\prime\neq\emptyset}\leq \gamma,
\end{align}
additional bisections have to be imposed. 
Here, $|T|$ denotes the diameter of $T$.
We refer to~\cite{affkp} for some extended 1D
bisection algorithm with~\eqref{refinement:sons}--\eqref{refinement:overlay} as
well as~\eqref{eq:gamma:shaperegular} for all $\TT\in\T$. There, the 
mesh-refinement guarantees that only finitely many shapes of, e.g., node
patches $\omega(\TT;z):=\bigcup\set{T\in\TT}{z\in T}$ occur. In particular,
the constant $\gamma\ge1$ depends only on the initial mesh $\TT_0$.

Even though the above mesh-refinement strategy seems fairly arbitrary, 
to the best of our knowledge, the newest vertex bisection for 
$d \ge 2$ is the only refinement strategy known to fulfil~\eqref{refinement:closure}--\eqref{refinement:overlay} for regular triangulations.
The proof of~\eqref{refinement:overlay} is found 
in~\cite{stevenson07} for $d=2$ and~\cite{ckns} for $d\ge2$. 
For the proof of~\eqref{refinement:closure}, we refer to~\cite{bdd} for $d=2$ 
and~\cite{stevenson08} for $d\ge2$.
The proof of~\eqref{refinement:sons2} is obvious for newest vertex
bisection in 2D and is valid 
in any dimension (the proof follows  with arguments from~\cite{stevenson08}) as pointed out by R. Stevenson in a private communication.
The works~\cite{bdd,stevenson08} assume an appropriate labelling of the edges of the initial mesh $\TT_0$ to prove~\eqref{refinement:closure}. \revision{This poses a combinatorial problem on the initial mesh $\TT_0$ but does not concern any of the following meshes $\TT_\ell$, $\ell\geq 1$. For $d=2$, it can be proven }
that each conforming triangular 
mesh $\TT$ allows for such a labelling, while no efficient algorithm is known to compute this in linear
complexity. For $d\geq 3$, such a result is missing. However, it is known that an appropriate
uniform refinement of an arbitrary conforming simplicial mesh $\TT$ for $d\ge2$ allows for such a labelling~\cite{stevenson08}.
Moreover, for $d=2$, it has recently been proved in~\cite{kpp} that~\eqref{refinement:closure} even holds 
without any further assumption on the initial mesh $\TT_0$.

If one admits hanging 
nodes, also the red-refinement strategy from~\cite{bn} can be used, 
where the order of hanging nodes is bounded.
Both mesh-refinement strategies, the one
of~\cite{bn} as well as newest vertex bisection, guarantee uniform boundedness of the $\gamma$-shape regularity
in the sense of
\begin{align}\label{refinement:shaperegular}
\begin{split}
 |T|^{1/d}&\leq \diam(T)\leq \gamma\,|T|^{1/d}\quad\text{for all }
 T\in\TT\in\T
\end{split}
\end{align}
with some fixed $\gamma\geq 1$ and the $d$-dimensional Lebesgue measure $|\cdot|$. As above, both mesh-refinement strategies
guarantee that only finitely many shapes of, e.g., node patches
$\omega(\TT;z):=\bigcup\set{T\in\TT}{z\in T}$ occur, and the constant $\gamma\ge1$
thus depends only on the initial mesh $\TT_0$.

Even the simple red-green-blue refinement from \cite{ccrgb} fails to 
satisfy~\eqref{refinement:overlay} as seen from a counterexample in~\cite[Satz 4.15]{pavbakk}.
\end{subequations}


\section{The Axioms}\label{section:axioms}
\noindent%
This section, states a set axioms that are sufficient for quasi-optimal convergence  of  Algorithm~\ref{algorithm} 
from Section~\ref{section:algorithm}. In other words, any numerical algorithm that fits into the general framework of Algorithm~\ref{algorithm} will converge with optimal rate if it satisfies~\eqref{A:stable}--\eqref{A:dlr} below.

\subsection{Set of axioms}\label{subsection:axioms}
The following four axioms for optimal convergence of Algorithm~\ref{algorithm} concern some fixed (unknown) limit $u\in\XX$ and the (computed) discrete approximation $U(\TT)\in\XX(\TT)$ for \revision{any} given mesh $\TT\in\T$.
\revision{The constants in~\eqref{A:stable}--\eqref{A:dlr} satisfy 
$\setc{stable},\setc{reduction}$, $\setc{hotConst}, \setc{dlr},\setc{refined},\setc{qosum}(\epsqo)\geq 1$ as well as $0<\setq{reduction}<1$ and depend solely on $\T$.}

\renewcommand{\theenumi}{{\rm A\arabic{enumi}}}%
\renewcommand{\labelenumi}{({\rm A\arabic{enumi}})}
\begin{enumerate}
\item \label{A:stable}\textbf{Stability on non-refined element domains}: For all refinements $\widehat\TT\in\T$ of a triangulation $\TT\in\T$, for all subsets $\mathcal{S}\subseteq\TT\cap\widehat\TT$ of non-refined element domains, and for all  $V\in\XX(\TT)$, $\widehat V\in\XX(\widehat\TT)$, it holds that
\begin{align*}
\Big|\Big(\sum_{T\in\mathcal{S}}\eta_T(\widehat\TT;\widehat V)^2\Big)^{1/2}-\Big(\sum_{T\in\mathcal{S}}\eta_T(\TT;V)^2\Big)^{1/2}\Big|\leq \c{stable}\,\dist[\widehat\TT]{\widehat V}{V}.
\end{align*}
\item \label{A:reduction}\textbf{Reduction property on refined element domains}: Any refinement $\widehat\TT\in\T$ of a triangulation $\TT\in\T$ satisfies
\begin{align*}
\sum_{T\in\widehat\TT\setminus\TT}\eta_T(\widehat\TT;U(\widehat\TT))^2\leq \q{reduction} \sum_{T\in\TT\setminus\widehat\TT}\eta_T(\TT;U(\TT))^2 + \c{reduction}\dist[\widehat\TT]{U(\widehat\TT)}{U(\TT)}^2.
\end{align*}
\item\label{A:qosum}\textbf{General quasi-orthogonality}: 
There exist constants
\begin{align*}
 0\leq \epsqo < \eps_{\rm qo}^\star(\theta)&:=\sup_{\delta>0}\frac{1-(1+\delta)(1-(1-\q{reduction})\theta)}{\c{reliable}^{2}(\c{reduction}+(1+\delta^{-1})\c{stable}^2)}
\end{align*}
and \revision{$\c{qosum}(\epsqo)\geq 1$} such that the output of Algorithm~\ref{algorithm} satisfies, for all $\ell,N\in\N_0$ with $N\ge\ell$, that
\begin{align*}
\sum_{ k=\ell}^N \big(\dist[\TT_{k+1}]{U(\TT_{k+1})}{U(\TT_k)}^2 -\epsqo \dist[\TT_k]{u}{U(\TT_k)}^2\big) \leq \c{qosum}(\epsqo)\eta(\TT_\ell;U(\TT_\ell))^2.
\end{align*}
\item\label{A:dlr}\textbf{Discrete reliability}:
For all refinements $\widehat\TT\in\T$ of a triangulation $\TT\in\T$,
there exists a subset $\RR(\TT,\widehat\TT)\subseteq\TT$ with $\TT\backslash\widehat\TT \subseteq\RR(\TT,\widehat\TT)$ and
$|\RR(\TT,\widehat\TT)|\le\c{refined}|\TT\backslash\widehat\TT|$ such that
\begin{align*}
 \dist[\widehat\TT]{U(\widehat\TT)}{U(\TT)}^2
 \leq \c{dlr}^2\sum_{T\in\RR(\TT,\widehat\TT)}\eta_T(\TT;U(\TT))^2.
\end{align*}
\end{enumerate}
%


\begin{remark}\label{rem:qosum1}
Proposition~\ref{prop:Rconv} and Proposition~\ref{prop:qosumiscool} below show that general quasi-orthogonality~\eqref{A:qosum} together with~\eqref{A:stable}--\eqref{A:reduction} and reliability~\eqref{A:reliable} implies~\eqref{A:qosum} even with $\epsqo=0$ and $0<\c{qosum}(0)<\infty$.
\end{remark}

\begin{remark}
In all examples of Section~\ref{section:examples1}--\ref{section:examples2} and Section~\ref{section:examples3}--\ref{section:nonlinear}, the axiom~\eqref{A:qosum} is proved for any $\epsqo>0$ instead of one single $0<\epsqo<\eps_{\rm qo}^\star(\theta)$ because the value of $\eps_{\rm qo}^\star(\theta)$ is involved.
Simple calculus allows to determine the maximum in~\eqref{A:qosum} as
\begin{align*}
\eps_{\rm qo}^\star(\theta)&=\Big(1-\frac{(1-(1-\q{reduction})\theta)\c{reduction}+D}{\c{reduction}+\c{stable}^2}\Big)\frac{D- (1-(1\q{reduction})\theta)\c{stable}^2}{\c{reliable}^2D(\c{reduction}+\c{stable}^2)}\\
&\geq 
\frac{\theta^2(1-\q{reduction})^2\c{stable}^2}{2\c{reliable}^2(\c{reduction}+\c{stable}^2)^2}>0.
\end{align*}
where $D:=\sqrt{1-(1-\q{reduction})\theta}\sqrt{\c{reduction}\c{stable}(1-\q{reduction})\theta+\c{stable}^2}>0$.
While Theorem~\ref{thm:main}~(i) holds for any choice $0<\theta\leq 1$, the optimality result of Theorem~\ref{thm:main}~(ii) is further restricted by $\theta<\theta_\star:=(1+\c{stable}^{2}\c{reliable}^2)^{-1}$.
\end{remark}

The following sections are dedicated to the relations between the different axioms 
and the corresponding implications. Figure~\ref{fig:map} outlines the 
convergence and quasi-optimality proof in Section~\ref{section:optimality} and visualizes how the different axioms interact.

\begin{figure}[h]
\newcommand\psBox[4][white]{\rput(#2){\rnode[rc]{#4}{%
  \psframebox[fillcolor=#1]{\tiny{\makebox{\tabular{c}#3\endtabular}}}}}}%
\newcommand\psbox[4][white]{\rput(#2){\rnode{#4}{%
  \psframebox[fillcolor=#1]{\tiny{\makebox{\tabular{c}#3\endtabular}}}}}}%
\newcommand\psBoxT[4][white]{\rput(#2){\rnode[lc]{#4}{%
  \psframebox[fillcolor=#1]{\tiny{\makebox{\tabular{c}#3\endtabular}}}}}}%
\psset{framearc=0.2,shadow=true,xunit=1mm,yunit=1mm,fillstyle=solid,shadowcolor=black!55}%
\begin{pspicture}(-8,0)(100,90)%
\psBox[blue!30]{10,80}{Reduction~\eqref{A:reduction}}{E1}%
\psBox[blue!30]{10,60}{Stability~\eqref{A:stable}}{E2}%
\psbox[green!30]{137,60}{Reliability~\eqref{A:reliable}\\(Lemma~\ref{lemma:dlr2rel})}{E3}%
\psBox[blue!30]{10,35}{Discrete reliability~\eqref{A:dlr}}{E4}%
\psbox[blue!30]{137,80}{Discrete reliability~\eqref{A:dlr}}{E4prime}%
\psBoxT[blue!30]{137,45}{Quasi-orthogonality~\eqref{A:qosum}}{E5}%
\psBoxT[green!30]{137,30}{Closure~\eqref{refinement:closure}}{E6}%
\psBoxT[green!30]{137,20}{Overlay~\eqref{refinement:overlay}}{E7}%
\psBox[green!30]{10,7}{Efficiency~\eqref{A:efficient}}{E8}%
\psbox[gray!30]{45,75}{Estimator reduction\\(Lemma~\ref{lem:estconv})}{EstRed}%
\psbox[gray!30]{90,45}{$R$-linear convergence\\of $\eta(\TT_\ell;U(\TT_\ell))$ (Proposition~\ref{prop:Rconv})}{Rconv}%
\psbox[gray!30]{95,75}{Convergence of $\eta(\TT_\ell;U(\TT_\ell))$}{EstConv}%
\psbox[red!70]{65,25}{Optimal Convergence\\of $\eta(\TT_\ell;U(\TT_\ell))$ (Proposition~\ref{prop:optimality})}{EstOpt}%
\psbox[gray!30]{95,63}{Convergence of $U(\TT_\ell)$}{Conv}%
\psbox[gray!30]{65,7}{Optimal Convergence\\of $U(\TT_\ell)$ (Proposition~\ref{prop:charAprox})}{Opt}%
\psbox[gray!30]{50,45}{Optimality of\\D\"orfler marking\\(Proposition~\ref{prop:doerfler})}{DoerflerOpt}%
\psset{nodesep=3pt,arrows=->,shortput=nab,shadow=false}%
\ncline{E1}{EstRed}%
\ncline{E2}{EstRed}%
\ncline{E3}{Conv}%
\ncline{E4}{DoerflerOpt}%
\ncline{E5}{Rconv}%
\ncline{E6}{EstOpt}%
\ncline{E7}{EstOpt}%
\ncline{E8}{Opt}%
\ncline{E2}{DoerflerOpt}%
\ncline{E3}{Rconv}%
\psset{nodesep=3pt,arrows=c->,shortput=nab,shadow=false,linewidth=0.1,arrowinset=0,arrowlength=0.8}%
\ncline{E4prime}{E3}%
\ncline{DoerflerOpt}{EstOpt}%
\ncline{Rconv}{EstOpt}%
\ncline{EstOpt}{Opt}%
\ncline{EstRed}{EstConv}%
\ncline{EstRed}{Rconv}%
\ncline{EstConv}{Conv}%
\end{pspicture}
\caption{Map of the quasi-optimality proof. \revision{The arrows mark the dependencies of the arguments.}}
\label{fig:map}
\end{figure}
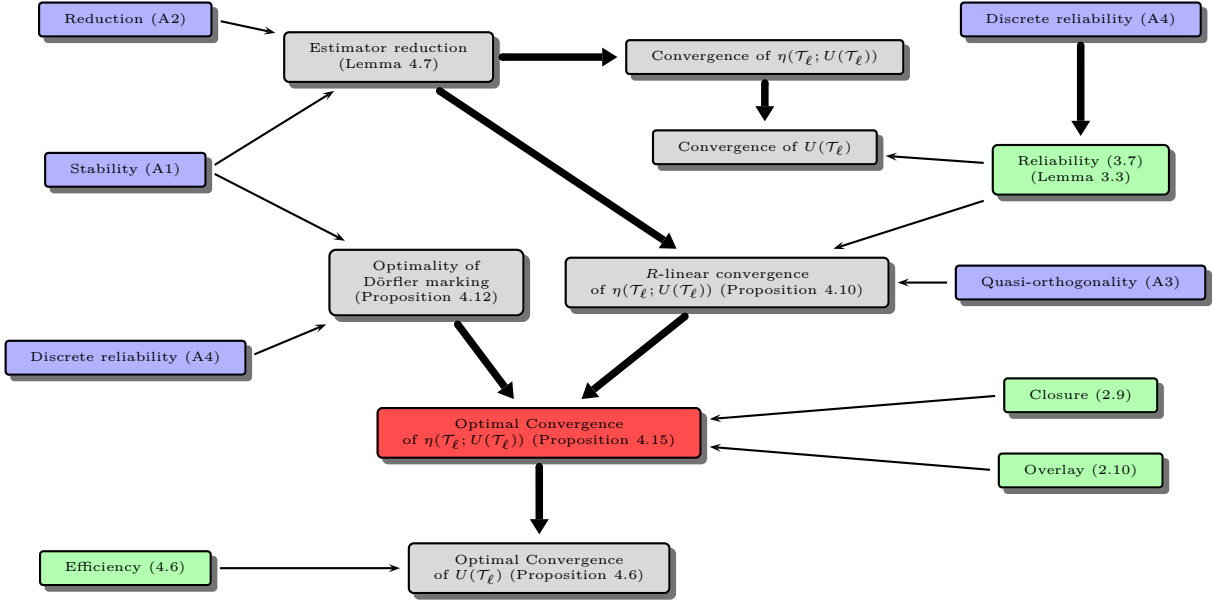

\subsection{Historic remarks}
This work provides some unifying framework on the theory of adaptive 
algorithms and the related convergence and quasi-optimality analysis. 
Some historic remarks are in order  on the development of the arguments over the years.
In one way or another, the axioms arose in various works throughout the literature.

{\Large\gray$\bullet$} {\bf Reliability~\eqref{A:reliable}.}\quad
Reliability basically states that the unknown error tends to zero if the
computable and hence known error bound is driven to zero by smart
adaptive algorithms. Since the invention of adaptive FEM in the 1970s, 
the question of reliability was thus a pressing matter and first results for 
FEM date back to the early works of 
\textsc{Babuska \& Rheinboldt}~\cite{br79} in 1D and 
\textsc{Babuska \& Miller}~\cite{bm87} in 2D. Therein, the error is
estimated by means of the residual. In the context of BEM, reliable 
residual-based error estimators date back to the works of 
\textsc{Carstensen \& Stephan}~\cite{cs96,cs95,cc97}. 
Since the actual adaptive algorithm only knows the estimator, reliability 
estimates have been a crucial ingredient for convergence proofs of 
adaptive schemes of any kind.

{\Large\gray$\bullet$} {\bf Efficiency~\eqref{A:efficient}.}\quad
Compared to reliability~\eqref{A:reliable}, 
efficiency~\eqref{A:efficient} provides the converse estimate and states
that the error is not overestimated by the estimator, up to some 
oscillation
terms $\eff{\cdot}{U(\cdot)}$ determined from the given data. An error estimator which 
satisfies both, reliability and efficiency, is mathematically guaranteed to 
asymptotically behave like the error, i.e.\ it decays with the same 
rate as the actual computational error.  Consequently, efficiency is a 
desirable property as soon as it comes to convergence rates. 
For FEM with residual error estimators, efficiency has first 
been proved by~\textsc{Verf\"urth}~\cite{v94}. He used appropriate inverse
estimates and localization by means of bubble functions.
In the frame of BEM, however, 
efficiency~\eqref{A:efficient} of the residual error estimators is widely 
open and only known for particular problems~\cite{affkp,cc96}, although 
observed empirically, see also Section~\ref{section:ex:symm}.

{\Large\gray$\bullet$} {\bf Discrete local efficiency and first 
convergence analysis of~\cite{d1996,mns}.}\quad
Reliability~\eqref{A:reliable} and efficiency~\eqref{A:efficient} 
are nowadays standard topics in textbooks on a~posteriori FEM error estimation~\cite{ao00,v96}, in contrast to the convergence of
adaptive algorithms. \textsc{Babuska \& Vogelius}~\cite{bv}
already observed for conforming discretizations, that the sequence
of discrete approximations $U(\TT_\ell)$ always converges.
The work of \textsc{D\"orfler}~\cite{d1996} introduced the marking 
strategy~\eqref{eq:doerfler} for the Poisson model problem
\begin{align}\label{history:poisson}
 -\Delta u = f \text{ in }\Omega
 \quad\text{and}\quad
 u = 0\text{ on }\Gamma=\partial\Omega
\end{align}
and conforming \revision{first-order} FEM \revision{to show} convergence up to any 
\revision{given} tolerance.
\textsc{Morin, Nochetto \& Siebert}~\cite{mns} refined this and the arguments of~\textsc{Verf\"urth}~\cite{v94} 
and~\textsc{D\"orfler}~\cite{d1996} and proved the discrete variant
\begin{align*}
C_{\rm eff}^{-2} \eta(\TT_\ell;U(\TT))^2\leq \norm{\nabla(U(\TT_{\ell+1})-U(\TT_\ell))}{L^2(\Omega)}^2 +  \revision{\eff[\TT_\ell]{\TT_\ell}{U(\TT_\ell)}^2 }
\end{align*}
of the efficiency~\eqref{A:efficient}.
See also~\cite{fop} for the explicit statement and proof. 
The proof relies on discrete bubble functions and thus required an 
\emph{interior node property}
of the local mesh-refinement, which is ensured by newest vertex bisection and five bisections for each refined element.
With this \cite{mns} proved \emph{error reduction} up to data oscillation terms in the sense of
\begin{align}\label{history:conv:mns}
 \norm{\nabla(u-U(\TT_{\ell+1}))}{L^2(\Omega)}^2
 \le \kappa\,\norm{\nabla(u-U(\TT_{\ell}))}{L^2(\Omega)}^2
 + C\, \eff{\TT_\ell}{U(\TT_\ell)}
\end{align}
with some $\ell$-independent constants $0<\kappa<1$ and $C>0$. This and additional enrichment of the 
marked elements $\MM_\ell\subseteq\TT_\ell$ to ensure $ \eff{\TT_\ell}{U(\TT_\ell)}\to0$ 
as $\ell\to\infty$ leads to convergence 
\begin{align}
 \norm{\nabla(u-U(\TT_{\ell}))}{L^2(\Omega)}
 \xrightarrow{\ell\to\infty}0.
\end{align}

{\Large\gray$\bullet$} {\bf Quasi-orthogonality~\eqref{A:qosum}.}\quad
The approach of~\cite{mns} has been generalized to non-symmetric operators in~\cite{mn2005}, to nonconforming
and mixed methods in~\cite{ch06a,ch06b},
as well as to the nonlinear obstacle problem in~\textsc{Braess, Carstensen
\& Hoppe}~\cite{bch07,bch09}.
One additional difficulty is the lack of the Galerkin orthogonality
which is circumvented 
with the quasi-orthogonality axiom~\eqref{A:qosum} in Section~\ref{section:quasiorth} below. 
Stronger variants of quasi-orthogonalities have  been used in~\cite{ch06a,ch06b,mn2005} and imply~\eqref{A:qosum} in 
Section~\ref{section:quasiorth} below.
In its current form, however, the general quasi-orthogonality~\eqref{A:qosum} 
goes back to~\cite{nonsymm} of~\textsc{Feischl, F\"uhrer \& 
Praetorius}  for nonsymmetric operators without 
artificial assumptions on the initial mesh as in~\cite{cn,mn2005}. Proposition~\ref{prop:Rconv} below shows that the present form~\eqref{A:qosum} of the quasi-orthogonality cannot be weakened if one aims to follow the analysis of~\cite{ckns,stevenson07} to prove quasi-optimal convergence rates.

{\Large\gray$\bullet$} {\bf Optimal convergence rates and discrete reliability~\eqref{A:dlr}.}\quad
The work of~\textsc{Binev, Dahmen \& DeVore}~\cite{bdd} was the first 
one to prove
algebraic convergence rates for adaptive FEM of the Poisson model 
problem~\eqref{history:poisson} and lowest-order FEM. They extended the adaptive
algorithm of~\cite{mns} by additional coarsening steps to avoid
over-refinement. \textsc{Stevenson}~\cite{stevenson07} removed
this artificial coarsening step and
introduced the axiom~\eqref{A:dlr} on discrete reliability. He implicitly introduced the concept of \emph{separate D\"orfler marking}: 
If the data oscillations $ \eff{\TT_\ell}{U(\TT_\ell)}$ are small compared to the error 
estimator $\eta(\TT_\ell;U(\TT_\ell))$, he used the common D\"orfler 
marking~\eqref{eq:doerfler} to single out the elements for refinement.
Otherwise, he suggested the D\"orfler marking~\eqref{eq:doerfler} for the local contributions 
$\eff[T]{\TT_\ell}{U(\TT_\ell)}$ of the data oscillation terms. The core proof 
of~\cite{stevenson07} then uses
the observation from~\cite{mn2005} that the so-called \emph{total error} is
contracted in each step of the adaptive loop in the sense of
\begin{align}\label{history:totalerror}
 \Delta_{\ell+1} &\le \kappa\,\Delta_\ell\quad\text{for}\quad
 \Delta_\ell := \norm{\nabla(u-U(\TT_\ell))}{L^2(\Omega)}^2
 + \gamma\,\eff{\TT_\ell}{U(\TT_\ell)}^2
\end{align}
with some $\ell$-independent constants $0<\kappa<1$ and $\gamma>0$. 

Moreover, the analysis of~\cite{stevenson07} \revision{shows that the}
D\"orfler marking~\eqref{eq:doerfler} is not only sufficient to guarantee
contraction~\eqref{history:totalerror}, but somehow even necessary, see Section~\ref{section:doerfler} for the refined analysis which avoids the use of efficiency~\eqref{A:efficient}.

{\Large\gray$\bullet$} {\bf Stability~\eqref{A:stable} and reduction~\eqref{A:reduction}.} \quad The AFEM analysis of~\cite{stevenson07} was simplified
by \textsc{Cascon, Kreuzer, Nochetto \& Siebert}~\cite{ckns} with the introduction of the \emph{estimator reduction} in the sense of 
\begin{align}\label{history:reduction}
\eta(\TT_{\ell+1};U(\TT_{\ell+1}))^2
\le \kappa\, \eta(\TT_\ell;U(\TT_\ell))^2
+ C\,\norm{\nabla(U(\TT_{\ell+1})-U(\TT_\ell))}{L^2(\Omega)}^2
\end{align}
with constants $0<\kappa<1$ and $C>0$. This is an immediate consequence of 
stability~\eqref{A:stable} and reduction~\eqref{A:reduction} in
Section~\ref{section:reduction} below and also ensures contraction
of the so-called quasi-error 
\begin{align}\label{history:quasierror}
 \Delta_{\ell+1} &\le \kappa \, \Delta_\ell\quad\text{for}\quad
 \Delta_\ell := \norm{\nabla(u-U(\TT_\ell))}{L^2(\Omega)}^2
 + \gamma\,\eta(\TT_\ell;U(\TT_\ell))^2
\end{align}
with some $\ell$-independent constants $0<\kappa<1$ and $\gamma>0$. The analysis of~\cite{ckns} removed the \emph{discrete local lower bound} from the set of necessary
axioms (and hence the \emph{interior node property}~\cite{mns}). Implicitly, the 
axioms~\eqref{A:stable}--\eqref{A:reduction} are part of the proof
of~\eqref{history:reduction} in~\cite{ckns}. While~\eqref{A:stable} essentially follows from the 
triangle inequality and appropriate inverse estimates in practice, the 
reduction~\eqref{A:reduction} builds on the observation that the element 
sizes of the sons of a refined element \revision{uniformly decreases.} 
For instance, bisection-based mesh-refinements yield $|T'|\le|T|/2$,
if $T'\in\TT_{\ell+1}\backslash\TT_\ell$ is a son of  
$T\in\TT_\ell\backslash\TT_{\ell+1}$.

{\Large\gray$\bullet$} {\bf Extensions of the analysis of~\cite{ckns}.}\quad
The
work~\cite{ks} considers lowest-order AFEM for the Poisson 
problem~\eqref{history:poisson} for error estimators which are locally 
equivalent to the residual error estimator. The works~\cite{cn,nonsymm} analyze
optimality of AFEM for linear, but non-symmetric elliptic operators. While~\cite{cn} required that the corresponding bilinear form induces a norm, 
such an assumption is dropped in~\cite{nonsymm}, so that the latter work concluded the AFEM analysis for linear second-order elliptic PDEs.
Convergence with quasi-optimal rates for adaptive boundary element methods
has independently been proved in~\cite{fkmp,gantumur}. The main additional 
difficulty was the development of appropriate \emph{local} inverse estimates
for the \emph{nonlocal} operators involved. The BEM analysis, however, still 
hinges on symmetric and elliptic integral operators and excludes boundary
integral formulations of mixed boundary value problems as well as the 
FEM-BEM coupling. AFEM with nonconforming and mixed FEMs is considered for various problems in~\cite{rabus10,CR2012,cpr13,cr11,bms09,mzs10}. AFEM with non-homogeneous Dirichlet and mixed Dirichlet-Neumann
boundary conditions are analyzed in~\cite{dirichlet2d} for 2D and
in~\cite{dirichlet3d} for 3D. The latter work adapts the separate 
D\"orfler marking from~\cite{stevenson07} to decide whether the refinement
relies on the error estimator for the discretization error or the approximation
error of the given continuous Dirichlet data, see Section~\ref{section:boundary}. The results of those
works are reproduced and partially even improved in the frame of the abstract
axioms~\eqref{A:stable}--\eqref{A:dlr} of this paper. Finally, the proofs 
of~\cite{dirichlet3d,nonsymm} simplified the core analysis 
of~\cite{stevenson07,ckns} in the sense that the optimality analysis avoids
the use of the \emph{total error} and solely works with the error estimator.

\subsection{Discrete reliability implies reliability}
The compatibility condition~\eqref{eq:unifconv} and the discrete reliability~\eqref{A:dlr} imply reliability.
\begin{lemma}\label{lemma:dlr2rel}
Discrete reliability implies reliability in the sense that any triangulation $\TT\in\T$ satisfies
\begin{align} \label{A:reliable}
\dist[\TT]{u}{U(\TT)} \leq \c{reliable} \eta(\TT;U(\TT)).
\end{align}
 \end{lemma}

 \begin{proof}
Given any $\eps>0$, the choice of $\widehat\TT$ in~\eqref{eq:unifconv} and the discrete reliability~\eqref{A:dlr} together with $\TT\setminus\widehat\TT\subseteq\RR(\TT,\widehat\TT)$ show 
\begin{align*}
\c{triangle}^{-1}\,\dist[\TT]{u}{U(\TT)} 
&=\c{triangle}^{-1}\,\dist[\widehat\TT]{u}{U(\TT)}
\\&\le \dist[\widehat\TT]{u}{U(\widehat\TT)}+\dist[\widehat\TT]{U(\widehat\TT)}{U(\TT)}\\
&\leq  \eps +\c{dlr}\Big(\sum_{T\in\RR(\TT,\widehat\TT)}\eta_T(\TT;U(\TT))^2\Big)^{1/2}\\
&\leq \eps +\c{dlr}\eta(\TT;U(\TT)).
\end{align*}
The arbitrariness of $\eps>0$ in the above estimate proves reliability of $\eta(\TT;U(\TT))$ with $\c{reliable}=\c{triangle}\c{dlr}$.
%
%
\end{proof}

\subsection{Quasi-monotonicity of the error estimator}
The first two lemmas show that the error estimator is quasi-monotone for many applications in the sense that there exists a constant $\setc{mon}>0$ such that all refinements 
$\widehat \TT\in\T$ of $\TT\in\T$ satisfy
\begin{align}\label{eq:mon}%
\eta(\widehat\TT;U(\widehat\TT))\leq\c{mon}\eta(\TT;U(\TT)).
\end{align}
{\renewcommand{\theequation}{\arabic{equation}}%
\numberwithin{equation}{section}%
Although reduction~\eqref{A:reduction} is assumed in the following, the assumption 
$\q{reduction}<1$ in~\eqref{A:reduction} is not needed in 
Lemma~\ref{lem:mon1} and Lemma~\ref{lem:mon2}.

\begin{lemma}\label{lem:mon1}
Stability~\eqref{A:stable}, reduction~\eqref{A:reduction}, and discrete reliability~\eqref{A:dlr} imply quasi-monotonicity~\eqref{eq:mon} of the estimator.
\end{lemma}

\begin{proof}
The stability~\eqref{A:stable} and the reduction estimate~\eqref{A:reduction} imply
\begin{align*}
 \eta(\widehat\TT;U(\widehat\TT))^2&\leq \q{reduction}\sum_{T\in\TT\setminus\widehat\TT} \eta_T(\TT;U(\TT))^2+2\sum_{T\in\TT\cap\widehat\TT} \eta_T(\TT;U(\TT))^2\\
&\qquad + (2\c{stable}^2 + \c{reduction})\dist[\widehat\TT]{U(\widehat\TT)}{U(\TT)}^2 =:\text{RHS}.
\end{align*}
The discrete reliability~\eqref{A:dlr}, leads to
\begin{align*}
\text{RHS} &\le\max\{2, \q{reduction}\}\eta(\TT;U(\TT))^2 +(2\c{stable}^2+\c{reduction})\c{dlr}^2\sum_{T\in\RR(\TT,\widehat\TT)}\eta_T(\TT;U(\TT))^2\\
&\leq \big(\max\{2, \q{reduction}\}+(2\c{stable}^2+\c{reduction})\c{dlr}^2\big)\eta(\TT;U(\TT))^2.
\end{align*}
This is~\eqref{eq:mon} with $\c{mon}:=\big(\max\{2, \q{reduction}\}+(2\c{stable}^2+\c{reduction})\c{dlr}^2\big)^{1/2}$.
\end{proof}

A C\'ea-type best approximation~\eqref{eq:cea} and reliability~\eqref{A:reliable} imply
monotonicity~\eqref{eq:mon}.

\begin{lemma}\label{lem:mon2}
Suppose stability~\eqref{A:stable}, reduction~\eqref{A:reduction}, and reliability~\eqref{A:reliable} and let $\setc{cea}>0$ be a constant such that
\begin{align}\label{eq:cea}
 \dist[\widehat\TT]{u}{U(\widehat\TT)}
 \leq \c{cea}\min_{V\in\XX(\widehat\TT)}\dist[\widehat\TT]{u}{V}
\end{align}
holds for any refinement $\widehat\TT$ of $\TT\in\T$. Suppose that the ansatz spaces $\XX(\TT)\subseteq \XX(\widehat\TT)$ are nested.  Then, the error estimator is quasi-monotone~\eqref{eq:mon}.
\end{lemma}

\begin{proof}
 As in the proof of Lemma~\ref{lem:mon1}, it follows
\begin{align*}
 \eta(\widehat\TT;U(\TT))^2 \leq \max\{2, \q{reduction}\}\eta(\TT;U(\TT))^2 +(2\c{stable}^2 + \c{reduction})\dist[\widehat\TT]{U(\widehat\TT)}{U(\TT)}^2.
\end{align*}
Recall $\dist[\widehat\TT]{u}{U(\TT)}=\dist[\TT]{u}{U(\TT)}$ and set $\widetilde C:=(2\c{stable}^2 + \c{reduction})\c{triangle}^2$.
Reliability~\eqref{A:reliable} and the quasi-triangle inequality yield
\begin{align*}
  \eta(\widehat\TT;U(\widehat\TT))^2 &\leq \max\{2, \q{reduction}\}\eta(\TT;U(\TT))^2 +\widetilde C\big(2\c{triangle}^2\dist[\widehat\TT]{u}{U(\widehat\TT)}^2+2\dist[\TT]{u}{U(\TT)}^2\big)\\
&\leq \max\{2, \q{reduction}\}\eta(\TT;U(\TT))^2 +\widetilde C\big(2\c{triangle}^2\c{cea}^2\dist[\TT]{u}{U(\TT)}^2+2\dist[\TT]{u}{U(\TT)}^2\big)\\
&\leq\big(\max\{2, \q{reduction}\} + 2\widetilde C(1+\c{triangle}^2\c{cea}^2)\c{reliable}^2\big)\eta(\TT;U(\TT))^2.
\end{align*}
This is~\eqref{eq:mon} with $\c{mon}:=\big(\max\{2, \q{reduction}\} + 2\widetilde C(1+\c{triangle}^2\c{cea}^2)\c{reliable}^2\big)^{1/2}.$
\end{proof}

\subsection{Quasi-orthogonality implies general quasi-orthogonality}\label{section:quasiorth}
The general quasi-orthogonality axiom~\eqref{A:qosum} generalizes the {\renewcommand{\theenumi}{\rm B3}\refstepcounter{enumi}\label{A:qo}}%
quasi-orthogonality~\eqref{A:qo}~:=~\eqref{A:qoa}\&\eqref{A:qob}.
\begin{enumerate}
\renewcommand{\theenumi}{{\rm B3}\alph{enumi}}
\renewcommand{\labelenumi}{({\rm B3}\alph{enumi})}
\item \label{A:qoa}There exists a function $\mu:\,\T \to \R$ such that for all $\eps>0$ there exists some constant $\setc{qo}(\eps)>0$ such that for all refinements $\widehat\TT\in\T$ of $\TT\in\T$ it holds
\begin{align*}
\dist[\widehat\TT]{U(\widehat\TT)}{U(\TT)}^2 
&\leq (1+\eps)\dist[\TT]{u}{U(\TT)}^2 - (1-\eps)\dist[\widehat\TT]{u}{U(\widehat\TT)}^2 \\
&\quad\quad+\c{qo}(\eps)\big(\mu(\TT)^2-\mu(\widehat\TT)^2\big).
\end{align*}
\item \label{A:qob}The function $\mu(\cdot)$ from~\eqref{A:qoa}
is dominated by $\eta(\cdot;U(\cdot))$ in the sense that\definec{muell}
\begin{align*}
 \sup_{\TT\in\T} \frac{\mu(\TT)^2}{\eta(\TT;U(\TT))^2} =: \c{muell}<\infty.
\end{align*}
\end{enumerate}

\begin{lemma}\label{lem:qo}
Reliability~\eqref{A:reliable} and quasi-orthogonality~\eqref{A:qo} imply the general quasi-orthogonality~\eqref{A:qosum} for all $\eps=\epsqo/2>0$.
\end{lemma}
\begin{proof}
Quasi-orthogonality~\eqref{A:qo} with $\eps=\epsqo/2$ and reliability~\eqref{A:reliable} show for any $N\in\N$ that
\begin{align*}
\sum_{ k=\ell}^N &\big(\dist[\TT_{k+1}]{U(\TT_{k+1})}{U(\TT_k)}^2 
-\epsqo \dist[\TT_k]{u}{U(\TT_k)}^2\big)\\
&\leq \sum_{k=\ell}^N\Big((1-\epsqo/2)\big(\dist[\TT_k]{u}{U(\TT_k)}^2
-\dist[\TT_{k+1}]{u}{U(\TT_{k+1})}^2 \big) + \c{qo}(\eps)\big(\mu(\TT_k)^2-\mu(\TT_{k+1})^2\big)\Big)\\
&\leq(1-\epsqo/2)\dist[\TT_\ell]{u}{U(\TT_\ell)}^2 
+\c{qo}(\epsqo/2)\mu(\TT_\ell)^2\lesssim \eta(\TT_\ell;U(\TT_\ell))^2.
\end{align*}
This follows from the telescoping series and~\eqref{A:qob}.  This concludes the proof of~\eqref{A:qosum}.
\end{proof}

\begin{remark}
In contrast to the common quasi-ortho\-go\-nality~\eqref{A:qo},
the general quasi-orthogonality~\eqref{A:qosum} holds for equivalent norms although with different $\epsqo$.
Therefore, general quasi-ortho\-gonality appears solely as an assumption
on the approximation property of the sequence $(U(\TT_\ell))_{\ell\in\N_0}$.
\end{remark}

\subsection{Conforming methods for elliptic problems}%
\label{section:conforming}%
This short section studies the particular case of conforming methods, which allows some interesting simplifications. Let $b(\cdot,\cdot)$ be a continuous and elliptic bilinear
form on the real Hilbert space $\XX$ with dual $\XX^*$. Given any $f\in\XX^*$, the Lax-Milgram lemma guarantees the existence and uniqueness of
the solution $u\in\XX$ to
\begin{align}
 b(u,v) = f(v)\quad\text{for all }v\in\XX.
\end{align}
Suppose $\XX(\TT)\subseteq\XX(\widehat\TT)\subseteq\XX$ for all triangulations $\TT\in\T$ 
and all refinements $\widehat\TT\in\T$ of $\TT$ and suppose that 
$\norm\cdot\XX$ 
is the Hilbert space norm on $\XX$ with $\dist[\TT]vw = \dist{v}{w} 
=\norm{v-w}\XX$ for all $\TT\in\T$ and $v,w\in\XX$.
Model problems follow in Section~\ref{section:examples1} and Section~\ref{section:examples2} below.

For \revision{any} closed subspace $\XX_\infty$ of $\XX$, the Lax-Milgram lemma \revision{implies the unique existence of a} solution $U_\infty\in\XX_\infty$  to
\begin{align}
 \bform{U_\infty}{V_\infty} = f(V_\infty)
 \quad\text{for all }U_\infty\in\XX_\infty
\end{align}
which satisfies the C\'ea lemma~\eqref{eq:cea}. 
In particular, this applies to the discrete spaces $\XX(\TT)$, so that the 
discrete Galerkin solutions $U(\TT)\in\XX(\TT)$ are unique and satisfy monotonicity of the error (defined in~\eqref{eq:errorquasimon} below)
\begin{align*}
 \norm{u-U(\widehat\TT)}{\XX}\leq \c{cea}\norm{u-U(\TT)}{\XX}\quad\text{for all } \TT\in\T\text{ and all refinements }\widehat\TT\in\T.
\end{align*}

{\Large\gray$\bullet$}
It has already been observed in the seminal work~\cite{bv} that in this conforming setting with nested spaces, there holds a~priori convergence
\begin{align}\label{eq:aconv}
 \lim_{\ell\to\infty}\norm{U_\infty-U(\TT_\ell)}{\XX} = 0
\end{align}
towards a certain (unknown) limit $U_\infty\in\XX$. 
Therefore
stability~\eqref{A:stable} and reduction~\eqref{A:reduction} combined with reliability~\eqref{A:reliable}
already imply convergence in Section~\ref{section:reduction}.


{\Large\gray$\bullet$}
Suppose that $b(\cdot,\cdot)$ is a scalar product on $\XX$ with induced norm
$\norm\cdot\XX$. Then, the Galerkin orthogonality
\begin{align}\label{eq:galerkin_orthogonality}
 \bform{u-U(\TT_{\ell+1})}{V} = 0 
 \quad\text{for all }V\in\XX(\TT_{\ell+1})
\end{align}
implies the Pythagoras theorem
\begin{align}\label{dp:pythagoras} 
 \norm{u-U(\TT_{\ell+1})}\XX^2 = \norm{u-U(\TT_{\ell})}\XX^2
  - \norm{U(\TT_{\ell+1})-U(\TT_{\ell})}\XX^2.
\end{align}
In particular, the quasi-orthogonality~\eqref{A:qo} is satisfied with $\epsqo=0=\c{qo}$ and $\mu(\cdot)=0$,
and Lemma~\ref{lem:qo} implies  the general quasi-orthogonality~\eqref{A:qosum}. In this frame, it thus only remains to verify~\eqref{A:stable}, \eqref{A:reduction}, and \eqref{A:dlr}.


\begin{remark}
The a~priori convergence~\eqref{eq:aconv} of conforming methods holds in a wider frame of (not necessarily linear) Petrov-Galerkin schemes as exploited in~\cite{msv,siebert,estconv,afp,page,pp} 
to prove convergence of adaptive FEM and BEM, and the adaptive FEM-BEM coupling.
\end{remark}

\section{Optimal Convergence Of The Adaptive Algorithm}\label{section:optimality}
\noindent%
The best possible algebraic convergence rate $0<s<\infty$  obtained by any local mesh refinement is characterized in terms of
\begin{align}\label{def:natapproxclass}
\norm{(u,U(\cdot))}{\A_s}:= \sup_{N\in\N_0}\min_{\TT\in\T(N)}(N+1)^s\dist[\TT]{u}{U(\TT)}<\infty.
\end{align}
The statement $\norm{(u,U(\cdot))}{\A_s}<\infty$ means $\dist[\TT]{u}{U(\TT)}=\OO(N^{-s})$ for the optimal triangulations $\TT\in\T(N)$, independently of the error estimator. 
Since the adaptive algorithm is steered by the error estimator $\eta(\cdot)$, it appears natural to consider the best algebraic convergence rate $\OO(N^{-s})$ in terms of $\eta(\cdot)$, characterized by
\begin{align}\label{def:approxclass}
\norm{(\eta(\cdot),U(\cdot))}{\B_s}:= \sup_{N\in\N_0}\min_{\TT\in\T(N)}((N+1)^s\eta(\TT;U(\TT))<\infty.
\end{align}
This implies the convergence rate $\eta(\TT;U(\TT)) = \OO(N^{-s})$ for the optimal triangulations $\TT\in\T(N)$.

The relation of $\norm{\cdot}{\A_s}$ and $\norm{\cdot}{\B_s}$ and the nonlinear approximation classes in~\cite{stevenson07,ckns,ks,cn} will be discussed in Section~\ref{section:approxclass} below.
\subsection{Optimal convergence rates for the error estimator}
The main results of this work state convergence and optimality of the adaptive algorithm in the sense that
the error estimator converges with optimal convergence rate. This is a generalization of existing results as discussed in Section~\ref{section:approxclass}. Moreover, if the error estimator $\eta(\cdot)$ satisfies an efficiency estimate, also optimal convergence of the error \revision{will be} guaranteed \revision{by} Theorem~\ref{thm:mainerr}.
On the other hand, Theorem~\ref{thm:main} and Theorem~\ref{thm:mainerr} \revision{show} that the adaptive algorithm characterizes the approximability of the limit $u\in\XX$ in terms of the error and the error estimator.

\begin{theorem}\label{thm:main}
Suppose stability~\eqref{A:stable}, reduction~\eqref{A:reduction}, and general quasi-orthogonality~\eqref{A:qosum}. Then, Algorithm~\ref{algorithm} guarantees  {\rm (i)}--{\rm (ii)}.
\begin{itemize}
\item[(i)] 
Discrete reliability~\eqref{A:dlr} resp. reliability~\eqref{A:reliable} imply for all $0<\theta\leq 1$ the $R$-linear convergence of the estimator in the sense that there exists $0<\setr{conv}<1$ and $\setc{conv}>0$ such that
\begin{align}\label{eq:thm:main:conv1}
\eta(\TT_{\ell+j};U(\TT_{\ell+j}))^2\leq \c{conv}\r{conv}^j\,\eta(\TT_\ell;U(\TT_\ell))^2\quad\text{for all }j,\ell\in\N_0.
\end{align}
In particular,
\begin{align}\label{eq:thm:main:conv2}
\c{reliable}^{-1} \dist[\TT_\ell]{u}{U(\TT_\ell)}
 \le \eta(\TT_\ell; U(\TT_\ell))
 \leq \c{conv}^{1/2}\r{conv}^{\ell/2}\,\eta(\TT_0;U(\TT_0))\quad\text{for all }\ell\in\N_0.
\end{align}
\item[(ii)] Discrete reliability~\eqref{A:dlr} and $0<\theta< \theta_\star:=(1+\c{stable}^2\c{dlr}^2)^{-1}$ imply quasi-optimal convergence of the estimator
in the sense of
\begin{align}\label{eq:thm:main:opt1}
c_{\rm opt}\norm{(\eta(\cdot),U(\cdot))}{\B_s} \leq\sup_{\ell\in\N_0}\frac{\eta(\TT_\ell;U(\TT_\ell))}{(|\TT_\ell|-|\TT_0|+1)^{-s}}\leq\c{optimal}\norm{(\eta(\cdot),U(\cdot))}{\B_s}
\end{align}
for all $s>0$.
\end{itemize}
The constants $\c{conv}, \r{conv} >0$ depend only on $\c{stable}, \q{reduction}, \c{reduction}, \c{qosum}(\epsqo)>0$ as well as on $\theta$.
Furthermore, the constant $\setc{optimal}>0$ depends only on 
$C_{\rm min}, \c{refined},\c{mesh},\c{stable},\c{dlr}$, $\c{reduction}, \c{qosum}(\epsqo), \q{reduction}>0$ as well as on~$\theta$ and $s$, while $c_{\rm opt}>0$ depends only on $C_{\rm son}$. 
\end{theorem}

\begin{remark}
Unlike prior work~\cite{ckns,cn,ks,stevenson07}, the upper bound $\theta_\star$ of the range of marking parameters $0<\theta<\theta_\star$ does \emph{not} depend on the efficiency constant
$\c{efficient}$  which is formally introduced in the following Section~\ref{section:approxclass}.
\end{remark}

\begin{remark}
The upper bound in~\eqref{eq:thm:main:opt1} states that given that
$\norm{(\eta(\cdot),U(\cdot))}{\B_s}<\infty$, the estimator sequence $\eta(\TT_\ell;U(\TT_\ell))$ of Algorithm~\ref{algorithm} will decay with order $s$, i.e., if a
decay with order $s$ is possible if the optimal meshes are chosen, this
decay will in fact be realized by the adaptive algorithm.
The lower bound in~\eqref{eq:thm:main:opt1} states that the asymptotic 
convergence rate of the estimator sequence, in fact, characterizes to which
approximation class $\B_s$ the problem and its discretization belong.
\end{remark}

\subsection{Optimal convergence rates for the error}\label{section:approxclass}
The following proposition relates the definition of optimality in~\eqref{def:natapproxclass} and~\eqref{def:approxclass} with the nonlinear approximation classes in~\cite{stevenson07,ckns,ks,cn}. To that end, efficiency comes into play: There exists $\c{efficient}>0$ such that
for all $\TT\in\T$, there exists a mapping $\eff{\TT}{\cdot}:\,\XX(\TT)\to[0,\infty]$ such that
any triangulation $\TT\in\T$ satisfies
\begin{align}\label{A:efficient}
  \c{efficient}^{-2}\,\eta(\TT;U(\TT))^2 \le \dist[\TT]{u}{U(\TT)}^2
  + \eff{\TT}{U(\TT)}^2,
\end{align}}
In particular, this implies that the data oscillations do \emph{not} have to be treated 
explicitly in the analysis.
The quality of the oscillation term $\osc(\cdot)$ is measured with
\begin{align}\label{eq:oscapprox}
\norm{\osc(\cdot)}{\O_s}:= \sup_{N\in\N_0}\min_{\TT \in \T(N)}(N+1)^{s}\eff{\TT}{U(\TT)}<\infty.
\end{align}

The following theorem shows, that the result of Theorem~\ref{thm:main} is a true generalization of the existing results in literature since 
the best possible rate for the error, measured in $\norm{\cdot}{\B_s}$, is equivalent to the best possible rate for the total error from e.g.~\cite{stevenson07,ckns,bdd}.
\begin{theorem}
The C\'ea lemma~\eqref{eq:cea} implies
\begin{align}\label{eq:approxclassbest}
 \c{cea}^{-1}\norm{(u,U(\cdot))}{\A_s}\leq \sup_{N\in\N_0}\min_{\TT\in\T(N)}\min_{V\in\XX(\TT)} (N+1)^s\dist[\TT]{u}{V}\leq \norm{(u,U(\cdot))}{\A_s}
\end{align}
for all $s>0$.
Additionally, suppose efficiency~\eqref{A:efficient} and the existence of $C_\osc>0$ such that all $\TT\in\T$ satisfy
\begin{align}
 C_\osc^{-1}\eff{\TT}{V}&\leq \eff{\TT}{W}+\dist[\TT]{V}{W}\quad\text{for all }V,W\in\XX(\TT),\label{eq:oscstab}\\
 \eff{\TT}{U(\TT)}&\leq C_\osc\eta(\TT;U(\TT)). \label{eq:oscdom}
\end{align}
Then,
\begin{align}\label{eq:approxclassckns}
\begin{split}
C_{\rm apx}^{-1}\norm{(u,U(\cdot))}{\B_s}&\leq
 \sup_{N\in\N_0}\min_{\TT\in\T(N)} \min_{V\in\XX(\TT)}(N+1)^s\big(\dist[\TT]{u}{V} + \eff{\TT}{V}\big)\\
 &\leq (\c{reliable}+C_\osc)\norm{(u,U(\cdot)}{\B_s}
 \end{split}
\end{align}
holds for all $s>0$.
The constant $C_{\rm apx}>0$ depends only on $\c{cea},\c{efficient}, C_\osc, \c{triangle}$.
\end{theorem}
\begin{proof}
The C\'ea lemma~\eqref{eq:cea} and hence
\begin{align*}
\c{cea}^{-1}\dist[\TT]{u}{U(\TT)} \leq \min_{V\in\XX(\TT)} \dist[\TT]{u}{V}\leq \dist[\TT]{u}{U(\TT)}\quad\text{for all }\TT\in\T
\end{align*}
imply the equivalence~\eqref{eq:approxclassbest}.
The characterization~\eqref{eq:approxclassckns} follows from the equivalence
\begin{align}\label{eq:blabla}
 \inf_{V\in\XX(\TT)}
 \big(\dist[\TT]{u}{V} + \eff{\TT}{V}\big)\simeq \eta(\TT;U(\TT))\quad\text{for all }\TT\in\T.
\end{align}
To prove~\eqref{eq:blabla}, the efficiency~\eqref{A:efficient} as well as the C\'ea lemma~\eqref{eq:cea} and~\eqref{eq:oscstab} lead to
\begin{align*}
\c{efficient}^{-1} \eta(\TT;U(\TT))&\leq \dist[\TT]{u}{U(\TT)} + \eff{\TT}{U(\TT)}\\
 &\leq \c{cea}\dist[\TT]{u}{V} + C_\osc\eff{\TT}{V}+C_\osc\dist[\TT]{U(\TT)}{V}\\
 &\leq (\c{cea}+ C_\osc\c{triangle}(\c{cea}+1))\dist[\TT]{u}{V} + C_\osc\eff{\TT}{V}
\end{align*}
for all $V\in\XX(\TT)$. The converse direction follows with reliability~\eqref{A:reliable} and~\eqref{eq:oscdom} via
\begin{align*}
 \inf_{V\in\XX(\TT)}
 \big(\dist[\TT]{u}{V} + \eff{\TT}{V}\big)\leq (\c{reliable}+C_\osc)\eta(\TT;U(\TT)). 
\end{align*}
This concludes the proof of~\eqref{eq:blabla} and of the proposition.
\end{proof}

Under certain assumptions on the oscillations $\osc(\cdot)$, the best possible rate for the estimator is characterized by the best possible
rate for the error. The following theorem shows that the adaptive algorithm reduces the error with the optimal rate and therefore at least as good as any
other algorithm which uses the same mesh-refinement.

\begin{theorem}\label{thm:mainerr}
Suppose \eqref{A:stable}--\eqref{A:dlr} as well as efficiency~\eqref{A:efficient} and quasi-monotonicity of oscillations and error in the sense that there exists a constant $C_{\rm emon}>0$ such that any $\TT\in\T$ and its refinements $\widehat\TT\in\T$ satisfy
\begin{align}\label{eq:errorquasimon}
\dist[\widehat\TT]{u}{U(\widehat\TT)}\leq C_{\rm emon}\dist[\TT]{u}{U(\TT)}\quad\text{and}\quad \eff{\widehat\TT}{U(\widehat\TT)}\leq C_{\rm emon}\eff{\TT}{U(\TT)}.
\end{align}%
Then, $ 0<\theta<\theta_\star:=(1+\c{stable}^2\c{dlr}^2)^{-1}$ implies quasi-optimal convergence of the error
\begin{align}\label{eq:thm:mainerr}
\begin{split}
 c_{\rm opt}\c{reliable}^{-1}\c{efficient}^{-1}\norm{(u,U(\cdot))}{\A_s}&\leq\sup_{\ell\in\N_0}\frac{\dist[\TT_\ell]{u}{U(\TT_\ell)}}{(|\TT_\ell| - |\TT_0|+1)^{-s}}+\norm{\osc(\cdot)}{\O_s}\\
 &\leq (\c{optimal}\c{reliable}C_{\rm apx}+1)(\norm{(u,U(\cdot))}{\A_s}+ \norm{\osc(\cdot)}{\O_s})
 \end{split}
\end{align}
for all $s>0$.
The constants $c_{\rm opt},\setc{optimal}>0$ are defined in Theorem~\ref{thm:main}~(ii), whereas the constant $C_{\rm apx}>0$ is defined in the
following Proposition~\ref{prop:charAprox}.
\end{theorem}

The proof \revision{of Theorem~\ref{thm:mainerr}} needs a relation of $\norm{\cdot}{\A_s}$ and $\norm{\cdot}{\B_s}$, which is given in the following proposition.

\begin{proposition}\label{prop:charAprox}
Suppose reliability~\eqref{A:reliable}, efficiency~\eqref{A:efficient}, quasi-monotonicity of the estimator~\eqref{eq:mon}, and 
quasi-monotonicity of error and oscillations~\eqref{eq:errorquasimon}.
Then,
\begin{align*}
\c{reliable}^{-1}\norm{(u,U(\cdot))}{\A_s}\leq \norm{(\eta(\cdot),U(\cdot))}{\B_s}\leq C_{\rm apx}(\norm{(u,U(\cdot))}{\A_s} + \norm{\osc(\cdot)}{\O_s}).
\end{align*}
holds for all $s>0$ with a constant $C_{\rm apx}>0$ which depends only on $C_{emon}$, $\c{efficient}$, and the validity of the overlay estimate~\eqref{refinement:overlay}.
\end{proposition}

\begin{proof}
 The reliability~\eqref{A:reliable} \revision{guarantees}
 \begin{align*}
 \norm{(u,U(\cdot))}{\A_s} \leq  \c{reliable} \norm{(\eta(\cdot),U(\cdot))}{\B_s}.
\end{align*}
Suppose $\norm{(u,U(\cdot))}{\A_s}+\norm{\osc(\cdot)}{\O_s}<\infty$ for some $s>0$. For any even $N\in\N_0$, this guarantees the existence of a triangulation $\TT_{N/2}\in\T(N/2)$ with
\begin{align*}
\dist[\TT_{N/2}]{u}{U(\TT_{N/2})}(N/2+1)^s\leq \norm{(u,U(\cdot))}{\A_s}
\end{align*}
and also the existence of a triangulation $\TT_{\rm osc}\in\T(N/2)$ with
\begin{align}\label{eq:efficienthot}
(N/2+1)^s\eff{\TT_\osc}{U(\TT_\osc)}\leq \norm{\osc(\cdot)}{\O_s}.
\end{align}
With monotonicity~\eqref{eq:errorquasimon}, the overlay $\TT_+:=\TT_{N/2}\oplus \TT_{\rm osc} \revision{\in \T(N)}$ satisfies 
\begin{align*}
\dist[\TT_+]{u}{U(\TT_+)}&\lesssim\dist[\TT_{N/2}]{u}{U(\TT_{N/2})}\lesssim (N/2+1)^{-s}\norm{(u,U(\cdot))}{\A_s},\\
\eff{\TT_+}{U(\TT_+)}&\lesssim \eff{\TT_\osc}{U(\TT_\osc)}\lesssim (N/2+1)^{-s}\norm{\osc(\cdot)}{\O_s}.
\end{align*}
This yields \revision{(with $2^{2s}\lesssim 1$) that}
\begin{align*}
(N+1)^{2s}\big(\dist[\TT_+]{u}{U(\TT_+)}^2+\eff{\TT_+}{U(\TT_+)}^2\big)
\lesssim \norm{(u,U(\cdot))}{\A_s}^2+\norm{\osc(\cdot)}{\O_s}^2.
 \end{align*}
 \revision{The efficiency~\eqref{A:efficient}} leads to
 \begin{align}
 \eta(\TT_+;U(\TT_+))^2\lesssim \dist[\TT_+]{u}{U(\TT_+)}^2+\eff{\TT_+}{U(\TT_+)}^2.
 \end{align}
Together with the previous estimate, this proves
 \begin{align}
 (N+1)^{2s} \eta(\TT_+;U(\TT_+))^2\lesssim  \norm{(u,U(\cdot))}{\A_s}^2+\norm{\osc(\cdot)}{\O_s}^2.
 \end{align}
 The overlay estimate~\eqref{refinement:overlay} finally yields $|\TT_+| -|\TT_0|\leq |\TT_{N/2}| + |\TT_{\rm osc}| -2|\TT_0| \leq N$. This proves $\norm{(\eta(\cdot),U(\cdot)}{\B_s}\lesssim \norm{(u,U(\cdot))}{\A_s}\revision{+\norm{\osc(\cdot)}{\O_s}}$.
\end{proof}

\begin{proof}[Proof of Theorem~\ref{thm:mainerr}]
According to Lemma~\ref{lem:mon1}, $\eta(\cdot)$ is quasi-monotone~\eqref{eq:mon}. Therefore all the claims of Proposition~\ref{prop:charAprox} are satisfied. Together with Theorem~\ref{thm:main}~(ii) \revision{(which will be proven at the very end of this section independently of this),} this shows
\begin{align*}
c_{\rm opt}\c{reliable}^{-1}\norm{(u,U(\cdot))}{\A_s} \leq\sup_{\ell\in\N_0}\frac{\eta(\TT_\ell;U(\TT_\ell))}{(|\TT_\ell|-|\TT_0|+1)^{-s}}\leq\c{optimal}C_{\rm apx}(\norm{(u,U(\cdot))}{\A_s}+\norm{\osc(\cdot)}{\O_s})
\end{align*}
Reliability~\eqref{A:reliable} implies
\begin{align*}
\sup_{\ell\in\N_0}\frac{\dist[\TT_\ell]{u}{U(\TT_\ell)}}{(|\TT_\ell|-|\TT_0|+1)^{-s}}\leq \c{reliable}\sup_{\ell\in\N_0}\frac{\eta(\TT_\ell;U(\TT_\ell))}{(|\TT_\ell|-|\TT_0|+1)^{-s}},
\end{align*}
whereas efficiency~\eqref{A:efficient} leads to
\begin{align*}
 \c{efficient}^{-1}\sup_{\ell\in\N_0}\frac{\eta(\TT_\ell;U(\TT_\ell))}{(|\TT_\ell|-|\TT_0|+1)^{-s}}\leq\sup_{\ell\in\N_0}\frac{\dist[\TT_\ell]{u}{U(\TT_\ell)}}{(|\TT_\ell|-|\TT_0|+1)^{-s}}
+\norm{\osc(\cdot)}{\O_s}.
\end{align*}
The combination of the last three estimates proves the assertion.
\end{proof}

\subsection{Estimator reduction and convergence of $\eta(\TT_\ell;U(\TT_\ell))$}
\label{section:reduction}
We start with the observation that stability~\eqref{A:stable} and reduction~\eqref{A:reduction}
lead to a perturbed contraction of the error estimator in each step of the adaptive loop.

\begin{lemma}\label{lem:estconv}
The stability~\eqref{A:stable} and reduction~\eqref{A:reduction} imply \revision{the estimator reduction}
\begin{align}\label{eq:estconv}
 \eta(\TT_{\ell+1};U(\TT_{\ell+1}))^2 \le \q{estconv}\,\eta(\TT_\ell;U(\TT_\ell))^2 + \c{estconv}\,\dist[\TT_{\ell+1}]{U(\TT_{\ell+1})}{U(\TT_\ell)}^2
\end{align}
for all $\ell\in\N_0$ with the constants $0<\setq{estconv}<1$ and $\setc{estconv}>0$ which relate via
\begin{align}\label{eq:estconvconst}
 \q{estconv}=(1+\delta)(1-(1-\q{reduction})\theta)\quad\text{and}\quad\c{estconv}=\c{reduction}+(1+\delta^{-1})\c{stable}^2
\end{align}
for all \revision{sufficiently small} $\delta>0$ such that $\q{estconv}<1$.
\end{lemma}

\begin{proof}
The Young inequality in combination with stability~\eqref{A:stable} and reduction~\eqref{A:reduction} 
shows for any $\delta >0$ and $\c{estconv} = \c{reduction} + (1+\delta^{-1})\c{stable}^2$ that
\begin{align*}
 \eta(\TT_{\ell+1};U(\TT_{\ell+1}))^2
 &= \sum_{T\in\TT_{\ell+1}\backslash\TT_\ell}\eta_T(\TT_{\ell+1};U(\TT_{\ell+1}))^2
 + \sum_{T\in\TT_{\ell+1}\cap\TT_\ell}\eta_T(\TT_{\ell+1};U(\TT_{\ell+1}))^2\\
 &\le \q{reduction}\sum_{T\in\TT_{\ell}\backslash\TT_{\ell+1}}\eta_T(\TT_\ell;U(\TT_\ell))^2
+ (1+\delta)\sum_{T\in\TT_\ell\cap\TT_{\ell+1}}\eta_T(\TT_\ell;U(\TT_\ell))^2
 \\&\qquad+ \c{estconv}\dist[\TT_{\ell+1}]{U(\TT_{\ell+1})}{U(\TT_\ell)}^2.
\end{align*}
Therefore,  the inclusion $\MM_\ell\subseteq\TT_{\ell}\backslash\TT_{\ell+1}$ and the D\"orfler marking~\eqref{eq:doerfler} lead to
\begin{align*}
 \eta(\TT_{\ell+1};U(\TT_{\ell+1}))^2&\leq(1+\delta)\Big(\eta(\TT_\ell;U(\TT_\ell))^2
 -(1-\q{reduction})\sum_{T\in\TT_{\ell}\backslash\TT_{\ell+1}}\eta_T(\TT_\ell;U(\TT_\ell))^2\Big)\\
&\qquad + \c{estconv}\dist[\TT_{\ell+1}]{U(\TT_{\ell+1})}{U(\TT_\ell)}^2\\
 &\le (1+\delta)\big(1-(1-\q{reduction})\theta\big)\eta(\TT_\ell;U(\TT_\ell))^2
 + \c{estconv}\dist[\TT_{\ell+1}]{U(\TT_{\ell+1})}{U(\TT_\ell)}^2.
\end{align*}
The choice of a  sufficiently small $\delta>0$ allows for $\q{estconv}=(1+\delta)\big(1-(1-\q{reduction})\theta\big)<1$.
\end{proof}

In particular situations (e.g.\ in Section~\ref{section:conforming}) the sequence of discrete approximations is
a~priori convergent towards some limit $U_\infty\in\XX$
\begin{align}\label{dpr:aprioriconv}
 \lim_{\ell\to\infty}\dist[\TT_\ell]{U_\infty}{U(\TT_\ell)} = 0.
\end{align}
Then, the \emph{estimator reduction}~\eqref{eq:estconv} implies
convergence of the adaptive algorithm.
This \emph{estimator reduction concept} is studied in~\cite{estconv} and applies to a general class of problems and error estimators.

\begin{corollary}\label{cor:convergence}
Suppose a priori convergence~\eqref{dpr:aprioriconv} in $\XX$. Then the estimator reduction~\eqref{eq:estconv} implies estimator convergence $\lim_{\ell\to\infty}\eta(\TT_\ell;U(\TT_\ell))=0$. Under reliability~\eqref{A:reliable}, this proves convergence of the adaptive algorithm $\lim_{\ell\to\infty}\dist[\TT_\ell]{u}{U(\TT_\ell)}=0$.
\end{corollary}

\begin{proof}
For a convenient reading, we recall the main arguments of~\cite[Lemma 2.3]{estconv} in the notation of this paper.
Mathematical induction on $\ell$ proves with~\eqref{eq:estconv} for all $\ell\in\N$
\begin{align}\label{eq:aprioriconvhelp}
\begin{split}
 \eta(\TT_{\ell+1};U(\TT_{\ell+1}))^2&\leq \q{estconv}^{\ell+1}\eta(\TT_0;U(\TT_0))^2 + \c{estconv}\sum_{j=0}^\ell\q{estconv}^{\ell-j}\dist[\TT_{\ell+1}]{U(\TT_{\ell+1})}{U(\TT_\ell)}^2\\
 &\lesssim \eta(\TT_0;U(\TT_0))^2+\sup_{\ell\in\N}\dist[\TT_{\ell+1}]{U(\TT_{\ell+1})}{U(\TT_\ell)}^2.
 \end{split}
\end{align}
The a~priori convergence of $U(\TT_\ell)$ implies $\dist[\TT_{\ell+1}]{U(\TT_{\ell+1})}{U(\TT_\ell)} \to 0$ and hence shows together with~\eqref{eq:aprioriconvhelp} that $\sup_{\ell\in\N}\eta(\TT_\ell;U(\TT_\ell))<\infty$. Moreover,~\eqref{eq:estconv} yields
\begin{align*}
 \lim\sup_{\ell\to\infty}\eta(\TT_{\ell+1};U(\TT_{\ell+1}))^2 &\le \lim\sup_{\ell\to\infty}\big(\q{estconv}\,\eta(\TT_\ell;U(\TT_\ell))^2 + \c{estconv}\,\dist[\TT_{\ell+1}]{U(\TT_{\ell+1})}{U(\TT_\ell)}^2\big)\\
 &=\q{estconv}\lim\sup_{\ell\to\infty}\eta(\TT_{\ell+1};U(\TT_{\ell+1}))^2.
\end{align*}
This shows $ \lim\sup_{\ell\to\infty}\eta(\TT_{\ell};U(\TT_{\ell}))^2=0$, and hence elementary calculus proves convergence $\eta(\TT_\ell;U(\TT_\ell))\to 0$. Under reliability~\eqref{A:reliable} this implies $U(\TT_\ell)\to u$ in $\XX$.
\end{proof}

\subsection{{Uniform $R$-linear convergence of $\eta(\TT_\ell;U(\TT_\ell))$ on any level}}
The quasi-orthogonality~\eqref{A:qosum} allows to improve~\eqref{eq:estconv} to $R$-linear convergence on any level. 
The following lemma is independent of the mesh-refinement in the sense that the critical properties~\eqref{refinement:closure}--\eqref{refinement:overlay} are not used throughout the proof. It thus remains valid e.g.\ for red-green-blue refinement.
\begin{lemma}\label{lem:Rconv}
The statements~{\rm (i)}--{\rm (iii)} are pairwise equivalent.
\renewcommand{\theenumi}{{\rm\roman{enumi}}}%
\renewcommand{\labelenumi}{({\rm\roman{enumi}})}%
\begin{enumerate}
 \item\label{contraction:estsum} Uniform summability: There exists a constant $\setc{estsum}>0$ such that
\begin{align}\label{eq:estsum}
 \sum_{k=\ell+1}^\infty\eta(\TT_k;U(\TT_k))^2\leq \c{estsum} \eta(\TT_\ell;U(\TT_\ell))^2 \quad \text{for all } \ell\in\N.
\end{align}
\item\label{contraction:invsum} Inverse  summability: \revision{For all $s>0$,} there exists a constant $\setc{invsum}>0$ such that
\begin{align}\label{eq:endest}
\sum_{k=0}^{\ell-1}\eta(\TT_k;U(\TT_k))^{-1/s} \leq \c{invsum} \eta(\TT_\ell;U(\TT_\ell))^{-1/s}\quad \text{for all } \ell\in\N.
\end{align}
\item\label{contraction:Rconv} Uniform $R$-linear convergence on any level: There exist constants $0<\setr{Rconv}<1$ and $\setc{Rconv}>0$ such that
\begin{align}\label{eq:propRconv}
 \eta(\TT_{\ell+k};U(\TT_{\ell+k}))^2\leq \c{Rconv}\r{Rconv}^k\,\eta(\TT_\ell;U(\TT_\ell))^2\quad\text{for all }k,\ell\in\revision{\N_0.}
\end{align}
\end{enumerate}
\end{lemma}
\begin{proof}
For sake of simplicity, we show the equivalence of~\eqref{contraction:estsum}--\eqref{contraction:Rconv} by proving the equivalences
$\eqref{contraction:Rconv}\Longleftrightarrow \eqref{contraction:estsum}$ 
and $\eqref{contraction:Rconv}\Longleftrightarrow \eqref{contraction:invsum}$.

For the proof of the implication $\eqref{contraction:Rconv} \Rightarrow \eqref{contraction:estsum}$, suppose~\eqref{contraction:Rconv} and use the convergence of the geometric series to see
\begin{align*}
  \sum_{k= \ell+1}^\infty\eta(\TT_k;U(\TT_k))^2\leq\c{Rconv}\eta(\TT_\ell;U(\TT_\ell))^2\sum_{k=\ell+1}^\infty\r{Rconv}^{k-\ell}= \c{Rconv}\r{Rconv}(1-\r{Rconv})^{-1} \eta(\TT_\ell;U(\TT_\ell))^2.
\end{align*}
This proves~\eqref{contraction:estsum} with $\c{estsum}=\c{Rconv}\r{Rconv}(1-\r{Rconv})^{-1}$.

Similarly, the implication $\eqref{contraction:Rconv} \Rightarrow \eqref{contraction:invsum}$ follows via
\begin{align*}
  \sum_{k=0}^{\ell-1}\eta(\TT_k;U(\TT_k))^{-1/s}&\leq\c{Rconv}^{1/(2s)}\eta(\TT_\ell;U(\TT_\ell))^{-1/s}\sum_{k=0}^{\ell-1}\r{Rconv}^{(\ell-k)/(2s)}\\
  &\leq \c{Rconv}^{1/(2s)}(1-\r{Rconv}^{1/(2s)})^{-1}\eta(\TT_\ell;U(\TT_\ell))^{-1/s}.
\end{align*}
This shows~\eqref{contraction:invsum} with $\c{invsum}=\c{Rconv}^{1/(2s)}(1-\r{Rconv}^{1/(2s)})^{-1}$.

For the proof of the implication $\eqref{contraction:estsum} \Rightarrow \eqref{contraction:Rconv}$, suppose~\eqref{contraction:estsum} and conclude
\begin{align*}
 (1+\c{estsum}^{-1})\sum_{j= \ell+1}^\infty\eta(\TT_j;U(\TT_j))^2 \leq \sum_{j= \ell+1}^\infty\eta(\TT_j;U(\TT_j))^2+ \eta(\TT_\ell;U(\TT_\ell))^2= \sum_{j= \ell}^\infty\eta(\TT_j;U(\TT_j))^2.
\end{align*}
By mathematical induction, this implies
\begin{align*}
 \eta(\TT_{\ell+k};U(\TT_{\ell+k}))^2 &\leq \sum_{j=\ell+k}^\infty\eta(\TT_j;U(\TT_j))^2 \leq (1+\c{estsum}^{-1})^{-k}\sum_{j=\ell}^\infty\eta(\TT_j;U(\TT_j))^2\\
 & \leq (1+\c{estsum})(1+\c{estsum}^{-1})^{-k}\eta(\TT_\ell;U(\TT_\ell))^2.
\end{align*}
This proves~\eqref{contraction:Rconv} with $\r{Rconv}=(1+\c{estsum}^{-1})^{-1}$ and $\c{Rconv}=(1+\c{estsum})$.

The implication  $\eqref{contraction:invsum} \Rightarrow \eqref{contraction:Rconv}$ follows analogously,
\begin{align*}
 (1+\c{invsum}^{-1})\sum_{j=0}^{\ell-1}\eta(\TT_j;U(\TT_j))^{-1/s} \leq  \sum_{j=0}^\ell\eta(\TT_j;U(\TT_j))^{-1/s}.
\end{align*}
This implies
\begin{align*}
 \eta(\TT_\ell;U(\TT_\ell))^{-1/s} &\leq \sum_{j=0}^\ell\eta(\TT_j;U(\TT_j))^{-1/s} \leq (1+\c{invsum}^{-1})^{-k}\sum_{j=0}^{\ell+k}\eta(\TT_{j};U(\TT_{j}))^{-1/s}\\
 &\leq
(1+\c{invsum})(1+\c{invsum}^{-1})^{-k}\eta(\TT_{\ell+k};U(\TT_{\ell+k}))^{-1/s}.
\end{align*}
This proves $\eta(\TT_{\ell+k};U(\TT_{\ell+k}))^2\leq (1+\c{invsum})^{2s}(1+\c{invsum}^{-1})^{-2sk}\eta(\TT_\ell;U(\TT_\ell))^2$. This is~\eqref{contraction:Rconv} with $\r{Rconv}=(1+\c{invsum}^{-1})^{-2s}$ and $\c{Rconv}=(1+\c{invsum})^{2s}$.
\end{proof}

\begin{proposition}\label{prop:Rconv}
Suppose estimator reduction~\eqref{eq:estconv} and reliability~\eqref{A:reliable}. Then, general quasi-ortho\-go\-nality~\eqref{A:qosum} implies~\eqref{eq:estsum}--\eqref{eq:propRconv}.
The constants $\c{estsum},\c{invsum},\c{Rconv}>0$ and $0<\r{Rconv}<1$ depend only on $\q{estconv},\c{estconv},\c{qosum}(\epsqo),s>0$.
\end{proposition}

\begin{proof}
In the following, the general quasi-orthogonality~\eqref{A:qosum} implies~\eqref{eq:estsum}--\eqref{eq:propRconv} since~\eqref{A:qosum} implies~\eqref{eq:estsum}. To that end, the estimator reduction~\eqref{eq:estconv} from Lemma~\ref{lem:estconv} yields for any $\nu>0$ that
\begin{align*}
  \sum_{k= \ell+1}^N\eta(\TT_k;U(\TT_k))^2&\leq \sum_{k=\ell+1}^N\big(\q{estconv}\eta(\TT_{k-1};U(\TT_{k-1}))^2+ \c{estconv}\dist[\TT_k]{U(\TT_k)}{U(\TT_{k-1})}^2\big)\\
&=\sum_{k=\ell+1}^N\Big((\q{estconv}+\nu)\eta(\TT_{k-1};U(\TT_{k-1}))^2\\
&\quad\quad+ \c{estconv}\big(\dist[\TT_k]{U(\TT_k)}{U(\TT_{k-1})}^2
-\nu\c{estconv}^{-1}\eta(\TT_{k-1};U(\TT_{k-1}))^2\big)\Big)\\
&=:{\rm RHS}.
\end{align*}
The use of reliability~\eqref{A:reliable} then shows 
\begin{align*}
{\rm RHS}\leq&\sum_{k=\ell+1}^N\Big((\q{estconv}+\nu)\eta(\TT_{k-1};U(\TT_{k-1}))^2\\
&\quad\quad+ \c{estconv}\big(\dist[\TT_k]{U(\TT_k)}{U(\TT_{k-1})}^2-\nu\c{estconv}^{-1}\c{reliable}^{-2}\dist[\TT_{k-1}]{u}{U(\TT_{k-1})}^2\big)\Big).
\end{align*}
With the constants $\q{estconv}$ and $\c{estconv}$ from~\eqref{eq:estconvconst}, the constraint on $\epsqo$ in~\eqref{A:qosum} reads
\begin{align*}
 0\leq \epsqo< \frac{1-\q{estconv}}{\c{reliable}^{2}\c{estconv}}=\frac{1-(1+\delta)(1-(1-\q{reduction})\theta}{\c{reliable}^{2}(\c{reduction}+(1+\delta^{-1})\c{stable}^2)}\leq \eps_{\rm qo}^\star
\end{align*}
for some choice of $\delta>0$. Note that this choice is valid since $\q{estconv}<1$. In particular, it exists $\nu<1-\q{estconv}$ such that $\epsqo\leq\nu\c{estconv}^{-1}\c{reliable}^{-2}$. This allows to apply general quasi-orthogonality~\eqref{A:qosum} to the last term before \revision{the limit} $N\to\infty$ \revision{proves} that
\begin{align*}
  \sum_{k=\ell+1}^\infty\eta(\TT_k;U(\TT_k))^2
\leq\sum_{k=\ell+1}^\infty(\q{estconv}+\nu)\eta(\TT_{k-1};U(\TT_{k-1}))^2+ \c{estconv}\c{qosum}(\epsqo)\eta(\TT_\ell;U(\TT_\ell))^2.
\end{align*}
Some rearrangement leads to
\begin{align*}
 (1-(\q{estconv}+\nu)) \sum_{k=\ell+1}^\infty\eta(\TT_k;U(\TT_k))^2\leq (\q{estconv}+\nu+\c{estconv}\c{qosum}(\epsqo))\eta(\TT_\ell;U(\TT_\ell))^2.
\end{align*}
This is~\eqref{eq:estsum} with $\c{estsum}=(\q{estconv}+\nu+\c{estconv}\c{qosum}(\epsqo))/(1-(\q{estconv}+\nu))$ and concludes the proof of \eqref{A:qosum} $\Rightarrow$ \eqref{eq:estsum}. Lemma~\ref{lem:Rconv} yields the equivalence~\eqref{eq:estsum}--\eqref{eq:propRconv}.
\end{proof}

Assume that~\eqref{A:stable}--\eqref{A:reduction} and reliability~\eqref{A:reliable} hold. The last proposition then proves that the quasi-orthogonality~\eqref{A:qosum} yields linear convergence~\eqref{eq:propRconv}. The following proposition shows that under the same assumptions, linear convergence~\eqref{eq:propRconv} implies the general quasi-orthogonality~\eqref{A:qosum}. 
\revision{This} means that linear convergence~\eqref{eq:propRconv} is equivalent to general quasi-orthogonality~\eqref{A:qosum}.

\begin{proposition}\label{prop:qosumiscool}
Reliability~\eqref{A:reliable} and each of the statements~\eqref{eq:estsum}--\eqref{eq:propRconv} imply general quasi-ortho\-go\-nality~\eqref{A:qosum} with $\epsqo=0$ and $\c{qosum}(0)>0$. 
\end{proposition}
\begin{proof}
With reliability~\eqref{A:reliable} and~\eqref{eq:estsum}, it holds 
\begin{align*}
\sum_{k=\ell}^N \dist[\TT_{k+1}]{U(\TT_{k+1})}{U(\TT_k)}^2&\lesssim\sum_{k=\ell}^N \dist[\TT_{k+1}]{u}{U(\TT_{k+1})}^2
+\dist[\TT_{k}]{u}{U(\TT_{k})}^2\\
&\leq 2 \sum_{k=\ell}^{N+1}\dist[\TT_{k}]{u}{U(\TT_{k})}^2\\
& \lesssim\sum_{k=\ell}^{N+1}\eta(\TT_k;U(\TT_k))^2\lesssim \eta(\TT_\ell;U(\TT_\ell))^2
\end{align*}
for all $\ell,N\in\N_0$. Let $N\to\infty$ to conclude~\eqref{A:qosum} with $\epsqo=0$ and $\c{qosum}(0)\simeq 1$.
\end{proof}


\begin{proof}[Proof of Theorem~\ref{thm:main},~(i)]
Reliability~\eqref{A:reliable} is implied by discrete reliability~\eqref{A:dlr} according to Lemma~\ref{lemma:dlr2rel}.
 Stability~\eqref{A:stable} and reduction~\eqref{A:reduction} \revision{guarantee} estimator reduction~\eqref{eq:estconv} for $\eta(\TT_\ell;U(\TT_\ell))$. Together with quasi-orthogonality~\eqref{A:qosum} and reliability~\eqref{A:reliable}, this allows to apply~\eqref{eq:propRconv}. In combination with reliability~\eqref{A:reliable}, \revision{this} proves Theorem~\ref{thm:main}~(i) with $\c{conv} = \c{Rconv}$ and $\r{conv} = \r{Rconv}$.
\end{proof}

\subsection{Optimality of D\"orfler marking}
\label{section:doerfler}
Lemma~\ref{lem:estconv} and Proposition~\ref{prop:Rconv} prove that D\"orfler marking~\eqref{eq:doerfler} essentially guarantees the (perturbed) contraction properties~\eqref{eq:estconv} and~\eqref{eq:estsum}--\eqref{eq:propRconv} and hence $\lim_{\ell\to\infty}\eta(\TT_\ell;U(\TT_\ell)) = 0$. The next statement asserts the converse.

\begin{proposition}\label{prop:doerfler}
Stability~\eqref{A:stable} and discrete reliability~\eqref{A:dlr} imply {\rm (i)--(ii)}.
\begin{itemize}
\item[{\rm(i)}] For all $0<\kappa_0<1$, there exists a constant $0<\theta_0<1$ such that all 
$0<\theta\le\theta_0$ and all refinements $\widehat\TT\in\T$ of $\TT\in\T$ satisfy
\begin{align}\label{eq:doerfler:opt}
 \eta(\widehat\TT;U(\widehat\TT))^2 \le \kappa_0\eta(\TT;U(\TT))^2
 \quad\Longrightarrow\quad
 \theta\,\eta(\TT;U(\TT))^2
 \le \sum_{T\in\RR(\TT,\widehat\TT)}\eta_T(\TT;U(\TT))^2
\end{align}
with $\TT\setminus\widehat\TT \subseteq\RR(\TT,\widehat\TT)\subseteq\TT$ from~\eqref{A:dlr}.
The constant $\theta_0$ depends only on $\c{stable}$, $\c{dlr}$ and $\kappa_0$.
\item[{\rm(ii)}] For all $0<\theta_0<\theta_\star:=(1+\c{stable}^2\c{dlr}^2)^{-1}$, there 
exists some $0<\kappa_0<1$ such that~\eqref{eq:doerfler:opt} holds for all $0<\theta\le\theta_0$ and all refinements $\widehat\TT$ of $\TT\in\T$.
The constant $\kappa_0$ depends only on $\c{stable}$, $\c{dlr}$, and $\theta_0$.
\end{itemize}
\end{proposition}

\begin{proof}
(i):
The Young inequality and stability~\eqref{A:stable} 
 show, for any $\delta>0$, that
\begin{align*}
 \eta(\TT;U(\TT))^2
 &= \sum_{T\in\TT\backslash\widehat\TT}\eta_T(\TT;U(\TT))^2
 + \sum_{T\in\TT\cap\widehat\TT}\eta_T(\TT;U(\TT))^2\\
 & \le \sum_{T\in\TT\backslash\widehat\TT}\eta_T(\TT;U(\TT))^2
 + (1+\delta)\sum_{T\in\TT\cap\widehat\TT}\eta_T(\widehat\TT;U(\widehat\TT))^2\\
 &\quad\quad
 + (1+\delta^{-1})\c{stable}^2\dist[\widehat\TT]{U(\widehat\TT)}{U(\TT)}^2 =: \text{RHS}.
 \end{align*}
 Recall $\TT\backslash\widehat\TT\subseteq\RR(\TT,\widehat\TT)$.
The application of the discrete reliability~\eqref{A:dlr} and the assumption $\eta(\widehat\TT;U(\widehat\TT))^2 \le \kappa_0\eta(\TT;U(\TT))^2$ yield
 \begin{align*}
\text{RHS} &\le (1+\delta)\kappa_0\eta(\TT;U(\TT))^2\\
 &\quad\quad + \big(1+(1+\delta^{-1})\c{stable}^2\c{dlr}^2\big)\,\sum_{T\in\RR(\TT,\widehat\TT)}\eta_T(\TT;U(\TT))^2.
\end{align*}
Some rearrangement of those terms reads
\begin{align*}
 \frac{1-(1+\delta)\kappa_0}{1+(1+\delta^{-1})\c{stable}^2\c{dlr}^2}\,\eta(\TT;U(\TT))^2
 \le \sum_{T\in\RR(\TT,\widehat\TT)}\eta_T(\TT;U(\TT))^2.
\end{align*}
For arbitrary $0<\kappa_0<1$ and sufficiently small $\delta>0$, this is~\eqref{eq:doerfler:opt} with
\begin{align}\label{eq:thetastar}
 \theta_0:= \frac{1-(1+\delta)\kappa_0}{1+(1+\delta^{-1})\c{stable}^2\c{dlr}^2}>0.
\end{align}
To see~(ii), choose $\delta>0$ sufficiently large and then determine $0<\kappa_0<1$ such that
\begin{align*}
 \theta_0 
 = \frac{1-(1+\delta)\kappa_0}{1+(1+\delta^{-1})\c{stable}^2\c{dlr}^2}
 < \frac{1}{1+(1+\delta^{-1})\c{stable}^2\c{dlr}^2}
 < \frac{1}{1+\c{stable}^2\c{dlr}^2} = \theta_\star.
\end{align*}
The arguments from~(i) conclude the proof.
\end{proof}

\begin{remark}
 Note that Proposition~\ref{prop:doerfler} states~\eqref{eq:doerfler:opt} for all $0<\kappa_0<1$. However, the subsequent quasi-optimality analysis relies, in principle, only on the fact that~\eqref{eq:doerfler:opt} holds for one particular $0<\kappa_0<1$. In this sense, the discrete reliability is sufficient to prove quasi-optimal convergence rates, but \revision{it} might not be necessary.

On the other hand, assume that the error estimator $\eta(\TT;U(\TT))$ is reliable~\eqref{A:reliable} and quasi-monotone \eqref{eq:mon}.
Then, the D\"orfler marking yields
\begin{align*}
\frac12\,\c{triangle}^{-2}\,
\dist[\widehat\TT]{U(\widehat\TT)}{U(\TT)}^2
 &\le \c{triangle}^2\dist[\widehat\TT]{u}{U(\widehat\TT)}^2+ \dist[\TT]{u}{U(\TT)}^2\\
 &\le\c{reliable}^2\,\big( \c{triangle}^2\eta(\widehat\TT;U(\widehat\TT))^2+\eta(\TT;U(\TT))^2\big)\\
 &\le \c{reliable}^2(1+ \c{triangle}^2\c{mon}^2)\,\eta(\TT;U(\TT))^2\\
&\leq\c{reliable}^2(1+ \c{triangle}^2\c{mon}^2)\theta^{-1}\sum_{T\in\MM}\eta_T(\TT;U(\TT))^2.
\end{align*}
The inclusion $\MM \subseteq \TT\backslash\widehat\TT =: \RR(\TT,\widehat\TT)$ thus shows that
discrete reliability~\eqref{A:dlr} holds with the constant $\c{dlr}^2 = 2\c{triangle}^2\theta^{-1}\c{reliable}^2(1+ \c{triangle}^2\c{mon}^2)$.
\end{remark}

\subsection{Quasi-optimality of adaptive algorithm}
This section provides quasi-optimal convergence rates for the estimator and \revision{thereby} the theoretical heart of the proof of Theorem~\ref{thm:main}~(ii). The first lemma states the existence of a quasi-optimal refinement $\widehat \TT$ of $\TT_\ell$ under certain assumptions guaranteed by Lemma~\ref{lem:mon1} in case that the estimator satisfies the axioms stability~\eqref{A:stable}, reduction~\eqref{A:reduction}, and discrete reliability~\eqref{A:dlr}. For sake of generality, however, the next statement is given independently of this context.
This step exploits the overlay 
estimate~\eqref{refinement:overlay}
for the mesh-refinement.
%
\begin{lemma}\label{lem:optimality}
Assume that the estimator is quasi-monotone~\eqref{eq:mon} and that the implication~\eqref{eq:doerfler:opt} is valid for one 
particular choice of $0<\kappa_0,\theta_0<1$. Then,
for $\norm{(\eta(\cdot),U(\cdot))}{\B_s}<\infty$ and $\TT_\ell \in \T$, there is a certain refinement $\widehat\TT\in\T$ of $\TT_\ell$ with
\begin{subequations}\label{eq:lemoptresult1}
\begin{align}
\eta(\widehat\TT;U(\widehat\TT))^2 &\leq  \kappa_0 \eta(\TT_\ell;U(\TT_\ell))^2,
\\
|\widehat\TT| - |\TT_\ell| &\leq \c{lemopthelp2} \norm{(\eta(\cdot),U(\cdot))}{\B_s}^{1/s}\,\eta(\TT_\ell;U(\TT_\ell))^{-1/s},
\end{align}
\end{subequations}%
where the set $\RR(\TT_\ell,\widehat\TT)\supseteq \TT_\ell\setminus\widehat\TT$ from Proposition~\ref{prop:doerfler} satisfies
\begin{align}\label{eq:lemoptresult2}
|\RR(\TT_\ell,\widehat\TT)|\leq \c{lemopthelp2} \norm{(\eta(\cdot),U(\cdot))}{\B_s}^{1/s}\,\eta(\TT_\ell;U(\TT_\ell))^{-1/s}
\end{align}
as well as the D\"orfler marking~\eqref{eq:doerfler} for all $0<\theta\leq \theta_0$. The constant $\setc{lemopthelp2}>0$ is independent of $\ell$ and depends only on
the constant $\c{mon}>0$ of quasi-monotonicity~\eqref{eq:mon} as well as on $\kappa_0$, $\c{refined}$, and $s>0$.
\end{lemma}

\begin{proof}
Without loss of generality, we may assume $\eta(\TT_0; U(\TT_0))>0$ since monotonicity~\eqref{eq:mon} predicts $\eta(\TT_\ell;U(\TT_\ell))\leq \c{mon}\eta(\TT_0;U(\TT_0))$ and the claim~\eqref{eq:lemoptresult1}--\eqref{eq:lemoptresult2} is trivially satisfied for $\eta(\TT_\ell;U(\TT_\ell))=0$ and $\widehat\TT=\TT_\ell$.
Define $\lambda:=\c{mon}^{-2}\kappa_0$ and due to quasi-mo\-notonicity~\eqref{eq:mon} also $0<\eps^2 := \lambda\eta(\TT_\ell;U(\TT_\ell))^2\leq \eta(\TT_0;U(\TT_0))^2\leq \norm{(\eta(\cdot),U(\cdot))}{\B_s}^2$.  The fact $\norm{(\eta(\cdot),U(\cdot))}{\B_s}<\infty$ implies the existence of some $N\in\N$ and $\TT_\eps\in\T(N)$ with
\begin{align}\label{eq:lemopthelp}
\eta(\TT_\eps;U(\TT_\eps))=\min_{\TT\in\T(N)}\eta(\TT;U(\TT)) \le (N+1)^{-s}\norm{(\eta(\cdot),U(\cdot))}{\B_s}\leq\eps.
\end{align}
Let $N\in\N_0$ be the smallest number such that the last estimate in~\eqref{eq:lemopthelp} holds. First, assume $N>0$. Then, it holds
\begin{align*}
\norm{(\eta(\cdot),U(\cdot))}{\B_s} >N^s\eps
\end{align*}
and therefore $N+1\leq 2N\leq 2 \norm{(\eta(\cdot),U(\cdot))}{\B_s}^{1/s}\eps^{-1/s}$. For $N=0$, it always holds $ 1\leq \norm{(\eta(\cdot),U(\cdot))}{\B_s}\eps^{-1/s}$
because of $\norm{(\eta(\cdot),U(\cdot))}{\B_s}\geq\eps$.
Altogether, $\TT_\eps$ fulfils
\begin{align}\label{eq1:lemopt}
 |\TT_\eps| - |\TT_0| \leq N\leq2\norm{(\eta(\cdot),U(\cdot))}{\B_s}^{1/s}\eps^{-1/s}
 \quad\text{and}\quad\eta(\TT_\eps;U(\TT_\eps)) \le \eps.
\end{align}
According to~\eqref{refinement:overlay}, the coarsest common refinement $\widehat\TT:=\TT_\eps\oplus\TT_\ell$ satisfies
\begin{align}\label{eq2:lemopt}
 |\widehat\TT| - |\TT_\ell| \le |\TT_\eps| + |\TT_\ell| - |\TT_0| -|\TT_\ell|=|\TT_\eps| - |\TT_0|.
\end{align}
The fact that $\widehat\TT$ is a refinement of $\TT_\eps$ allows for quasi-monotonicity~\eqref{eq:mon} and
\begin{align}\label{eq:quasimonoton}
 \eta(\widehat\TT;U(\widehat\TT))^2 \leq\c{mon}^2 \eta(\TT_\eps;U(\TT_\eps))^2 \le \c{mon}^2\lambda\eta(\TT_\ell;U(\TT_\ell))^2=\kappa_0 \eta(\TT_\ell;U(\TT_\ell))^2.
\end{align}
Proposition~\ref{prop:doerfler} thus \revision{guarantees} that the set $\RR(\TT_\ell,\widehat\TT)\subseteq\TT_\ell$ with $|\RR(\TT_\ell,\widehat\TT)|\simeq |(\TT_\ell\setminus\widehat\TT)|$  satisfies the D\"orfler marking. There holds $\eps\simeq \eta(\TT_\ell;U(\TT_\ell))$ and together with~\eqref{eq1:lemopt} and~\eqref{eq2:lemopt}, this proves~\eqref{eq:lemoptresult1}. Estimate~\eqref{eq:lemoptresult2} follows from~\eqref{refinement:sons} and~\eqref{eq2:lemopt}, i.e.
\begin{align*}
 |\RR(\TT_\ell,\widehat\TT)|
& \lesssim |(\TT_\ell\backslash\widehat\TT)|
 \le |\widehat\TT| - |\TT_\ell|
 \le |\TT_\eps| - |\TT_0|\\
 &\lesssim \norm{(\eta(\cdot),U(\cdot))}{\B_s}^{1/s}\eps^{-1/s}
 \simeq \norm{(\eta(\cdot),U(\cdot))}{\B_s}^{1/s}\eta(\TT_\ell;U(\TT_\ell))^{-1/s}.
\end{align*}
This concludes the proof.
\end{proof}

The subsequent proposition states the quasi-optimality for a general adaptive algorithm which fits in the framework of this section under the axioms stability~\eqref{A:stable}, reduction~\eqref{A:reduction}, and discrete reliability~\eqref{A:dlr}. For the upper bound in~\eqref{eq:optimality}, the mesh 
refinement strategy has to fulfil the mesh closure estimate~\eqref{refinement:closure}, while the lower bound hinges only on~\eqref{refinement:sons2}.

\begin{proposition}\label{prop:optimality}
Suppose that~\eqref{eq:lemoptresult1}--\eqref{eq:lemoptresult2} of Lemma~\ref{lem:optimality} are valid
for one particular choice of $0<\kappa_0,\theta_0<1$, and assume that the estimator $\eta(\TT_\ell;U(\TT_\ell))$ satisfies the equivalent estimates~\eqref{eq:estsum}--\eqref{eq:propRconv} from Lemma~\ref{lem:Rconv}.
For $0<\theta\leq \theta_0$, then the equivalence
\begin{align}\label{eq:optimality}
c_{\rm opt}\norm{(\eta(\cdot),U(\cdot))}{\B_s} \leq\sup_{\ell\in\N_0}\frac{\eta(\TT_\ell;U(\TT_\ell))}{(|\TT_\ell|-|\TT_0|+1)^{-s}}\leq\c{optimal}\norm{(\eta(\cdot),U(\cdot))}{\B_s}
\end{align}
holds for all $s>0$. 
The constant $\c{optimal}>0$ depends only on  $C_{\rm min},\c{mesh},\c{invsum},\c{lemopthelp2},C_{\rm son}>0$
and $s>0$, while $c_{\rm opt}>0$ depends only on $C_{\rm son}$.
\end{proposition}

\begin{proof}
For the proof of the upper bound in~\eqref{eq:optimality}, suppose that the right-hand side of~\eqref{eq:optimality} is finite. Otherwise, the upper bound holds trivially.
Step~(iii) of Algorithm~\ref{algorithm} selects some set
$\MM_\ell$ with \emph{(almost) minimal} cardinality which satisfies the D\"orfler marking~\eqref{eq:doerfler}. Since the set $\RR(\TT_\ell;\widehat\TT)$ also satisfies the D\"orfler marking,~\eqref{eq:lemoptresult2} implies
\begin{align}\label{eq:missinglink}
|\MM_\ell|\lesssim |\RR(\TT_\ell;\widehat\TT)|\lesssim\norm{(\eta(\cdot),U(\cdot))}{\B_s}^{1/s}\, \eta(\TT_\ell;U(\TT_\ell))^{-1/s}.
\end{align}
The equivalence of the estimates~\eqref{eq:estsum}--\eqref{eq:propRconv} in Lemma~\ref{lem:Rconv} together with Proposition~\ref{prop:Rconv} allows to employ~\eqref{eq:endest} as well as~\eqref{eq:missinglink} and the optimality of the mesh closure~\eqref{refinement:closure}. For all $\ell\in\N$, this implies
\begin{align}\label{eq:markedest}
\begin{split}
 |\TT_\ell|-|\TT_0|+1\lesssim |\TT_\ell| - |\TT_0|
 &\lesssim \sum_{j=0}^{\ell-1}|\MM_j|
 \lesssim \norm{(\eta(\cdot),U(\cdot))}{\B_s}^{1/s}\,\sum_{j=0}^{\ell-1}\eta(\TT_j;U(\TT_j))^{-1/s}\\
&\lesssim \norm{(\eta(\cdot),U(\cdot))}{\B_s}^{1/s}\,\eta(\TT_\ell;U(\TT_\ell))^{-1/s}.
\end{split}
\end{align}
Consequently,
\begin{align*}
 \eta(\TT_\ell;U(\TT_\ell))(|\TT_\ell| - |\TT_0|+1)^{s}\lesssim \norm{(\eta(\cdot),U(\cdot))}{\B_s} \quad \text{for all } \ell\in\N.
\end{align*}
Since $\eta(\TT_0;U(\TT_0))\leq \norm{(\eta(\cdot),U(\cdot))}{\B_s}$ , this leads to the upper bound in~\eqref{eq:optimality}.

For the proof of the lower bound in~\eqref{eq:optimality}, suppose the middle supremum is finite. Otherwise the lower bound holds trivially. Choose $N\in\N_0$ and the largest possible $\ell\in\N_0$ with $|\TT_\ell|-|\TT_0|\leq N$.
Due to maximality of $\ell$, $N+1<|\TT_{\ell+1}|-|\TT_0|+1\leq C_{\rm son}|\TT_\ell| -|\TT_0|+1\lesssim |\TT_\ell|-|\TT_0|+1$.
This leads to
\begin{align*}
\inf_{\TT\in\T(N)}(N+1)^{s}\eta(\TT;U(\TT)) \lesssim (|\TT_\ell|-|\TT_0|+1)^{s}\eta(\TT_\ell;U(\TT_\ell))
\end{align*}
and concludes the proof.
\end{proof}

\begin{remark}\label{remark:doerfler}
As mentioned above, the axioms stability~\eqref{A:stable}, reduction~\eqref{A:reduction}, and discrete reliability~\eqref{A:dlr} allow for an application of Proposition~\ref{prop:optimality}.
In this case, Lemma~\ref{lem:optimality} implies
the quasi-optimality~\eqref{eq:optimality} of Proposition~\ref{prop:optimality} for all 
$0<\theta<\theta_\star:=(1+\c{stable}^2\c{dlr}^2)^{-1}$, see Proposition~\ref{prop:doerfler}~(ii). 
Moreover, $\c{optimal}>0$ then depends only on $C_{\rm min}$, $\c{mesh}$, $\c{stable},\c{reduction},\c{dlr},\c{qosum}(\epsqo)>0 $ as well as on $\theta$ and $s>0$.
\end{remark}

\begin{proof}[Proof of Theorem~\ref{thm:main}~(ii)]
Lemma~\ref{lemma:dlr2rel} proves that discrete reliability~\eqref{A:dlr} implies reliability~\eqref{A:reliable}.
Stability~\eqref{A:stable} and discrete reliability~\eqref{A:dlr} guarantee that~\eqref{eq:doerfler:opt} holds for all $\kappa_0\in(0,1)$. Together with quasi-monotonicity~\eqref{eq:mon} from Lemma~\ref{lem:mon1}, this implies that~\eqref{eq:lemoptresult1}--\eqref{eq:lemoptresult2} of Lemma~\ref{lem:optimality} are valid.  Moreover,~\eqref{A:stable} and~\eqref{A:reduction} prove estimator reduction~\eqref{eq:estconv} from Lemma~\ref{lem:estconv}. This and quasi-orthogonality~\eqref{A:qosum} together with reliability~\eqref{A:reliable} allow to employ Proposition~\ref{prop:Rconv} which ensures that~\eqref{eq:estsum}--\eqref{eq:propRconv} hold. Finally, 
choose $0<\theta\le\theta_0<\theta_\star$ in 
Proposition~\ref{prop:doerfler}~(ii) with corresponding $0<\kappa_0<1$.
Then, the application of Proposition~\ref{prop:optimality} concludes
the proof.
\end{proof}



\section{Laplace Problem with Residual Error Estimator}
\label{section:examples1}

\noindent
This section applies the abstract analysis of the preceding sections to different discretizations of the Laplace problem.
The examples are taken from conforming, nonconforming, and mixed finite element methods (FEM) as well as the boundary element methods (BEM) for weakly-singular and hyper-singular integral equations.


\subsection{Conforming FEM}
\label{section:ex:poisson}%
In the context of conforming FEM for symmetric operators, the convergence and quasi-optimality of the adaptive
algorithm has finally been analyzed in the seminal works~\cite{ckns,stevenson07}.
In this section, we show that their results can be reproduced in the abstract
framework developed. Moreover, our approach adapts the idea 
of~\cite{dirichlet3d}, and efficiency~\eqref{A:efficient} is only used to
characterize the approximation class. This provides a qualitative improvement
over~\cite{ckns,stevenson07} in the sense that the upper bound $\theta_\star$
for optimal adaptivity parameters $0<\theta<\theta_\star$ does not depend on
the efficiency constant $\c{efficient}$.

Let $\Omega\subset\R^d$, $d\ge2$, be a bounded Lipschitz domain with polyhedral boundary 
$\Gamma:=\partial\Omega$. With given volume force $f\in L^2(\Omega)$,
we consider the Poisson model problem
\begin{align}\label{ex:poisson}
-\Delta u &= f\quad\text{in }\Omega
\quad\text{and}\quad
 u=0\quad\text{on }\Gamma.
\end{align}
For the weak formulation, let 
$\XX := H^1_0(\Omega) $ denote the usual Sobolev space, with the equivalent $H^1$-norm 
$\norm{v}{H^1_0(\Omega)}:=\norm{\nabla v}{L^2(\Omega)}$ associated with the scalar product 
\begin{align}\label{eq:poisson:bilinear}
 b(u,v) := \int_\Omega\nabla u\cdot\nabla v\,dx = \int_\Omega fv\,dx
 \quad\text{for all } v\in H^1_0(\Omega).
\end{align}
Then, the weak form of~\eqref{ex:poisson} admits a unique solution $u\in H^1_0(\Omega)$.
Based on a regular triangulation $\TT$ of $\Omega$ into simplices, the conforming finite 
element spaces $\XX(\TT) := \SS^p_0(\TT) := \PP^p(\TT)\cap H^1_0(\Omega)$ of fixed 
polynomial order $p\ge1$ read
\begin{align}\label{eq:polynomials}
 \PP^p(\TT):=\set{v\in L^2(\Omega)}{\forall T\in\TT\quad v|_T \text{ is a polynomial of degree}\,\leq p}.
\end{align}
The discrete formulation
\begin{align}\label{eq:poisson:discrete}
 b(U(\TT),V) = \int_\Omega fV\,dx
 \quad\text{for all }V\in\SS^p_0(\TT)
\end{align}
also admits a unique FE solution $U(\TT)\in\SS^p_0(\TT)$. In particular, all assumptions
of Section~\ref{sec:setting} are satisfied with 
$\dist[\TT]{v}{w} = \norm{v-w}{H^1_0(\Omega)}$ and $\c{triangle}=1$.
The standard residual error estimator consists of the local contributions
\begin{align}\label{ex:poisson:estimator}
 \eta_T(\TT;V)^2 := h_T^2\,\norm{f+\Delta V}{L^2(T)}^2
 + h_T\,\norm{[\partial_nV]}{L^2(\partial T\cap\Omega)}^2\quad \text{for all }T\in\TT,
\end{align}
see e.g.~\cite{ao00,v96} as well as e.g.~\cite{ckns,stevenson07}.

Here, $[\partial_n V]$ denotes the jump of the normal derivative over interior
facets of $\TT$. Following~\cite{ckns}, we use the local mesh-width function
\begin{align}\label{ex:laplace:meshsize}
 h(\TT)\in\PP^0(\TT)
 \quad\text{with}\quad
 h(\TT)|_T := h_T = |T|^{1/d}, 
\end{align}
where $|T|$ denotes the volume 
of an element $T\in\TT$. We employ newest vertex bisection for mesh-refinement and 
stress that the sons $T'$ of a refined element $T$ satisfy $h_{T'}\le 2^{-1/d}h_T$.
Since the admissible meshes $\TT\in\T$ are uniformly shape regular, we note that 
$h_T \simeq \diam(T)$ with the Euclidean diameter $\diam(T)$.
In particular, $\eta(\cdot)$ coincides,
up to a multiplicative constant, with the usual definition found in textbooks,
cf.\ e.g.~\cite{ao00,v96}. We refer to Section~\ref{section:examples3:meshsize} for the proof that the choice of the mesh-size function does not affect convergence and quasi-optimality of the adaptive algorithm.

Recall that the \revision{problem under consideration} involves some symmetric and elliptic \revision{bilinear form} $b(\cdot,\cdot)$.
According to the abstract analysis in Section~\ref{section:conforming}, \revision{it remains} to verify that the residual error
estimator $\eta(\TT;V)$ satisfies stability~\eqref{A:stable}, 
reduction~\eqref{A:reduction}, and discrete reliability~\eqref{A:dlr}.

\begin{proposition}\label{prop:ex:poisson:assumptions}
The conforming discretization of the Poisson problem~\eqref{ex:poisson} with
residual error estimator~\eqref{ex:poisson:estimator} satisfies
stability~\eqref{A:stable}, reduction~\eqref{A:reduction}
with $\q{reduction}=2^{-1/d}$, discrete reliability~\eqref{A:dlr} with 
$\RR(\TT,\widehat\TT) = \TT\backslash\widehat\TT$ and efficiency~\eqref{A:efficient} 
with 
\begin{align}\label{ex:poisson:osc}
\eff{\TT}{U(\TT)}
 := \osc(\TT) 
 := \min_{F\in\PP^{p-1}(\TT)}\norm{h(\TT)\,(f-F)}{L^2(\Omega)},
\end{align}
where $\norm{\osc(\cdot)}{\O_{1/d}}<\infty$ is guaranteed.
The constants $\c{stable},\c{reduction},\c{dlr},\c{efficient}>0$
depend only on the polynomial degree $p\in\N$ and 
on $\T$.
\end{proposition}

\begin{proof}
Stability on non-refined elements~\eqref{A:stable} as well as reduction on 
refined elements~\eqref{A:reduction} are part of the proof 
of~\cite[Corollary~3.4]{ckns}. The discrete reliability~\eqref{A:dlr} is 
found in~\cite[Lemma~3.6]{ckns}. Efficiency~\eqref{A:efficient} is well exposed in text books on a~posteriori
error estimation, see e.g.~\cite{ao00,v96}, and $\norm{\osc(\cdot)}{\O_{1/d}}<\infty$ follows by definition~\eqref{eq:oscapprox} for a sequence of uniform meshes.
\end{proof}

\begin{consequence}\label{con:conforming}
The adaptive algorithm leads to convergence with quasi-optimal rate
for the estimator $\eta(\cdot)$ in the sense of Theorem~\ref{thm:main}. 
Theorem~\ref{thm:mainerr} proves that for 
lowest-order elements $p=1$, even optimal rates for the discretization error
are achieved, while for higher-order elements $p\geq 2$ additional regularity of
$f$ has to be imposed, e.g., $f\in H^1(\Omega)$ for $p=2$. \qed
\end{consequence}

Numerical examples for the 2D Laplacian with mixed Dirichlet-Neumann boundary conditions are found in~\cite{p1afem} together with a detailed discussion of the implementation. Examples for 3D are found in~\cite{ckns}.

\subsection{Lowest-order nonconforming FEM}
\label{section:ex:nonconforming}
The convergence analysis of adaptive nonconforming finite element techniques is much younger than that of conforming ones. The lack of the Galerkin orthogonality led to the
invention of the quasi-orthogonality in \cite{ch06b} and  thereafter in \cite{rabus10,bms09,mzs10}. 

Consider the Poisson model problem \eqref{ex:poisson} of Subection~\ref{section:ex:poisson} with the
weak formulation \eqref{eq:poisson:bilinear} in the Hilbert space
$\XX := H^1_0(\Omega)$. Let $\TT\in\T$ be a regular triangulation, and let $\EE(\TT)$ denote the set of element facets. The discrete problem is based on the piecewise gradients for piecewise linear polynomials
\begin{align}\label{ex:nonconfPMP:CR}
\begin{split}
\XX(\TT):=CR^1_0(\TT):=\set{V\in \PP^1(\TT)}{&V\text{ is continuous at }
\mmid(\EE(\TT))\cap \Omega\\
&\text{ and }  V=0 \text{ at }    \mmid(\EE(\TT)) \cap \Gamma }
\end{split}
\end{align}
where $\mmid(\EE(\TT))$ denotes the set of barycenters of all facets of $\TT$. Given $U,V\in  CR^1_0(\TT)$ in the nonconforming $P_1$-FEM also sometimes named after Crouzeix and Raviart \cite{CR73}, the piecewise version of the bilinear form,
\begin{equation}
b(U,V) := \sum_{T\in\TT} \int_T   \nabla U\cdot \nabla V\, dx,
\end{equation}
where the weak gradient $\nabla(\cdot)$ is replaced by the $\TT$-piecewise gradient $\nabla_\TT(\cdot)$,
defines a scalar product on $CR^1_0(\TT)$. The induced norm 
\[
\norm{\cdot}{\XX(\TT)} = \left(  \sum_{T\in\TT} \norm{\nabla_\TT(\cdot)}{L^2(T)}^2 \right)^{1/2}
\]
equals the piecewise $H^1$-seminorm and controls the $L^2$-norm in the sense of a discrete Friedrichs inequality~\cite{SR}, and all assumptions of Section~\ref{sec:setting} hold with 
$\dist[\TT]{v}{w} = \norm{v-w}{\XX(\TT)}$.
Hence,  the discrete formulation
\begin{align}\label{eq:poisson:CRdiscrete}
 b(U(\TT),V) = \int_\Omega fV\,dx
 \quad\text{for all }V\in CR^1_0(\TT)
\end{align}
admits a unique FE solution $U(\TT)\in CR^1_0(\TT)$. We adopt the mesh-size function $h_T=h(\TT)|_T$ from~\eqref{ex:laplace:meshsize}. Note that analogously to Section~\ref{section:ex:poisson}, we use newest vertex bisection and obtain $h_{\widehat T}\leq 2^{-1/d}h_T$ for all $T\in\TT$ and its successors $\widehat T\in\widehat\TT\setminus\TT$ with $\widehat T \subset T$.
The explicit residual-based error estimator consists of the local contributions
\begin{align}\label{ex:poisson:CRestimator}
 \eta_T(\TT;V)^2 := h_T^2\,\norm{f}{L^2(T)}^2
 + h_T\,\norm{[\partial_t V]}{L^2(\partial T)}^2\quad \text{for all }T\in\TT.
\end{align}
Here $[\partial_t V]$ denotes the jump of the $(d-1)$-dimensional tangential derivatives across  interior
facets of $\TT$ and (for the homogeneous Dirichlet boundary conditions at hand)
$[\partial_t V]:=\partial_t V$ along boundary facets $E\in\EE(\TT)$ with $E\subset \partial\Omega$.

\begin{proposition}\label{prop:ex:poisson:CRassumptions}
The nonconforming discretization of the Poisson problem~\eqref{ex:poisson} with
residual error estimator~\eqref{ex:poisson:CRestimator} satisfies
stability~\eqref{A:stable}, reduction~\eqref{A:reduction}
with $\q{reduction}=2^{-1/d}$, general quasi-orthogonality~\eqref{A:qosum}, discrete reliability~\eqref{A:dlr} with 
$\RR(\TT,\widehat\TT) = \TT\backslash\widehat\TT$ and efficiency~\eqref{A:efficient} 
with $ \eff{\TT}{U(\TT)}:= \osc(\TT)$ from~\eqref{ex:poisson:osc}
and hence $\norm{\osc(\cdot)}{\O_{1/d}}<\infty$. 
The constants $\c{stable},\c{reduction},\c{dlr},\c{efficient}>0$
depend only on $\T$.
\end{proposition}

\begin{proof}
Stability~\eqref{A:stable} and reduction~\eqref{A:reduction} follow as for the conforming case by reduction of the mesh-size function and standard inverse inequalities.
Efficiency~\eqref{A:efficient} is established in \cite{DDPV,CBJ,CCJH}, while the discrete reliability~\eqref{A:dlr} is shown in \cite[Sect.~4]{rabus10} for $d=2$, but the proof essentially 
applies to any dimension. The aforementioned contributions utilize a continuous or discrete Helmholtz decomposition and are therefore restricted to simply connected domains. The general case is exposed in \cite{CGS}. Notice the abbreviation for the $L^2$-norm on the refined domain $\bigcup( \TT\setminus\widehat\TT )$
\[
\|  \cdot \|_{\TT\setminus\widehat\TT}:=\Big(   \sum_{T\in \TT\setminus\widehat\TT}
\|  \cdot \|_{L^2(T)}^2\Big)^{1/2}.
\]
The quasi-orthogonality of \cite{bms09,ch06b} 
guarantees, e.g., in the form of \cite[Lemma~2.2]{rabus10}, that
\begin{align*}
\enorm{U(\widehat\TT)-U(\TT)}_{\XX(\widehat\TT)}^2 &\leq 
\enorm{u-U(\TT)}_{\XX(\TT)}^2 - \enorm{u-U(\widehat\TT)}_{\XX(\widehat\TT)}^2 \\
&\quad\quad+ C\enorm{u-U(\widehat\TT)}_{\XX(\widehat\TT)}
\|  h(\TT)\, f\|_{\TT\setminus\widehat\TT}
\end{align*}
for some generic constant $C\simeq 1$ which depends on $\T$ . For any
$0<\epsqo<1$, the Young inequality yields
\[
C\enorm{u-U(\widehat\TT)}_{\XX(\widehat\TT)}
\|  h(\TT)\, f\|_{\TT\setminus\widehat\TT} \le 
\epsqo\enorm{u-U(\widehat\TT)}_{\XX(\widehat\TT)}^2
+ C^2\|  h(\TT)\, f\|_{\TT\setminus\widehat\TT}^2/(4\epsqo).
\]
The analysis of the last term starts with the observation that 
\[
\mu(\TT) :=  \|  h(\TT) f  \|_{L^2(\Omega)}
\]
defines a function $\mu:\T\to\R$ with 
\[
\|  h(\TT)\, f\|_{\TT\setminus\widehat\TT}^2\le  \left(  
\mu(\TT)^2-\mu(\widehat\TT)^2\right)/(1-2^{2/d}).
\]
In fact, any contribution for $T\in \TT\cap\widehat\TT$ vanishes on both sides while
for any 
$ \widehat{T}\in\widehat\TT$ and $T\in\TT \setminus\widehat\TT$ 
with  $ \widehat{T}\subset T$, the local mesh-size satisfies  $h_{\widehat{T}}\le 2^{-1/d} h_T$. 
The combination of the aforementioned estimates result in \eqref{A:qoa}
with $\c{qo}(\epsqo):= C^2/\left(\epsqo  (1-2^{2/d}) 4 \right)$.
Since the term $\mu(\TT)$ is part of the estimator $\eta(\TT,V)$, it follows
$C_2=1$ in \eqref{A:qob}.
This and Lemma~\ref{lem:qo} imply the general quasi-orthogonality~\eqref{A:qosum}.
\end{proof}

\begin{consequence}
The adaptive algorithm leads to convergence with quasi-optimal rate
in the sense of Theorem~\ref{thm:main} and Theorem~\ref{thm:mainerr}. \qed
\end{consequence}
Numerical examples in 2D that underline the above result can be found in~\cite{ch06a}.

%
\subsection{Mixed FEM}
\label{section:ex:mixed}
The mixed formulation of the Poisson model problem~\eqref{ex:poisson} involves
the product Hilbert space
$\XX:=H(\ddiv,\Omega)\times L^2(\Omega) $ with 
\[
H(\ddiv,\Omega):=\{q\in L^2(\Omega;\R^n):\ddiv q\in L^2(\Omega)\}
\]
equipped with the \revision{corresponding} norms i.e. 
\begin{align*}
\norm{(q,v)}{\XX}^2:= \normLtwo{q}{\Omega}^2+ \normLtwo{\ddiv q}{\Omega}^2 + \normLtwo{v}{\Omega}^2 
\end{align*}
and $\dist[\TT]{(p,u)}{(q,v)} = \norm{(p-q,u-v)}{\XX}$.
The weak formulation~\eqref{eq:poisson:bilinear} now involves the bilinear form $b:\XX\times\XX\to\R$ and 
right-hand side $F\in \XX^*$ defined
for any $(p,u),(q,v)\in H(\ddiv,\Omega)\times L^2(\Omega) $ by
\begin{align*}
b((p,u),(q,v))&:=\int_\Omega\left(  p\cdot q+ u\ddiv q+v\ddiv p\right)dx, \\
F(q,v)&=\int_\Omega fv\, dx ,
\end{align*}
where $f\in L^2(\Omega)$ is the right-hand side in the Poisson problem. Let \revision{$\TT\in\T$} be a regular triangulation, and let $\EE(\TT)$ denote the set of element facets.
The conforming mixed finite element function spaces read
\[
\XX(\TT):= \{ Q\in H(\ddiv,\Omega): \forall T\in\TT,\,
Q|_T\in M_k(T)\}\times \PP^k(\TT)\subset \XX,
\]
with the Raviart-Thomas (RT) mixed finite element space
\begin{align*}
M_k(T):=\set{ Q\in \PP^{k+1}(T;\R^d) }{ 
& \exists a_1,\dots,a_d,b\in \PP^k(T)\,\forall x\in T,\\
& Q(x)=\big(a_1(x)+b(x)\,x_1,\dots,a_d(x)+ b(x)\, x_d\big)   }
\end{align*}
or the Brezzi-Douglas-Marini (BDM) mixed finite element 
space $M_k(T):=\PP^{k+1}(T;\R^d) $ amongst many other examples for $k\in\N_0$.
The discrete formulation 
\begin{align*}
b\big((P(\TT),U(\TT)),(Q,V)\big)=F(Q,V)\quad\text{for all } (Q,V)\in \XX(\TT),
\end{align*}
admits a unique solution $(P(\TT),U(\TT))\in\XX(\TT)$ cf.\ e.g.~\cite{brezzi-fortin}.
With the local mesh-size function $h_T:=h(\TT)|_T$ from~\eqref{ex:laplace:meshsize},
the explicit residual-based a~posteriori error estimator for $d=2,3$ 
consists of the local contributions, for all $T\in\TT$,
\begin{align}\label{ex:poisson:MFEMestimator}
 \eta_T(\TT;Q)^2 := h_T^2\,\norm{\curl Q}{L^2(T)}^2
 + h_T^2\,\norm{f-\Pi_k f}{L^2(T)}^2
 + h_T\,\norm{[Q\times \nu]}{L^2(\partial T)}^2.
\end{align}
Here,  $\curl $ denotes the rotation operator ($=\partial\cdot/\partial x_2-\partial\cdot/\partial x_1$ in 2D) and  $[Q\times \nu]$ denotes the
jump of the $(d-1)$-dimensional tangential derivatives across  interior
facets $E\in\EE(\TT)$ with $E\subseteq\Omega$ with unit normal $\nu$ \revision{along $\partial T$.} For the homogeneous Dirichlet 
boundary conditions at hand, we define
$[Q\times \nu]:=Q\times \nu$ along boundary facets $E\in\EE(\TT)$ with $E\subseteq \partial\Omega$. Finally,
$\Pi_k:\,L^2(\Omega)\to \PP^k(\TT)$ is the $L^2$-orthogonal projection onto $P^k(\TT)$.

\revision{The} newest-vertex bisection for mesh refinement \revision{allows} the following result.
\begin{proposition}\label{prop:ex:poisson:MFEMassumptions}
The mixed formulation of the Poisson problem~\eqref{ex:poisson} on a simply connected 
Litschitz domain $\Omega$ in $d=2,3$ dimensions with
residual error estimator~\eqref{ex:poisson:MFEMestimator} satisfies
stability~\eqref{A:stable}, reduction~\eqref{A:reduction}
with $\q{reduction}=2^{-1/d}$, general quasi-orthogonality~\eqref{A:qosum}, discrete reliability~\eqref{A:dlr} with 
$\RR(\TT,\widehat\TT) = \TT\backslash\widehat\TT$ and efficiency~\eqref{A:efficient}  with $ \eff{\TT}{U(\TT)}:= \osc(\TT)$ from~\eqref{ex:poisson:osc}
and hence $\norm{\osc(\cdot)}{\O_{1/d}}<\infty$. 
As above, the constants $\c{stable}$, $\c{reduction}$, $\c{dlr}$, $\c{efficient}>0$
depend only on the polynomial degree $k$ and on $\T$.
\end{proposition}

\begin{proof}
Stability~\eqref{A:stable} and reduction~\eqref{A:reduction} follow as for the conforming case.
Efficiency~\eqref{A:efficient} dates back to the independent work \cite{Alonso1996,Carstensen1997};
the first version and the notion of quasi-orthogonality~\eqref{A:qosum} has been introduced in
\cite{ch06a} and refined in \cite{LCMHJX}. For the two mentioned versions RT-MFEM  and 
BDM-MFEM, the work \cite{HuangXu} presents discrete reliability~\eqref{A:dlr} and 
the quasi-orthogonality \eqref{A:qo} in the form
\begin{align*}
\norm{P(\widehat\TT)-P(\TT)}{L^2(\Omega)}^2 &\leq 
\norm{p-P(\TT)}{L^2(\Omega)}^2 - \norm{p-P(\widehat\TT)}{L^2(\Omega)}^2 \\
&\quad\quad+ C\norm{p-P(\widehat\TT)}{L^2(\Omega)}
\osc(\TT\setminus\widehat\TT,f)
\end{align*}
for some generic constant $C\simeq 1$.  The rearrangements  of the previous subsection 
with $\mu(\TT):= \osc(\TT;f)$ result in~\eqref{A:qoa}
for {\em any } $0<\epsqo<1$ and $\c{qo}(\epsqo):= C^2/\left(\epsqo  (1-2^{2/d}) 4 \right)$ and $C_2=1$ in
\eqref{A:qob}. 
\end{proof}

\begin{consequence}
The adaptive algorithm leads to convergence with quasi-optimal rate
for the estimator $\eta(\cdot)$ in the sense of Theorem~\ref{thm:main} and Theorem~\ref{thm:mainerr}.
 \qed
\end{consequence}
Numerical examples that underline the above result can be found in~\cite{hw97,cb02}.


\def\slp{\mathfrak V}
\def\dlp{\mathfrak K}
\subsection{Conforming BEM for weakly-singular integral equation}
\label{section:ex:symm}
In this section, we consider adaptive mesh-refinement for the weighted-residual error estimator
in the context of BEM for integral operators of order $-1$. Unlike FEM, the efficiency of this error estimator is \revision{still an open question} in general and mathematically guaranteed only for particular situations~\cite{affkp} \revision{while} typically observed throughout, see e.g.~\cite{cc97,cms,cs95,cs96}. Nevertheless, the abstract framework of
Section~\ref{section:optimality} provides the means to analyze
convergence and quasi-optimality of the adaptive algorithm. 

In a \revision{specific} setting, optimal convergence of adaptive mesh-refinement has independently first been proved by~\cite{fkmp,gantumur} for lowest-order BEM. While the analysis of~\cite{gantumur} covers general operators, but is restricted to smooth boundaries $\Gamma$, the analysis of~\cite{fkmp} focusses on the Laplace equation only, but allows for polyhedral boundaries. In~\cite{ffkmp:part1}, these results are generalized to BEM with ansatz functions of arbitrary, but fixed polynomial order.

Let $\Omega\subset\R^d$ be a bounded Lipschitz domain with polyhedral boundary
$\partial\Omega$ and $d=2,3$. Let $\Gamma\subseteq\partial\Omega$ be a relatively
open subset. For given $f\in H^{1/2}(\Gamma) := \set{\phi|_\Gamma}{\phi\in H^1(\Omega)}$, 
we consider the weakly-singular first-kind integral equation
\begin{align}\label{ex:symm}
 \slp u(x) = f(x)
 \quad\text{for  }x \in \Gamma.
\end{align}
The sought solution satisfies $u\in \H^{-1/2}(\Gamma)$. The
negative-order Sobolev space $\H^{-1/2}(\Gamma)$ is
the dual space of $H^{1/2}(\Gamma)$ with respect to the extended $L^2(\Gamma)$-scalar
product $\dual{\cdot}{\cdot}_{L^2(\Gamma)}$. We refer to the 
monographs~\cite{hw,mclean,ss}
for details and proofs of this as well as of the following facts on the functional analytic 
setting: With the fundamental solution of the Laplacian
\begin{align}\label{ex:symm:G}
 G(z) := \begin{cases}
 -\frac{1}{2\pi}\log|z|\quad&\text{for }d=2,\\
 +\frac{1}{4\pi}\frac{1}{|z|}\quad&\text{for }d=3,
 \end{cases}
\end{align}
the {\em simple-layer potential} reads
\begin{align}\label{eq:symm:V}
 \slp u(x) &:= \int_\Gamma G(x-y)u(y)\,d\Gamma(y)
 \quad\text{for }x\in\Gamma.
\end{align}
We note that $\slp\in L(H^{-1/2+s}(\Gamma);H^{1/2+s}(\Gamma))$ is a linear, continuous, and symmetric 
operator for all $-1/2\le s\le 1/2$. For 2D, we assume $\diam(\Omega)<1$ which can always 
be achieved by scaling. Then, $\slp$ is also elliptic, i.e.\
\begin{align}
 b(u,v) := \dual{\slp u}{v}_{L^2(\Gamma)}
\end{align}
defines an equivalent scalar product on $\XX:=\H^{-1/2}(\Gamma)$. We equip $\H^{-1/2}(\Gamma)$ 
with the induced Hilbert space norm $\norm{v}{\H^{-1/2}(\Gamma)}^2 := \dual{\slp v}{v}_{L^2(\Gamma)}$.
According to the Hahn-Banach theorem,~\eqref{ex:symm} is equivalent to the variational formulation
\begin{align}\label{ex:symm:weakform}
 b(u,v) = \dual{f}{v}_{L^2(\Gamma)}
 \quad\text{for all }v\in \H^{-1/2}(\Gamma).
\end{align}
It relies on the the scalar product $b(\cdot,\cdot)$ and hence admits a unique solution $u\in \H^{-1/2}(\Gamma)$ of~\eqref{ex:symm:weakform}. 

Let $\TT$ be a regular triangulation of $\Gamma$. For each element $T\in\TT$, let
$\gamma_T:T_{\rm ref}\to T$ denote an affine bijection from the reference element 
$T_{\rm ref} = [0,1]$ for $d=2$ resp.\ $T_{\rm ref} = {\rm conv}\{(0,0),(0,1),(1,0)\}$
for $d=3$ onto $T$. We employ conforming boundary elements 
$\XX(\TT) := \PP^p(\TT) \subset H^{-1/2}(\Gamma)$
of order $p\ge0$, where
\begin{align*}
 \PP^p(\TT) := \set{V\in L^2(\Gamma)}{V\circ\gamma_T\text{ is a polynomial of degree }\le p
 \text{ on }T_{\rm ref},
 \text{ for all }T\in\TT}.
\end{align*}
The discrete formulation
\begin{align*}
 b(U(\TT),V) = \dual{f}{V}_{L^2(\Gamma)}
 \quad\text{for all }V\in\PP^p(\TT)
\end{align*}
admits a unique BE solution $U(\TT)\in\PP^p(\TT)$. In particular, all assumptions
of Section~\ref{sec:setting} are satisfied with 
$\dist[\TT]{v}{w} = \norm{v-w}{\H^{-1/2}(\Gamma)}$ and $\c{triangle}=1$.

Under additional regularity of the data $f\in H^1(\Gamma)$, we consider the
weighted-residual error estimator of~\cite{cc97,cms,cs95,cs96}
with local contributions
\begin{align}\label{ex:symm:eta}
 \eta_T(\TT;V)^2 := h_T\,\norm{\nabla_\Gamma(f-\slp V)}{L^2(T)}^2\quad\text{for all }T\in\TT.
\end{align}
Here, $\nabla_\Gamma(\cdot)$ denotes the surface gradient and
$h(\TT)\in\PP^0(\TT)$ denotes the local-mesh width~\eqref{ex:laplace:meshsize} which now reads $h(\TT)|_T=|T|^{1/(d-1)}$ for all $T\in\TT$ as $\Gamma$ is a $(d-1)$-dimensional manifold.
We note that the
analysis of~\cite{cc97,cms,cs95,cs96} \revision{relies} on a Poincar\'e-type estimate
$\norm{R(\TT)}{H^{1/2}(\Gamma)}\lesssim\norm{h(\TT)^{1/2}\nabla_\Gamma R(\TT)}{L^2(\Gamma)}$ for
the Galerkin residual $R(\TT)=f-\slp U(\TT)$ and \revision{requires} shape-regularity of the triangulation 
$\TT$ for \revision{$d=3$, in particular, the fact} that $h_T\simeq \diam(T)$. We employ newest vertex bisection for $d=3$ \revision{and} the bisection
algorithm of~\cite{affkp} for $d=2$.

As in Section~\ref{section:ex:poisson}, the \revision{problem under consideration} involves a symmetric and
elliptic \revision{bilinear form} $b(\cdot,\cdot)$ and conforming discretizations. Therefore, it only remains to 
discuss stability~\eqref{A:stable}, 
reduction~\eqref{A:reduction}, and discrete reliability~\eqref{A:dlr}, see Section~\ref{section:conforming}.

\begin{proposition}\label{prop:example1symm}
The conforming BEM of the weakly-singular integral equations~\eqref{ex:symm} with
weighted-residual error estimator~\eqref{ex:symm:eta} satisfies
stability~\eqref{A:stable}, reduction~\eqref{A:reduction}
with $\q{reduction}=2^{-1/(d-1)}$, and discrete reliability~\eqref{A:dlr} with 
\begin{align}\label{dp:patch}
 \RR(\TT,\widehat\TT) 
 = \omega(\TT;\TT\backslash\widehat\TT)
 := \set{T\in\TT}{\exists T'\in\TT\backslash\widehat\TT\quad T\cap T'\neq\emptyset},
\end{align}
i.e.\ $\RR(\TT,\widehat\TT)$ contains all refined elements plus one additional layer of elements.
The constants $\c{stable}$, $\c{reduction}$, $\c{dlr}>0$ depend only on 
the polynomial degree $p\in\N_0$ and
on $\T$.
\end{proposition}

\begin{proof}
Stability on non-refined elements~\eqref{A:stable} as well as reduction on 
refined elements~\eqref{A:reduction} are part of the proof 
of~\cite[Proposition~4.2]{fkmp}. The proof essentially follows~\cite{ckns}, but additionally
relies on the novel inverse-type estimate
\begin{align*}
 \norm{h(\TT)^{1/2}\nabla_\Gamma \slp V}{L^2(\Gamma)} 
 \lesssim \norm{V}{\H^{-1/2}(\Gamma)}
 \quad\text{for all }V\in\PP^p(\TT).
\end{align*}
While the work~\cite{fkmp} is concerned with the lowest-order case $p=0$ only, we 
refer to~\cite[Corollary~2]{afembem} for general $p\ge0$ so 
that~\cite[Proposition~4.2]{fkmp} transfers to $p\ge0$. Discrete reliability
is proved in~\cite[Proposition~5.3]{fkmp} for $p=0$, but the proof holds 
accordingly for arbitrary $p\ge0$.
\end{proof}

\begin{consequence}
The adaptive algorithm leads to convergence with quasi-optimal rate
for the estimator $\eta(\cdot)$ in the sense of Theorem~\ref{thm:main}.\qed
\end{consequence}

Numerical examples that underline the above result can be found in~\cite{cs95}.

Efficiency~\eqref{A:efficient} of the weighted-residual error 
estimator~\eqref{ex:symm:eta} \revision{remains an open question.} The only result available~\cite{affkp} is
for $d=2$, and it exploits the equivalence of~\eqref{ex:symm} to some 
Dirichlet-Laplace problem: Assume $\Gamma = \partial\Omega$ and let 
\begin{align}\label{def:doublelayer}
 \dlp g (x) &:= \int_\Gamma \partial_{n(y)}G(x-y)\,g(y)\,d\Gamma(y)
\end{align}
denote the double-layer potential $\dlp\in L(H^{1/2+s}(\Gamma);H^{1/2+s}(\Gamma))$,
for all $-1/2\le s\le 1/2$. Then, the weakly-singular integral equation~\eqref{ex:symm:laplace}
for given Dirichlet data $g\in H^{1/2}(\Gamma)$ and $f:=(\dlp+1/2)g$ is 
an equivalent formulation of the Dirichlet-Laplace problem
\begin{align}\label{ex:symm:laplace}
-\Delta \phi &= 0\quad\text{in }\Omega
\quad\text{and}\quad
 \phi=g\quad\text{on }\Gamma = \partial\Omega.
\end{align}
The density $u\in\H^{-1/2}(\Gamma)$, which is sought in~\eqref{ex:symm}, is the normal 
derivative $u=\partial_n \phi$ to the potential $\phi\in H^1(\Omega)$ of~\eqref{ex:symm:laplace}.

For this special situation and lowest-order elements $p=0$, efficiency~\eqref{A:efficient} 
of $\eta(\TT)$ is proved in~\cite[Theorem~4]{affkp}.

\begin{proposition}\label{dp:symm:proposition}
We consider lowest-order BEM $p=0$ for $d=2$ and $\Gamma=\partial\Omega$.
Let $\sigma>2$ and $g\in H^\sigma(\partial\Omega):=\set{\phi|_{\partial\Omega}}{\phi\in H^{\sigma+1/2}(\Omega)}$.
For $f:=(\dlp+1/2)g$, the weighted-residual error estimator~\eqref{ex:symm:eta} satisfies~\eqref{A:efficient}
with $\norm{\osc(\cdot)}{\O_{3/2}}<\infty$.
\end{proposition}

\begin{proof}
The statement on efficiency of $\eta(\TT)$ is found in~\cite[Theorem~4]{affkp}, 
where $\eff{\TT}{U(\TT)}$ is based on the regular part of the exact solution $u$.
It holds $\eff{\TT}{U(\TT)}=\OO(h^{3/2+\eps})$ for uniform meshes
with mesh-size $h$ and some $\sigma$-dependent $\eps>0$, see~\cite{affkp}.
\end{proof}

For some smooth exact solution $u$, the generically optimal order of convergence is $\OO(h^{3/2})$ for lowest-order elements $p=0$,
where $h$ denotes the maximal mesh-width. 
For quasi-uniform meshes with $N$ elements and 2D BEM, this corresponds to
$\OO(N^{-3/2})$ and hence $s=3/2$.
With the foregoing proposition and according to Theorem~\ref{thm:mainerr}, the adaptive algorithm attains \revision{any} possible convergence order $0<s \le 3/2$ and the generically quasi-optimal rate is thus achieved.

\begin{consequence}\label{dp:symm:consequence}
Under the assumptions of Proposition~\ref{dp:symm:proposition},
the adaptive algorithm leads to the generically optimal rate 
for the discretization error in the sense of Theorem~\ref{thm:mainerr}.\qed
\end{consequence}

Numerical examples that underline the above result can be found in~\cite{affkp,cc97,cms,cs95,cs96,fkmp}.

\def\hyp{\mathfrak W}
\def\stab{\mathfrak S}
\subsection{Conforming BEM for hyper-singular integral equation}
\label{section:ex:hypsing}
In this section, we consider adaptive BEM for hyper-singular integral 
equations, where the \revision{hyper-singular} operator is of order $+1$.
In this frame, convergence and quasi-optimality of the adaptive algorithm 
has first been proved in~\cite{gantumur}, while the necessary technical 
tools have independently been developed in~\cite{afembem}. While the analysis
of~\cite{gantumur} only covers the lowest-order case $p=1$ and smooth 
boundaries, the recent work~\cite{ffkmp:part2} generalizes this to BEM
with ansatz functions of arbitrary, but fixed polynomial order $p\ge1$
and polyhedral boundaries.

Throughout, we use the notation from Section~\ref{section:ex:symm}. 
Additionally, we assume that $\Gamma\subseteq \partial\Omega$ is  
connected. We consider the hyper-singular integral
equation
\begin{align}\label{eq:def:hypsing:screen}
 \hyp u(x) = f(x)\quad\text{for }x\in\Gamma,
\end{align}
where the hyper-singular integral operator formally reads
\begin{align}\label{def:hyper-singular}
\hyp v(x) := \partial_{n(x)}\int_\Gamma \partial_{n(y)}G(x-y)v(y)\, d\Gamma(y).
\end{align}
By definition, there holds $\hyp g(x)=\partial_n \dlp g(x)$ if the double-layer potential $\dlp g(x)$ is considered as a function on $\Omega$ by evaluating~\eqref{def:doublelayer} for $x\in\Omega$. 
Again, we refer to the monographs~\cite{hw,mclean,ss} for details and proofs
of the following facts on the functional analytic setting:
The hyper-singular integral operator $\hyp$ is symmetric as well as
positive semi-definite and has a one-dimensional kernel which consists of the constant functions, i.e.\ $\hyp 1 = 0$. To deal with this kernel and to obtain an elliptic formulation, 
we distinguish the cases $\Gamma\subsetneqq\partial\Omega$ and
$\Gamma=\partial\Omega$.

\subsubsection{Screen problem $\Gamma\subsetneqq \partial\Omega$}
\label{section:ex:hypsing:screen}
On the screen, the hyper-singular integral operator $\hyp:\,\widetilde H^{1/2+s}(\Gamma)\to H^{-1/2+s}(\Gamma)$ is a continuous mapping for all $-1/2\leq s\leq1/2$.
Here, $\widetilde H^{1/2+s}(\Gamma):=\set{v|_\Gamma}{v\in H^{1/2+s}(\partial \Omega)\text{ with }\text{supp}(v)\subseteq \overline\Gamma}$ denotes the space of functions which can be extended by zero to the entire boundary,
and $H^{-1/2+s}(\Gamma)$ denotes the dual space of $H^{1/2-s}(\Gamma)$.
For given $f\in H^{-1/2}(\Gamma)$, we seek the solution
$u\in \widetilde H^{1/2}(\Gamma)$ of~\eqref{eq:def:hypsing:screen}.

We note that $1\notin \widetilde H^{1/2}(\Gamma)$ and $\hyp: \,\widetilde H^{1/2}(\Gamma)\to H^{-1/2}(\Gamma)$ is a symmetric and elliptic operator.
In particular, 
\begin{align}
b(u,v):=\dual{\hyp u}{v}_{L^2(\Gamma)}
\end{align}
defines an equivalent scalar product on $\XX:=\widetilde H^{1/2}(\Gamma)$. We equip $\widetilde H^{1/2}(\Gamma)$ with the induced Hilbert space norm $\norm{v}{\widetilde H^{1/2}(\Gamma)}^2:=b(v,v)$. 
The hyper-singular integral equation is thus equivalently stated as\begin{align}\label{hypsing:screen:weak}
 b(u,v)=\dual{f}{v}_{L^2(\Omega)}\quad\text{for all }v\in \widetilde H^{1/2}(\Gamma)
\end{align}
and admits a unique solution.

Given a regular triangulation $\TT$ and a polynomial degree $p\ge1$, we 
employ conforming boundary elements
$
\XX(\TT):=\SS^p_0(\TT)
:=\PP^{p}(\TT)\cap \widetilde H^{1/2}(\Gamma).
$
The discrete formulation
\begin{align*}
 b(U(\TT),V) = \dual{f}{V}_{L^2(\Gamma)}
 \quad\text{for all }V\in\SS^p_0(\TT)
\end{align*}
admits a unique BE solution $U(\TT)\in\SS^p_0(\TT)$. In particular, all assumptions
of Section~\ref{sec:setting} are satisfied with 
$\dist[\TT]{v}{w} = \norm{v-w}{\widetilde H^{1/2}(\Gamma)}$ and $\c{triangle}=1$.

Under additional regularity of the data $f\in L^2(\Gamma)$, we may define the weighted-residual error estimator from \cite{cc97,cmps,cs95,cs96} with local contributions
\begin{align}\label{eq:def:hypsing:est}
 \eta_T(\TT;V)^2:=h_T\normLtwo{f-\hyp V}{T}^2\quad\text{for all }T\in\TT.
\end{align}
As in Section~\ref{section:ex:poisson}, the \revision{problem under consideration} involves symmetric and
elliptic $b(\cdot,\cdot)$ and conforming discretizations. Therefore, it only remains to 
discuss stability~\eqref{A:stable}, 
reduction~\eqref{A:reduction}, and discrete reliability~\eqref{A:dlr},  see Section~\ref{section:conforming}. As for for the weakly-singular integral
equation from Section~\ref{section:ex:symm}, efficiency~\eqref{A:efficient} 
is only observed empirically~\cite{cc97,cmps,cs95,cs96}, but a rigorous 
mathematical proof \revision{remains as an open question.}

\begin{proposition}
 The conforming BEM of the hyper-singular integral equation~\eqref{eq:def:hypsing:screen} on the screen with
weighted-residual error estimator~\eqref{eq:def:hypsing:est} satisfies
stability~\eqref{A:stable}, reduction~\eqref{A:reduction}
with $\q{reduction}=2^{-1/(d-1)}$, and discrete reliability~\eqref{A:dlr} with 
\begin{align*}
 \RR(\TT,\widehat\TT) 
 = \TT\setminus\widehat\TT.
\end{align*}
The constants $\c{stable},\c{reduction},\c{dlr}>0$ depend only on 
the polynomial degree $p\in\N$ and on
$\T$.
\end{proposition}

\begin{proof}
The discrete reliability~\eqref{A:dlr} follows with the techniques from~\cite{ckns} which are combined with the localization techniques for the $H^{1/2}(\Gamma)$-norm from~\cite{cmps}. We refer to~\cite{ffkmp:part2} 
for details. For the lowest-order case $p=1$, an alternate proof is found 
in~\cite[Section~4]{gantumur}, where $\RR(\TT,\widehat\TT)=\omega(\TT;\TT\backslash\widehat\TT)$ are the refined elements plus one additional
layer of elements, see~\eqref{dp:patch}.
Stability~\eqref{A:stable} and reduction~\eqref{A:reduction} are proved in~\cite{ffkmp:part2} and use the inverse estimate from~\cite[Corollary~2]{afembem}.
\end{proof}

\begin{consequence}
The adaptive algorithm leads to convergence with quasi-optimal rate
for the estimator $\eta(\cdot)$ in the sense of Theorem~\ref{thm:main}.\qed
\end{consequence}

Numerical examples that underline the above result can be found in~\cite{cs95}.

\subsubsection{Laplace-Neumann problem $\Gamma=\partial\Omega$}
On the closed boundary $\Gamma=\partial\Omega$, the hyper-singular integral operator~\eqref{def:hyper-singular} is continuous for all $-1/2\leq s\leq 1/2$
\begin{align*}
 \hyp:\, H^{1/2+s}(\Gamma)\to H^{-1/2+s}(\Gamma).
\end{align*}
Due to $1\in H^{1/2}(\Gamma)$, we have to stabilize $\hyp$, e.g., with the rank-one operator $\stab v:= \dual{v}{1}_{L^2(\Omega)}\,1$. Alternatively, one could consider $\hyp$ on the factor space $H^{1/2}(\Gamma)/\R\simeq H^{1/2}_\star(\Gamma):=\set{v\in H^{1/2}(\Gamma)}{\int_\Gamma v\,ds =0}$. The (stabilized) hyper-singular integral equation reads
\begin{align}\label{eq:def:hypsing}
 (\hyp+\stab) u(x)=f(x)\quad\text{for } x\in \Gamma.
\end{align}
The sought solution satisfies $u\in\XX:= H^{1/2}(\Gamma)$.
The stabilization $\stab$ allows to define an equivalent scalar product on $H^{1/2}(\Gamma)$ by
\begin{align*}
 b(u,v):=\dual{\hyp u}{v}_{L^2(\Gamma)}+ \dual{u}{1}_{L^2(\Gamma)}\dual{v}{1}_{L^2(\Gamma)}.
\end{align*}
We equip $H^{1/2}(\Gamma)$ with the induced Hilbert space norm $\normHeh{v}{\Gamma}^2=b(v,v)$.
Then,~\eqref{eq:def:hypsing} is equivalent to
\begin{align}\label{eq:hypsing:weak}
 b(u,v)=\dual{f}{v}_{L^2(\Gamma)}\quad\text{for all }v\in H^{1/2}(\Gamma).
\end{align}
In case of $\dual{f}{1}_{L^2(\Gamma)}=0$, we see that $\dual{u}{1}_{L^2(\Gamma)}=0$ by choice of the test function $v=1$. Then, the above formulation~\eqref{eq:def:hypsing} resp.~\eqref{eq:hypsing:weak} is equivalent 
to~\eqref{eq:def:hypsing:screen}.

For given $g\in H^{-1/2}(\Gamma)$ and the special right-hand side
$f=(1/2-\dlp')g$, it holds $\langle f,1\rangle_{L^2(\Gamma)}=0$.
Moreover,~\eqref{eq:def:hypsing:screen} resp.~\eqref{eq:def:hypsing}
is an equivalent formulation of the Laplace-Neumann problem
\begin{align}
 -\Delta \phi=0\quad\text{in }\Omega\quad\text{and}\quad \partial_n\phi=g\quad\text{on }\Gamma=\partial\Omega.
\end{align}
Clearly, the solution $\phi\in H^1(\Omega)$ is only unique up to an additive 
constant. If we fix this constant by $\dual\phi1_{L^2(\Gamma)}=0$,
the density $u\in H^{1/2}(\Gamma)$ which is sought in~\eqref{eq:def:hypsing:screen}
for $f=(1/2-\dlp')g$, is the trace $u=\phi|_{\Gamma}$ of the potential $\phi$.

For fixed $p\ge1$ and a regular triangulation $\TT$ of $\Gamma$, we employ
conforming boundary elements $\XX(\TT):=\SS^p(\TT):=\PP^p(\TT)\cap H^{1/2}(\Gamma)$.
The discrete formulation
\begin{align}\label{eq:def:hypsing:discrete}
 b(U(\TT), V)=\dual{f}{V}_{L^2(\Gamma)}\quad\text{for all }V\in \SS^p(\TT)
\end{align}
admits a unique solution $U(\TT)\in\SS^p(\TT)$. In particular, all assumptions 
of Section~\ref{sec:setting} are satisfied with 
$\dist[\TT]{v}{w} = \norm{v-w}{H^{1/2}(\Gamma)}$ and $\c{triangle}=1$.
In case of $\dual{f}{1}_{L^2(\Gamma)}=0$, it follows as for the continuous case that $\dual{U(\TT)}{1}_{\Gamma}=0$ and therefore $\stab U(\TT) =0$. Hence, the definition of the error estimator as well as the proof of the axioms~\eqref{A:stable},
\eqref{A:reduction}, and~\eqref{A:dlr}
is verbatim to the screen problem in Section~\ref{section:ex:hypsing:screen} and therefore omitted.

\begin{consequence}
The adaptive algorithm leads to convergence with quasi-optimal rate
for the estimator $\eta(\cdot)$ in the sense of Theorem~\ref{thm:main}.\qed
\end{consequence}

Numerical examples that underline the above result can be found in~\cite{cmps}.

Although one may expect an efficiency result~\eqref{A:efficient} similar to that from~\cite{affkp} 
for Symm's integral equation from Section~\ref{section:ex:symm}, see
Consequence~\ref{dp:symm:consequence}, the details have not been worked out 
yet. In particular, quasi-optimality of the adaptive algorithm in the sense of Theorem~\ref{thm:mainerr} \revision{remains as an open question.}

\section{General Second-Order Elliptic Equations}
\label{section:examples2}

\noindent
This section collects further fields of applications for the abstract theory 
developed in the previous Sections~\ref{sec:setting}--\ref{section:optimality}
beyond the Laplace model problem from Section~\ref{section:examples1}.
This includes general second-order linear elliptic operators as well as different
FEM discretizations of the Stokes problem and linear elasticity.

\def\operator#1{{\mathcal{#1}}}
\def\matrix#1{{\boldsymbol{#1}}}
\def\div{{\rm div\;}}
\subsection{Conforming FEM for non-symmetric problems}
\label{section:ex:nonsymm}
On the bounded Lipschitz domain $\Omega\subset \R^d$, we consider the following linear second-order PDE
\begin{align}\label{ex:nonsymm}
 \operator{L}u
 := -\text{div}\matrix{A}\nabla u + \matrix{b}\cdot\nabla u + c u 
 = f\quad\text{in }\Omega
\quad\text{and}\quad
 u=0\quad\text{on }\Gamma.
\end{align}
For all $x\in\Omega$, $\matrix{A}(x)\in \R^{d\times d}$ is a symmetric matrix with $\matrix{A}\in W^{1,\infty}(\Omega; \R^{d\times d}_{\rm sym})$. Moreover, $\matrix{b}(x)\in\R^d$ is a vector with $\matrix{b}\in L^\infty(\Omega;\R^d)$ and $c(x)\in \R$ is a scalar with $c\in L^\infty(\Omega)$.
Note that $\operator{L}$ is non-symmetric as
\begin{align*}
\operator{L}\neq \operator{L}^T=-\text{div}\matrix{A}\nabla u -\matrix{b}\cdot\nabla u + (c-\text{div}\matrix{b}) u.
\end{align*}
We assume that the induced bilinear form 
\begin{align*}
 \bform{u}{v}:=\dual{\operator{L}u}{v}
 = \int_\Omega \matrix{A}\nabla u \cdot \nabla v + \matrix{b}\cdot\nabla u v + cu v\,dx
 \quad\text{for }u,v\in\XX:=H^1_0(\Omega)
\end{align*}
is continuous and \revision{$H^1_0(\Omega)$-elliptic} and denote by $\quasinorm{v}{H^1_0(\Omega)}^2:=\bform{v}{v}$ the induced \emph{quasi-norm} on $H^1_0(\Omega)$.
According to the Lax-Milgram lemma and for given $f\in L^2(\Omega)$, the weak formulation
\begin{align}\label{eq:nonsymmweakform}
 \bform{u}{v} = \int_\Omega fv\,dx
 \quad\text{for all }v\in H^1_0(\Omega)
\end{align}
admits a unique solution $u\in H^1_0(\Omega)$.

Historically, the convergence and quasi-optimality analysis for the
adaptive algorithm has first been developed for elliptic and symmetric
operators, 
e.g.~\cite{d1996,mns,bdd,stevenson07,ckns} to name \revision{some} milestones,
and the analysis strongly used the Pythagoras theorem for
the energy norm~\eqref{dp:pythagoras}. The work~\cite{mn2005} introduced
an appropriate quasi-orthogonality~\eqref{A:qoa} in the $H^1$-norm 
to prove linear convergence of the so-called \emph{total error} which is the 
weighted sum of error plus oscillations. Later,~\cite{cn} used this
approach to prove quasi-optimal convergence rates. However,~\cite{mn2005,cn}
are restricted to $\div\matrix{b}=0$ and sufficiently fine initial meshes 
$\TT_0$ to
prove this quasi-orthogonality. The recent work~\cite{nonsymm} removes
these artificial assumption by proving the general quasi-orthogonality~\eqref{A:qosum}
with respect to the induced energy \emph{quasi-norm} 
$\quasinorm\cdot{H^1_0(\Omega)}$.
Moreover, the latter analysis also provides a framework for convergence
and quasi-optimality if $\bform\cdot\cdot$ is not uniformly elliptic, but 
only satisfies some Garding inequality. For details, the reader is 
referred to~\cite[Section~6]{nonsymm}.

The discretization of~\eqref{eq:nonsymmweakform} is done as in 
Section~\ref{section:ex:poisson}, from where we adopt the notation:
For a given regular triangulation $\TT$
and a polynomial degree $p\ge1$, we consider 
$\XX(\TT) := \SS^p_0(\TT) := \PP^p(\TT)\cap H^1_0(\Omega)$
with $\PP^p(\TT)$ from~\eqref{eq:polynomials}. 
The discrete formulation
also fits into the frame of the Lax-Milgram lemma and
\begin{align}
 b(U(\TT),V) = \int_\Omega fV\,dx
 \quad\text{for all }V\in\SS^p_0(\TT)
\end{align}
hence admits a unique FE solution $U(\TT)\in\SS^p_0(\TT)$. In particular, all assumptions
of Section~\ref{sec:setting} are satisfied with the \emph{quasi-norm}
$\quasinorm{\cdot}{H^1_0(\Omega)}$ and
$\dist[\TT]{v}{w} = \quasinorm{v-w}{H^1_0(\Omega)}$ and some constant
$\c{triangle}\ge1$ which depends only on $\operator{L}$.

The residual error-estimator $\eta(\cdot)$ differs slightly from \revision{the} above, namely
\begin{align}\label{ex:nonsymm:estimator}
\eta_T(\TT;V)^2:=h_T^2\normLtwo{\operator{L}|_T V -f}{T}^2 
+ h_T\normLtwo{[\matrix{A}\nabla V\cdot n]}{\partial T\cap\Omega}^2
\end{align}
for all $T\in\TT$ and all $V\in\SS^p_0(\TT)$ and $\operator{L}|_T V:=-\text{div}|_T \matrix{A}(\nabla|_T V )+ \matrix{b}\cdot \nabla|_T V + c V$, see e.g.~\cite{ao00,v96}.

The \revision{problem under consideration} involves the elliptic \revision{bilinear form} $\bform\cdot\cdot$ and thus fits
into \revision{the framework of} Section~\ref{section:conforming}.

\begin{proposition}\label{prop:ex:nonsymm:assumptions}
The conforming discretization of problem~\eqref{ex:nonsymm} with
residual error estimator~\eqref{ex:nonsymm:estimator} satisfies
stability~\eqref{A:stable}, reduction~\eqref{A:reduction}
with $\q{reduction}=2^{-1/d}$, generalized 
quasi-orthogonality~\eqref{A:qosum}, discrete reliability~\eqref{A:dlr} 
with $\RR(\TT,\widehat\TT) = \TT\backslash\widehat\TT$, and efficiency~\eqref{A:efficient} 
with 
\begin{align}
\begin{split}
 \eff{\TT}{U(\TT)}^2
&:= \min_{F\in\PP^{q}(\TT)}\sum_{T\in\TT}h_T^2\norm{\operator{L}|_T U(\TT)-f-F}{L^2(T)}^2\\
&\qquad\qquad+\min_{F^\prime\in\PP^{q^\prime}(\TT)}\sum_{T\in\TT}h_T\normLtwo{[\matrix{A}\nabla U(\TT)\cdot n]-F}{\partial T\cap\Omega}^2,
\end{split}
\end{align}
where $q,q^\prime\in\N_0$ are arbitrary and $\c{efficient}$ depends on $q,q^\prime$ and on $\T$.
If the differential operator $\operator{L}$ has piecewise polynomial coefficients, 
sufficiently large $q,q^\prime\in\N_0$ even provides~\eqref{A:efficient} with
\begin{align}\label{ex:nonsymm:osc}
 \eff{\TT}{U(\TT)}
 := \osc(\TT) 
 = \min_{F\in\PP^{p-1}(\TT)}\norm{h(\TT)\,(f-F)}{L^2(\Omega)}.
\end{align}%
In particular, there holds $\norm{\osc(\cdot)}{\O_{1/d}}<\infty$ in this case.
The constants $\c{stable},\c{reduction},\c{dlr}>0$
depend only on 
the polynomial degrees $q,q'\in\N$ and on $\T$.
\end{proposition}

\begin{proof}
Stability~\eqref{A:stable}, reduction~\eqref{A:reduction}, and
discrete reliability~\eqref{A:dlr} follow as for the Poisson model problem
from Section~\ref{section:ex:poisson}. Standard arguments from e.g.~\cite{ao00,v96} provide the efficiency~\eqref{A:efficient}. 
The proof of the general quasi-orthogonality~\eqref{A:qosum} follows as in~\cite{nonsymm}:
First, according to Corollary~\ref{cor:convergence}, a~priori convergence of 
$U(\TT_\ell)\to U_\infty$ implies convergence $U(\TT_\ell)\to u$ as 
$\ell\to\infty$. Without loss of generality, assume that $u\neq U_\ell$
for all $\ell\ge0$.

Second, by general Hilbert space arguments, the a~priori convergence 
implies that 
the sequences $e_\ell := (u-U_\ell)/\norm{\nabla(u-U_\ell)}{L^2(\Omega)}$
as well as $E_\ell := (U_{\ell+1}-U_\ell)/\norm{\nabla(u-U_\ell)}{L^2(\Omega)}$
tend weakly to zero, see~\cite[Lemma~6]{nonsymm}.

Third, the non-symmetric part $\operator{K}u:=\matrix{b}\cdot\nabla u$ of 
$\operator{L}$ is a 
compact perturbation. Hence, $\operator{K} e_\ell$ as well as $\operator{K} E_\ell$ converge to zero even strongly in $H^{-1}(\Omega):=H^1_0(\Omega)^*$.

Since $\operator{L}-\operator{K}$ is
symmetric, the following quasi-orthogonality is established as in~\cite[Proposition~7]{nonsymm}. For all $\eps>0$, there exists some index
$\ell_0\in\N$ such that 	
\begin{align*}
 \quasinorm{U(\TT_{\ell+1})-U(\TT_\ell)}{H^1_0(\Omega)}^2
 \leq \frac{1}{1-\eps}\quasinorm{u-U(\TT_\ell)}{H^1_0(\Omega)}^2
 -\quasinorm{u-U(\TT_{\ell+1})}{H^1_0(\Omega)}^2
 \quad\text{for all }\ell\geq \ell_0.
\end{align*}
As shown in~\cite[Proof of Theorem~8]{nonsymm}, one may now choose 
$\eps>0$ sufficiently small to derive for all $N\geq \ell$
\begin{align*}
 \sum_{k= \ell}^N \big(\quasinorm{U(\TT_{k+1})-U(\TT_k)}{H^1_0(\Omega)}^2
 - \epsqo\quasinorm{u-U(\TT_k)}{H^1_0(\Omega)}^2\big) 
 \lesssim \quasinorm{u-U(\TT_\ell)}{H^1_0(\Omega)}^2
 \text{ for all }\ell\geq \ell_0. 
\end{align*}
For the remaining indices $0\leq \ell< \ell_0$,  recall that $\quasinorm{u-U(\TT_\ell)}{H^1_0(\Omega)}=0$ implies 
$\quasinorm{U(\TT_{k+1})-U(\TT_k)}{H^1_0(\Omega)}=0$ for all $k\geq \ell$. 
With the convention $\infty\cdot 0 = 0$, it holds
\begin{align*}
\max_{\ell=0,\ldots,\ell_0-1} \quasinorm{u-U(\TT_\ell)}{H^1_0(\Omega)}^{-2}
\sum_{k=\ell}^{\ell_0-1} \quasinorm{U(\TT_{k+1})-U(\TT_k)}{H^1_0(\Omega)}^2<\infty.
\end{align*}
The combination of the last two estimates finally yields the
general quasi-or\-tho\-go\-na\-lity~\eqref{A:qosum} and concludes the proof.
\end{proof}

\begin{consequence}
The adaptive algorithm leads to convergence with quasi-optimal rate for
the estimator $\eta(\cdot)$ in the sense of 
Theorem~\ref{thm:main}.
For lowest-order elements $p=1$ and piecewise polynomial coefficients of $\operator{L}$, even quasi-optimal rates for the 
discretization error are achieved in the sense of Theorem~\ref{thm:mainerr}.
At the current state of research, higher-order elements $p\ge2$ formally require piecewise polynomial coefficients of $\operator{L}$ 
and additional regularity of $f$.\qed
\end{consequence}
Numerical examples for the symmetric case that underline the above result can be found in~\cite{mn2005}.
\subsection{Lowest-order nonconforming FEM for Stokes}
\label{example:stokes}
The simplest model example for computational fluid dynamics is the  stationary Stokes equations
 \begin{align} 
-\Delta u  +\nabla p= f \quad\text{and}\quad \div u= 0 \quad\text{in }\Omega
\quad\text{and}\quad u = 0\quad\text{on }\Gamma = \partial\Omega
\label{ex:stokes:strongform}
\end{align}
with Dirichlet boundary conditions for the velocity field $u\in H^1_0(\Omega;\R^d)$ along the boundary $\Gamma$ and the pressure field $p\in L^2_0(\Omega):=\set{q\in L^2(\Omega)}{ \int_\Omega q\, dx=0} $. 
 The weak fomulation involves the Hilbert space $\XX := H^1_0(\Omega;\R^d)\times L^2_0(\Omega) $ and the bilinear form
$b(\cdot,\cdot)$ and linear form $F(\cdot)$ with 
\[
b((u,p),(v,q)):=\int_\Omega (  Du:Dv- p\,\ddiv v- q\,\ddiv u)\,dx
\quad\text{and}\quad
F(v,q):=\int_\Omega f\cdot v\, dx
\]
for $(u,p),(v,q)\in\XX$ with  the Frobenius scalar product of 
matrices $A:B:=\sum_{j,k=1}^d A_{jk} B_{jk}$ and the Jacobian $D$. The weak problem 
\[
b((u,p),(v,q))=F(v,q)\quad\text{for all }(v,q)\in\XX
\]
has a unique solution $(u,p)\in\XX$. This and the conforming and nonconforming discretisation is e.g.\ included in
textbooks \cite{braess,SR,brezzi-fortin,Gi_Ra}.

The first contributions on adaptive FEMs for the Stokes Equations involved the Uzawa algorithm \cite{BMN,YS} to overcome the residuals from the 
divergence term. In contrast to this, the nonconforming scheme naturally 
satisfies the side constraint $\ddiv u =0$ piecewise \cite{DDP}
and so enables the convergence and optimality proof \cite{BeMao10,HX07,cpr13}.

Given a regular triangulation $\TT\in\T$, the nonconforming discretisation starts with the nonconforming Crouzeix-Raviart space $CR^1_0(\TT)$ from \eqref{ex:nonconfPMP:CR} and
\begin{equation}\label{ex:nonconfstokes:CR}
\XX(\TT):=CR^1_0(\TT)^d\times  (\PP^0(\TT)\cap L^2_0(\Omega)),
\end{equation}
equipped with the product norm ($\nabla_\TT(\cdot)$ denotes the piecewise gradient)
\begin{align*}
 \norm{(V,Q)}{\XX(\TT)}^2:=\norm{\nabla_\TT V}{L^2(\Omega)}^2 + \normLtwo{Q}{\Omega}^2
\end{align*}
and
$\dist[\TT]{(U,P)}{(V,Q)} = \norm{(U-V,P-Q)}{\XX(\TT)}$.
The differential operators in $b(\cdot,\cdot)$ are understood in the piecewise sense
\[
b\big(\TT;(U,P),(V,Q)\big) := \sum_{T\in\TT} \int_T    (  DU:DV- P\,\ddiv V- Q\,\ddiv U) dx
\]
for all $(U,P), (V,Q)\in\XX(\TT)$. 
The discrete problem
\begin{align}\label{eq:Stokes:CRdiscrete}
b\big(\TT; (U(\TT),P(\TT)),(V,Q)) = \int_\Omega fV\,dx \quad\text{for all }(V,Q)\in \XX(\TT)
\end{align}
admits a unique FE solution $((U(\TT),P(\TT))\in \XX(\TT)$  \cite{braess,SR,brezzi-fortin,Gi_Ra}. 
Recall the jumps of the  tangential derivatives from Section~\ref{section:ex:nonconforming}
and define  the local contributions of the explicit residual-based error estimator  \cite{DDP}
\begin{align}\label{ex:stokes:CRestimator}
 \eta_T(\TT;(V,Q))^2 := h_T^2\,\norm{f}{L^2(T)}^2
 + h_T\,\norm{[\partial_t V]}{L^2(\partial T)}^2\quad \text{for all }T\in\TT,
\end{align}
where $V$ is some part of the discrete test function $Y=(V,Q)\in\XX(\TT)$ and $h_T$ is the local mesh-size defined in~\eqref{ex:laplace:meshsize}.

\begin{proposition}\label{prop:ex:stokes:CRassumptions}
The nonconforming discretization~\eqref{eq:Stokes:CRdiscrete} of the Stokes problem~\eqref{ex:stokes:strongform} on a simply connected domain $\Omega$ with
residual error estimator~\eqref{ex:stokes:CRestimator} satisfies
stability~\eqref{A:stable}, reduction~\eqref{A:reduction}
with $\q{reduction}=2^{-1/d}$, general quasi-orthogonality~\eqref{A:qosum}, discrete reliability~\eqref{A:dlr} with 
$\RR(\TT,\widehat\TT) = \TT\backslash\widehat\TT$ and efficiency~\eqref{A:efficient} 
with $ \eff{\TT}{U(\TT)}:= \osc(\TT)$ from~\eqref{ex:poisson:osc}
and hence $\norm{\osc(\cdot)}{\O_{1/d}}<\infty$. 
The constants $\c{stable},\c{reduction},\c{dlr},\c{efficient}>0$
depend only on $\T$.
\end{proposition}

\begin{proof}
Stability~\eqref{A:stable} and reduction~\eqref{A:reduction} follow as in 
Proposition~\ref{prop:ex:poisson:CRassumptions}.
Efficiency~\eqref{A:efficient} is established in \cite{DDP}, 
while the discrete reliability~\eqref{A:dlr} is shown in \cite[Theorem~3.1]{cpr13} for $d=2$, but the proof essentially 
applies also to the case $d=3$. The aforementioned contributions utilize a continuous or discrete Helmholtz decomposition and are therefore restricted to simply connected domains. The
general case is clarified in \cite{CGS}.

The quasi-orthogonality in the version of \cite[Lemma~4.3]{cpr13} allows an analysis analogous to that of 
Proposition~\ref{prop:ex:poisson:CRassumptions} with the same $\mu(\TT)$ (applied to $f$ in $d$ components rather than one) to prove \eqref{A:qob}.
This and Lemma~\ref{lem:qo} imply the general quasi-orthogonality~\eqref{A:qosum}.
\end{proof}

\begin{consequence}
The adaptive algorithm leads to convergence with quasi-optimal rate in the sense of Theorem~\ref{thm:main} and Theorem~\ref{thm:mainerr}. \qed
\end{consequence}
Numerical examples for 2D that underline the above result can be found in~\cite{BeMao10}.
\subsection{Mixed FEM for Stokes}
\label{example:mixedstokes}
The pseudostress formulation for the Stokes equations \eqref{ex:stokes:strongform} starts with the stress 
$\sigma:= Du -p\, I$ for the $d\times d$ unit matrix $I$ and the velocity $u\in H^1_0(\Omega;\R^d)$ and the
pressure $p\in L^2_0(\Omega)$. Since $u$ is divergence free, the trace free part 
$ \dev  \, \sigma := \sigma- \tr(\sigma) /d\, I$ (with the trace $\tr \sigma:=\sigma_{11}+\dots+\sigma_{dd}=\sigma:I$)
equals the Jacobian matrix  $Du$. With this notation,~\eqref{ex:stokes:strongform} reads
\begin{align*}
 \dev  \, \sigma=Du\quad\text{and}\quad
 f+\ddiv\sigma=0,
 \end{align*}
 where the divergence acts row-wise.
 With the Hilbert space
\[
H:=\set{\tau\in H(\ddiv,\Omega;\R^{d\times d})}{  \int_\Omega \tr(\tau)dx = 0 }
\]
for the stresses, the mixed weak formulation reads 
\begin{equation*}
\begin{array}{rcll}
 \int_\Omega ( \sigma:\dev  \tau + u\cdot \ddiv  \tau)dx  &=& 0&\text{for all }\tau\in H,\\
 \int_\Omega v\cdot \ddiv\sigma\,dx  &=& -\int_\Omega v\cdot f\,dx &\text{for all } v\in L^2(\Omega;\R^d).
\end{array}
\end{equation*}
Given a regular triangulation $\TT\in\T$,
the discrete spaces for the Raviart-Thomas discretisation read
\[
\XX(\TT):= (M_k(\TT)^d\cap H)\times \PP^k(\TT;\R^2)\subset \XX:=H\times   L^2(\Omega;\R^d)
\]
with $M_k(\TT)$ from Section~\ref{section:ex:mixed} and equipped with the norm
$\dist[\TT]{(U,\sigma)}{(V,\tau)} = (\norm{U-V}{L^2(\Omega)}^2+\norm{\sigma-\tau}{H(\ddiv,\Omega;\R^{d\times d})}^2)^{1/2}$.
The discrete formulation
\begin{align*}
 \int_\Omega ( \Sigma(\TT):\dev  \mbox{\large$\tau$} + U(\TT)\cdot \ddiv  \mbox{\large$\tau$})dx  &= 0,\\
 \int_\Omega V\cdot \ddiv\Sigma(\TT)\,dx  &= -\int_\Omega V\cdot f\,dx 
\end{align*}
for all $Y:=(V,\mbox{\large$\tau$})\in\XX(\TT)$ admits a unique solution $X(\TT)=(U(\TT),\Sigma(\TT))\in\XX(\TT)$~\cite{CR73}.
For 
$Y(\TT)=(V,\tau)\in  \XX(\TT)$,
the a posteriori error analysis of  \cite{CKP11} leads to the local contribution 
\begin{align}\label{e:defestimator}
\begin{split}
  \eta_T(\TT;V)^2 :=  \osc^2(f,T)  
  + h_T^2\,  \norm{   \curl(\dev V)  }{L^2(T)}^2  
   + h_T  \,\norm{    [\dev( V)  \times \nu ]   }{L^2(\partial T)}^2
  \end{split}
 \end{align}
with the jumps $ [\dev( V) \times \nu]$ of the tangential  components of the deviatoric part of the stress approximation as in Section~\ref{section:ex:mixed}. 

\begin{proposition}\label{prop:ex:stokes:MFEMassumptions}
The pseudostress formulation of the Stokes equations on a simply connected 
Lipschitz domain $\Omega$ in $d=2$  with
residual error estimator~\eqref{e:defestimator} satisfies
stability~\eqref{A:stable}, reduction~\eqref{A:reduction}
with $\q{reduction}=2^{-1/d}$, general quasi-orthogonality~\eqref{A:qosum}, discrete reliability~\eqref{A:dlr} with 
$\RR(\TT,\widehat\TT) = \TT\backslash\widehat\TT$ and efficiency~\eqref{A:efficient} 
with $ \eff{\TT}{U(\TT)}:= \osc(\TT)$ from~\eqref{ex:poisson:osc}
and hence $\norm{\osc(\cdot)}{\O_{1/d}}<\infty$. 
As above,  the constants $\c{stable}$, $\c{reduction}$, $\c{dlr}$, $\c{efficient}>0$
depend only on 
the polynomial degree $k$ and on
$\T$.
\end{proposition}

\begin{proof}
Stability~\eqref{A:stable} and reduction~\eqref{A:reduction} follow as above --- some details can be found in the proof of
\cite[Theorem 4.1]{CGSpseudo}.
Efficiency~\eqref{A:efficient} 
is contained in \cite{CKP11,CGSpseudo}.  
The recent work \cite{CGSpseudo} presents 
discrete reliability~\eqref{A:dlr}  \cite[Theorem 5.1]{CGSpseudo} and
quasi-orthogonality  in the form
\begin{align*}
\norm{\Sigma(\widehat\TT)-\Sigma(\TT)}{L^2(\Omega)}^2 &\leq 
\norm{\sigma-\Sigma(\TT)}{L^2(\Omega)}^2 - \norm{\sigma-\Sigma(\widehat\TT)}{L^2(\Omega)}^2 \\
&\quad\quad+ C\norm{u-U(\widehat\TT)}{L^2(\Omega)}
\osc(\TT\setminus\widehat\TT,f)
\end{align*}
for some generic constant $C\simeq 1$  \cite[Theorem 4.2]{CGSpseudo}. The proof is based on a discrete Helmholtz decomposition and an equivalence result of the pseudostress method with the nonconforming FEM of the previous subsection and so restricted to $d=2$.
The rearrangements as in Section~\ref{section:ex:mixed} 
with $\mu(\TT):= \osc(\TT;f)$ result in~\eqref{A:qoa}
for {\em any } $0<\epsqo<1$ \revision{with} $C_1:= C^2/\left(\epsqo  (1-2^{-1/2}) 4 \right)$ and $C_2=1$ in
\eqref{A:qob}. 
\end{proof}

\begin{consequence}
The adaptive algorithm leads to convergence with quasi-optimal rate
for the estimator $\eta(\cdot)$ in the sense of Theorem~\ref{thm:main} and Theorem~\ref{thm:mainerr}. \qed
\end{consequence}
Numerical examples that underline the above result can be found in~\cite{CKP11}.
\subsection{Lowest-order nonconforming FEM for linear elasticity}
\label{example:lame}
The Navier Lam\'e equations  form the simplest model problem in solid mechanics with isotropic homogeneous positive material and the Lame parameters $\lambda$ and $\mu$.
Given a polyhedral Lipschitz domain $\Omega\subset\R^d$ and $f\in L^2(\Omega;\R^d)$, the displacement field $u\in \XX:= H^1_0(\Omega; \mathbb R^d)$ satisfies 
\begin{equation}\label{ex:lame:strongformulation}
	-\mu\Delta u - (\lambda + \mu) \nabla \ddiv u=f \text{ in } \Omega.
\end{equation}
The existence and uniqueness of weak solutions with the bilinear form $b(\cdot,\cdot)$ and the linear form $F(\cdot)$ 
and the conforming and nonconforming discretisation is included in the
textbooks~\cite{braess,SR}.
The weak form of~\eqref{ex:lame:strongformulation} reads
\[
b( u,v)   :=  
\int_\Omega \left(  \mu  D u : D v + (\lambda+\mu)  (\ddiv u) (\ddiv v) \right) dx.
\]
Given a regular triangulation $\TT\in\T$, let $\XX(\TT):=CR^1_0(\TT)^d$ denote the nonconforming Crouzeix-Raviart space from Section~\ref{section:ex:nonconforming} and let $\dist[\TT]{\cdot}{\cdot}$ be defined as in Section~\ref{section:ex:nonconforming}. There exists a unique discrete solution $U(\TT)\in \XX(\TT)$ 
such that 
\begin{align}\label{eq:lame:CRdiscrete}
b(\TT; U(\TT),V) = \int_\Omega f\cdot V\,dx \quad\text{for all }V\in \XX(\TT),
\end{align}
where
\[
b(\TT; U(\TT),V)   :=  \sum_{T\in\TT}
\int_T \left(  \mu  D U(\TT) : D V + (\lambda+\mu)  (\ddiv U(\TT)) (\ddiv V) \right) dx.
\]
The error estimator reads  \eqref{ex:poisson:CRestimator} as in Section~\ref{section:ex:nonconforming}
with the little difference that $V$ and $f$ are  no longer scalar but $d$-dimensional. 

\begin{proposition}\label{prop:ex:lame:CRassumptions}
The nonconforming discretization   \eqref{eq:lame:CRdiscrete} of the Navier-Lame equations \eqref{ex:lame:strongformulation}
on the simply connected domain $\Omega$ in 2D with
residual error estimator~\eqref{ex:poisson:CRestimator} satisfies
stability~\eqref{A:stable}, reduction~\eqref{A:reduction}
with $\q{reduction}=2^{-1/d}$, general quasi-orthogonality~\eqref{A:qosum}, discrete reliability~\eqref{A:dlr} with 
$\RR(\TT,\widehat\TT) = \TT\backslash\widehat\TT$ and efficiency~\eqref{A:efficient} 
with $ \hot(\TT):= \osc(\TT;f)$ from Section~\ref{section:ex:poisson}
and hence $\norm{\osc(\cdot)}{\O_{1/d}}<\infty$. 
The constants $\c{stable}$, $\c{reduction}$, $\c{dlr}$, $\c{efficient}>0$
depend only on $\T$ and constraints on $\mu$, but do not depend on $\lambda$.
\end{proposition}

\begin{proof}
Stability~\eqref{A:stable} and reduction~\eqref{A:reduction},
Efficiency~\eqref{A:efficient} plus 
the discrete reliability~\eqref{A:dlr} and  the quasi-orthogonality  follow as in 
Section~\ref{section:ex:nonconforming}. The novel aspect is that 
all the generic constants are independent of $\lambda$ which follows with an application of the  $\tr$-$\dev$-$\ddiv$ lemma  
\cite{Carstensen:2005:Unifying,CR2012}. A discrete Helmholtz decomposition in  \cite{CR2012}
leads to discrete reliability and so  restricts the assertion  to simply connected domains $\Omega$ for $d=2$. 
\end{proof}

\begin{consequence}
The adaptive algorithm leads to robust convergence with quasi-optimal rate in the sense of Theorem~\ref{thm:main} and Theorem~\ref{thm:mainerr}. All constants are independent of $\lambda$. \qed
\end{consequence}

Numerical examples that underline the above result and provide a comparison to conforming finite element simulations can be found in~\cite{CR2012}.

\section{Incorporation of an Inexact Solver Algorithm}
\label{section:inexact}
\noindent
Any evaluation of the solver $U(\cdot)$ depends on the solution of some linear or nonlinear system of equations and may be \revision{polluted} by computational errors.
This contradicts the verification of the axioms~\eqref{A:stable}--\eqref{A:dlr} in Section~\ref{section:examples1}--\ref{section:examples2} for the exact evaluation of the solver $U(\cdot)$.
This section is devoted to the incorporation of this additional error into the optimality analysis.

\subsection{Discrete problem}
In contrast to the previous sections, we do not assume that the discrete approximation $U(\TT)$ is computed exactly by the numerical solver. 
Instead, given some $0<\vartheta<\infty$, we assume that we can compute another discrete approximation
$\widetilde{U}(\TT) \in \XX(\TT)$ such that
\begin{align}\label{eq:iediscrete}
\begin{split}
\dist[\TT]{U(\TT)}{\widetilde{U}(\TT)}
\leq \vartheta \,\eta(\TT;\widetilde{U}(\TT)).
\end{split}
\end{align}
Here, we assume that the error introduced by the \revision{inexact solve} is controlled by
the corresponding error estimator. A similar criterion is found in~\cite[Section~2]{BeMao09}.
Since $\vartheta=0$ in~\eqref{eq:iediscrete} implies $U(\TT)=\widetilde U(\TT)$, the results of this section generalize those of Section~\ref{sec:setting}--\ref{section:optimality}.

\subsection{Residual control of approximation error}
This section illustrates the condition~\eqref{eq:iediscrete} in the context of an iterative solver.
Suppose $\dist[\TT]{v}{w} 
= \quasinorm{v-w}{\XX(\TT)}$ stems from a quasi-norm on $\XX+\XX(\TT)$ and let $\YY(\TT)$ be a suitable 
normed test space. Suppose that $\Bform{\cdot}{\cdot}{\TT}:\XX(\TT)\times\YY(\TT)\to\R$
is linear in \revision{the second component in} $\YY(\TT)$. Given any linear function $F(\TT;\cdot)$, suppose that $U(\TT)$ solves the variational equality
\begin{align}
 \Bform{U(\TT)}{V}{\TT}=F(\TT;V)\quad\text{for all }V\in\YY(\TT).
\end{align}
An iterative solver \revision{terminates} after a finite computation and so specifies an inexact solver
\begin{align*}
\widetilde U(\cdot): \T \to \XX(\cdot).
\end{align*}
Given an accuracy $\eps>0$, common iterative solvers allow to monitor
the discrete residual
\begin{align}\label{eq:ieresest}
 \norm{F(\TT;\cdot)-\Bform{\widetilde U(\TT)}{\cdot}{\TT}}{\YY(\TT)^*}\leq \eps
\end{align}
in terms of the dual norm $\norm{\cdot}{\YY(\TT)^*}:=\sup_{V\in\YY(\TT)\setminus\{0\}}\dual{\cdot}{V}/\norm{V}{\YY(\TT)}$. 
Suppose that $\Bform\cdot\cdot\TT$ satisfies a uniform LBB condition in the sense that
\begin{align}\label{eq:inexact:lbb}
\quasinorm{V}{\XX(\TT)}
 \le \c{lbb}\,\norm{\Bform{V}{\cdot}\TT}{\YY(\TT)^*}
 \quad\text{for all }V\in\XX(\TT)
\end{align}
with some  universal constant $\setc{lbb}>0$. Then, the estimate~\eqref{eq:ieresest}
guarantees
\begin{align*}
 \dist[\TT]{U(\TT)}{\widetilde U(\TT)} \leq \c{lbb}\eps.
\end{align*}
Altogether, the termination with $\eps:=\c{lbb}^{-1}\vartheta\,\eta(\TT;\widetilde U(\TT))$ guarantees~\eqref{eq:iediscrete}.

In particular, the above assumptions are met for the uniformly elliptic 
problems of Section~\ref{section:conforming}
as well as for \revision{the} strongly monotone
operators of Section~\ref{section:nonlinear}.

\subsection{Adaptive algorithm for an inexact solver}
The only difference between the following adaptive algorithm and Algorithm~\ref{algorithm} of Section~\ref{sec:setting} is that the \revision{inexact solve} computes the discrete approximations in Step~(i).
\begin{algorithm}\label{iealgorithm}
\textsc{Input:} Initial triangulation $\TT_0$, parameters $0<\theta,\vartheta<1$.\\
\textbf{Loop: }For $\ell=0,1,2,\ldots$ do ${\rm (i)}-{\rm(iii)}$
\begin{itemize}
\item[\rm(i)] Compute \emph{approximate} discrete approximation $\widetilde{U}(\TT_\ell)\in\XX(\TT)$ as well as the corresponding error estimator $\eta(\TT_\ell;\widetilde{U}(\TT_\ell))$ which satisfy~\eqref{eq:iediscrete}.
\item[\rm(ii)] Determine set $\MM_\ell\subseteq\TT_\ell$ of (almost) minimal cardinality such that
\begin{align}\label{eq:iedoerfler}
 \theta\,\eta(\TT_\ell;\widetilde{U}(\TT_\ell))^2 \le \sum_{T\in\MM_\ell}\eta_T(\TT_\ell;\widetilde{U}(\TT_\ell))^2.
\end{align}
\item[\rm(iii)] Refine (at least) the marked elements $T\in\MM_\ell$ to design new triangulation $\TT_{\ell+1}$.
\end{itemize}
\textsc{Output:} Approximate solutions $\widetilde{U}(\TT_\ell)$ and error estimators
$\eta(\TT_\ell;\widetilde{U}(\TT_\ell))$ for all $\ell\in\N$.
\end{algorithm}

\subsection{Optimal convergence rates}
The following is the main result of this section.

\begin{theorem}\label{thm:main2}
Suppose stability~\eqref{A:stable}, reduction~\eqref{A:reduction}, and general quasi-orthogonality~\eqref{A:qosum}. Then, Algorithm~\ref{iealgorithm} guarantees  {\rm (i)}--{\rm (ii)}.
\begin{itemize}
\item[(i)] 
Discrete reliability~\eqref{A:dlr} \revision{or} reliability~\eqref{A:reliable} and $0\leq \vartheta^2\c{stable}^{2}< \theta$ imply $R$-linear convergence of the estimator in the sense that there exists $0<\setr{conv}<1$ and $\setc{conv}>0$ such that
\begin{align}\label{eq:thm:main2:conv}
\eta(\TT_{\ell+j};\widetilde U(\TT_{\ell+j}))^2\leq \c{conv}\r{conv}^j\,\eta(\TT_\ell;\widetilde U(\TT_\ell))^2\quad\text{for all }j,\ell\in\N_0.
\end{align}
In particular,
\begin{align}\label{eq:thm:main2:conv2}
 \widetilde C_{\rm rel}^{-1}\,\dist[\TT_\ell]{u}{\widetilde U(\TT_\ell)}
 \leq \eta(\TT_\ell;\widetilde U(\TT_\ell))
 \leq \c{conv}^{1/2}\r{conv}^{\ell/2}\,\eta(\TT_0;\widetilde U(\TT_0))\quad\text{for all }\ell\in\N_0.
\end{align}
\item[(ii)] Discrete reliability~\eqref{A:dlr} together with $0<\theta< \theta_\star:=(1+\c{stable}^2\c{dlr}^2)^{-1}$ and
\begin{align}\label{eq:inexact:theta}
 0 < \theta < \sup_{\delta>0}\frac{(1-\vartheta\c{stable})^2\theta_\star-(1+\delta^{-1})\vartheta^2\c{stable}^2}{(1+\delta)}
\end{align}
imply quasi-optimal convergence of the estimator in the sense of
\begin{align}\label{eq:thm:main2:opt1}
c_{\rm opt}\norm{(\eta(\cdot),U(\cdot))}{\B_s} \leq\sup_{\ell\in\N_0}\frac{\eta(\TT_\ell;\widetilde U(\TT_\ell))}{(|\TT_\ell|-|\TT_0|+1)^{-s}}\leq\c{optimal}\norm{(\eta(\cdot),U(\cdot))}{\B_s}
\end{align}
for all $s>0$.
\end{itemize}
The constants $\c{conv}, \r{conv} >0$ depend only on $\c{stable}, \q{reduction}, \c{reduction}, \c{qosum}(\epsqo)>0$ as well as on $\theta$ and $\vartheta$.
Furthermore, the constant $\setc{optimal}>0$ depends only on 
$C_{\rm min}, \c{refined},\c{mesh},\c{stable},\c{dlr}$, $\c{reduction}, C_{\rm son}, \c{qosum}(\epsqo), \q{reduction}>0$ as well as on~$\theta,\vartheta$ and $s$, while $c_{\rm opt}>0$ depends only on $\vartheta, \c{stable}$, and $C_{\rm son}$.
\end{theorem}

\revision{The proof of Theorem~\ref{thm:main2} is the overall subject of this section and found below.}
The following theorem transfers the results of Theorem~\ref{thm:mainerr} to \revision{inexact solve} $\widetilde U(\cdot)$.

\begin{theorem}\label{thm:mainerr2}
Suppose~\eqref{A:stable}--\eqref{A:dlr} as well as efficiency~\eqref{A:efficient} and quasi-monotonicity of error and oscillations~\eqref{eq:errorquasimon}.
Then,~\eqref{eq:inexact:theta} implies quasi-optimal convergence of the error
\begin{align}\label{eq:thm:mainerr2}
\begin{split}
 c_{\rm opt}C_{\rm ie}^{-1}\norm{(u,U(\cdot))}{\A_s}&\leq\sup_{\ell\in\N_0}\frac{\dist[\TT_\ell]{u}{\widetilde U(\TT_\ell)}}{(|\TT_\ell| - |\TT_0|+1)^{-s}}+\norm{\osc(\cdot)}{\O_s}\\
 &\leq \c{optimal}C_{\rm ie}(\norm{(u,U(\cdot))}{\A_s}+ \norm{\osc(\cdot)}{\O_s})
 \end{split}
\end{align}
for all $s>0$. The constants $c_{\rm opt},\setc{optimal}>0$ are defined in Theorem~\ref{thm:main2}.
The constant $C_{\rm ie}>0$ depends only on $\vartheta,\c{stable},\c{reliable},\c{efficient},C_{\rm apx}$.
\end{theorem}

The proof of Theorem~\ref{thm:main2} first establishes that the D\"orfler marking~\eqref{eq:iedoerfler} for $\eta(\TT; \widetilde{U}(\TT))$  implies the D\"orfler marking~\eqref{eq:doerfler} for $\eta(\TT; U(\TT))$ with a different 
parameter $0<\widetilde\theta<1$ and vice versa.

\begin{lemma}\label{lem:doerflerequiv}
Suppose that $\eta(\cdot)$ satisfies stability~\eqref{A:stable}.
Then, any $\TT\in\T$ and
$0<\theta_1,\theta_2,\vartheta<1$ satisfy {\rm (i)--(iii)}.
\begin{itemize}
\item[\rm(i)] 
 $(1-\vartheta\c{stable})\eta(\TT;\widetilde U(\TT))\leq \eta(\TT;U(\TT)) \leq (1+\vartheta\c{stable})\eta(\TT;\widetilde U(\TT)).$
\item[\rm(ii)] Assume that $\theta_2=\theta$ satisfies~\eqref{eq:inexact:theta} with $\theta_1=\theta_\star$.
If $\MM\subset\TT$ satisfies
\begin{align}\label{eq:stardoerfler}
\theta_1 \,\eta(\TT; U(\TT))^2 \leq \sum_{T\in\MM}\eta_T(\TT;U(\TT))^2,
\end{align}
then it follows
\begin{align}\label{eq:pertdoerfler}
\theta_2\, \eta(\TT; \widetilde U(\TT))^2 \leq \sum_{T\in\MM}\eta_T(\TT;\widetilde U(\TT))^2.
\end{align}
\item[\rm(iii)] Provided that $\vartheta^2\c{stable}^2<\theta_2$, there exists $0<\theta_0<\theta_2$
which depends only on $\theta_2$, $\vartheta$, and $\c{stable}$, such that $0<\theta_1\le\theta_0$
guarantees that~\eqref{eq:pertdoerfler} implies~\eqref{eq:stardoerfler}.
\end{itemize}
\end{lemma}

\begin{proof}[Proof of (i)]
Stability~\eqref{A:stable} and the definition of $\widetilde{U}(\TT)$ in~\eqref{eq:iediscrete} show
\begin{align*}
\eta(\TT;U(\TT))&\leq \eta(\TT;\widetilde U(\TT)) + \c{stable}\dist[\TT]{\widetilde U(\TT)}{U(\TT)}\\
&\leq (1+ \vartheta\c{stable})\eta(\TT;\widetilde U(\TT)).
\end{align*}
Analogously, one derives
\begin{align*}
\eta(\TT;\widetilde U(\TT))&\leq \eta(\TT; U(\TT)) + \c{stable}\dist[\TT]{\widetilde U(\TT)}{U(\TT)}\\
&\leq \eta(\TT;U(\TT))+ \vartheta\c{stable}\eta(\TT;\widetilde U(\TT)).
\end{align*}
This implies~(i).
\end{proof}

\begin{proof}[Proof of (ii)]
Suppose that~\eqref{eq:stardoerfler} holds. With~(i), 
stability~\eqref{A:stable} as well as~\eqref{eq:iediscrete} and the Young inequality, it follows, for each $\delta>0$, that
\begin{align*}
(1-\vartheta\c{stable})^2\theta_1\eta(\TT;\widetilde U(\TT))^2&\leq  \theta_1\eta(\TT;U(\TT))^2\leq \sum_{T\in\MM} \eta_T(\TT;U(\TT))^2\\
&\leq
 (1+\delta)\sum_{T\in\MM} \eta_T(\TT;\widetilde U(\TT))^2 + (1+\delta^{-1})\vartheta^2\c{stable}^2 \eta(\TT;\widetilde U(\TT))^2.
\end{align*}
The absorption of the last term proves~\eqref{eq:pertdoerfler} for all
\begin{align}\label{eq:wtheta}
0<\theta_2\leq \sup_{\delta>0}\frac{(1-\vartheta\c{stable})^2\theta_1-(1+\delta^{-1})\vartheta^2\c{stable}^2}{(1+\delta)}.
\end{align}
Therefore, assumption~\eqref{eq:inexact:theta} with $\theta = \theta_2$ and $\theta_\star = \theta_1$
implies~\eqref{eq:pertdoerfler}.
\end{proof}

\begin{proof}[Proof of (iii)]
Suppose that~\eqref{eq:pertdoerfler} holds. The aforementioned arguments show, for each $\delta>0$, that
\begin{align*}
 \theta_2\eta(\TT;\widetilde U(\TT))^2 &\leq \sum_{T\in\MM}\eta_T(\TT;\widetilde U(\TT))^2 \\
&\leq
(1+\delta)\sum_{T\in\MM} \eta_T(\TT;U(\TT))^2 + (1+\delta^{-1})\vartheta^2\c{stable}^2\eta(\TT;\widetilde U(\TT))^2.
\end{align*}
This implies
\begin{align*}
\big( \theta_2-(1+\delta^{-1})\vartheta^2\c{stable}^2\big)\eta(\TT;\widetilde{U}(\TT))^2
&\leq
(1+\delta)\sum_{T\in\MM}\eta_T(\TT;U(\TT))^2.
\end{align*}
The combination with (i) and
\begin{align}\label{eq:k1}
0<\theta_1\leq \sup_{\delta>0}\frac{ \theta_2-(1+\delta^{-1})\vartheta^2\c{stable}^2}{(1+\delta)(1+\vartheta\c{stable})^2}=:\theta_0<\theta_2
\end{align}
establishes~\eqref{eq:stardoerfler}. This concludes the proof.
\end{proof}

\begin{proof}[Proof of Theorem~\ref{thm:mainerr2}]
Proposition~\ref{prop:charAprox},~\eqref{eq:thm:main2:opt1}, and the equivalence from Lemma~\ref{lem:doerflerequiv}~(i) lead to
\begin{align*}
(1+\vartheta\c{stable})^{-1}\c{optimal}^{-1}\c{reliable}^{-1}\norm{(u,U(\cdot))}{\A_s} &\leq\sup_{\ell\in\N_0}\frac{\eta(\TT_\ell;U(\TT_\ell))}{(|\TT_\ell|-|\TT_0|+1)^{-s}}\\
&\leq(1-\vartheta\c{stable})^{-1}\c{optimal}C_{\rm apx}(\norm{(u,U(\cdot))}{\A_s}+\norm{\osc(\cdot)}{\O_s}).
\end{align*}
The arguments of the proof of Theorem~\ref{thm:mainerr} imply
\begin{align}\label{eq:proofmainerr2}
\begin{split}
 (1+\vartheta\c{stable})^{-1}\c{optimal}^{-1}\c{reliable}^{-1}&\c{efficient}^{-1}\norm{(u,U(\cdot))}{\A_s}\leq\sup_{\ell\in\N_0}\frac{\dist[\TT_\ell]{u}{U(\TT_\ell)}}{(|\TT_\ell| - |\TT_0|+1)^{-s}}+\norm{\osc(\cdot)}{\O_s}\\
 &\leq ((1-\vartheta\c{stable})^{-1}\c{optimal}\c{reliable}C_{\rm apx}+1)(\norm{(u,U(\cdot))}{\A_s}+ \norm{\osc(\cdot)}{\O_s}).
\end{split}
 \end{align}
The arguments of~\eqref{eq:iereliable} together with~\eqref{A:stable}, efficiency~\eqref{A:efficient}, and~\eqref{eq:iediscrete} yield
\begin{align*}
\dist[\TT]{u}{U(\TT)} 
&\lesssim \dist[\TT]{u}{\widetilde U(\TT)} 
+ \dist[\TT]{U(\TT)}{\widetilde U(\TT)}\\
&\lesssim \dist[\TT]{u}{\widetilde U(\TT)}+\vartheta \, \eta(\TT;\widetilde U(\TT))\\
&\lesssim \dist[\TT]{u}{\widetilde U(\TT)}+\vartheta \, \eta(\TT; U(\TT))
 + \vartheta \, \dist[\TT]{U(\TT)}{\widetilde U(\TT)}\\
&\lesssim (1+\vartheta)  \dist[\TT]{u}{\widetilde U(\TT)} + \vartheta\, \dist[\TT]{u}{U(\TT)}
 + \vartheta\,\eff{\TT}{U(\TT)}.
\end{align*}
For some sufficiently small $\vartheta$, it follows
\begin{align*}
\dist[\TT]{u}{U(\TT)}
\lesssim \dist[\TT]{u}{\widetilde U(\TT)}+\eff{\TT}{U(\TT)}.
\end{align*}
The converse estimate follows analogously
\begin{align*}
\dist[\TT]{u}{\widetilde U(\TT)}
\lesssim \dist[\TT]{u}{U(\TT)} + \vartheta\,\eta(\TT,U(\TT))
\lesssim \dist[\TT]{u}{U(\TT)} + \vartheta \, \eff{\TT}{U(\TT)}.
\end{align*}
This leads to the equivalence
\begin{align*}
\dist[\TT]{u}{U(\TT)} + \eff{\TT}{U(\TT)}
\simeq \dist[\TT]{u}{\widetilde U(\TT)} + \eff{\TT}{U(\TT)}.
\end{align*}
\revision{The combination with~\eqref{eq:proofmainerr2} concludes the proof.}
\end{proof}%

\begin{proof}[Proof of Theorem~\ref{thm:main2}~(i)]
With $\vartheta^2\c{stable}^{2}<\theta$ and $\theta_2=\theta$, Lemma~\ref{lem:doerflerequiv}~(iii) shows that the
D\"orfler marking~\eqref{eq:doerfler} holds for some $0<\widetilde\theta<1$ in the sense of
\begin{align*}
\widetilde \theta\eta(\TT_\ell;U(\TT_\ell))^2\leq \sum_{T\in\MM_\ell}\eta_T(\TT_\ell;U(\TT_\ell))^2.
\end{align*}
Proposition~\ref{prop:Rconv} provides $R$-linear convergence~\eqref{eq:propRconv} of $\eta(\TT_\ell;U(\TT_\ell))$. 
This and Lemma~\ref{lem:doerflerequiv}~(i) imply $R$-linear convergence of $\eta(\TT_\ell;\widetilde U(\TT_\ell))$ and hence~\eqref{eq:thm:main2:conv}. The reliability~\eqref{A:reliable}, assumption~\eqref{eq:iediscrete},
and Lemma~\ref{lem:doerflerequiv}~(i) lead to
\begin{align}\label{eq:iereliable}
 \begin{split}
 \c{triangle}^{-1} \dist[\TT]{u}{\widetilde U(\TT)}
 &\leq  \dist[\TT]{u}{U(\TT)} + \dist[\TT]{U(\TT)}{\widetilde U(\TT)}\\
 &\leq \c{reliable}\eta(\TT;U(\TT)) + \vartheta\eta(\TT;\widetilde U(\TT))\\
 &\leq (\c{reliable}(1+\vartheta\c{stable}) + \vartheta)\,\eta(\TT;\widetilde U(\TT)),
 \end{split}
\end{align}
i.e.\ reliability of $\eta(\TT;\widetilde U(\TT))$ with $\widetilde C_{\rm rel}:=\c{triangle} (\c{reliable}(1+\vartheta\c{stable}) + \vartheta)$.
This, $\TT=\TT_\ell$ in the estimate above, and~\eqref{eq:thm:main2:conv} conclude the proof of~\eqref{eq:thm:main2:conv2}.
\end{proof}

The following lemma \revision{asserts,} in particular, that the approximation class $\B_s$ from~\eqref{def:approxclass} is a suitable approximation class for the inexact problem~\eqref{eq:iediscrete}.

\begin{lemma}\label{lem:approxClass}
Provided $\vartheta\c{stable}<1$ and $s>0$, it holds
\begin{align}\label{def:ieapproxclass}
(1-\vartheta\c{stable})\norm{(\eta(\cdot),\widetilde U(\cdot))}{\B_s}\leq\norm{(\eta(\cdot),U(\cdot))}{\B_s}\leq(1+\vartheta\c{stable})\norm{(\eta(\cdot),\widetilde U(\cdot))}{\B_s}.
\end{align}
\end{lemma}
\begin{proof}
 The statement follows immediately from Lemma~\ref{lem:doerflerequiv} (i).
\end{proof}

\begin{proof}[Proof of Theorem~\ref{thm:main2}~(ii)]
According to~\eqref{eq:inexact:theta}, there exists $0<\theta_0<\theta_\star$ such that
\begin{align}\label{eq2:inexact:theta}
 0 < \theta < \sup_{\delta>0}\frac{(1-\vartheta\c{stable})^2\theta_0-(1+\delta^{-1})\vartheta^2\c{stable}^2}{(1+\delta)}.
\end{align}
Given $\theta_0$, Proposition~\ref{prop:doerfler}~(ii) provides an appropriate $0<\kappa_0<1$ 
and allows for Lemma~\ref{lem:optimality}. For $ \norm{(\eta(\cdot),U(\cdot))}{\B_s}<\infty$ and $\TT_\ell \in \T$, this guarantees the
existence of a certain refinement $\widehat\TT\in\T$ of $\TT_\ell$ with
\begin{align*}
\eta(\widehat\TT;U(\widehat\TT))^2 \leq  \kappa_0 \eta(\TT_\ell,U(\TT_\ell))^2 \quad \text{and} \quad |\widehat\TT|- |\TT_\ell| \leq 2 \norm{(\eta(\cdot),U(\cdot))}{\B_s}^{1/s}\eta(\TT_\ell,U(\TT_\ell))^{-1/s}
\end{align*}
\revision{for} some set $\RR(\TT_\ell,\widehat\TT)\supseteq \TT_\ell\setminus\widehat\TT$ from 
Proposition~\ref{prop:doerfler}, which satisfies
\begin{align*}
|\RR(\TT_\ell,\widehat\TT)|\leq \c{lemopthelp2} \norm{(\eta(\cdot),U(\cdot))}{\B_s}^{1/s}\eta(\TT_\ell,U(\TT_\ell))^{-1/s},
\end{align*}
as well as the D\"orfler marking~\eqref{eq:doerfler} for $\theta_0$ and $\eta(\cdot;U(\cdot))$.
With~\eqref{eq2:inexact:theta}, Lemma~\ref{lem:doerflerequiv}~(ii) yields that $\RR(\TT_\ell,\widehat\TT)$ 
satisfies the D\"orfler marking~\eqref{eq:iedoerfler}
\begin{align*}
\theta\eta(\TT_\ell;\widetilde U(\TT_\ell))^2\leq \sum_{T\in\RR(\TT_\ell,\widehat\TT)}\eta_T(\TT_\ell;\widetilde U(\TT_\ell))^2.
\end{align*}
The (almost) minimal cardinality of $\MM_\ell$ in Algorithm~\ref{iealgorithm} results in
\begin{align*}
|\MM_\ell| \lesssim |\RR(\TT_\ell,\widehat\TT)| \lesssim\norm{(\eta(\cdot),U(\cdot))}{\B_s}^{1/s} \eta(\TT_\ell,U(\TT_\ell))^{-1/s}\quad \text{for all }\ell\in\N.
\end{align*}
According to Theorem~\ref{thm:main2}~(i), $\eta(\TT_\ell;U(\TT_\ell))$ is $R$-linear convergent. The arguments of the proof of Proposition~\ref{prop:optimality} show
\begin{align*}
 \eta(\TT_\ell;U(\TT_\ell))(|\TT_\ell|-|\TT_0|+1)^{s}\lesssim \norm{(\eta(\cdot),U(\cdot))}{\B_s}\quad\text{for all }\ell\in\N.
\end{align*}
Hence $\eta(\TT_\ell;U(\TT_\ell))$ decays with the optimal algebraic rate. 
The equivalence $\eta(\TT_\ell;U(\TT_\ell))\simeq\eta(\TT_\ell;\widetilde U(\TT_\ell))$ from
Lemma~\ref{lem:doerflerequiv}~(i) proves the upper bound in~\eqref{eq:thm:main2:opt1}.
The lower bound follows as in the proof of Theorem~\ref{thm:main}~(ii) by use of Lemma~\ref{lem:doerflerequiv}~(i).
\end{proof}

%


\section{\revision{Equivalent Error Estimators}}
\label{section:locequiv}
\noindent%
Some error estimators $\varrho(\cdot)$ do not immediately match the abstract framework of Section~\ref{section:axioms}, but are (locally) equivalent to other estimators $\eta(\cdot)$ that do.
Moreover, the local contributions of an error estimator may rather be
associated with facets and/or nodes than with elements.
This section shows quasi-optimal convergence rates for an estimator $\varrho(\cdot)$ if D\"orfler marking with $\varrho(\cdot)$ is equivalent to D\"orfler marking with some mesh-width based error estimator $\eta(\cdot)$ that satisfies the axioms of Section~\ref{section:axioms}. 
Moreover, the discrete reliability axiom~\eqref{A:dlr} is generalized to allow for strong non-linear problems like the $p$-Laplace.
This generalizes~\cite{ks,bdk}.

Affirmative examples and applications are found in Section~\ref{section:examples3}
and~Section~\ref{section:nonlinear} below.

\subsection{Additional assumptions on mesh-refinement}
\label{section:kpatch}
The following assumptions are 
satisfied for all mesh-refinement strategies of Section~\ref{section:refinement}.
The element domains $T\in\TT$ are
compact subsets of $\R^D$ with positive $d$-dimensional measure $|T|>0$ for a fixed $d\leq D$. 
The meshes 
$\TT\in\T$ are uniformly $\gamma$-shape regular in the sense 
of~\eqref{refinement:shapereg} and each refined element domain
$T\in\TT\backslash\widehat\TT$ is the union of its successors, i.e., 
$T = \bigcup\set{\widehat T\in\widehat\TT}{\widehat T\subset T}$. Moreover, 
two different successors 
$\widehat T,\widehat T'\in\widehat\TT$ of $T\in\TT$ are essentially disjoint in the sense that $\widehat T\cap\widehat T'$ has measure zero.  
Finally, for each $\TT\in\T$, let $h(\TT)\in \PP^0(\TT)$ denote 
the piecewise constant  mesh-size function defined by
$h(\TT)|_T = |T|^{1/d}$ as in
Section~\ref{section:examples1} and Section~\ref{section:examples2}. Suppose that there exists a contraction constant $0<\setr{refinereduction}<1$
(which depends only on $\T$), such that all successors $\widehat T\in\widehat\TT$ 
of a refined element $T\in\TT\backslash\widehat\TT$ satisfy
\begin{align}\label{refinement:sons:measure}
|\widehat T|\leq \r{refinereduction}|T|.
\end{align}
The strategies from 
Section~\ref{section:refinement} imply~\eqref{refinement:sons:measure} with $\r{refinereduction}=1/2$. 

Additional notation is required throughout this section.
The $k$-patch $\omega^k(\TT;\SS)\subseteq \TT$ of a subset $\SS\subseteq \TT\in\T$ is successively  defined by
\begin{align*}
\omega(\TT;\SS)&:=\omega^1(\TT;\SS):=\set{T\in\TT}{\text{exists }T^\prime \in\SS \text{ such that }T^\prime\cap T \neq \emptyset}\quad\text{and}\\
\omega^k(\TT;\SS)&:=\omega(\TT;\omega^{k-1}(\TT;\SS))\quad\text{for }k=2,3,\ldots
\end{align*}
To abbreviate notation, set $\omega^k(\TT;T):=\omega^k(\TT;\{T\})$.
The $\gamma$-shape regularity, implies
\begin{align}\label{eq:patchbound}
 |\omega^k(\TT;\SS)|\leq \c{patchbound}|\SS|
 \quad\text{for all }\SS\subseteq\TT\in\T
\end{align}
with some constant $\setc{patchbound}$, which depends only on $\T$ and $k\in\N$.

\subsection{Assumptions on abstract index set}\label{section:index}
For each mesh $\TT\in\T$, let $\II(\TT)$ denote an index set. For each index
$\tau\in\II(\TT)$, let $\TT(\tau)\subseteq\TT$ be a nonempty subset of associated elements. Recall the counting measure $|\cdot|$ for finite sets and suppose uniform boundedness
\begin{align}
 |\TT(\tau)| \le \c{dp1}
 \quad\text{for all }\tau\in\II(\TT)
\end{align}
with a universal constant $\setc{dp1}\ge1$. For each subset
$\Sigma\subseteq\II(\TT)$ of indices, abbreviate 
$\TT(\Sigma):=\bigcup_{\tau\in\Sigma}\TT(\tau)$ and, with a universal constant $\setc{dp2}\ge1$, assume that
\begin{align}\label{ass:dp1}
 \big|\set{\tau\in\II(\TT)}{\TT(\tau)\cap\SS\neq\emptyset}\big|
 \le \c{dp2}\,|\SS|
 \quad\text{for all }\SS\subseteq\TT.
\end{align}

In typical applications, the local contributions of $\varrho(\cdot)$ are 
associated with the element domains $T\in\TT$, the facets $E\in\EE(\TT)$ of $\TT$, and/or the nodes
$z\in\KK(\TT)$ of $\TT$, i.e.\ it holds $\II(\TT)\subseteq\TT\cup\EE(\TT)\cup\KK(\TT)$.
In those cases, $\TT(\tau)$ usually is either the whole corresponding patch or just
one (arbitrary) element of the patch and $\c{dp1},\c{dp2}>0$ depend only
on $\gamma$-shape regularity and hence only on $\T$.

\subsection{Adaptive algorithm}
For each $\TT\in\T$ and $\tau\in\II(\TT)$, let $\varrho_\tau(\TT,\cdot):\,\XX(\TT)\to[0,\infty)$  
denote a function on the discrete space $\XX(\TT)$ with the corresponding error estimator
\begin{align}
 \varrho(\TT,V)^2:=\sum_{\tau\in\II(\TT)}\varrho_\tau(\TT,V)^2
 \quad\text{for all }\TT\in\T\text{ and }V\in\XX(\TT).
\end{align}
The difference between Algorithm~\ref{algorithmrho} below and Algorithm~\ref{algorithm} of Section~\ref{sec:setting} is that  instead of $\eta(\cdot)$, $\varrho(\cdot)$ marks indices $\II(\TT_\ell)$ for refinement in Step~(iii). The refinement step~(iv) refines the element domains $\TT(\MM_\ell)$ associated with the marked indices.

\begin{algorithm}\label{algorithmrho}
\textsc{Input:} Initial triangulation $\TT_0$ and $0<\theta<1$.\\
\textbf{Loop: }for $\ell=0,1,2,\ldots$ do ${\rm (i)}-{\rm(iv)}$
\begin{itemize}
\item[\rm(i)] Compute discrete approximation $U(\TT_\ell)$.
\item[\rm(ii)] Compute refinement indicators $\varrho_\tau(\TT_\ell;U(\TT_\ell))$ for all $\tau\in\II(\TT_\ell)$.
\item[\rm(iii)] Determine set $\MM_\ell\subseteq\II(\TT_\ell)$ of (almost) minimal cardinality such that
\begin{align}\label{eq:doerflerrho}
 \theta\,\varrho(\TT_\ell;U(\TT_\ell))^2 \le \sum_{\tau\in\MM_\ell}\varrho_\tau(\TT_\ell;U(\TT_\ell))^2.
\end{align}
\item[\rm(iv)] Refine (at least) the element domains $T\in\TT(\MM_\ell)$ corresponding to marked indices, to design new triangulation $\TT_{\ell+1}$.
\end{itemize}
\textsc{Output:} Discrete approximations $U(\TT_\ell)$ and error estimators
$\varrho(\TT_\ell;U(\TT_\ell))$ for all $\ell\in\N$.
\end{algorithm}

\subsection{Assumptions on equivalent mesh-width weighted error estimator}\label{section:meshwidtherrest}
\label{section8:assumption}%
Let $\eta(\cdot)$ be a given error estimator of the form
\begin{align}\label{eq:hbasedest}
\eta_T(\TT;V)^2=\eta_T(\TT,\widehat h(\TT);V)^2\quad\text{for all }T\in\TT
\end{align}
with some local mesh-width function $\widehat h(\TT)\in L^\infty(\bigcup\TT)$, either $\widehat h(\TT) = h(\TT)$ or $\widehat h(\TT) = h(\TT, k)$ with  the equivalent mesh-width function $h(\TT,k)$ from Section~\ref{section:modh} below.

Suppose that 
$\varrho(\cdot)$ and $\eta(\cdot)$ are globally equivalent in the sense that, with a universal constant $\setc{dp3}>0$,
\begin{align}\label{eq:glob}
 \c{dp3}^{-1}\,\eta(\TT,h(\TT);U(\TT))^2 \le \varrho(\TT;U(\TT))^2 \le \c{dp3}\,\eta(\TT,h(\TT);U(\TT))^2
 \quad\text{for all }\TT\in\T.
\end{align}

Suppose that D\"orfler marking for $\eta(\cdot)$ and $\varrho(\cdot)$ is equivalent 
in the  sense that there exist constants $k\in\N$ and $\setc{dp4}\ge1$ such that for 
all $\TT\in\T$ the following conditions (i)--(ii) hold:
\begin{itemize}
\item[(i)] If $\MM\subseteq\II(\TT)$ and $0<\theta<1$ satisfy the D\"orfler marking criterion
\begin{subequations}\label{dp:doerfler}
\begin{align}\label{dp1:doerfler}
 \theta\,\varrho(\TT;U(\TT))^2 \le \sum_{\tau\in\MM}\varrho_\tau(\TT;U(\TT))^2,
\end{align}
then, $\widetilde\theta := \c{dp4}^{-1}\theta$ and the $k$-patch $\widetilde\MM:=\omega^k(\TT;\TT(\MM))$ satisfy
\begin{align}\label{dp2:doerfler}
 \widetilde\theta\,\eta(\TT;U(\TT))^2 \le \sum_{T\in\widetilde\MM} \eta_T(\TT;U(\TT))^2.
\end{align}
\end{subequations}
\item[(ii)] Conversely, if $\widetilde\MM\subseteq\TT$ satisfies the D\"orfler marking 
criterion~\eqref{dp2:doerfler} with $0<\widetilde\theta<1$, the set
$\MM:=\set{\tau\in\II(\TT)}{\TT(\tau)\subseteq\omega^k(\TT;\widetilde\MM)\neq\emptyset}$
satisfies~\eqref{dp1:doerfler} with $\theta := \c{dp4}^{-1}\widetilde\theta$.
\end{itemize}

In addition to the general assumptions of Section~\ref{section:setting:continuous}, suppose~\eqref{A:kernelhomo}--\eqref{A:kernelstab}.
\begin{enumerate}
\renewcommand{\theenumi}{{\rm B\arabic{enumi}}}%
\setcounter{enumi}{-1}
\item\label{A:kernelhomo} 
\textbf{Homogeneity}: There exist universal constants $0<r_+\leq r_-<\infty$ such that for
all $T\in\TT\in\T$, $V\in\XX(\TT)$, and $\alpha\in L^\infty(T;[0,1])$ it holds
\begin{align*}
\norm{\alpha}{L^\infty(T)}^{r_-}\eta_T(\TT,\widehat h(\TT);V) \leq \eta_T(\TT,\alpha \widehat h(\TT);V)\leq \norm{\alpha}{L^\infty(T)}^{r_+}\eta_T(\TT,\widehat h(\TT);V).
\end{align*}
\item\label{A:kernelstab} 
\textbf{Stability}: There exists a constant $\setc{stab2}>0$ such that all refinements 
$\widehat\TT\in\T$ of $\TT\in\T$, all functions $\widehat V\in\XX(\widehat\TT)$ and  $V\in\XX(\TT)$,
as well as all $\widehat h(\TT)\in \PP^0(\widehat\TT)$ with $\widehat h(\TT)\leq h(\TT)$ satisfy
\begin{align*}
\begin{split}
\Big|\big(\sum_{T\in\widehat\SS}\eta_T(\widehat \TT, \widehat h(\TT);\widehat V)^2\big)^{1/2}-\big(\sum_{T\in\SS}\eta_T(\TT, \widehat h(\TT);V)^2\big)^{1/2}\Big|
\leq
 \c{stab2}\,\dist[\widehat\TT]{\widehat V}{V}
\end{split}
\end{align*}
for all subsets $\widehat\SS\subseteq\widehat\TT$, $\SS\subseteq \TT$ with $\bigcup\widehat\SS = \bigcup\SS$. 
\end{enumerate}
Note that~\eqref{A:kernelstab} is slightly stronger than~\eqref{A:stable}, since it includes 
the case $\SS\subseteq \TT\cap\widehat\TT$ and $\widehat h(\TT)=h(\TT)$ formulated in~\eqref{A:stable} 
 with $\c{stable}=\c{stab2}$.
Lemma~\ref{lemma:etaestconv1} below asserts that~\eqref{A:kernelhomo}--\eqref{A:kernelstab} imply the reduction 
\revision{axiom~\eqref{A:reduction}. Section~\ref{section:examples3}} below studies the application for the residual FEM error estimator.

Finally, the discrete reliability axiom~\eqref{A:dlr} is weakened.
\begin{enumerate}
\renewcommand{\theenumi}{{\rm B\arabic{enumi}}}%
\setcounter{enumi}{3}
\item\label{B:dlr}\textbf{Weak discrete reliability}:
For all refinements $\widehat\TT\in\T$ of a triangulation $\TT\in\T$ and all $\eps>0$,
there exists a subset $\RR(\eps;\TT,\widehat\TT)\subseteq\TT$ with 
$\TT\backslash\widehat\TT \subseteq\RR(\eps;\TT,\widehat\TT)$ and
$|\RR(\eps;\TT,\widehat\TT)|\le\c{refined}(\eps)\,|\TT\backslash\widehat\TT|$ such that
\begin{align*}
 \dist[\widehat\TT]{U(\widehat\TT)}{U(\TT)}^2
\leq\eps\,\eta(\TT;U(\TT))^2
 + \c{dlr}(\eps)^2 \sum_{T\in\RR(\eps;\TT,\widehat\TT)}\eta_T(\TT;U(\TT))^2.
\end{align*}
\end{enumerate}
The constants $\c{refined}(\eps),\c{dlr}(\eps)>0$ depend only on $\T$ and $\eps>0$.

\begin{lemma}\label{lemma:wdlr2rel}
Discrete reliability~\eqref{A:dlr} implies weak discrete reliability with $\eps=0$ and $\c{dlr}(0)=\c{dlr}$.
Weak discrete reliability~\eqref{B:dlr} implies reliability~\eqref{A:reliable} with $\c{reliable}=\inf_{\eps>0}(\eps+ \c{triangle}\c{dlr}(\eps))$.
\end{lemma}
\begin{proof}
The first statement is obvious. The proof of the second statement follows the lines of the proof of Lemma~\ref{lemma:dlr2rel} with obvious modifications.
\end{proof}

\subsection{Locally equivalent weighted error estimator}
The presentation in~\cite{ks} concerns locally equivalent FEM error estimators which implies~\eqref{eq:glob} and the 
equivalence~\eqref{dp:doerfler}. To prove this,
assume that 
\begin{subequations}\label{eq:loc}
\begin{align}\label{eq:loca}
\varrho_\tau(\TT;U(\TT))^2
&\leq\c{loc}\sum_{T\in\omega^k(\TT;\TT(\tau))}\eta_{T}(\TT;U(\TT))^2,\\
\eta_T(\TT;U(\TT))^2&\leq\c{loc}\sum_{\genfrac{}{}{0pt}{2}{\tau\in\II(\TT)}{\TT(\tau)\cap\omega^k(\TT;T)\neq\emptyset}}\varrho_{\tau}(\TT;U(\TT))^2 \label{eq:locb}
\end{align}
\end{subequations}
for some \revision{fixed} $k\in\N$ and a universal constant $\setc{loc}\geq 1$. 

\begin{lemma}\label{dp:lemma:locequiv}
The local equivalence~\eqref{eq:loc} implies~\eqref{eq:glob} with 
$\c{dp3} = \c{loc}\,\max\{\c{dp1},\c{dp2}\}$. Moreover,~\eqref{dp1:doerfler}
implies~\eqref{dp2:doerfler} with~$\c{dp4}=\c{dp1}\c{dp2}\c{loc}^2$ and vice versa.
\end{lemma}

\begin{proof}
For all $\Sigma\subseteq\II(\TT)$ and $\SS\subseteq\TT$, the local equivalence~\eqref{eq:loc}
yields
\begin{subequations}\label{eq2:loc}
\begin{align}\label{eq2:loca}
\sum_{\tau\in\Sigma}\varrho_\tau(\TT;U(\TT))^2
&\leq\c{dp1}\c{loc}\sum_{T\in\omega^k(\TT;\TT(\Sigma))}\eta_{T}(\TT;U(\TT))^2,
\\
\sum_{T\in\SS}\eta_T(\TT;U(\TT))^2
&\leq\c{dp2}\c{loc}\sum_{\genfrac{}{}{0pt}{2}{\tau\in\II(\TT)}{\TT(\tau)\cap\omega^k(\TT;\SS)\neq\emptyset}}\varrho_{\tau}(\TT;U(\TT))^2. \label{eq2:locb}
\end{align}
\end{subequations}%
For $\SS=\TT$ and $\Sigma=\II(\TT)$, this shows the global equivalences~\eqref{eq:glob}
with $\c{dp3} = \max\{\c{dp1}\c{loc},\c{dp2}\c{loc}\}$.
Moreover, D\"orfler marking~\eqref{dp1:doerfler} for $\varrho(\cdot)$ yields
\begin{align*}
 \theta\,\eta(\TT,h(\TT);U(\TT))^2
 &\le \c{dp2}\c{loc}\,\theta\,\varrho(\TT;U(\TT))^2
 \le \c{dp2}\c{loc}\,\sum_{\tau\in\MM}\varrho_\tau(\TT;U(\TT))^2
 \\&\le \c{dp1}\c{dp2}\c{loc}^2\!\!\!\sum_{T\in\omega^k(\TT;\TT(\MM))}\!\!\!\eta_T(\TT,h(\TT);U(\TT))^2.
\end{align*}
This leads to the D\"orfler marking~\eqref{dp2:doerfler} with 
$\widetilde\theta=\c{dp1}^{-1}\c{dp2}^{-1}\c{loc}^{-2}\,\theta$
and $\widetilde\MM = \omega^k(\TT;\TT(\MM))$. The converse implication follows analogously.
\end{proof}

\subsection{Main result}
The following two theorems are the main result of this section.
Note that the global equivalence~\eqref{eq:glob} of the error estimators implies
\begin{align}
 \norm{(\eta(\cdot),U(\cdot))}{\B_s} \simeq \norm{(\varrho(\cdot),U(\cdot))}{\B_s}.
\end{align}
In particular, $R$-linear convergence and optimal convergence rates do not 
depend on the particular estimator $\eta(\TT_\ell,h(\TT_\ell);U(\TT_\ell))$ or $\varrho(\TT_\ell;U(\TT_\ell))$ considered. To avoid \revision{additional} constants, the main
theorems are therefore formulated with respect to $\eta(\TT_\ell,h(\TT_\ell);U(\TT_\ell))$, although $\varrho(\TT_\ell;U(\TT_\ell))$ is used to \revision{drive} the 
mesh-refinement.

\begin{theorem}\label{thm:main3}
In addition to the assumptions of Section~\ref{section8:assumption}, suppose that $\eta(\cdot)$ satisfies stability~\eqref{A:stable}, reduction~\eqref{A:reduction}, and general quasi-orthogonality~\eqref{A:qosum}. Then, Algorithm~\ref{algorithmrho} guarantees  {\rm (i)}--{\rm (ii)}.
\begin{itemize}
\item[(i)] Weak discrete reliability~\eqref{B:dlr} \revision{or} reliability~\eqref{A:reliable} imply $R$-linear convergence of the estimator in the sense that there exists $0<\setr{conv}<1$ and $\setc{conv}>0$ such that
\begin{align}\label{eq:thm:main3:conv1}
\eta(\TT_{\ell+j},h(\TT_{\ell+j});U(\TT_{\ell+j}))^2\leq \c{conv}\r{conv}^{j}\,\eta(\TT_\ell,h(\TT_\ell);U(\TT_\ell))^2\quad\text{for all }j,\ell\in\N.
\end{align}
In particular, this yields
\begin{align}\label{eq:thm:main3:conv2}
\c{reliable}^{-1} \dist[\TT_\ell]{u}{U(\TT_\ell)} \le \eta(\TT_\ell,h(\TT_\ell); U(\TT_\ell))
 \leq \c{conv}^{1/2}\r{conv}^{\ell/2}\,\eta(\TT_0,h(\TT_0);U(\TT_0))\quad\text{for all }\ell\ge 0.
\end{align}
\item[(ii)] Weak discrete reliability~\eqref{B:dlr} and 
$0<\theta<\c{dp4}^{-1}\,\overline\theta_\star$ with
\begin{align}\label{eq:overlinethetastar}
\overline\theta_\star := \sup_{\eps>0}\frac{1-\c{stable}^2\,\eps}{1+\c{stable}^2\c{dlr}(\eps)^2}
\end{align}
imply quasi-optimal convergence of the estimator
in the sense \revision{that}
\begin{align}\label{eq:thm:main3:opt1}
c_{\rm opt}\norm{(\eta(\cdot),U(\cdot))}{\B_s} \leq\sup_{\ell\in\N_0}\frac{\eta(\TT_\ell,h(\TT_\ell);U(\TT_\ell))}{(|\TT_\ell|-|\TT_0|+1)^{-s}}\leq\c{optimal}\norm{(\eta(\cdot),U(\cdot))}{\B_s}
\end{align}
for all $s>0$.
\end{itemize}
The constants $\c{conv}, \r{conv} >0$ depend only on $\c{stable}, \q{reduction}, \c{reduction}, \c{qosum}(\epsqo)>0$ as well as on $\theta$.
Furthermore, the constant $\setc{optimal}>0$ depends only on 
$C_{\rm min}, \c{refined},\c{mesh},\c{stable},\c{dlr}$, $\c{reduction}, \c{qosum}(\epsqo), \q{reduction}>0$ as well as on~$\theta$ and $s$, while $c_{\rm opt}>0$ depends only on $C_{\rm son}$.
\end{theorem}

\revision{The proof of Theorem~\ref{thm:main3} follows in Section~\ref{section:modh}--\ref{section:proof3} below. An analogous optimality result can also be obtained for the error under the assumption that the error estimator is efficient. }

\begin{theorem}\label{thm:mainerr3}
In addition to the assumptions of Section~\ref{section8:assumption}, supposethat $\eta(\cdot)$ satisfies \eqref{A:stable}--\eqref{A:qosum},~\eqref{B:dlr} as well as efficiency~\eqref{A:efficient} and quasi-monotonicity of oscillations and error~\eqref{eq:errorquasimon}.
 Then, $0<\theta<\c{dp4}^{-1}\,\overline\theta_\star$ with $\overline\theta_\star$ from~\eqref{eq:inexact:theta} implies quasi-optimal convergence of the error
\begin{align}\label{eq:thm:mainerr3}
\begin{split}
 \c{optimal}^{-1}\c{reliable}^{-1}\c{efficient}^{-1}\norm{(u,U(\cdot))}{\A_s}&\leq\sup_{\ell\in\N_0}\frac{\dist[\TT_\ell]{u}{ U(\TT_\ell)}}{(|\TT_\ell| - |\TT_0|+1)^{-s}}+\norm{\osc(\cdot)}{\O_s}\\
 &\leq (\c{optimal}\c{reliable}C_{\rm apx}+1)(\norm{(u,U(\cdot))}{\A_s}+ \norm{\osc(\cdot)}{\O_s})
 \end{split}
\end{align}
for all $s>0$.
The constant $\setc{optimal}>0$ is defined in Theorem~\ref{thm:main3}.
\end{theorem}
\begin{proof}
Since the error estimator $\varrho(\cdot)$ and $\eta(\cdot)$ are equivalent, the arguments of the proof of Theorem~\ref{thm:mainerr} apply and prove the statement.
\end{proof}

This section concludes with an overview on its main arguments.
In general, D\"orfler marking~\eqref{eq:doerflerrho} for $\varrho(\cdot)$ 
does not imply D\"orfler marking~\eqref{eq:doerfler} for $\eta(\cdot)$
with $\TT_\ell(\MM_\ell)$, but may be satisfied with the larger set $\omega^k(\TT_\ell;\TT_\ell(\MM_\ell))$ by
virtue of~\eqref{dp:doerfler}. To ensure the estimator 
reduction~\eqref{eq:estconv} for $\eta(\cdot)$ and to simultaneously avoid the refinement of $T\in\omega^k(\TT_\ell;\TT_\ell(\MM_\ell))$ (as proposed in~\cite{ks}) the analysis of this section modifies $\eta(\cdot)$ by changing the mesh-size function 
$h(\cdot)\simeq h(\cdot,k)$
such that the
resulting error estimator $\eta(\cdot,h(\cdot,k);\cdot)$ 
satisfies~\eqref{eq:estconv} although only  $\TT_\ell(\MM_\ell)$
is refined. Since $\eta(\cdot)$ and $\eta(\cdot,h(\cdot,k);\cdot)$ are even
$\TT$-elementwise equivalent, all properties transfer to $\eta(\cdot,h(\cdot,k);\cdot)$ and lead to $R$-linear convergence for $\eta(\cdot,h(\cdot,k);\cdot)$
and therefore for $\eta(\cdot)$. The optimality analysis, utilizes $\eta(\cdot)$ to obtain corresponding results for $\varrho(\cdot)$
via the global equivalence~\eqref{eq:glob}.

\subsection{Equivalent mesh-width function}\label{section:modh}
The following equivalent mesh-size function is contracted on a patch if at least one element domain of the patch is refined. Its design only requires a mesh-refinement strategy which ensures uniform $\gamma$-shape regularity.
The following proposition generalizes a result from~\cite{ffkmp:part1}.

\begin{proposition}\label{prop:htilde}
Given any $k\in\N$, there exists a modified mesh-width function $h(\cdot,k):\, \T\to L^\infty(\Omega)$ which satisfies~{\rm (i)--(iii)}.
\begin{itemize}
 \item[(i)] Equivalence: For all $\TT\in\T$, it holds
\begin{align}\label{eq:htildeequiv}
 \c{htildeequiv}^{-1}h(\TT)\leq h(\TT,k)\leq h(\TT) \quad \text{almost everywhere in } \bigcup\TT.
\end{align}
\item[(ii)] Contraction on the $k$-patch: All refinements $\widehat\TT\in\T$ of a triangulation $\TT\in\T$ satisfy
\begin{align}\label{eq:htildecontraction}
h(\widehat\TT,k)|_T \leq \q{htc}h(\TT,k)|_T \quad\text{for all }T\in\omega^k(\TT;\TT\setminus\widehat\TT).
\end{align}
\item[(iii)] Monotonicity: All refinements $\widehat\TT\in\T$ of a triangulation $\TT\in\T$ satisfy
\begin{align}\label{eq:htildemon}
h(\widehat\TT,k)\leq h(\TT,k) \quad \text{ almost everywhere in }\bigcup\TT.
\end{align}
\end{itemize}
The constants $\setc{htildeequiv}\geq 1$ and $0<\setq{htc}<1$ depend only on the $\gamma$-shape regularity of the meshes in $\T$, on $k\in\N$, as well as on $\r{refinereduction}$.
\end{proposition}

\begin{proof}
 The $\gamma$-shape regularity of the meshes in $\T$ implies that the mesh-size ratio of the elements in the $k$-patch is uniformly bounded in the sense that
\begin{align}\label{eq:meshbound}
 \c{meshbound}^{-1}\leq h(\TT)|_T/h(\TT)|_{T^\prime} \leq \c{meshbound}\quad \text{for all }T^\prime\in\omega^k(\TT;T),\,T \in\TT,\,\TT\in\T.
\end{align}
The constant $\setc{meshbound}>0$ depends only on the $\gamma$-shape regularity and on $k\in \N$.
Moreover, the number of element domains in the $k$-patch is bounded with~\eqref{eq:patchbound}, i.e.
\begin{align}\label{eq:meshbound2}
 |\omega^k(\TT;T)|\leq \c{patchbound} \quad\text{for all }T \in\TT,\,\TT\in\T.
\end{align}

The first three steps of the proof consider a sequence of consecutive triangulations $(\TT_\ell)_{\ell\in\N}\subset\T$ such that $\TT_{\ell+1}$ is a refinement of $\TT_\ell$.

\textbf{Step~1 }
proves that an element domain $T^\prime \in \omega^k(\TT_\ell;T)$ cannot be refined arbitrarily often (and still be in the $k$-patch of $T$) without refining $T\in\TT_\ell$ itself.
 Suppose that there exist consecutively refined elements 
\begin{align*}
T^\prime=T_0^\prime\supsetneqq T_1^\prime\supsetneqq \ldots \supsetneqq T_N^\prime \quad\text{with } T_i^\prime\in\omega^k(\TT_{\ell+m_i};T) \text{ for all }i=0,\ldots,N
\end{align*}
with a strictly monotone sequence $m_{i+1}>m_i>m_0= 0$, $i=0,\ldots,N-1$  (Note that, in particular $T\in\TT_{\ell+m_i}$ for all $i=0,\ldots,N$).
Assumption~\eqref{refinement:sons:measure} implies $h(\TT_{\ell+m_N})|_{T_N^\prime} \leq \r{refinereduction}^N h(\TT)|_{T_0^\prime}$. The estimate~\eqref{eq:meshbound} and the fact $T\in\TT_{\ell+m_N}$ yield
\begin{align}\label{eq:refineboundfirst}
 h(\TT_{\ell})|_T=h(\TT_{\ell+m_N})|_T \leq \c{meshbound} h(\TT_{\ell+m_N})|_{T_N^\prime}\leq \c{meshbound} \r{refinereduction}^{N} h(\TT_{\ell})|_{T_0^\prime} = \c{meshbound}^2 \r{refinereduction}^{N} h(\TT_{\ell})|_T.
\end{align}
This implies $N\leq N_0$ for the maximal $N_0\in\N$ with $1\leq \c{meshbound}^2\r{refinereduction}^{N_0}$. Note that $N_0\in\N$ solely depends on $\c{meshbound}$ (and hence on $\gamma$ and $k\in\N$ as well as on $\r{refinereduction}$), but neither on $T\in\TT_\ell$ nor on $\TT_\ell\in\T$.

\textbf{Step 2 }
provides a bound on the \revision{number of} refinements which may take place in the $k$-patch of $T$ without refining $T$ itself.
Suppose that
\begin{align}\label{eq:htildehelp}
 \omega^k(\TT_{\ell+m_i};T)\cap (\TT_{\ell+m_i}\setminus\TT_{\ell+m_i+1})\neq\emptyset \quad \text{for }i=1,\ldots,n_T
\end{align}
for a strictly monotone sequence $m_{i+1}>m_i>m_0= 0$, $i=1,\ldots,n_T$. This means that at least $n_T$ elements are refined in the $k$-patch of $T$ without $T$ itself being refined. Introduce counters $c(T^\prime)=0$ for all $T^\prime \in\omega^k(\TT_\ell;T)$ and apply the following algorithm.
\begin{quote}
\textbf{for $i=1,\ldots,n_T$ do}
\begin{itemize}
 \item Determine the unique ancestor $T^\prime\in\omega^k(\TT_\ell;T)$ of each 
 \begin{align*}
 T^{\prime\prime} \in \omega^k(\TT_{\ell+m_i};T)\cap (\TT_{\ell+m_i}\setminus\TT_{\ell+m_i+1})
 \end{align*}
 and increment its counter $c(T^\prime) \mapsto c(T^\prime)+1$.
\end{itemize}
\end{quote}
The bound~\eqref{eq:meshbound2} and the fact that at least one counter is incremented in each iteration of the loop show that there exists at least one counter $c(T^\prime)\geq n_T/\c{patchbound}$ for some $T^\prime\in\omega^k(\TT_\ell;T)$. The definition of the  above algorithm implies the existence of consecutively refined elements $T^\prime=T_0^\prime\supsetneqq T_1^\prime\supsetneqq \ldots \supsetneqq T_{c(T^\prime)}^\prime$ with $T_i^\prime \in \omega^k(\TT_{\ell+m_i};T)$. This and Step~1 show
\begin{align*}
 n_T/\c{patchbound}\leq c(T^\prime)\leq N_0.
\end{align*}
Hence $n_T\leq n_{\rm max}:=N_0\c{patchbound}$ is uniformly bounded and the bound $n_{\rm max}$ depends only on $\gamma$-shape regularity, $\r{refinereduction}$, and on $k\in\N$.

\textbf{Step~3 } successively defines a preliminary modified mesh-width function $\widetilde h(\TT_\ell,k)$ for the particular sequence $\TT_\ell$ of meshes. For $\ell=0$, set $\widetilde h(\TT_0,k)=h(\TT_0)$. For $\ell\geq 0$ and for all $T\in\TT_\ell$ set
\begin{align*}
 \widetilde h(\TT_{\ell+1},k)|_T := \begin{cases}  h(\TT_{\ell+1})|_T & T\in\TT_\ell\setminus \TT_{\ell+1},\\
                              \r{refinereduction}^{1/(n_{\rm max}+1)}\,\widetilde h(\TT_\ell,k)|_T & T\in \omega^k(\TT_\ell;\TT_\ell\setminus \TT_{\ell+1})\setminus (\TT_\ell\setminus \TT_{\ell+1}),\\
			      \widetilde h(\TT_\ell,k)|_T & T\in \TT_\ell\setminus \omega^k(\TT_\ell;\TT_\ell\setminus \TT_{\ell+1}).
                             \end{cases}
\end{align*}
The claim~\eqref{eq:htildeequiv} follows from 
\begin{align}\label{eq:htildeequivpre}
 \r{refinereduction}^{n_{\rm max}/(n_{\rm max}+1)} h(\TT_\ell)|_T\leq \widetilde h(\TT_\ell,k)|_T \leq h(\TT_\ell)|_T\quad\text{for all }T\in\TT.
\end{align}
The upper bound \revision{in~\eqref{eq:htildeequivpre}} follows immediately by mathematical induction on $\ell\in\N$.
The lower \revision{bound in~\eqref{eq:htildeequivpre}} follows by contradiction. Consider an element domain $T\in\TT_j$, $j\in\N$, with
\begin{align}\label{eq:contrastart}
 \widetilde h(\TT_j,k)|_T < \r{refinereduction}^{n_{\rm max}/(n_{\rm max}+1)} h(\TT_j)|_T.
\end{align}  
Let $\ell \le j$ be an index with $\widetilde h(\TT_\ell, k)|_T = h(\TT_\ell)|_T$. Such an index always exists. To see this, assume that the element domain $T$ is refined at some point, i.e. $T \in \TT_{\ell-1}\backslash \TT_\ell$. Then, $\widetilde h(\TT_\ell,k)|T = h(\TT_\ell)|_T$ by definition of $\widetilde h$. Otherwise, assume that $T$ is never refined. Then, the definition states $\widetilde h(\TT_0, k) = h(\TT_0)$.
Hence, to obtain~\eqref{eq:contrastart} there must exist at least $n_{\rm max}+1$ indices $\ell+m_i< j$ with~\eqref{eq:htildehelp}. In terms of Step~2, this means $n_T\geq n_{\rm max}+1$. This contradiction  proves~\eqref{eq:htildeequivpre}.

To prove the contraction estimate~\eqref{eq:htildecontraction} for $\widetilde h(\TT_\ell,k)$, distinguish two cases. If $T\in\TT_\ell\setminus \TT_{\ell+1}$, then, with the lower bound in~\eqref{eq:htildeequivpre}, it holds
\begin{align}\label{eq:hhatcont1}
\begin{split}
 \widetilde h(\TT_{\ell+1},k)|_T &= h(\TT_{\ell+1})|_T\leq \r{refinereduction}\, h(\TT_\ell)|_T \\
&\leq \r{refinereduction}\,\r{refinereduction}^{-n_{\rm max}/(n_{\rm max}+1)}\, \widetilde h(\TT_\ell,k)|_T=\r{refinereduction}^{1/(n_{\rm max}+1)}\, \widetilde h(\TT_\ell,k).
\end{split}
\end{align}
If $T\in\omega^k(\TT_\ell;\TT_\ell\setminus \TT_{\ell+1})\setminus (\TT_\ell\setminus \TT_{\ell+1})$, then, it holds
\begin{align}\label{eq:hhatcont2}
 \widetilde h(\TT_{\ell+1},k)|_T =\r{refinereduction}^{-1/(n_{\rm max}+1)}\,\widetilde h(\TT_\ell,k)|_T. 
\end{align}
Each case leads to some contraction with constant $\q{htc}=\r{refinereduction}^{-1/(n_{\rm max}+1)} \in (0,1)$. 

For $T\in\TT_\ell \setminus \omega^k(\TT_\ell;\TT_\ell\setminus\TT_{\ell+1})$, the definition shows
\begin{align*}
 \widetilde h(\TT_{\ell+1},k)|_T=\widetilde h(\TT_\ell,k)|_T.
\end{align*}
This implies $\widetilde h(\TT_{\ell+1},k)\leq \widetilde h(\TT_\ell,k)$ almost everywhere.

\textbf{Step~4 } improves the preliminary modified mesh-width function $\widetilde h(\TT_\ell,k)$ by removing the dependence on the sequence of meshes $\TT_0,\TT_1,\ldots$ which lead to $\TT_\ell$. So far, for $T\in\TT_\ell$, it holds
\begin{align*}
 \widetilde h(\TT_\ell,k)|_T =\widetilde h(\TT_0,\ldots,\TT_\ell;k)|_T.
\end{align*}
Define the set of all sequences which lead to a particular mesh $\TT\in\T$, i.e.
\begin{align*}
 \T(\TT):=\set{(\TT_0,\ldots,\TT_\ell=\TT)}{ \ell\in\N,\,\TT_{j+1}\text{ is a refinement of }\TT_j \text{ for all }j=0,\ldots,\ell-1}.
\end{align*}
Define $h(\TT,k)\in \PP^0(\TT)$ by
\begin{align*}
h(\TT,k)|_T := \min_{(\TT_0,\ldots,\TT_\ell)\in\T(\TT)} \widetilde h(\TT_0,\ldots,\TT_\ell,k)|_T \quad \text{ for all }T\in\TT\in\T.
\end{align*}
Note that it is valid to take the minimum in the definition above, since the set $\T(\TT)$ is finite up to mesh repetition, i.e. $\TT_{j+1}=\TT_j$.
Equivalence~\eqref{eq:htildeequiv} follows from the fact that all the $\widetilde h(\TT_0,\ldots,\TT_\ell;k)$ are equivalent with the same constants as shown in~\eqref{eq:htildeequivpre}.
The contraction property~\eqref{eq:htildecontraction} can be seen for $T\in\omega^k(\TT;\TT\setminus\widehat\TT)$
\begin{align}\label{eq:htilde1}
  h(\widehat\TT,k)|_T \leq \widetilde h(\TT_0^*,\ldots,\TT_\ell^*,\widehat\TT,k)|_T \leq \q{htc} \widetilde h(\TT_0^*,\ldots,\TT_\ell^*,k) = \q{htc} h(\TT,k)|_T,
\end{align} 
where $(\TT_0^*,\ldots,\TT_\ell^*=\TT)\in\T(\TT)$ is chosen such that $ h(\TT,k)|_T = \widetilde h(\TT_0^*,\ldots,\TT_\ell^*)|_T$.
Finally, monotonicity~\eqref{eq:htildemon} follows with the same arguments that lead to~\eqref{eq:htilde1} by replacing $\rho$ with $1$.
This concludes the proof.
\end{proof}

\subsection{Proof of Theorem~\ref{thm:main3}}\label{section:proof3}
This section transfers the convergence and quasi-optimality results for $\eta(\cdot)$ 
to the locally equivalent estimator $\varrho(\cdot)$ with the help a third error 
estimator.

\begin{lemma}\label{lem:errestequiv}
There exists a constant $\setc{errestequiv}\geq 1$ 
which depends only on $\c{htildeequiv}$ and on the constants $r_+$ and $r_-$ in the 
homogeneity~\eqref{A:kernelhomo},
such that all $T\in\TT\in\T$ and all $V\in\XX(\TT)$ satisfy 
\begin{align}\label{eq:errestequiv}
\c{errestequiv}^{-1} \eta_T(\TT,h(\TT);V)^2\leq \eta_T(\TT,h(\TT,k);V)^2\leq  \eta_T(\TT,h(\TT);V)^2.
\end{align}
In particular, the assumptions general quasi-orthogonality~\eqref{A:qosum}, reliability~\eqref{A:reliable}, weak discrete reliability~\eqref{B:dlr}, and efficiency~\eqref{A:efficient} hold true for $\eta(\cdot,h(\cdot,k);\cdot)$ if and only if their corresponding counterpart holds true for 
$\eta(\cdot,h(\cdot);\cdot)$. Moreover, D\"orfler marking~\eqref{dp2:doerfler}
for $\eta(\cdot,h(\cdot);\cdot))$ and some set
$\widetilde\MM\subseteq\TT$ and $0<\widetilde\theta<1$ implies D\"orfler marking
\begin{align}\label{eq:etakdoerfler}
 \overline\theta\,\eta(\TT,h(\TT,k);U(\TT))^2
 \le \sum_{T\in\widetilde\MM} \eta_T(\TT,h(\TT,k);U(\TT))^2
\end{align}
for $\eta(\cdot,h(\cdot,k);\cdot))$ with the same set $\widetilde\MM$ and $\overline\theta:=\c{errestequiv}^{-1}\,\widetilde\theta$.
\end{lemma}

\begin{proof}
The function $\alpha=h(\TT,k)/h(\TT)\in \PP^0(\TT)$ satisfies 
$\c{htildeequiv}^{-1} \le \norm{\alpha}{L^\infty(\bigcup\TT)} \leq 1$ because of~\eqref{eq:htildeequiv}. Therefore, homogeneity~\eqref{A:kernelhomo} with $\widehat h(\TT) = h(\TT)$
proves~\eqref{eq:errestequiv} with $\c{errestequiv} = \c{htildeequiv}^{r_-}$. 
The remaining statements follow with~\eqref{eq:errestequiv}.
\end{proof}

\begin{lemma}\label{lemma:etaestconv1}
Algorithm~\ref{algorithmrho} enforces for all $\ell\geq 0$ the estimator reduction%
\begin{align}\label{eq:estconvrho}%
\begin{split}%
 \eta(\TT_{\ell+1},h(\TT_{\ell+1}&,k);U(\TT_{\ell+1}))^2\\%
 &\le \q{estconv}\,\eta(\TT_\ell,h(\TT_\ell,k);U(\TT_\ell))^2 + \c{estconv}
 \,\dist[\TT_{\ell+1}]{U(\TT_{\ell+1})}{U(\TT_\ell)}^2.%
\end{split}%
\end{align}%
The constants $0<\q{estconv}<1$ and $\c{estconv}>0$ depend only on 
$\c{stab2},\q{htc},$ as well as $\c{patchbound},\c{dp4},\c{loc},\c{errestequiv}>0$, 
 and on the marking parameter $0<\theta<1$ of the D\"orfler marking~\eqref{eq:doerflerrho} from Proposition~\ref{prop:htilde}.
\end{lemma}

\begin{proof}
First, the estimator is split into a contracting and a non-contracting part
\begin{align}\label{eq:estconvhelp1}
 \eta(\TT_{\ell+1},h(\TT_{\ell+1},k);U(\TT_{\ell+1}))^2
 &= \sum_{T\in\omega^k(\TT_{\ell+1};\TT_{\ell+1}\backslash\TT_{\ell})}\eta_T(\TT_{\ell+1},h(\TT_{\ell+1},k);U(\TT_{\ell+1}))^2\\
 &\quad\quad+ \sum_{T\in\TT_{\ell+1}\setminus\omega^k(\TT_{\ell+1};\TT_{\ell+1}\backslash\TT_{\ell})}\eta_T(\TT_{\ell+1},h(\TT_{\ell+1},k);U(\TT_{\ell+1}))^2.\nonumber
 \end{align}
In the following, stability~\eqref{A:kernelstab} comes into play. Note that $\SS=\omega^k(\TT_{\ell};\TT_{\ell}\backslash\TT_{\ell+1})$ and $\widehat\SS=\omega^k(\TT_{\ell+1};\TT_{\ell+1}\backslash\TT_{\ell})$ satisfy $\bigcup\SS=\bigcup\widehat\SS$. Moreover, due to~\eqref{eq:htildeequiv}, we have $\alpha=h(\TT,k)\leq h(\TT)$ in~\eqref{A:kernelstab}.  The Young inequality and~\eqref{A:kernelstab} imply (for each $\delta>0$) that
 \begin{align*}
  \sum_{T\in\omega^k(\TT_{\ell+1};\TT_{\ell+1}\backslash\TT_{\ell})}\eta_T(\TT_{\ell+1}&,h(\TT_{\ell+1},k);U(\TT_{\ell+1}))^2\\
  &\leq (1+\delta) \sum_{T\in\omega^k(\TT_{\ell};\TT_{\ell}\backslash\TT_{\ell+1})} \eta_T(\TT_\ell,h(\TT_{\ell+1},k);U(\TT_\ell))^2\\
  &\qquad +(1+\delta^{-1})\c{stab2}^2\dist[\TT_{\ell+1}]{U(\TT_{\ell+1})}{U(\TT_\ell)}^2.
  \end{align*}
  Homogeneity~\eqref{A:kernelhomo} with $\alpha = h(\TT_{\ell+1}, k)/h(\TT_{\ell},k)$ and $\widehat h(\TT) = h(\TT_\ell,k)$, and the contraction~\eqref{eq:htildecontraction} yield
    \begin{align*}
   \sum_{T\in\omega^k(\TT_{\ell+1};\TT_{\ell+1}\backslash\TT_{\ell})}\eta_T(\TT_{\ell+1}&,h(\TT_{\ell+1},k);U(\TT_{\ell+1}))^2\\
    &\leq (1+\delta)\q{htc}^{2r_+} \sum_{T\in\omega^k(\TT_{\ell};\TT_{\ell}\backslash\TT_{\ell+1})} \eta_T(\TT_\ell,h(\TT_{\ell},k);U(\TT_\ell))^2\\
  &\qquad +(1+\delta^{-1})\c{stab2}^2\dist[\TT_{\ell+1}]{U(\TT_{\ell+1})}{U(\TT_\ell)}^2.
 \end{align*}
 The second term on the right-hand side of~\eqref{eq:estconvhelp1} is similarly estimated by use of monotonicity~\eqref{eq:htildemon} instead of~\eqref{eq:htildecontraction}. This proves
 \begin{align}\label{eq:stabhelp}
\begin{split}
  \sum_{T\in\TT_{\ell+1}\setminus\omega^k(\TT_{\ell+1};\TT_{\ell+1}\backslash\TT_{\ell})}&\eta_T(\TT_{\ell+1},h(\TT_{\ell+1},k);U(\TT_{\ell+1}))^2\\
  &  \leq  (1+\delta)\sum_{T \in \TT_\ell \setminus \omega^k(\TT_\ell; \TT_\ell\setminus \TT_{\ell+1})}
  \eta_T(\TT_\ell,h(\TT_\ell,k);U(\TT_\ell))^2\\
   &\quad\quad+  (1+\delta^{-1}) \c{stab2}^2\dist[\TT_{\ell+1}]{U(\TT_{\ell+1})}{U(\TT_\ell)}^2.
\end{split}
 \end{align}
Assumption~\eqref{dp:doerfler} and Lemma~\ref{lem:errestequiv} imply that
$\eta(\TT_\ell,h(\TT_\ell,k),\cdot)$ satisfies the D\"orfler marking~\eqref{eq:etakdoerfler} with $\widetilde\MM=\omega^k(\TT_\ell;\TT_\ell(\MM_\ell))$
and $\overline\theta=\c{dp4}^{-1}\c{errestequiv}^{-1}\,\theta$.  
 Therefore, $\omega^k(\TT_\ell;\TT_\ell(\MM_\ell))\subseteq\omega^k(\TT_\ell;\TT_{\ell}\backslash\TT_{\ell+1})$, and the sum of the last two estimates yields
 for $\c{estconv}:=2(1+\delta^{-1})\c{stab2}^2$
 \begin{align*}
 \eta(\TT_{\ell+1},h(\TT_{\ell+1},k);U(\TT_{\ell+1}))^2&\leq(1+\delta)\eta(\TT_\ell,h(\TT_{\ell},k);U(\TT_\ell))^2\\
&\qquad 
-(1+\delta)(1-\q{htc}^{2r_+})
\sum_{T\in\omega^k(\TT_\ell;\TT_{\ell}\backslash\TT_{\ell+1})}\eta_T(\TT_\ell,h(\TT_{\ell},k);U(\TT_\ell))^2\\
&\qquad + \c{estconv}\dist[\TT_{\ell+1}]{U(\TT_{\ell+1})}{U(\TT_\ell)}^2\\
 &\le \big((1+\delta)-(1+\delta)(1-\q{htc}^{2r_+})
 \overline\theta\big)\eta(\TT_\ell,h(\TT_{\ell},k);U(\TT_\ell))^2\\
&\quad\quad + \c{estconv}\dist[\TT_{\ell+1}]{U(\TT_{\ell+1})}{U(\TT_\ell)}^2.
\end{align*}
For sufficiently small $\delta>0$, this proves~\eqref{eq:estconvrho} with $\q{estconv}=(1+\delta)\big(1-(1-\q{htc}^{2r_+})\big)
 \overline\theta
 <1$. 
\end{proof}

\begin{proof}[Proof of Theorem~\ref{thm:main3} (i)]
According to Lemma~\ref{lem:errestequiv} and Lemma~\ref{lemma:wdlr2rel}, $\eta(\cdot,h(\cdot,k);\cdot)$ satisfies general quasi-orthogonality~\eqref{A:qosum} and reliability~\eqref{A:reliable}, and $\eta(\TT_\ell,h(\TT_\ell,k);U(\TT_\ell))$ satisfies the estimator reduction~\eqref{eq:estconvrho}. With~\eqref{eq:estconv} replaced by~\eqref{eq:estconvrho}, Proposition~\ref{prop:Rconv} proves $R$-linear convergence of $\eta(\TT_\ell,h(\TT_\ell,k);U(\TT_\ell))$ to zero, i.e.~\eqref{eq:propRconv} holds
with $\eta(\cdot)$ replaced by $\eta(\TT_\ell,h(\TT_\ell,k);U(\TT_\ell))$.
Since there holds $\eta(\TT_\ell,h(\TT_\ell,k);U(\TT_\ell))\le\eta(\TT_\ell,h(\TT_\ell);U(\TT_\ell))\lesssim\eta(\TT_\ell,h(\TT_\ell,k);U(\TT_\ell))$, the
reliability~\eqref{A:reliable} concludes the proof.
\end{proof}

The following lemma shows that the weak discrete reliability axiom~\eqref{B:dlr} guarantees the optimality of the D\"orfler marking. In particular, the main results
of Section~\ref{section:optimality} and Section~\ref{section:inexact} remain valid 
with~\eqref{A:dlr} replaced by~\eqref{B:dlr}.

\begin{lemma}\label{lemma:etaestconv2}
Suppose that $\eta(\cdot)$ satisfies the weak discrete reliability axiom~\eqref{B:dlr}. 
Then, for all  $0<\theta_0<\overline\theta_\star$ with $\overline\theta_\star$ from~\eqref{eq:overlinethetastar}
there exists some $0<\kappa_0<1$ and $\eps_0>0$ such that 
for all $0<\widetilde\theta\le\theta_0$ and all refinements $\widehat\TT\in\T$ of $\TT\in\T$, 
the assumption
\begin{subequations}\label{eq:equivestoptdoerfler}
\begin{align}
 \eta(\widehat\TT,h(\TT);U(\widehat\TT))^2 &\le \kappa_0\eta(\TT,h(\TT);U(\TT))^2
\end{align}
implies
\begin{align}
 \widetilde\theta\,\eta(\TT,h(\TT);U(\TT))^2
 &\le \sum_{T\in\RR(\eps_0;\TT,\widehat\TT)}\eta_T(\TT,h(\TT);U(\TT))^2
\end{align}
\end{subequations}
with $\RR(\eps_0;\TT,\widehat\TT)$ from~\eqref{B:dlr}.  
The constants $\eps_0$ and $\kappa_0$ depend only on $\theta_0$ and $\overline\theta_\star$.
\end{lemma}

\begin{proof}
Recall from the definition of~\eqref{A:kernelstab} that stability~\eqref{A:stable} holds with $\c{stable}=\c{stab2}$.
The proof of the lemma follows that of Proposition~\ref{prop:doerfler}~(ii) with free variables
$\eps_0>0$ and $\kappa_0$ to be fixed below.
The Young inequality and stability~\eqref{A:stable} show
\begin{align*}
 \eta(\TT,h(\TT);U(\TT))^2
 &= \sum_{T\in\TT\backslash\widehat\TT}\eta_T(\TT,h(\TT);U(\TT))^2
 + \sum_{T\in\TT\cap\widehat\TT}\eta_T(\TT,h(\TT);U(\TT))^2\\
 & \le \sum_{T\in\TT\backslash\widehat\TT}\eta_T(\TT,h(\TT);U(\TT))^2
 + (1+\delta)\sum_{T\in\TT\cap\widehat\TT}\eta_T(\widehat\TT,h(\widehat\TT);U(\widehat\TT))^2\\
 &\quad\quad
 + (1+\delta^{-1})\c{stable}^2\dist[\widehat\TT]{U(\widehat\TT)}{U(\TT)}^2 =: \text{RHS}.
 \end{align*}
 Recall $\TT\backslash\widehat\TT\subseteq\RR(\eps_0;\TT,\widehat\TT)$.
The application of the weak discrete reliability~\eqref{B:dlr} and the assumption 
$\eta(\widehat\TT,h(\widehat\TT);U(\widehat\TT))^2 \le \kappa_0\eta(\TT,h(\TT);U(\TT))^2$ yield
 \begin{align*}
\text{RHS} 
&\le \big((1+\delta)\kappa_0+(1+\delta^{-1})\c{stable}^2\,\eps_0\big)\,\eta(\TT,h(\TT);U(\TT))^2\\
 &\quad\quad + \big(1+(1+\delta^{-1})\c{stable}^2\c{dlr}(\eps_0)^2\big)\,\sum_{T\in\RR(\eps_0;\TT,\widehat\TT)}\eta_T(\TT,h(\TT);U(\TT))^2.
\end{align*}
Some rearrangements prove
\begin{align*}
 \frac{1-(1+\delta)\kappa_0-(1+\delta^{-1})\c{stable}^2\,\eps_0}{1+(1+\delta^{-1})\c{stable}^2\c{dlr}(\eps_0)^2}\,\eta(\TT,h(\TT);U(\TT))^2
 \le \sum_{T\in\RR(\eps_0;\TT,\widehat\TT)}\eta_T(\TT,h(\TT);U(\TT))^2.
\end{align*}
For arbitrary $0<\theta_0<\overline\theta_\star$,
fix $\eps_0>0$, choose $\delta>0$ sufficiently large and then determine $0<\kappa_0<1$ 
with
\begin{align*}
 \theta_0 = \frac{1-(1+\delta)\kappa_0-(1+\delta^{-1})\c{stable}^2\,\eps_0}{1+(1+\delta^{-1})\c{stable}^2\c{dlr}(\eps_0)^2}
 < \frac{1-\c{stable}^2\,\eps_0}{1\!+\!\c{stable}^2\c{dlr}(\eps_0)^2}
 \le \sup_{\eps>0}\frac{1-\c{stable}^2\,\eps}{1\!+\!\c{stable}^2\c{dlr}(\eps)^2}
 = \overline\theta_\star.
\end{align*}
The claim follows for $\theta_0$ and hence for all $0<\widetilde\theta\le \theta_0$.
\end{proof}

\begin{proof}[Proof of Theorem~\ref{thm:main3} (ii)]
Recall linear convergence~\eqref{eq:thm:main3:conv1} of $\eta(\TT_\ell,h(\TT_\ell);U(\TT_\ell))$.
By assumption, $0<\theta<\c{dp4}^{-1}\,\overline\theta_\star$, and set $\widetilde\theta:=\c{dp4}\,\theta<\overline\theta_\star$.
The proof follows that of Proposition~\ref{prop:optimality} with
the difference that the set of marked indices (and hence elements) is determined by 
$\varrho(\cdot)$ instead of $\eta(\cdot,h(\cdot);\cdot)$.
First, Lemma~\ref{lemma:etaestconv2} provides the means to use 
Lemma~\ref{lem:optimality}. Then, $\RR(\eps_0;\TT_\ell,\widehat\TT)$ and hence its superset
$\widetilde\MM:=\omega^k(\TT_\ell;\RR(\eps_0;\TT_\ell,\widehat\TT))$
satisfy the D\"orfler marking~\eqref{eq:doerfler} for $\eta(\cdot,h(\cdot);\cdot)$ with $\widetilde\theta$. 
Second, by assumption~\eqref{dp:doerfler}, $\MM:=\set{\tau\in\II(\TT_\ell)}{\TT_\ell(\tau)\cap\omega^k(\TT_\ell;\widetilde\MM)\neq\emptyset}$ 
satisfies the D\"orfler marking~\eqref{eq:doerflerrho} for $\varrho(\cdot)$ with
$\c{dp4}^{-1}\,\widetilde\theta = \theta$. According to the almost minimal cardinality of 
$\MM_\ell$, assumption~\eqref{ass:dp1}, and uniform shape regularity, it follows that
\begin{align*}
 |\MM_{\ell}| \le |\MM| 
 \le \c{dp2}\,|\omega^k(\TT_\ell;\widetilde\MM)|
 \simeq |\widetilde\MM|
 \simeq|\widetilde\RR(\eps_0;\TT_\ell,\widehat\TT)|.
\end{align*}
The remaining steps are verbatim to the proof of Proposition~\ref{prop:optimality} and are therefore omitted.
\end{proof}

\section{Locally Equivalent Error Estimators for the Poisson Problem}\label{section:examples3}
\noindent
This section applies the analysis of the previous one to 
a \revision{specific} model problem, where the adaptive algorithm is steered by 
some locally equivalent and possibly non-residual error estimator.
This improves the work~\cite{ks}, where all patches of marked element domains are refined. Theorem~\ref{thm:main3} states optimal convergence behaviour of Algorithm~\ref{algorithmrho}, where solely the element domains associated to marked indices are refined.

\subsection{Poisson model problem}
In the spirit of~\cite{ks}, consider the Poisson model 
problem~\eqref{ex:poisson} in $\Omega\subseteq \R^d$,
\begin{align*}
-\Delta u &= f\quad\text{in }\Omega
\quad\text{and}\quad
 u=0\quad\text{on }\Gamma,
\end{align*}
and recall the weak formulation~\eqref{eq:poisson:bilinear}, the FE
discretization~\eqref{eq:poisson:discrete} by means of piecewise polynomials
$\SS^p_0(\TT)=\PP^p(\TT)\cap H^1_0(\Omega)$
of degree $p\ge1$, as well as the definition of $\dist\cdot\cdot:= \norm{\nabla\cdot}{L^2(\Omega)}$.
The residual error 
estimator $\eta(\cdot)$ with local contributions 
\begin{align}\label{exx:poisson:estimator}
 \eta_T(\TT;V)^2
 = \eta_T(\TT,h(\TT);V)^2 := h_T^2\,\norm{f+\Delta_\TT V}{L^2(T)}^2
 + h_T\,\norm{[\partial_nV]}{L^2(\partial T\cap\Omega)}^2
\end{align}
with $h_T:=h(\TT)|_T = |T|^{1/d}$ for all $T\in\TT$
and $\Delta_\TT$ the $\TT$-elementwise Laplacian serves as a theoretical tool. With newest vertex bisection (NVB) the assumptions~\eqref{refinement:sons}--\eqref{refinement:overlay} as well as uniform $\gamma$-shape regularity~\eqref{refinement:shapereg} and all further assumptions of Section~\ref{section:kpatch} are satisfied.

\begin{proposition}
In addition to the axioms~\eqref{A:stable}--\eqref{A:dlr},
the residual error \revision{estimator~\eqref{exx:poisson:estimator}} satisfies efficiency~\eqref{A:efficient},
homogeneity~\eqref{A:kernelhomo}
with $r_+=1/2$ and $r_-=1$, as well as stability~\eqref{A:kernelstab}.
\end{proposition}

\begin{proof}
Proposition~\ref{prop:ex:poisson:assumptions} verifies 
the axioms~\eqref{A:stable}--\eqref{A:dlr} as well as efficiency~\eqref{A:efficient}. Stability~\eqref{A:kernelstab} is well-known and follows by use of the 
triangle inequality as well as standard inverse estimates analogously to 
the proof of~\cite[Corollary~3.4]{ckns}. The homogeneity~\eqref{A:kernelhomo} 
is obvious.
\end{proof} 

The following sections concern different error estimators $\varrho(\cdot)$ which
are equivalent to $\eta(\cdot)$ and fit into the framework of Section~\ref{section:locequiv}.
Section~\ref{section:examples3:meshsize} studies the influence of equivalent choices 
of the mesh-size function $h(\TT)$ for the residual error estimator, Section~\ref{section:dp:facet-based} concerns
a facet-based formulation of $\eta(\cdot)$, while Section~\ref{section:est:zz}
analyzes recovery-based error estimators.
Further examples for the lowest-order case $p=1$, which also fit in the frame of the analysis from Section~\ref{section:locequiv},  are found in~\cite{ks}.

\subsection{Estimator based on equivalent mesh-size function}
\label{section:examples3:meshsize}
Instead of $h_T=|T|^{1/d}$ for weighting the local contributions of
$\eta(\cdot)$, one can also use the local diameter $\diam(T)$. This leads to
\begin{align*}
 \varrho_T(\TT;V)^2 
 = \eta_T(\TT,\widehat h(\TT);V)^2 
 := \diam(T)^2\,\norm{f+\Delta V}{L^2(T)}^2
 + \diam(T)\,\norm{[\partial_nV]}{L^2(\partial T\cap\Omega)}^2
\end{align*}
with the modified mesh-width function $\widehat h(\TT)|_T:=\diam(T)$. 
This variant of $\eta(\cdot)$ is usually found in textbooks as 
e.g.~\cite{ao00,v96}.
The uniform $\gamma$-shape regularity~\eqref{refinement:shapereg}
of newest vertex bisection leads to $h(\TT) \le \widehat h(\TT) 
\le \gamma\,h(\TT)$. In particular, $\eta(\cdot)$ and 
$\varrho(\cdot)$ are elementwise equivalent and so match all the assumptions of
Section~\ref{section:meshwidtherrest}.

\begin{proposition}
The estimators $\eta(\cdot)$ and $\varrho(\cdot)$ are globally
equivalent in the sense that~\eqref{eq:glob} holds with $\c{dp3}=\gamma^2$. Moreover, the
equivalence of D\"orfler marking~\eqref{dp:doerfler} holds with $k=1$, $\MM=\widetilde\MM$,
and $\c{dp4}=\gamma^2$.\qed
\end{proposition}

\begin{consequence}
Convergence and optimal rates for the adaptive algorithm steered by the 
residual error estimator in the sense of Theorem~\ref{thm:main} 
and Theorem~\ref{thm:main3} resp. Theorem~\ref{thm:mainerr3} hold independently of the equivalent mesh-width
function chosen.\qed
\end{consequence}

\subsection{Facet-based formulation of residual error estimator}
\label{section:dp:facet-based}
For a given triangulation $\TT\in\T$, let $\EE(\TT)$ denote the corresponding set
of facets which lie inside $\Omega$, i.e.\ for each $E\in\EE(\TT)$ there
are two unique elements $T,T'\in\TT$ with $T\neq T'$ and $E=T\cap T'$.
Let $\omega(\TT;E):= \{T,T'\}$ and $\bigcup\omega(\TT;E)=T\cup T'$ denote 
the patch of $E\in\EE(\TT)$.
Assume that each element $T\in\TT$ has 
at most one facet on the boundary $\Gamma=\partial\Omega$ which is a \revision{minor additional} assumption on the 
initial mesh $\TT_0$ to exclude pathological cases. 
In particular, each element $T\in\TT$ has at least one node $z\in\KK(\TT)$ 
inside $\Omega$.
For each facet $E\in\EE(\TT)$, let $F_E\in\PP^{p-1}(\bigcup\omega(\TT;E))$ be the unique 
polynomial of degree $p-1$ such that
\begin{align}\label{dp:def:FE}
 \norm{\Delta_\TT V - f - F_E}{L^2(\bigcup\omega(\TT;E))} 
 = \min_{F\in\PP^{p-1}(\bigcup\omega(\TT;E))}\norm{\Delta_\TT V - f - F}{L^2(\bigcup\omega(\TT;E))}. 
\end{align}
With the introduced notation, consider the following facet-based variant of the residual error estimator~\eqref{exx:poisson:estimator}
\begin{subequations}
\begin{align}
 \varrho(\TT;V)^2 &= \sum_{E\in\EE(\TT)}\varrho_E(\TT;V)^2,\\
 \varrho_E(\TT;V)^2 
 &= \diam(E)^2\,\norm{\Delta_\TT V - f - F_E}{L^2(\bigcup\omega(\TT;E))}^2
 + \diam(E)\,\norm{[\partial_nV]}{L^2(E)}^2.
\end{align}
\end{subequations}
Convergence and quasi-optimality for this estimator is directly proved for $d=2$ and
$p=1$ in~\cite{dirichlet2d} via the technical and non-obvious
observation that the edge oscillations are contractive~\cite{pp,page}. The novel approach of this paper
generalizes the mentioned works to arbitrary dimension $d\ge2$ and 
polynomial degree $p\ge1$.

For each facet $E=T\cap T'\in\EE(\TT)$, define $\TT(E):=\{T,T'\}$. In other words
if the edge $E\in\EE(\TT)$ is marked in step~(iii) of Algorithm~\ref{algorithmrho}, the elements of the patch of $E$ will be refined. This does not necessarily imply that the
facet $E$ is refined.
To apply Theorem~\ref{thm:main3} and thus derive convergence with quasi-optimal
rates, it remains to show that $\varrho(\cdot)$ and $\eta(\cdot)$ meet
the assumptions of Section~\ref{section8:assumption}.

\begin{proposition}\label{dp:prop1}
The estimators $\eta(\cdot)$ and $\varrho(\cdot)$ are globally
equivalent~\eqref{eq:glob}. Moreover, equivalence of D\"orfler 
marking~\eqref{dp:doerfler} holds with $k=0$. The constants 
$\c{dp3},\c{dp4}>0$ depend only on $\T$, the polynomial degree
$p\ge1$, and the use of newest vertex bisection.
\end{proposition}

The proof requires some technical lemmas and some further notation: For an interior node $z\in\KK(\TT)\cap\Omega$ of
$\TT$, define the star $\Sigma(\TT;z):=\set{E\in\EE(\TT)}{z\in E}$ as well
as the patch $\omega(\TT;z):=\set{T\in\TT}{z\in T}$. To abbreviate notation, write $\bigcup\omega(\TT;z):=\bigcup_{T\in\omega(\TT;z)}T$. Finally, 
$F_z\in\PP^{p-1}(\bigcup\omega(\TT;z))$ denotes the unique polynomial of degree 
$p-1$ such that
\begin{align}\label{dp:osc:node}
 \norm{\Delta_\TT V - f - F_z}{L^2(\bigcup\omega(\TT;z))} 
 = \min_{F\in\PP^{p-1}(\bigcup\omega(\TT;z))}\norm{\Delta_\TT V - f - F}{L^2(\bigcup\omega(\TT;z))}. 
\end{align}
To abbreviate notation, write $\res(\TT):=\Delta_\TT U(\TT) - f$ for
the residual.

\begin{lemma}\label{dp:lemma1}
Any interior node $z\in\KK(\TT)\cap\Omega$ and $T\in\TT$ with
$z\in T$ satisfies
\begin{align}\label{dp:est1}
\begin{split}
 \c{est1}^{-1}h_T^2\,&\norm{\res(\TT)}{L^2(T)}^2
 \le h_T\,\norm{[\partial_nU(\TT)]}{L^2(\bigcup\Sigma(\TT;z))}^2
 + h_T^2\,\norm{\res(\TT)-F_z}{L^2(\bigcup\omega(\TT;z))}^2.
\end{split}
\end{align}
The constant $\setc{est1}>0$ depends only on $\gamma$-shape regularity 
and hence on $\T$.
\end{lemma}

\begin{proof} 
Consider the nodal basis function $\phi_z\in\SS^1(\TT)$ characterized by 
$\phi_z(z)=1$ and $\phi_z(z')=0$ for all $z'\in\KK(\TT)$ with $z\neq z'$.
In particular, $\text{supp}(\phi_z)= \bigcup\omega(\TT;z)$. 
Let $\Pi_{p-1}(\TT):\,L^2(\bigcup\omega(\TT;z))\to \PP^{p-1}(\bigcup\omega(\TT;z))$ be the $L^2$-orthogonal projection and
note that $F_z=\Pi_{p-1}(\TT)\res(\TT)$. 
A scaling argument and $\norm{\phi_z}{L^\infty(\Omega)}=1$ prove
\begin{align*}
\norm{F_z&}{L^2(\bigcup\omega(\TT;z))}^2
\lesssim\norm{\phi_z^{1/2}F_z}{L^2(\bigcup\omega(\TT;z))}^2\\
&= \int_{\bigcup\omega(\TT;z)}\res(\TT)\phi_zF_z\,dx
- \int_{\bigcup\omega(\TT;z)}\big((1-\Pi_{p-1}(\TT))\res(\TT)\big)\phi_zF_z\,dx\\
&\le \int_{\bigcup\omega(\TT;z)}\res(\TT)\phi_zF_z\,dx
+\normLtwo{(1-\Pi_{p-1}(\TT))\res(\TT)}{\bigcup\omega(\TT;z)}\normLtwo{F_z}{\bigcup\omega(\TT;z)}.
 \end{align*}
 Consider the first term on the right-hand side
and use that $V:=\phi_zF_z\in\SS^p_0(\TT)$ is a suitable test function. With the Galerkin formulation~\eqref{eq:poisson:discrete} and elementwise integration by parts, it follows that
 \begin{align*}
  \int_{\bigcup\omega(\TT;z)} &\res(\TT)\phi_zF_z\,dx
  = \int_{\bigcup\omega(\TT;z)} \res(\TT)V\,dx\\
  &= \int_{\bigcup\omega(\TT;z)} \Delta_\TT U(\TT)\,V\,dx+
   \int_{\bigcup\omega(\TT;z)} \nabla U(\TT)\cdot\nabla V\,dx\\
   &= \int_{\bigcup\Sigma(\TT;z)}[\partial_n U(\TT)]\,\phi_zF_z \,dx\\
   &\leq \normLtwo{[\partial_n U(\TT)]}{\bigcup\Sigma(\TT;z)}\normLtwo{F_z}{\bigcup\Sigma(\TT;z)}.
 \end{align*}
Since $F_z\in\PP^{p-1}(\bigcup\omega(\TT;z))$, an inverse-type inequality 
with $h_z:=\diam(\bigcup\omega(\TT;z))$ shows 
\begin{align*}
 \normLtwo{F_z}{\bigcup\Sigma(\TT;z)}
 \lesssim h_z^{-1/2}\,\norm{F_z}{L^2(\bigcup\omega(\TT;z))}.
\end{align*}
The hidden constant depends only on $\gamma$-shape 
regularity~\eqref{refinement:shapereg} and hence on $\T$.
The combination of the previous arguments implies
\begin{align*}
\norm{F_z&}{L^2(\bigcup\omega(\TT;z))}^2
\lesssim \big(h_z^{-1/2}\,\normLtwo{[\partial_n U(\TT)]}{\bigcup\Sigma(\TT;z)} 
+\normLtwo{\res(\TT)-F_z}{\bigcup\omega(\TT;z)}\big)
\normLtwo{F_z}{\bigcup\omega(\TT;z)}.
\end{align*}
The triangle inequality together with $h_z\simeq h_T$ proves
\begin{align*}
\begin{split}
h_T^2\normLtwo{\Delta_\TT U(\TT)+f}{T}^2
&\lesssim h_T^2\normLtwo{F_z}{\bigcup\omega(\TT;z)}^2+ h_T^2\normLtwo{\res(\TT)-F_z}{\bigcup\omega(\TT;z)}\\
&\lesssim h_T\normLtwo{[\partial_n U(\TT)]}{\bigcup\Sigma(\TT;z)}^2
+h_T^2\normLtwo{\res(\TT)-F_z}{\bigcup\omega(\TT;z)}^2.
\end{split}
\end{align*}
This concludes the proof.
\end{proof}

The following lemma shows that edge oscillations~\eqref{dp:def:FE} and node oscillations~\eqref{dp:osc:node} are equivalent on patches.
\begin{lemma}\label{dp:lemma2}
Any interior node $z\in\KK(\TT)\cap\Omega$ and $T\in\TT$ with
$z\in T$ satisfies
\begin{align}\label{dp:est2}
\begin{split}
 \c{est2}^{-1}\,\norm{\res(\TT)-F_z}{L^2(\bigcup\omega(\TT;z))}^2
 &\le \sum_{E\in\Sigma(\TT;z)}\norm{\res(\TT)-F_E}{L^2(\bigcup\omega(\TT;E))}^2
 \\&
 \le \c{est3}\,\norm{\res(\TT)-F_z}{L^2(\bigcup\omega(\TT;z))}^2.
\end{split}
\end{align}
The constants $\setc{est2},\setc{est3}>0$ depend only on $\T$, the polynomial
degree $p\ge1$, and the use of newest vertex bisection.
\end{lemma}

\begin{proof}
The upper bound \revision{in~\eqref{dp:est2}} follows from
\begin{align*}
 \norm{\res(\TT)-F_E}{L^2(\bigcup\omega(\TT;E))}
 \le \norm{\res(\TT)-F_z}{L^2(\bigcup\omega(\TT;E))}
 \le \norm{\res(\TT)-F_z}{L^2(\bigcup\omega(\TT;z))}
\end{align*}
for all $E\in\Sigma(\TT;z)$ and the fact that the cardinality 
$|\Sigma(\TT;z)|$ is uniformly bounded in terms of the uniform shape 
regularity constant $\gamma$. 

The lower bound \revision{in~\eqref{dp:est2}} is first proved for a piecewise polynomial
$f\in\PP^{p-1}(\TT)$. This yields $\res(\TT) \in \PP^{p-1}(\TT)$. We employ equivalence of seminorms on finite dimensional spaces and scaling arguments.
Note that both terms in~\eqref{dp:est2} define seminorms on $\PP^{p-1}(\omega(\TT;z))$
with the kernel $\PP^{p-1}(\bigcup\omega(\TT;z))$ and hence are
equivalent with constants $\c{est2},\c{est3}>0$.
A scaling argument proves that these constants depend
only on the shape of $\bigcup\omega(\TT;E)$ or $\bigcup\Sigma(\TT;z)$.
Since newest vertex bisection only leads to finitely many shapes of 
triangles and hence patches and facet stars, this proves that 
$\c{est2}$ and $\c{est3}$ depend only on $\T$, $p$, and the use of newest 
vertex bisection.

It remains to prove the lower bound \revision{in~\eqref{dp:est2}} for general
$f\in L^2(\Omega)$. Let $\Pi(\TT):L^2(\Omega)\to\PP^{p-1}(\TT)$ denote
the $L^2$-projection \revision{so that} $F(\TT) = \Pi(\TT)\res(\TT)$ \revision{is
the unique solution to} 
\begin{align*}
 \norm{\res(\TT)-F(\TT)}{L^2(T)}
 = \min_{F\in\PP^{p-1}(T)} \norm{\res(\TT)-F}{L^2(T)}
 \quad\text{for all }T\in\TT.
\end{align*}
Note that $\PP^{p-1}(\bigcup\omega(\TT;E))\subset\PP^{p-1}(\omega(\TT;E))$.
Since $F_E$ and $F(\TT)$ are the corresponding $L^2$-orthogonal projections
of $\res(\TT)$, this yields
\begin{align}
 \norm{F(\TT) - F_E}{L^2(\bigcup\omega(\TT;E))} 
 = \min_{F\in\PP^{p-1}(\bigcup\omega(\TT;E))}\norm{F(\TT) - F}{L^2(\bigcup\omega(\TT;E))}.
\end{align} 
According to the $\TT$-elementwise Pythagoras theorem and the foregoing 
discussion for a $\TT$-piecewise polynomial $f$, it follows
\begin{align*}
 \norm{\res(\TT)-&F_z}{L^2(\bigcup\omega(\TT;z))}^2
 = \norm{\res(\TT)-F(\TT)}{L^2(\bigcup\omega(\TT;z))}^2
 + \norm{F(\TT)-F_z}{L^2(\bigcup\omega(\TT;z))}^2 
 \\&
 \lesssim \sum_{E\in\Sigma(\TT;z)}
 \big(\norm{\res(\TT)-F(\TT)}{L^2(\bigcup\omega(\TT;E))}^2
 + \norm{F(\TT)-F_E}{L^2(\bigcup\omega(\TT;E))}^2\big)
 \\&
 = \sum_{E\in\Sigma(\TT;z)}\norm{\res(\TT)-F_E}{L^2(\bigcup\omega(\TT;E))}^2.
\end{align*}
This concludes the proof.
\end{proof}

\begin{proof}[Proof of Proposition~\ref{dp:prop1}]
According to Lemma~\ref{dp:lemma:locequiv}, it remains to verify~\eqref{eq:loc}.
The uniform $\gamma$-shape regularity~\eqref{refinement:shapereg} yields
$h_E=\diam(E) \simeq h_T$ for all $E\in\EE(\TT)$ and $T\in\TT$ with
$E\subseteq T$. Hence
\begin{align*}
 \varrho_E(\TT;U(\TT))^2
 &= h_E^2\,\norm{\res(\TT)-F_E}{L^2(\bigcup\omega(\TT;E))}^2
 + h_E\,\norm{[\partial_nU(\TT)}{L^2(E)}^2\\
 &\lesssim \sum_{T\in\omega(\TT;E)}\big(
  h_E^2\norm{\res(\TT)}{L^2(T)}^2 
  + h_E\,\norm{[\partial_nU(\TT)}{L^2(\partial T\cap\Omega)}^2
 \big)\\
 &\simeq \sum_{T\in\TT(E)}\eta_T(\TT;U(\TT))^2.
\end{align*}
This proves~\eqref{eq:loca}.
For each interior node $z\in\KK(\TT)\cap\Omega$ of $T\in\TT$, 
Lemma~\ref{dp:lemma1} and~\ref{dp:lemma2} imply
\begin{align*}
 \eta_T(\TT,h(\TT);U(\TT))^2
   &= h_T^2\,\norm{\res(\TT)}{L^2(T)}^2 
   + h_T\,\norm{[\partial_nU(\TT)]}{L^2(\partial T\cap\Omega)}^2\\
   &\le h_T^2\,\norm{\res(\TT)-F_z}{L^2(\bigcup\omega(T;z))}^2 
   + h_T\,\norm{[\partial_nU(\TT)]}{L^2(\bigcup\Sigma(\TT;z))}^2\\ 
   &\simeq \sum_{E\in\Sigma(\TT;z)}\big(
   h_T^2\,\norm{\res(\TT)-F_E}{L^2(\bigcup\omega(T;E))}^2 
   + h_T\,\norm{[\partial_nU(\TT)]}{L^2(E)}^2\big)\\
   &\simeq \sum_{E\in\Sigma(\TT;z)}\varrho_E(\TT;U(\TT))^2.
\end{align*}
Since $\Sigma(\TT;z) \subseteq \set{E\in\EE(\TT)}{E\cap T\neq\emptyset}$,
this concludes the proof of~\eqref{eq:locb}. 
\end{proof}

\begin{remark}
This section concerns the natural choice $\TT(E)=\{T,T'\}$ for 
$E=T\cap T'\in\EE(\TT)$ for the relation between the index set $\EE(\TT)$
and the elements $\TT$. Remarkably, the abstract analysis of 
Section~\ref{section:locequiv} would even guarantee convergence with 
optimal rates, for fixed $k\in\N_0$, if $\TT(E)$ is an \emph{arbitrary} nonempty subset of $\omega^k(\TT(E))$.\qed
\end{remark}

\begin{consequence}
Convergence and optimal rates for the adaptive algorithm in the sense of Theorem~\ref{thm:main} 
and Theorem~\ref{thm:main3} resp.\ Theorem~\ref{thm:mainerr3} hold even for the facet-based error estimator.\qed
\end{consequence}
Numerical examples that underline the above result can be found in for 2D and lowest-order elements in~\cite{fpp}. Moreover, numerical examples for the obstacle problem with the facet-based estimator are found in~\cite{pp,page}.
\subsection{Recovery-based error estimator $\boldsymbol{\varrho(\cdot)}$}
\label{section:est:zz}
In this section, we consider recovery-based error estimators for FEM
which are occasionally also called ZZ-estimators after Zienkiewicz and 
Zhu~\cite{zz}. These estimators are popular in computational science and engineering because of their implementational ease and striking performance in many applications. Reliability has independently
shown by~\cite{rod94,cb02} for lowest-order elements $p=1$ and
later generalized to higher-order elements $p\ge1$ in~\cite{cb02p}. 
For the lowest-order case, convergence and quasi-optimality of the related
adaptive mesh-refining algorithm has been analyzed in~\cite{ks}. In the following, the result of~\cite{ks} is reproduced and even generalized to higher-order elements $p\geq 1$. Moreover, the abstract analysis of Section~\ref{section:locequiv} removes the artificial refinements~\cite{ks}.

Adopt the definition of $\varrho(\cdot)$ and the notation of Section~\ref{section:dp:facet-based} and let $G(\TT):\,L^2(\Omega) \to \SS^p_0(\TT)$ denote the local 
averaging operator which is defined as follows:
\begin{itemize}
 \item For lowest-order polynomials $p=1$, define $G(\TT)(v)\in\SS^1_0(\TT)$ by
 \begin{align*}
  G(\TT)(v)(z):= \frac{1}{|\omega(\TT;z)|}\,\int_{\bigcup\omega(\TT;z)} v\,dx\quad\text{for all inner nodes }z\in\KK(\TT)\cap\Omega.
 \end{align*}
\item For the general case $p\ge1$, define $G(\TT) = J(\TT): H^1_0(\Omega)\to \SS^p_0(\TT)$ as the Scott-Zhang projection from~\cite{scottzhang}.
\end{itemize}
Based on $G(\TT)$, the local estimator contributions of 
the recovery-based error estimator $\varrho(\cdot)$ read
\begin{align}\label{eq:recest}
\begin{split}
\varrho_\tau(\TT;U(\TT))^2 
&:=
\begin{cases}
 \normLtwo{(1 - G(\TT)) \nabla U(\TT)}{T}^2
 &\text{for }\tau = T \in\TT,\\
 \diam(E)^2\,\norm{\Delta_\TT U(\TT)-f-F_E}{L^2(E)}^2\quad
 &\text{for }\tau = E \in\EE(\TT).
\end{cases}
\end{split}
\end{align}
Note that $\II(\TT) = \TT \cup \EE(\TT)$ with respect to the abstract
notation of Section~\ref{section:locequiv}. We define $\TT(T) = \{T\}$
for $T\in\TT$ and $\TT(E) = \{T,T'\}$ for $E=T\cap T'\in\EE(\TT)$.

\begin{proposition}\label{dp:prop2}
For general polynomial degree $p\ge1$, the error estimators $\eta(\cdot)$ and 
$\varrho(\cdot)$ satisfy the local equivalences~\eqref{eq:loc} for $k=2$.
\end{proposition}

The proof requires the following lemma which states that the normal jumps
are locally equivalent to averaging. The result is well-known 
for the lowest-order case, \revision{and} its proof is included for the convenience
of the reader.

\begin{lemma}\label{dp:lemma3}
For some interior node $z\in\KK(\TT)\cap\Omega$, it holds
\begin{align}\label{dp:est4}
\begin{split}
 \c{est4}^{-1}\,h_T\,\normLtwo{[\partial_n U(\TT)]}{\bigcup\Sigma(\TT;z)}^2
 &\le \normLtwo{(1 - G(\TT))\nabla U(\TT)}{\bigcup\omega(\TT;z)}^2
 \\&
 \le \c{est5}
  \sum_{z^\prime\in\Sigma(\TT;z)\cap\KK(\TT)\cap\Omega} h_{z'}
 \normLtwo{[\partial_n U(\TT)]}{\bigcup\Sigma(\TT;z^\prime)}^2.
\end{split}
\end{align}
The constants $\setc{est4},\setc{est5}>0$ depend only on $\T$, the polynomial
degree $p\ge1$, and the use of newest vertex bisection.
\end{lemma}

\begin{proof}
We use equivalence of seminorms on finite dimensional spaces  and 
scaling arguments.
To prove~\eqref{dp:est4}, it thus suffices to show that the chain of inequalities holds true if one term is zero. 

First, assume $(1 - G(\TT))\nabla U(\TT)=0$ on $\bigcup\omega(\TT;z)$. This implies $\nabla U(\TT)\in\SS^p(\omega(\TT;z))$ and hence $[\partial_n U(\TT)]=0$ on $\bigcup\Sigma(\TT;z)$. 

Second, assume $[\partial_n U(\TT)]=0$ on $\bigcup\Sigma(\TT;z^\prime)$ for all inner nodes $z^\prime$ of $\Sigma(\TT;z)$. This shows that the normal jumps of $\nabla U(\TT)$ are zero over $\bigcup\Sigma(\TT;z^\prime)$. Since $U(\TT)\in H^1(\Omega)$, the tangential jumps of $\nabla U(\TT)$  also vanish over $\Sigma(\TT;z^\prime)$. Altogether, this implies $\nabla U(\TT)\in \SS^{p-1}(\omega(\TT;z^\prime))$ for all $z^\prime$. If the Scott-Zhang projection defines the averaging, $G(\TT)(\nabla U(\TT)(z^\prime)$ depends only on $\nabla U(\TT)|_{\omega(\TT;z^\prime)}$, this implies $G(\TT)\nabla U(\TT)=\nabla U(\TT)$. In the particular case $p=1$ and patch averaging, $\nabla U(\TT)$ is constant on $\omega(\TT;z^\prime)$. In any case, we thus derive $(1-G(\TT))\nabla U(\TT) = 0$ on $\bigcup\omega(\TT;z)$. 

The constants in~\eqref{dp:est4} depend on the shapes of patches 
$\bigcup\omega(\TT;z')$ involved.
Since NVB leads to only finitely many patch shapes, we deduce that the 
these constants depend only on the polynomial degree $p\in\N$ and on $\T$.
\end{proof}

\begin{proof}[Proof of Proposition~\ref{dp:prop2}]
In order to prove the local equivalence~\eqref{eq:loc} of $\varrho(\cdot)$
and $\eta(\cdot)$, let $z\in\KK(\TT)\cap\Omega$ be an interior node
of $T\in\TT$. The upper estimate in~\eqref{dp:est4} yields
\begin{align*}
 \varrho_T(\TT;U(\TT)^2
 \lesssim \sum_{T'\in\omega^2(\TT;T)}\eta_{T'}(\TT;U(\TT))^2.
\end{align*}
For $E=T'\cap T\in\EE(\TT)$, it holds
\begin{align*}
 \varrho_E(\TT;U(\TT))^2
 \le 2\,\norm{\res(\TT)}{L^2(T)}^2 + 2\,\norm{\res(\TT)}{L^2(T')}^2
 \le 2\,\sum_{T'\in\omega(\TT;T)}\eta_{T'}(\TT;U(\TT))^2.
\end{align*}
The combination of the last two estimates proves~\eqref{eq:loca}. The proof of~\eqref{eq:locb} employs
Lemma~\ref{dp:lemma1} and~\ref{dp:lemma2} as well as the lower bound in~\eqref{dp:est4}. 
For an interior node $z\in\KK(\TT)\cap\Omega$ of $T\in\TT$, it follows
\begin{align*}
 \eta_T(\TT;U(\TT)^2
 &\lesssim h_T\,\norm{[\partial_n U(\TT)]}{L^2(\bigcup\Sigma(\TT;z))}^2
 + h_T^2\,\sum_{E\in\Sigma(\TT;z)}\norm{\res(\TT)-F_E}{L^2(\bigcup\omega(\TT;E)}^2\\
 &\lesssim \sum_{\genfrac{}{}{0pt}{2}{\tau\in\TT\cup\EE(\TT)}{\tau\cap T\neq\emptyset}}
 \varrho_\tau(\TT;U(\TT))^2.
\end{align*}
This concludes the proof.
\end{proof}

\begin{consequence}
The adaptive algorithm leads to convergence with quasi-optimal rate
for the estimator $\varrho(\cdot)$ in the sense of Theorem~\ref{thm:main3}. For 
lowest-order elements $p=1$, Theorem~\ref{thm:mainerr3} states optimal rates for the discretization error, while for higher-order elements $p\geq 1$, additional regularity of
$f$ has to be imposed, e.g., $f\in H^1(\Omega)$ for $p=2$. \qed
\end{consequence}
\def\operator#1{{\mathcal{#1}}}
\def\matrix#1{{\boldsymbol{#1}}}
\def\div{{\rm div\;}}
\section{Adaptive FEM for Nonlinear Model Problems}
\label{section:nonlinear}
In this section, we give three examples of adaptive FEM for nonlinear problems. Each problem relies on different approaches, however, all fit into the abstract analysis of Section~\ref{section:optimality} resp. Section~\ref{section:locequiv}. 
\subsection{Conforming FEM for certain strongly-monotone operators}
In this section, we consider a possibly nonlinear  generalization of the model problem of Section~\ref{section:ex:nonsymm}. 
On a Lipschitz domain $\Omega\subset \R^d$, $d=2,3$, consider the nonlinear, second-order PDE
\begin{align}\label{ex:nonlin:strong}
\begin{split}
\operator{L}u(x):= -\div\matrix{A}\big(x,\nabla u(x)\big) + g\big(x,u(x),\nabla u(x)\big)&=f\quad\text{in }\Omega,\\
u&=0\quad\text{on }\Gamma=\partial\Omega.
\end{split}
\end{align}
The work~\cite{gmz12} considers strongly monotone operators $\operator{L}$ with $\matrix{A}(x,\nabla u(x))\linebreak=\alpha(x,|\nabla u(x)|^2)\,\nabla u(x)$ with $\alpha(\cdot,\cdot)\in\R$ and $g(x,u(x),\nabla u(x))=0$.
The discretization consists of first-order polynomials. Although the analysis is, in principle, not limited to the lowest-order case, this avoids further regularity assumptions on the nonlinearity of the operator $\operator{L}$ to guarantee reduction~\eqref{A:reduction} of the estimator.
In the frame of strongly monotone operators, suppose the coefficient functions to satisfy
\begin{subequations}
\begin{align}
 \normLtwo{\matrix{A}(\cdot,\nabla v)-\matrix{A}(\cdot,\nabla w)}{\Omega}&\leq \c{nllip}\normLtwo{\nabla(v-w)}{\Omega},\label{eq:smon1a}\\
 \normLtwo{g(\cdot,w,\nabla v)-g(\cdot,w,\nabla v)}{\Omega}&\leq \c{nllip}\normLtwo{\nabla(v-w)}{\Omega}\label{eq:smon1b}\\
\c{nlelliptic} \normLtwo{\nabla(v-w)}{\Omega}^2&\leq  \dual{\operator{L}v-\operator{L}w}{v-w}\label{eq:smon2}
\end{align}
\end{subequations}
for all $v,w\in H^1_0(\Omega)$ and some constants $\setc{nllip},\setc{nlelliptic}>0$. Here and throughout this paper, $\dual{\cdot}{\cdot}$ is the dual pairing of $H^1_0(\Omega)$ and $H^{-1}(\Omega)$ and all differential operators are understood in the weak sense.
Note that~\eqref{eq:smon1a}--\eqref{eq:smon1b} implies, in particular, that the operator $\operator{L}:H^1_0(\Omega)\to H^{-1}(\Omega):=H^1_0(\Omega)^*$ is Lipschitz continuous. Together with~\eqref{eq:smon2} and $f\in L^2(\Omega)$ the main theorem on strongly monotone operators~\cite[Theorem~26.A]{zeidler} guarantees that the weak form
\begin{align}\label{ex:nonlin:weak}
\dual{\operator{L} u}{v}:=\int_\Omega \matrix{A}(x,\nabla u(x))\cdot \nabla v+g(x,u(x),\nabla u(x))v\,dx=\int_\Omega fv\,dx\quad\text{for all }v\in H^1_0(\Omega)
\end{align}
admits a unique solution $u\in H^1_0(\Omega)$.
The discretization of~\eqref{ex:nonlin:weak} as well as the notation follow Section~\ref{section:ex:nonsymm}.
For a given regular triangulation $\TT$, consider 
$\XX(\TT) := \SS^1_0(\TT) := \PP^1(\TT)\cap H^1_0(\Omega)$
with $\PP^p(\TT)$ from~\eqref{eq:polynomials}. 
The discrete formulation also fits in the framework of strongly monotone operators and
\begin{align}\label{ex:nonlin:discrete}
\dual{\operator{L}U(\TT)}{V}=\int_\Omega fV\,dx\quad\text{for all }V\in\SS^1_0(\TT)
\end{align}
admits a unique solution $U(\TT)\in\SS^1_0(\TT)$.
Define the symmetric error-measure $\dist[\TT]{v}{w}=\dist{v}{w}:=\dual{\operator{L}v-\operator{L}w}{v-w}$, which is equivalent to the $H^1$-norm in the sense that
\begin{align}\label{ex:nonlin:equiv}
\normLtwo{\nabla(v-w)}{\Omega}\lesssim \dist{v}{w}\lesssim \normLtwo{\nabla(v-w)}{\Omega}\quad\text{for all }v,w\in H^1_0(\TT).
\end{align}
Therefore, $\dist{\cdot}{\cdot}$ satisfies the quasi-triangle inequality with $\c{triangle}>0$ which depends only on $\operator{L}$ and $\Omega$.
With $\XX:=H^1_0(\Omega)$, all the assumptions of Section~\ref{sec:setting} are satisfied.
and the C\'ea lemma~\eqref{eq:cea} holds with the constant $\c{cea}=2\c{nllip}/\c{nlelliptic}$.

For ease of notation, set $\operator{A}v:= -\div\matrix{A}(\cdot,\nabla v(\cdot))$ as well as $\operator{K}:=\operator{L}-\operator{A}$.
To define the error estimator and to verify our axioms of adaptivity, suppose that $\matrix{A}:\,\Omega\times\R^d\to\R^d$ is Lipschitz continuous and $\operator{L}:\, H_0^1(\Omega)\to H^{-1}(\Omega)$ as well as $\operator{A}:\, H_0^1(\Omega)\to H^{-1}(\Omega)$ to be twice Fr\'echet differentiable
\begin{align}\label{eq:frechet}
 \begin{split}
  D\operator{L},D\operator{A}:&\, H^1_0(\Omega)\to L(H^1_0(\Omega),H^{-1}(\Omega)),\\
  D^2\operator{L},D^2\operator{A}:&\, H^1_0(\Omega)\to L\big(H^1_0(\Omega),L(H^1_0(\Omega),H^{-1}(\Omega))\big).
 \end{split}
\end{align}
Assume that the second derivative is bounded locally around the solution $u$ of~\eqref{ex:nonlin:weak} i.e., there exists $\eps_{\ell oc}>0$ with \definec{nlbound}
\begin{align}\label{eq:nlbounded}
\begin{split}
 \c{nlbound}:=\sup_{\genfrac{}{}{0pt}{2}{v\in H^1_0(\Omega)}{\normLtwo{\nabla(u-v)}{\Omega}<\eps_{\ell oc}}}\Big(\norm{&D^2\operator{L}(v)}{L\big(H^1_0(\Omega),L(H^1_0(\Omega),H^{-1}(\Omega))\big)}\\
&+\norm{D^2\operator{A}(v)}{L\big(H^1_0(\Omega),L(H^1_0(\Omega),H^{-1}(\Omega))\big)}\Big)<\infty.
\end{split}
\end{align}
Assume that $D\operator{A}(v):\,H^1_0(\Omega)\to H^{-1}(\Omega)$ is symmetric for all $v\in H^1_0(\Omega)$ in the sense that $\dual{D\operator{A}(v)w_1}{w_2}=\dual{w_1}{D\operator{A}(v)w_2}$ for all $w_1,w_2\in H^1_0(\Omega)$.

The residual error estimator is similar to the linear case~\eqref{ex:nonsymm:estimator} and reads
\begin{align}\label{ex:nonlin:estimator}
\eta_T(\TT;V):=h_T^2\normLtwo{\operator{L}|_TV - f}{T}^2 + h_T\normLtwo{[\matrix{A}(\cdot,\nabla V)\cdot n]}{\partial T\cap\Omega}^2
\end{align}
for all $V\in\SS^1_0(\TT)$ and $T\in\TT$, see~\cite[Section~6.5]{nonsymm}.

Suppose
newest vertex bisection (NVB) so that the assumptions~\eqref{refinement:sons}--\eqref{refinement:overlay} as well as uniform $\gamma$-shape regularity~\eqref{refinement:shapereg} hold.

While the axioms stability~\eqref{A:stable}, reduction~\eqref{A:reduction},  and discrete reliability~\eqref{A:dlr} follow from the same arguments as for the linear case, the general quasi-orthogonality requires some additional analysis.

\begin{proposition}\label{prop:ex:nonlin:assumptions}
The conforming discretization of~\eqref{ex:nonlin:discrete} with
residual error estimator~\eqref{ex:nonlin:estimator} satisfies
stability~\eqref{A:stable}, reduction~\eqref{A:reduction}
with $\q{reduction}=2^{-1/d}$, generalized 
quasi-orthogonality~\eqref{A:qosum},  and discrete reliability~\eqref{A:dlr} 
with $\RR(\TT,\widehat\TT) = \TT\backslash\widehat\TT$.
The constants $\c{stable},\c{reduction},\c{reliable},\c{qosum}(\epsqo),\c{dlr}>0$
depend only on the polynomial degree $p\in\N$ and the shape regularity and hence
on $\T$.
\end{proposition}
\begin{proof}
Stability~\eqref{A:stable} and reduction~\eqref{A:reduction} follow similarly as for the Poisson model problem
from Section~\ref{section:ex:poisson} by use of Lipschitz continuity of $\matrix{A}:\, \Omega\times \R^d \to \R^d$. The equivalence~\eqref{ex:nonlin:equiv} allows to apply the techniques of~\cite{stevenson07,ckns} to prove discrete reliability~\eqref{A:dlr} which implies reliability~\eqref{A:reliable} according to Lemma~\ref{lemma:dlr2rel}. As in the proof of Proposition~\ref{prop:ex:nonsymm:assumptions}, a priori convergence already implies convergence $U(\TT_\ell)\to u$ as $\ell\to\infty$.

The proof of the general quasi-orthogonality~\eqref{A:qosum}, mimics the proof of Proposition~\ref{prop:ex:nonsymm:assumptions} (see~\cite[Section~6.5]{nonsymm} for details). The Taylor approximation in the neighbourhood of the solution $u$ of~\eqref{ex:nonlin:weak} and boundedness~\eqref{eq:nlbounded} lead to (cf. e.g.~\cite[Theorem~6.5]{cupr})
\begin{align}\label{eq:taylordef}
\normHme{\operator{A}v-\operator{A}w-D\operator{A}(v)(v-w)}{\Omega}\leq C_{\rm Taylor} \normLtwo{\nabla(v-w)}{\Omega}^2
\end{align}
for all $v,w\in H^1_0(\Omega)$ with $\normLtwo{\nabla(u-v)}{\Omega}+\normLtwo{\nabla(u-w)}{\Omega}\leq\eps_{\ell oc}$. The constant $C_{\rm Taylor}>0$ depends only on $\operator{L}$. Since $D\operator{A}(v)$ is symmetric and since  $U(\TT_\ell)\to u$, there exists $\ell_0\in\N$ such that $U(\TT_\ell)$ and $U(\TT_{\ell+1})$ are sufficiently close to $u$ for all $\ell\geq \ell_0$. According to the C\'ea lemma~\eqref{eq:cea} and $U(\TT_\ell)\in\SS^1_0(\TT_{\ell+1})$, this yields
\begin{align*}
\Big|\dual{\operator{A}u&-\operator{A}U(\TT_{\ell+1})}{U(\TT_{\ell+1})-U(\TT_\ell)}\Big|\\
&\leq
\Big|\dual{u-U(\TT_{\ell+1})}{D\operator{A}(U(\TT_{\ell+1}))\big(U(\TT_{\ell+1})-U(\TT_\ell)\big)}\Big|+C \normLtwo{\nabla(u-U(\TT_\ell))}{\Omega}^3\\
&\leq 
\Big|\dual{u-U(\TT_{\ell+1})}{\operator{A}U(\TT_{\ell+1}))-\operator{A}U(\TT_\ell)}\Big|+C^\prime \normLtwo{\nabla(u-U(\TT_\ell))}{\Omega}^3.
\end{align*}
Here, $C\simeq  \c{cea}C_{\rm Taylor}>0$  and $C^\prime\simeq \c{cea}^3 C_{\rm Taylor}>0$ depend only on $\operator{L}$.
For any $\delta>0$, there exists $\ell_0\in\N$ such that for all $\ell \geq \ell_0$
\begin{align}\label{ex:nonlin:help1}
\begin{split}
\Big|\dual{\operator{A}u&-\operator{A}U(\TT_{\ell+1})}{U(\TT_{\ell+1})-U(\TT_\ell)}\Big|\\
&\leq 
\Big|\dual{u-U(\TT_{\ell+1})}{\operator{A}U(\TT_{\ell+1}))-\operator{A}U(\TT_\ell)}\Big|+\delta\normLtwo{\nabla(u-U(\TT_\ell))}{\Omega}^2.
\end{split}
\end{align}

The weak convergence to zero of the sequence $e_\ell := (u-U_\ell)/\norm{\nabla(u-U_\ell)}{L^2(\Omega)}\in H^1_0(\Omega)$ is proved in~\cite[Lemma~17]{nonsymm} and allows for~\eqref{ex:nonlin:help1} also for the (nonlinear) compact operator $\operator{K}$ in the sense that, for all $\ell\geq \ell_0$
\begin{align*}
\Big|\dual{\operator{K}u&-\operator{K}U(\TT_{\ell+1})}{U(\TT_{\ell+1})-U(\TT_\ell)}\Big|\\
&\leq 
\Big|\dual{u-U(\TT_{\ell+1})}{\operator{K}U(\TT_{\ell+1}))-\operator{K}U(\TT_\ell)}\Big|+\delta\normLtwo{\nabla(u-U(\TT_\ell))}{\Omega}^2.
\end{align*}
The remaining parts of the proof follow those of Proposition~\ref{prop:ex:nonsymm:assumptions} (see~\cite[Section~6.5]{nonsymm} for details). Given some $\eps>0$, there exists $\ell_0\in\N$ such that
\begin{align*}
\dist{U(\TT_{\ell+1})}{U(\TT_\ell)}^2\leq (1-\eps)^{-1}\dist{u}{U(\TT_\ell)}^2-\dist{u}{U(\TT_{\ell+1})}^2\quad\text{for all }\ell\geq\ell_0.
\end{align*}
Combined with the equivalence~\eqref{ex:nonlin:equiv}, this allows for the
general quasi-orthogonality~\eqref{A:qosum} and so concludes the proof.
\end{proof}

\begin{consequence}
The adaptive algorithm leads to convergence with quasi-optimal rate for
the estimator $\eta(\cdot)$ in the sense of 
Theorem~\ref{thm:main}~(i)--(ii).\qed
\end{consequence}

\subsection{Conforming FEM for the $p$-Laplacian}
\label{section:ex:plaplace}
The $p$-Laplacian allows for a review of the results of~\cite{bdk} in terms of the abstract framework from Section~\ref{section:locequiv}. Since the lower error bound is not required, this paper provides some slight improvement.
The analysis allows generalizations to $N$-functions as in~\cite{bdk}.

Consider the energy minimization problem
\begin{align}\label{ex:plaplace:min}
\operator{J}(u)=\min_{v\in W^{1,p}_0(\Omega)}\operator{J}(v)\quad\text{with}\quad\operator{J}(v):=\frac{1}{p}\int_\Omega |\nabla v|^p\,dx - \int_\Omega fv\,dx
\end{align}
for $p>1$ and $W^{1,p}_0(\Omega)$ equipped with the norm $\norm{v}{W^{1,p}(\Omega)}:=\big( \norm{v}{L^p(\Omega)}^2+\norm{\nabla v}{L^p(\Omega)}^2\big)^{1/2}$. The direct method of the calculus of variations yields existence and strict convexity of $\operator{J}(\cdot)$ even uniqueness of the solution $u\in W^{1,p}_0(\Omega)$.
With the nonlinearity 
\begin{align*}
\matrix{A}:\R^d\to\R^d,\;\matrix{A}(Q)=|Q|^{p-2}Q,
\end{align*}
the Euler-Lagrange equations associated to~\eqref{ex:plaplace:min} read
\begin{align}\label{ex:plaplaceweak}
 \dual{\operator{L}u}{v}
 = \int_\Omega \matrix{A}(\nabla u) \cdot \nabla v =\int_\Omega fv\,dx
 \quad\text{for }u,v\in\XX:=W^{1,p}_0(\Omega).
\end{align}
Define $\matrix{F}(Q):=|Q|^{p/2-1}Q$ for all $Q\in\R^d$ as well as the error measure
\begin{align}\label{eq:quasimetric}
\dist[\TT]{v}{w}=\dist{v}{w}:=\normLtwo{F(|\nabla v|)-F(|\nabla w|)}{\Omega} \quad\text{for all }v,w\in W^{1,p}_0(\Omega).
\end{align}
The error measure $\dist{\cdot}{\cdot}$ is symmetric and satisfies the quasi-triangle inequality and coercivity
\begin{align*}
\dual{\operator{L}v-\operator{L}w}{v-w} \simeq \normLtwo{F(|\nabla v|)-F(|\nabla w|)}{\Omega}^2\quad\text{for all }v,w\in W^{1,p}_0(\Omega),
\end{align*}
with hidden constants which depend solely on $p> 1$. 

The discretization of~\eqref{ex:plaplaceweak} and the notation follows Section~\ref{section:ex:poisson}.
For a given regular triangulation $\TT$,
we consider the lowest-order Courant finite element space
$\XX(\TT) := \SS^1_0(\TT) := \PP^1(\TT)\cap H^1_0(\Omega)$
with $\PP^1(\TT)$ from~\eqref{eq:polynomials}. 
Arguing as in the continuous case, the minimization problem 
\begin{align}\label{ex:plaplace:discmin}
\operator{J}(U(\TT))=\min_{V\in\SS^1_0(\TT)}\operator{J}(V)
\end{align}
admits a unique discrete solution $U(\TT)\in\SS^1_0(\TT)$, which satisfies
\begin{align}\label{ex:plaplace:discrete}
 \dual{\operator{L}U(\TT)}{V} = \int_\Omega fV\,dx
 \quad\text{for all }V\in\SS^1_0(\TT).
\end{align}
All assumptions
of Section~\ref{sec:setting} are satisfied with
newest vertex bisection (NVB).
The residual error estimator $\varrho(\cdot)$ reads
\begin{align}\label{ex:plaplace:estimator}
\varrho_T(\TT;V)^2:=h_T^2\int_T \big(|\nabla V|^{p-1}+ h_T|f|\big)^{q-2}|f|^2\, dx
+ h_T\normLtwo{[\matrix{F}(\nabla V)\cdot n]}{\partial T\cap\Omega}^2
\end{align}
for all $T\in\TT$ and all $V\in\SS^1_0(\TT)$~\cite[Section~3.2]{bdk}. 
Since the error estimator $\varrho(\cdot)$ is associated with elements, $\II(\TT)=\TT$ in the notation of Section~\ref{section:locequiv}.
Since the first term of $\varrho(\cdot)$ depends nonlinearly on $V$,~~\cite[Section~3.2]{bdk} introduces an equivalent error estimator $\eta(\cdot)$ with local contributions
\begin{align*}
\eta_T(\TT,h(\TT);V)^2:=h_T^2\int_T \big(|\nabla u|^{p-1}+ h_T|f|\big)^{q-2}|f|^2\, dx
+ h_T\normLtwo{[\matrix{F}(\nabla V)\cdot n]}{\partial T\cap\Omega}^2
\end{align*}
for all $T\in\TT$ and all $V\in\SS^1_0(\TT)$. Note that $\eta(\cdot)$ can only serve as a theoretical tool as it employs the unknown solution $u$.

\begin{proposition}\label{prop:ex:plaplace:assumptions1}
The error estimators $\eta(\cdot)$ and $\varrho(\cdot)$ are globally equivalent in the sense of~\eqref{eq:glob} and they satisfy the equivalence of D\"orfler marking~\eqref{dp:doerfler} with $k=0$.
Moreover, $\eta(\cdot)$ satisfies the axioms homogeneity~\eqref{A:kernelhomo}, stability~\eqref{A:kernelstab}, reduction~\eqref{A:reduction} with $\q{reduction}=2^{-1/d}$, general quasi-orthogonality~\eqref{A:qosum},  and weak discrete reliability~\eqref{B:dlr} with $\RR(\eps;\TT,\widehat\TT) = \TT\backslash\widehat\TT$.
The constants $\c{stab},\c{reduction},\c{qosum}(\epsqo),\c{reliable},\c{dlr}(\eps)>0$ 
depend only on $\T$ as well as on $p>1$.
\end{proposition}
\begin{proof}
The global equivalence~\eqref{eq:glob} is proved in~\cite[Corollary~4.3]{bdk}. The equivalence of D\"orfler marking~\eqref{dp:doerfler} is part of the proof of~\cite[Lemma~4.6]{bdk}.
The homogeneity~\eqref{A:kernelhomo} follows as in Section~\ref{section:examples3} with $r_+=1/2$ and $r_-=1$. Since the first term of $\eta(\cdot)$ does not depend on the argument $V$, standard inverse estimates as for the linear case prove stability~\eqref{A:kernelstab} as in~\cite[Proposition~4.4]{bdk} and Proposition~\ref{prop:ex:poisson:assumptions}. Reduction~\eqref{A:reduction} follows with the arguments from Proposition~\ref{prop:ex:poisson:assumptions} as in~\cite[Lemma~4.6]{bdk}. 

The discrete reliability~\eqref{A:dlr} for $\varrho(\cdot)$ with $\RR(\TT,\widehat\TT) = \TT\backslash\widehat\TT$
follows from~\cite[Lemma~3.7]{bdk}. Together with the equivalence from~\cite[Proposition~4.2]{bdk}, there holds for all $\delta>0$
\begin{align*}
\dist{U(\widehat\TT)}{U(\TT)}^2&\leq\c{dlr} \sum_{T\in\RR(\TT,\widehat\TT)}\varrho_T(\TT;U(\TT))^2\\
&\leq \c{dlr}C_\delta \sum_{T\in\RR(\TT,\widehat\TT)}\eta_T(\TT,h(\TT);U(\TT))^2 + \c{dlr}\delta \dist{u}{U(\TT)}^2.
\end{align*}
The constant $C_\delta>0$ is defined in~\cite[Proposition~4.2]{bdk}.
This proves weak discrete reliability~\eqref{B:dlr} with $\RR(\eps;\widehat\TT,\TT):=\TT\setminus\widehat\TT$ and $\c{dlr}(\eps):=\c{dlr}C_\delta$ and $\delta=\eps/\c{dlr}$ and particularly implies reliability~\eqref{A:reliable} as proved in Lemma~\ref{lemma:wdlr2rel}.

The general quasi-orthogonality~\eqref{A:qosum} follows from the fact that the equivalence of the error measure to the energy of the problem with $\epsqo=0$ and $\c{qosum}>0$ independent of $\epsqo$. As stated in~\cite[Lemma~3.2]{bdk}, each arbitrary refinement $\widehat\TT\in\T$ of $\TT\in\T$ satisfies
\begin{align*}
\operator{J}(U(\widehat\TT))-\operator{J}(u)&\simeq \dist{u}{U(\widehat\TT)}^2,\\
\operator{J}(U(\TT))-\operator{J}(U(\widehat\TT))&\simeq \dist{U(\widehat\TT)}{U(\TT)}^2
\end{align*}
with hidden constants, which depend only on $p>1$. This immediately implies for all $\ell\leq N\in\N$ that
\begin{align*}
\sum_{k=\ell}^N \dist{U(\TT_{k+1})}{U(\TT_k)}^2&\lesssim \sum_{k=\ell}^N \operator{J}(U(\TT_\ell))-\operator{J}(U(\TT_{\ell+1}))\\
&=\operator{J}(U(\TT_\ell))-\operator{J}(U(\TT_{N+1}))\\
&\leq \operator{J}(U(\TT_\ell))-\operator{J}(u)\lesssim \dist{u}{U(\TT_\ell)}^2.
\end{align*}
Together with reliability~\eqref{A:reliable}, this implies~\eqref{A:qosum} with $\epsqo=0$, and $0<\c{qosum}(0)<\infty$ depend only on $p>1$ and $\c{reliable}$.
\end{proof}

\begin{consequence}
The adaptive algorithm leads to convergence with quasi-optimal rate for
the estimator $\varrho(\cdot)$ in the sense of 
Theorem~\ref{thm:main3}~(i)--(ii).\qed
\end{consequence}
Numerical examples for 2D that underline the above result can be found in~\cite{bdk}.
%
\subsection{Conforming FEM for some elliptic eigenvalue problem}\label{section:eigenvalue}
This subsection is devoted to the optimal adaptive computation of an
eigenpair $(\lambda,u)\in \XX:=\R\times V$ for 
$V:=H^1_0(\Omega)$ with energy norm $| \cdot |_{H^1(\Omega)}= a(\cdot,\cdot)^{1/2}$ 
for the energy scalar product $a(\cdot,\cdot)$ (denoted $b(\cdot,\cdot)$ in (5.2))
\begin{align*}
 a(v,w):=\int_\Omega \nabla v\cdot\nabla w\,dx \quad\text{for all }v,w\in H^1_0(\Omega)
\end{align*}
and the $L^2$-scalar product $b(\cdot,\cdot)$ 
\begin{align*}
 b(v,w):=\int_\Omega vw\,dx\quad\text{for all }v,w\in H^1_0(\Omega)
 \end{align*}
of the model eigenvalue problem 
\[
-\Delta u = \lambda u\quad\textrm{in }\Omega
\qquad\textrm{and}\qquad u= 0\phantom{u}\quad\textrm{on }\partial\Omega.
\]
The weak form of the eigenvalue problem reads
\begin{align}\label{eig:weak}
a(u,v)=\lambda b(u,v)\quad\text{for all }v\in V.
\end{align}
On the continuous level, there exists a countable number of such eigenpairs with positive eigenvalues  ordered increasingly which essentially depend on the polyhedral bounded Lipschitz domain $\Omega\subset\R^d$.

Throughout this subsection, let $(\lambda,u)$ denote one fixed eigenpair with the
$\nu$-th  {\em simple eigenvalue}   $\lambda$ 
and corresponding eigenvector 
normalized via $\norm{u}{L^2(\Omega)}=1$ \revision{(the first eigenvalue for $\nu=1$ is always simple~\cite{evansPDE}).}
For simplicity, this subsection presents an analysis for the conforming FEM of order $p$
with $V(\mathcal{T}):= \SS^p_0(\TT)$,
which embeds the optimality results of~\cite{dxz,cg12} in the general setting of this paper. 

Let  the number $\nu$ of the simple eigenvalue $\lambda$ be fixed throughout this section and kept constant also \revision{on any} discrete level (without extra notation for this) and suppose  
$\norm{h(\TT_0)}{L^\infty(\Omega)}$ is so small that
the discrete eigenvalue problem has at least $\nu$ degrees of freedom and pick the
discrete eigenpair  
\[
U(\mathcal{T}):=(\lambda(\mathcal{T}), u(\mathcal{T}))\in  \XX(\mathcal{T}):=\R\times V(\mathcal{T})
\] 
(of that fixed number $\nu$) with $\norm{ u(\mathcal{T})}{L^2(\Omega)}=1$ and 
\begin{align}\label{eig:discrete}
a(u(\mathcal{T}),V)=\lambda(\mathcal{T})\,  b(u(\mathcal{T}),V)\quad\text{for all }V\in V(\mathcal{T})=\SS^p_0(\TT).
\end{align}
Notice that the discrete eigenvalue 
$(\lambda(\mathcal{T})$ of number $\nu$) is simple  for sufficiently small 
$\norm{h(\TT_0)}{L^\infty(\Omega)}$
%
and 
that the adaptive algorithm is supposed to solve the algebraic eigenvalue problems 
exactly with an (arbitrary) choice of the sign of  $u(\mathcal{T})$ (which 
does {\em not} enter the adaptive algorithm below but is assumed in the error measures to be somehow selected aligned with $u$). In fact, given 
any $\mathcal{T}\in\TT$ and $(\mu,V)\in \XX(\mathcal{T}):=\R\times V(\mathcal{T})$, 
the residual-based a~posteriori error estimator consists of the local contributions  
\begin{align*}
\eta_T(\TT; (\mu,V))^2 &:= 
h_T^2\,\norm{\mu V+  \Delta_\TT V}{L^2(T)}^2
 + h_T\,\norm{[\partial_nV]}{L^2(\partial T\cap\Omega)}^2.
\end{align*}
The axioms (A1)-(A4) follow 
for the error measure  $\dist[\TT]\cdot\cdot$ defined (independently of $\mathcal{T}$) by
\[
\dist[\TT]{(\lambda,u)}{(\mu,v)}:= \Bigl( \norm{\lambda u-\mu v}{L^2(\Omega)}^2+
 | u- v|_{H^1(\Omega)}^2
\Bigr)^{1/2}\quad\text{for all }(\lambda,u),(\mu,v)\in \XX .
\]
The arguments which imply optimal convergence of the adaptive algorithm are essentially contained in~\cite{dxz,cg12} while the introduction of the error measure to allow for the abstract framework of this paper is a novel ingredient.

\begin{proposition}\label{prop:ex:eigenvalues} Given positive integers $\nu$ and 
$p$ such that the  $\nu$-th  eigenvalue $\lambda$  is simple and provided the initial mesh-size 
$\norm{h(\TT_0)}{L^\infty(\Omega)}$ of the  initial triangulation $\mathcal{T}_0$ is sufficiently small, 
the $p$-th order conforming finite element discretization~\eqref{eig:discrete} of the eigenvalue problem~\eqref{eig:weak} with
the above  residual-based error estimator  \revision{satisfy}
stability~\eqref{A:stable}, reduction~\eqref{A:reduction}
with $\q{reduction}=2^{-1/d}$, general quasi-orthogonality~\eqref{A:qosum}, discrete reliability~\eqref{A:dlr} with  $\RR(\TT,\widehat\TT) = \TT\backslash\widehat\TT$,
and efficiency~\eqref{A:efficient} with $\osc(\TT;U(\TT))=0$.
The constants $\c{stable}$,
$\c{reduction}$, $\c{qosum}, \c{dlr},\c{eff}>0$
depend only on  $\mathbb{T}$ and the polynomial degree $p\in\N$.
\end{proposition}

The proof of the proposition requires two straight-forward algebraic identities.

\begin{lemma}
Suppose that $(\mu,v) :=(\lambda(\TT),u(\TT)) \in \XX(\TT)$  and  
$ (\widehat \mu,\widehat v):=(\lambda(\widehat\TT),u(\widehat\TT))\in \XX(\widehat \TT)$ denote  the discrete eigenpairs 
with respect to  some refinement $\widehat{\mathcal{T}}\in\mathbb{T} $ 
of  $\mathcal{T}\in\mathbb{T}$
and let  $(\lambda,u)$ denote the exact eigenpair. Then, it holds
\begin{eqnarray}
| \widehat v- v|_{H^1(\Omega)}^2  &=&\mu-\widehat\mu +
\widehat\mu \norm{\widehat v-v  }{L^2(\Omega)}^2\ge \mu-\widehat\mu\ge 0,
\label{eq:eigenvalue0}
\\  \label{eq:eigenvalue2}
\norm{ \widehat\mu \widehat v -\mu  v }{L^2(\Omega)}^2&=& (\widehat\mu-\mu)^2+\mu\widehat\mu\,
\norm{ \widehat v - v}{L^2(\Omega)}^2\le (\mu-\widehat\mu)  | \widehat v- v|_{H^1(\Omega)}^2  .
\end{eqnarray}
\end{lemma}

\begin{proof}
The  Rayleigh-Ritz principle for the conforming discretizations leads to 
\[
\lambda = | u |_{H^1(\Omega)}^2  \le \widehat\mu\le \mu 
:=\lambda(\mathcal{T})=| u(\mathcal{T})|^2_{H^1(\Omega)}\le \lambda(\TT_0).
\]
In particular,  the differences of discrete eigenvalues in \eqref{eq:eigenvalue0}-\eqref{eq:eigenvalue2} are all non-negative.
Direct calculations with $b(\widehat v + v, \widehat v - v)=0$ from 
$\norm{\widehat v}{L^2(\Omega)}=1=\norm{v}{L^2(\Omega)}$  prove
\begin{equation}
\label{eq:eigenvalue2a}
 b( \widehat v, \widehat v- v)=\frac 12  \norm{ \widehat v- v}{L^2(\Omega)}^2=b(v-\widehat v,v).
\end{equation}
The eigenvalue relations $(| \widehat v |^2_{H^1(\Omega)}=\widehat\mu$ etc.)
show
\[
| \widehat v- v|^2_{H^1(\Omega)}=\mu-\widehat \mu  + 2 a(\widehat v,\widehat v-v )
=\mu-\widehat \mu  + 2 \widehat \mu b(\widehat v,\widehat v-v).
\]
Together with the first equation in \eqref{eq:eigenvalue2a}, 
this implies \eqref{eq:eigenvalue0},
which has been used before, e.g., in~\cite{cg12}. The  
left-hand side of \eqref{eq:eigenvalue2}  equals
\[
\norm{\widehat\mu(\widehat v-v)
+(\widehat\mu-\mu)v}{L^2(\Omega)}^2    
=(\widehat\mu-\mu)^2 
+2 \widehat\mu (\widehat\mu-\mu)\, 
b(\widehat v-v,v)
+ \widehat\mu^2 \norm{ \widehat v-v}{L^2(\Omega)}^2  .
\]
This and the second equation in \eqref{eq:eigenvalue2a}, prove
the equality in \eqref{eq:eigenvalue2}.

The substitution of $\mu-\widehat\mu$ from 
\eqref{eq:eigenvalue0} in one factor of $(\mu-\widehat\mu)^2$ in \eqref{eq:eigenvalue2}
proves 
\[
\norm{\widehat\mu \widehat v -\mu  v }{L^2(\Omega)}^2+\widehat\mu^2   \norm{\widehat v -v  }{L^2(\Omega)}^2 =
(\mu-\widehat\mu)| \widehat v- v|_{H^1(\Omega)}^2.
\]
This concludes the proof of the inequality in  \eqref{eq:eigenvalue2}.
\end{proof}

{\em Proof of Proposition~\ref{prop:ex:eigenvalues}.}
The stability~\eqref{A:stable}  follows as in Proposition~\ref{prop:ex:poisson:assumptions}
for $(\widehat\mu,\widehat v)\in \XX(\widehat\TT)$ and $ (\mu,v) \in \XX(\TT)$
up to sums of squares of some additional terms  
\[
h_T\,\norm{\widehat\mu\widehat v - \mu v}{L^2(T)}\quad\text{for }T\in\SS\subset\TT. 
\]
Those extra terms motivate the error measure $\dist[\TT]\cdot\cdot$ and,
because of $h_T\le   \norm{h(\TT_0)}{L^\infty(\Omega)}$,  lead to the proof of 
\eqref{A:stable} without  additional difficulty. The proof of the  
reduction~\eqref{A:reduction} with $\q{reduction}=2^{-1/d}$ 
follows the same lines and hence is not outlined here.

The remaining parts of the proof require a brief  discussion on a sufficiently small mesh size $\norm{h(\TT_0)}{L^\infty(\Omega)}$  of the initial  triangulation $\mathcal{T}_0$. Textbook analysis~\cite{sfbook} proves  
the uniqueness of the  algebraic eigenpair  $(\lambda(\mathcal{T}),U(\mathcal{T}))$
for sufficiently small  $\norm{ h(\TT_0)}{L^\infty(\Omega)}$ and that 
the direction of $\pm U(\mathcal{T})$ of the discrete \revision{eigenfunction} $U(\mathcal{T})$
can and will be chosen in alignment to $u$ (via $b(u, u(\TT))>0$  in this proof)  
such that
\begin{equation}
\label{eq:eigenvalue3}
\norm{u- u(\mathcal{T})}{L^2(\Omega)}^2
\le o(\norm{h(\TT)}{L^\infty(\Omega)}) \, | u- u(\mathcal{T})|_{H^1(\Omega)}^2
\end{equation}
holds for some Landau symbol with $\lim_{\delta\to 0} o(\delta)=0$ 
uniformly for all triangulations $\mathcal{T}$ $\in\mathbb{T}$.  
A less well-known discrete analog of  \eqref{eq:eigenvalue3} for 
all refinements $\widehat{\mathcal{T}}\in\mathbb{T} $ of  $\mathcal{T}\in\mathbb{T}$ 
with mesh-size $h(\TT)\in P_0(\TT)$ reads
\begin{equation}
\label{eq:eigenvalue4}
\norm{ u(\widehat{\mathcal{T}})- u(\mathcal{T})}{L^2(\Omega)}^2
\le o(\norm{h(\TT) }{L^\infty(\Omega)}) \, | u(\widehat{\mathcal{T}})- u(\mathcal{T})|_{H^1(\Omega)}^2.
\end{equation}
The proof of \eqref{eq:eigenvalue4} follows from elliptic regularity and 
the combination of~\cite[Lemma~3.3--3.4]{cg12} for a simple eigenvalue 
$\lambda$. 
Without loss of generality, we may and will suppose that the function $o(\delta)$ is monotone increasing in $\delta$ so that  $o(\norm{ h(\TT)}{L^\infty(\Omega)} )\le o(\norm{h(\TT_0)}{L^\infty(\Omega)})\le  1/(2\lambda(\TT_0))$. 
The reliability of the error estimators requires some sufficiently small mesh-size as well. 
For some sufficiently small mesh-size $\norm{ h(\TT_0)}{L^\infty(\Omega)}$ of the  initial 
triangulation $\mathcal{T}_0$, \cite[Lemma~3.5]{cg12} reads, in the above notation, as 
\[
 | u(\widehat{\mathcal{T}})- u(\mathcal{T})|_{H^1(\Omega)}^2
 \lesssim  \norm{\text{Res}(\TT; U(\TT))}{V(\widehat\TT)^*}^2
\]
in terms of the discrete dual norm $\norm{\cdot}{V(\widehat\TT)^*}$.
It is a standard argument in the linear theory of Section~4.1 to estimate the 
discrete dual norm of the residual 
\[
\text{Res}(\TT; U(\TT)):= b(\lambda(\TT)\, u(\TT) ,\cdot)-a(u(\TT) ,\cdot)\in V(\widehat\TT)^*
\] 
(just replace $f$ on the right hand side by  
$\lambda(\TT)\, u(\TT)$ in the Poisson model problem) by
\[
 \norm{\text{Res}(\TT; U(\TT))}{V(\widehat\TT)^*}^2\lesssim 
 \sum_{T\in\RR(\TT,\widehat\TT)}\eta_T(\TT;U(\TT))^2
\]
with $\RR(\TT,\widehat\TT):=\widehat\TT\setminus\TT$.
The combination of the aforementioned estimates verifies 
\[
 | u(\widehat{\mathcal{T}})- u(\mathcal{T}) |_{H^1(\Omega)}^2
\lesssim 
 \sum_{T\in\RR(\TT,\widehat\TT)}\eta_T(\TT;U(\TT))^2.
\]
The inequality in \eqref{eq:eigenvalue2} and $\widehat\mu\le\mu\le\lambda(\TT_0)$ prove 
for $(\mu,v):= U(\TT)$ and $(\widehat \mu,\widehat v):= U(\widehat \TT)$ that
\begin{align}
\begin{split}
\dist[\TT]{U(\widehat\TT)}{U(\TT)}^2&= 
 \norm{\widehat\mu \widehat v-\mu v }{L^2(\Omega)}^2+
 | \widehat v- v|_{H^1(\Omega)}^2 
\\ \label{eq:eigenvalue5}
&\le  (1+\lambda(\TT_0))| u(\widehat{\mathcal{T}})- u(\mathcal{T}) |_{H^1(\Omega)}^2.
\end{split}
\end{align}
The combination of the previous two displayed estimates proves 
the  discrete reliability~\eqref{A:dlr} with a constant $\c{reliable}$ which depends 
only on $\mathbb{T}$. 

The convergence of the conforming finite element discretization is understood from Textbook analysis~\cite{sfbook} or \eqref{eq:eigenvalue3} 
and so  Lemma~\ref{lemma:dlr2rel} reveals reliability  \eqref{A:reliable}.
Efficiency is proved in~\cite[Lemma~4.2]{cg13} for general $p\ge1$ and relies on sufficiently small $\norm{h_0}{L^\infty(\Omega)}\ll1$.

The proof of  the general quasi-orthogonality~\eqref{A:qosum} with $\epsqo=0$ starts 
with a combination of  \eqref{eq:eigenvalue0} and \eqref{eq:eigenvalue4}
with $o(\norm{h(\TT)}{L^\infty(\Omega)} )\le o(\norm{h(\TT_0)}{L^\infty(\Omega)})\le  1/(2\lambda(\TT_0))$. This proves
\[
\left| | u(\widehat{\mathcal{T}})- u(\mathcal{T})|_{H^1(\Omega)}^2
-\lambda(\TT) + \lambda(\widehat\TT)\right|\le 
\frac 12 | u(\widehat{\mathcal{T}})- u(\mathcal{T})|_{H^1(\Omega)}^2.
\]
The first conclusion is the equivalence 
\[
 | u(\widehat{\mathcal{T}})- u(\mathcal{T})|_{H^1(\Omega)}^2\simeq
 \lambda(\TT) - \lambda(\widehat\TT).
\]
With \eqref{eq:eigenvalue5}, the second equivalence is,  for 
all refinements $\widehat{\mathcal{T}}\in\mathbb{T} ~\eqref{A:efficient}$ of  $\mathcal{T}\in\mathbb{T}$,
that
\begin{equation}
\label{eq:eigenvalue6}
\dist[\TT]{U(\widehat\TT)}{U(\TT)}^2\simeq
 \lambda(\TT) - \lambda(\widehat\TT).
\end{equation}
Exploit the equivalence \eqref{eq:eigenvalue6}  in the proof of the general 
quasi-orthogonality 
with $(\lambda_k,u_k):=U_k:= U(\TT_k)$  to verify,
for  any  $\ell, N\in \N_0$ with $N\ge \ell$, that
\[
\sum_{k=\ell}^N  \dist[\TT]{U_{k+1}}{U_k}^2\lesssim 
\lambda_\ell-\lambda_{N+1}\lesssim \dist[\TT]{U_{N+1}}{U_\ell}^2.
\]
The combination with  the  discrete reliability~\eqref{A:dlr}
concludes the proof of the 
general quasi-orthogonality~\eqref{A:qosum} with $\epsqo=0$ such that $\c{qosum}$ only depends on $\mathbb{T}$. \hfill $\Box$

\begin{consequence}
Given sufficiently small $\norm{h(\TT_0)}{L^\infty(\Omega)}$, the adaptive algorithm leads to convergence with quasi-optimal rate 
in the sense of Theorem~\ref{thm:main} and Theorem~\ref{thm:mainerr}.\qed
\end{consequence}

Numerical examples can be found in~\cite{cg13,cg12} with the generalization to inexact
solve and even optimal computational complexity under realistic assumptions on the
performance of the underlying algebraic eigenvalue solver~\cite{cg13}.

This section focussed on a simple eigenvalue $\lambda$ while 
clusters of eigenvalues require a simultaneous adaptive mesh-refinement with respect to all
affected eigenvectors \cite{Gallistl2013} beyond the scope of this paper.
An optimal nonconforming adaptive FEM has recently been analyzed in~\cite{ccgs13} \revision{with guaranteed lower eigenvalue bounds.}
%
\section{\revision{Non-Trivial} Boundary Conditions}
\label{section:boundary}

\def\meas{\text{meas}}
\def\rhs{RHS}
\newcommand{\eqdef}{\mathrel{\widehat{=}}}

\noindent
The literature on
adaptive finite elements focusses on 
homogeneous Dirichlet conditions with the only exception of~\cite{mns03,dirichlet2d,dirichlet3d}. This section extends the previous results to non-homogeneous boundary conditions of mixed Dirichlet-Neumann-Robin
type where inhomogeneous Dirichlet conditions enforce some additional 
discretization error. The present section improves~\cite{dirichlet3d} and shows that standard D\"orfler marking~\eqref{eq:doerfler} leads to convergence with
quasi-optimal rates if the Scott-Zhang projection~\cite{scottzhang} is used 
for the discretization of the Dirichlet data~\cite{dirichlet3d,sv}. The 
heart of the analysis is the application of the modified mesh-width function 
$h(\TT,k)$ from Proposition~\ref{prop:htilde}.

\subsection{Model problem}
The Laplace model problem in $\R^d$ for $d\geq 2$ with mixed Dirichlet-Neumann-Robin boundary conditions splits the boundary $\Gamma$
of the Lipschitz domain $\Omega\subset \R^d$ into three (relatively) open and pairwise disjoint boundary parts $\partial\Omega=\overline{\Gamma_D\cup \Gamma_N \cup \Gamma_R}$. Given data $f\in L^2(\Omega)$, $g_D\in H^{1}(\Gamma_D)$, 
$\phi_N\in L^2(\Gamma_N)$, $\phi_R \in L^2(\Gamma_R)$, and $\alpha \in L^\infty(\Gamma_R)$ with $\alpha\geq\alpha_0>0$ almost everywhere on $\Gamma_R$, the problem seeks $u\in H^1(\Omega)$ with
\begin{subequations}
\label{eq:mixedbc}
\begin{align}
  -\Delta u &= f\qquad\text{in }\Omega,\label{eq:mixedbca}\\
  u&=g_D\quad\; \text{on }\Gamma_D,\label{eq:mixedbcb}\\
  \partial_n u &=\phi_N\quad\;\text{on }\Gamma_N,\label{eq:mixedbcc}\\
  \phi_R - \alpha u&=\partial_n u\quad\text{on }\Gamma_R.\label{eq:mixedbcd}
\end{align}
\end{subequations}
The presentation focusses on the case that $|\Gamma_D|,|\Gamma_R|>0$, with possibly $\Gamma_N=\emptyset$. The cases $\Gamma_D=\emptyset$ and $|\Gamma_R|>0$, $|\Gamma_D|>0$ and $\Gamma_R=\emptyset$, as well as the pure Neumann problem $\Gamma_N=\partial\Omega$ are also covered by the abstract analysis of Sections~\ref{sec:setting}--\ref{section:optimality}.

\subsection{Weak formulation}
The weak formulation of~\eqref{eq:mixedbc} seeks $u\in \XX := H^1(\Omega)$  such that 
\begin{subequations}\label{eq:mixedbcweak}
\begin{align}
\label{eq:mixedbcweak:trace}
u=g_D \text{ on }\Gamma_D \text{ in the sense of traces }
\end{align}
and all $v\in H^1_{D}(\Omega):= \set{v\in H^1(\Omega)}{v=0\text{ on }\Gamma_D}$ satisfy
\begin{align}
\label{eq:mixedbcweak:form}
b(u,v):=\int_\Omega\nabla u\cdot \nabla v\,dx + \int_{\Gamma_R}\alpha u v\,ds
 = \rhs(v)
\end{align}
with 
\begin{align}
\label{eq:mixedbcweak:rhs}
\rhs(v):=\int_\Omega f v\,dx + \int_{\Gamma_N}\phi_N v\,ds + \int_{\Gamma_R}\phi_R v\,ds.
\end{align}
\end{subequations}
Since $|\Gamma_R|>0$ and $\alpha \ge \alpha_0>0$, the norm
$\norm{\cdot}{}:=b(\cdot,\cdot)^{1/2}$ is equivalent to the $H^1(\Omega)$-norm. 

Let  $u_D\in H^1(\Omega)$ with $u_D|_\Gamma=g_D$ be an arbitrary lifting of
the given Dirichlet data and set $u_0 := u - u_D \in H^1_D(\Omega)$. Then,~\eqref{eq:mixedbcweak} is equivalent to seek $u_0\in H^1_{D}(\Omega)$ with
\begin{align}\label{eq:mixedbcaux}
 b(u_0,v)=\rhs(v)-b( u_D,v)
\quad\text{for all }v\in H^1_D(\Omega).
\end{align}
According to the Lax-Milgram theorem, the auxiliary problem~\eqref{eq:mixedbcaux} 
admits a unique solution $u_0 \in H^1(\Omega)$ and thus $u:= u_0 + u_D$ is the unique solution of ~\eqref{eq:mixedbcweak}.

\subsection{FEM discretization and approximation of Dirichlet data}
Assume the initial
triangulation $\TT_0$, and hence all triangulations $\TT\in \T $ of $\Omega$, to 
resolve the boundary conditions in the sense that for all facets $E\subset \partial\Omega$ on 
the boundary, there holds $E\subseteq \overline\gamma$ for some $\gamma\in\{\Gamma_D,\Gamma_N,\Gamma_R\}$ and suppose newest vertex bisection.
Let $\XX(\TT) = \SS^p(\TT) := \PP^p(\TT)\cap H^1(\Omega)$ 
and $\SS^p_D(\TT) := \PP^p(\TT)\cap H^1_D(\Omega)$ with fixed polynomial order $p\ge1$ and $\PP^p(\TT)$ from~\eqref{eq:polynomials} above.
To discretize the given Dirichlet data $g_D$, for any given mesh $\TT\in\T$, choose an approximation 
\begin{align*}
G_D(\TT)\in\SS^p(\TT|_{\Gamma_D}):=\set{V|_{\Gamma_D}}{V\in\SS^p(\TT)}
\end{align*}
of the Dirichlet data $g_D$. Here and throughout this section, let $\TT|_{\Gamma_D}:=\set{T|_{\Gamma_D}}{T\in\TT}$ denote the restriction of the volume mesh to the Dirichlet boundary $\Gamma_D$, and $\SS^p(\TT|_{\Gamma_D})$ is the discrete trace space. A convenient way to choose this approximation independently of the spatial dimension is the Scott-Zhang projection $J(\TT):H^1(\Omega) \to \SS^p(\TT)$ from~\cite{scottzhang}. The formal definition also allows for an operator $J(\TT|_{\Gamma_D}): L^2(\Gamma_D) \to \SS^{p}(\TT|_{\Gamma_D})$ on the boundary. The reader is referred to~\cite{dirichlet3d} for details and further discussions.

The discrete counterpart of~\eqref{eq:mixedbcweak} seeks $U(\TT)\in\SS^p(\TT)$ such that%
\begin{subequations}\label{eq:mixedbcdiscrete}%
\begin{align}
 U(\TT)|_{\Gamma_D} &= G_D(\TT),\\
b(U(\TT),V)&=f(V)
 \quad\text{for all }V\in\SS^p_D(\TT).
\end{align}
\end{subequations}
As in the continuous case,~\eqref{eq:mixedbcdiscrete} admits
a unique solution and satisfies all assumptions
of Section~\ref{sec:setting} with $\dist[\TT]vw = \norm{v-w}{}$
and $\c{triangle}=1$.

\subsection{Quasi-optimal convergence}
The derivation of the residual-based error estimator $\eta(\TT,\cdot)$ follows similarly
to the homogeneous case and differs only by adding an oscillation term to control the approximation of the Dirichlet data~\cite{dirichlet3d, bcd, dirichlet2d, sv}. With the local mesh-width function $h(\TT)$ from Section~\ref{section:locequiv},
the local contributions read
\begin{align*}
 \eta_T(\TT;V)&:=\normLtwo{h(\TT)(f+\Delta_\TT V)}{T}^2 +\normLtwo{h(\TT)^{1/2}[\partial_n V]}{\partial T\cap\Omega}^2\\
 &\qquad+ \normLtwo{h(\TT)^{1/2}(\phi_R - \alpha V -\partial_n V)}{\partial T\cap\Gamma_R}^2
 +\normLtwo{h(\TT)^{1/2}(\phi_N-\partial_n V)}{\partial T\cap\Gamma_N}^2\\
 &\qquad+ {\rm dir}_T(\TT)^2,
\end{align*}
where
\begin{align*}
{\rm dir}_T(\TT):=\normLtwo{h(\TT)^{1/2}(1-\Pi_{p-1}(\TT|_{\Gamma_D}))\nabla_\Gamma g_D}{\partial T\cap\Gamma_D}
\end{align*}
and $\Pi_{p-1}(\TT|_{\Gamma_D}):\,L^2(\Gamma_D)\to \PP^{p-1}(\TT|_{\Gamma_D}) := \set{V|_{\Gamma_D}}{V\in\PP^{p-1}(\TT)}$ is the (piecewise) $L^2$-orthogonal projection, and 
$\nabla_\Gamma(\cdot)$ denotes the surface gradient. 

For each facet $E\subset\partial\Omega$, there exists a unique element 
$T\in\TT$ such that $E\subset\partial T$. In particular, $h(\TT)$ also induces a local 
mesh-size function on $\gamma\in\{\Gamma_D,\Gamma_N,\Gamma_R\}$.

The following proposition shows that inhomogeneous (and mixed) boundary data 
fit in the framework of our abstract analysis. Emphasis is on the novel
quasi-orthogonality~\eqref{A:qosum} which improves the analysis of~\cite{dirichlet3d} on separate D\"orfler marking.
The novel mesh-size function $h(\TT,k)$ establishes optimal convergence
of Algorithm~\ref{algorithm} with the standard D\"orfler marking~\eqref{eq:doerfler}.

\begin{proposition}\label{prop:inhom}
The estimator $\eta(\cdot)$ satisfies stability~\eqref{A:stable}, reduction~\eqref{A:reduction}, quasi-orthogonality~\eqref{A:qosum}, discrete reliability~\eqref{A:dlr}, and efficiency~\eqref{A:efficient}. The discrete reliability~\eqref{A:dlr} holds with $\RR(\TT,\widehat\TT):=\omega^5(\TT;\widehat\TT\setminus\TT)$ (as defined in Section~\ref{section:kpatch}), and the oscillation terms in the efficiency axiom~\eqref{A:efficient} reads
\begin{align}\label{eq:mixedbceff}
\begin{split}
\eff{\TT}{U(\TT)}^2&:={\rm dir}(\TT)^2+ \min_{F\in\PP^{p-1}(\TT)}\normLtwo{h(\TT)(f-F)}{\Omega}^2\\
&\qquad+ \min_{\Phi\in\PP^{p-1}(\TT|_{\Gamma_N})}\normLtwo{h(\TT)^{1/2}(\phi_N-\Phi)}{\Gamma_N}^2\\
&\qquad+ \min_{\Phi\in\PP^{p-1}(\TT|_{\Gamma_R})}\normLtwo{h(\TT)^{1/2}(\phi_R-\Phi)}{\Gamma_R}^2.
\end{split}
\end{align}
\end{proposition}

\begin{proof}
Efficiency~\eqref{A:efficient} can be found in\cite{bcd, sv} or~\cite[Proposition 3]{dirichlet3d}. The proof of~\eqref{eq:mixedbceff} follows similarly to that of Proposition~\ref{prop:ex:poisson:assumptions} and exploits that 
$\Delta_\TT U(\TT)|_T$ is a polynomial of degree $\leq p-2$.

The proofs of stability~\eqref{A:stable} and reduction~\eqref{A:reduction} are verbatim to the case with $\Gamma_R=\emptyset$ from~\cite[Proposition~11]{dirichlet3d}. 
The proof of discrete reliability~\eqref{A:dlr} is more involved, however, the difficulties arise only due to the approximation of the Dirichlet data and the non-local $H^{1/2}(\Gamma_D)$-norm.
The proof in~\cite[Proposition 21]{dirichlet3d} for $\Gamma_R = \emptyset$ generalizes to the present case. 

It remains to verify the quasi-othogonality~\eqref{A:qo} which implies~\eqref{A:qosum} by virtue of Lemma~\ref{lem:qo}. 
Recall the modified mesh-size function $h(\TT,5)$ 
and the patch $\omega^5(\TT;\TT\backslash\widehat\TT)\subseteq\TT$ from Section~\ref{section:locequiv} for $k=5$. 
It is proved in~\cite[Lemma~20]{dirichlet3d} for $\Gamma_R=\emptyset$ that there holds for all $\epsqo>0$
 \begin{align}\label{eq:dirichlet3d}
 \begin{split}
  \enorm{U(\widehat\TT)-U(\TT)}^2&\leq \enorm{u-U(\TT)}^2 -(1-\epsqo)\enorm{u-U(\widehat\TT)}^2\\
  &\qquad+ C_{\rm pyth}\epsqo^{-1}\normHeh{(J(\widehat\TT|_{\Gamma_D})-J(\TT|_{\Gamma_D}))g_D}{\Gamma_D}^2,
  \end{split}
 \end{align}
 where $C_{\rm pyth}>0$ depends only on $\T$ and $\Gamma_D$. Although~\cite{dirichlet3d} considers $\Gamma_R=\emptyset$ and hence $\enorm{\cdot}=\normLtwo{\nabla(\cdot)}{\Omega}$, the proof transfers to the present case.

The focus in the derivation of~\eqref{A:qo} is the last term on the right-hand side $\mu(\TT)^2-\mu(\widehat \TT)^2$.
First, let $\omega_D^5(\TT;\TT\backslash\widehat\TT)\subseteq \TT|_{\Gamma_D}$
denote the set of all facets $E$ of $\TT$ with $E \subseteq \overline\Gamma_D\cap\bigcup\omega^5(\TT;\TT\backslash\widehat\TT)$.
It is part of the proof of~\cite[Proposition 21]{dirichlet3d} that there exists a uniform constant $\setc{szdrel}>0$ such that any mesh $\TT\in\T$ and all refinements $\widehat\TT$ of $\TT$ satisfy
\begin{align*}
  \normHeh{(J(\widehat\TT|_{\Gamma_D})-J(\TT|_{\Gamma_D}))v}{\Gamma_D}\leq \c{szdrel}\normLtwo{h(\TT)^{1/2}(1-\Pi_{p-1}(\TT|_{\Gamma_D}))\nabla_\Gamma v}{\bigcup\omega^5_D(\TT;\TT\setminus\widehat\TT)}
\end{align*}
for all $v\in H^{1}(\Gamma_D)$. We note that this estimate hinges on the use of
newest vertex bisection in the sense that the the constant $\c{szdrel}>0$ depends
on the shape of all possible patches. For newest vertex bisection, only finitely many
pairwise different patch shapes can occur.

Secondly, this estimate is applied for $v = g_D$. The definition of $h(\TT,5)$ 
in Proposition~\ref{prop:htilde} implies
 \begin{align*}
  h(\widehat\TT,5) &\le h(\TT,5)
 \quad\hspace*{3.7mm}\text{pointwise on all $T\in\TT$},\\
  h(\widehat\TT,5) &\leq \q{htc}h(\TT,5)
 \quad \text{pointwise on all }T\in\omega^5(\TT;\TT\setminus\widehat\TT),
 \end{align*}
for some independent constant $0< \q{htc} <1$. Hence
\begin{align*}
 (1-\q{htc}) \, h(\TT,5)|_{\bigcup\omega^5(\TT;\TT\setminus\widehat\TT)}
\leq h(\TT,5)-h(\widehat\TT,5)
\quad\text{pointwise on all }T\in\TT.
\end{align*}
This implies
\begin{align*}
& (1-\q{htc})\normLtwo{h(\TT,5)^{1/2}(1-\Pi_{p-1}(\TT|_{\Gamma_D}))\nabla_\Gamma g_D}{\bigcup\omega_D^5(\TT;\TT\setminus\widehat\TT)}^2\\
  &\leq  \normLtwo{h(\TT,5)^{1/2}(1-\Pi_{p-1}(\TT|_{\Gamma_D}))\nabla_\Gamma g_D}{\Gamma_D}^2-
    \normLtwo{h(\widehat\TT,5)^{1/2}(1-\Pi_{p-1}(\TT|_{\Gamma_D}))\nabla_\Gamma g_D}{\Gamma_D}^2.
\end{align*}
This \revision{and} the elementwise best-approximation property of 
$\Pi_{p-1}(\widehat\TT|_{\Gamma_D})$ prove that
\begin{align*}
\normLtwo{h(\widehat\TT,5)^{1/2}(1-\Pi_{p-1}(\widehat\TT|_{\Gamma_D}))\nabla_\Gamma g_D}{\Gamma_D}^2\leq \normLtwo{h(\widehat\TT,5)^{1/2}(1-\Pi_{p-1}(\TT|_{\Gamma_D}))\nabla_\Gamma g_D}{\Gamma_D}^2.
\end{align*}
With $h(\TT)\leq \c{htildeequiv} h(\TT,5)$
from Proposition~\ref{prop:htilde}, this implies
\begin{align*}
(1-\q{htc})\c{htildeequiv}^{-1}\normLtwo{&h(\TT)^{1/2}(1-\Pi_{p-1}(\TT{|_{\Gamma_D}}))\nabla_\Gamma g_D}{\bigcup\omega_D^5(\TT;\TT\setminus\widehat\TT)}^2\leq\mu(\TT)^2-\mu(\widehat\TT)^2.
  \end{align*}
The combination of the previous arguments leads to
 \begin{align*}
   \normHeh{(J(\widehat\TT|_{\Gamma_D})-J(\TT|_{\Gamma_D}))g_D}{\Gamma_D}^2\le \c{szdrel}\c{htildeequiv}(1-\q{htc})^{-1}\,\big( \mu(\TT)^2-\mu(\widehat\TT)^2\big).
 \end{align*}
 Since 
$\mu(\TT)^2 \le \sum_{T\in\TT} \osc_T(\TT)^2 \le \eta(\TT;U(\TT))$, there
also holds~\eqref{A:qob}.
This concludes the proof.
\end{proof}

\begin{remark}\label{rem:gammaR}
We briefly comment on the case $\Gamma_R=\emptyset$ with
\begin{align*}
\enorm{v}^2:=\norm{\nabla v}{L^2(\Omega)}^2+\norm{v}{H^{1/2}(\Gamma_D)}^2\neq b(v,v)
\end{align*}
The Rellich compactness theorem guarantees that 
$\norm\cdot{}$ is an equivalent norm in $H^1(\Omega)$.
The combination with~\cite[Lemma~20]{dirichlet3d} (i.e.~\eqref{eq:dirichlet3d} with $\norm{\cdot}{}=\normLtwo{\nabla(\cdot)}{\Omega}$) 
proves for sufficiently small $\epsqo\ll1$ that
 \begin{align}\label{eq:dirichlet3dmod}
 \begin{split}
 \enorm{U(\widehat\TT)-U(\TT)}^2&\leq \normLtwo{\nabla(u-U(\TT))}{\Omega}^2 -(1-\epsqo)\normLtwo{\nabla(u-U(\widehat\TT))}{\Omega}^2\\
  &\qquad+ \widetilde{C}_{\rm pyth}\epsqo^{-1}\normHeh{(J(\widehat\TT|_{\Gamma_D})-J(\TT|_{\Gamma_D}))g_D}{\Gamma_D}^2.
  \end{split}
 \end{align}
With~\eqref{eq:dirichlet3dmod} instead of~\eqref{eq:dirichlet3d}, the arguments in the proof of Proposition~\ref{prop:inhom} remain valid.
\end{remark}

The adaptive FEM for the mixed boundary value boundary~\eqref{eq:mixedbc} satisfies all assumptions of the abstract framework.

\begin{consequence}
The adaptive algorithm leads to convergence with quasi-optimal rate
for the estimator $\eta(\TT;U(\TT))$ in the sense of Theorem~\ref{thm:main}. For quasi-optimal rates of the discretization error in the sense of Theorem~\ref{thm:mainerr}, additional regularity of the data has to be imposed for higher-order elements $p \ge 1$, cf.\ Consequence~\ref{con:conforming}. \qed
\end{consequence}

Numerical examples which underline the above result can be found for 2D in~\cite{fpp} and for 3D in~\cite{dirichlet3d}.

\bigskip
\textbf{Acknowledgement}: The authors (MF, MP, DP) acknowledge financial support through the Austrian Science Fund (FWF) under grant P21732 as well as through the Viennese Science and Technology Fund (WWTF) under grant MA09-29.

\bibliographystyle{elsarticle-num}
\bibliography{literature}

\begin{thebibliography}{100}
\expandafter\ifx\csname url\endcsname\relax
  \def\url#1{\texttt{#1}}\fi
\expandafter\ifx\csname urlprefix\endcsname\relax\def\urlprefix{URL }\fi
\expandafter\ifx\csname href\endcsname\relax
  \def\href#1#2{#2} \def\path#1{#1}\fi

\bibitem{ao00}
M.~Ainsworth, J.~T. Oden, A posteriori error estimation in finite element
  analysis, Pure and Applied Mathematics (New York), Wiley-Interscience, New
  York, 2000.

\bibitem{v96}
R.~Verf\"urth, A review of a~posteriori error estimation and adaptive
  mesh-refinement techniques, Teubner, Stuttgart, 1996.

\bibitem{babs}
I.~Babuska, T.~Strouboulis, The finite element method and its reliability,
  Numerical Mathematics and Scientific Computation, Oxford University Press,
  New York, 2001.

\bibitem{repin}
P.~Neittaanm{\"a}ki, S.~Repin, Reliable methods for computer simulation,
  Vol.~33 of Studies in Mathematics and its Applications, Elsevier Science
  B.V., Amsterdam, 2004.

\bibitem{repinalleinzuhaus}
S.~Repin, A posteriori estimates for partial differential equations, Vol.~4 of
  Radon Series on Computational and Applied Mathematics, Walter de Gruyter GmbH
  \& Co. KG, Berlin, 2008.

\bibitem{CCMer}
C.~Carstensen, C.~Merdon, Estimator competition for {P}oisson problems, J.
  Comput. Math. 28~(3) (2010) 309--330.

\bibitem{CBK}
C.~Carstensen, S.~Bartels, R.~Klose, An experimental survey of a posteriori
  courant finite element error control for the {P}oisson equation, Adv. Comput.
  Math. 15~(1-4) (2001) 79--106.

\bibitem{Carstensen:2005:Unifying}
C.~Carstensen, A unifying theory of a posteriori finite element error control,
  Numer. Math. 100~(4) (2005) 617--637.

\bibitem{CCJH}
C.~Carstensen, J.~Hu, A unifying theory of a posteriori error control for
  nonconforming finite element methods, J. Numer. Math. 107~(3) (2007)
  473--502.

\bibitem{CCEigHop}
C.~Carstensen, M.~Eigel, R.~W. Hoppe, C.~L{\"o}bhard, A review of unified a
  posteriori finite element error control, Numer. Math. Theory Methods Appl.
  5~(4) (2012) 509--558.

\bibitem{d1996}
W.~D{\"o}rfler, A convergent adaptive algorithm for {P}oisson's equation, SIAM
  J. Numer. Anal. 33~(3) (1996) 1106--1124.

\bibitem{mns}
P.~Morin, R.~H. Nochetto, K.~G. Siebert, Data oscillation and convergence of
  adaptive {FEM}, SIAM J. Numer. Anal. 38~(2) (2000) 466--488.

\bibitem{bdd}
P.~Binev, W.~Dahmen, R.~DeVore, Adaptive finite element methods with
  convergence rates, Numer. Math. 97~(2) (2004) 219--268.

\bibitem{stevenson07}
R.~Stevenson, Optimality of a standard adaptive finite element method, Found.
  Comput. Math. 7~(2) (2007) 245--269.

\bibitem{ckns}
J.~M. Cascon, C.~Kreuzer, R.~H. Nochetto, K.~G. Siebert, Quasi-optimal
  convergence rate for an adaptive finite element method, SIAM J. Numer. Anal.
  46~(5) (2008) 2524--2550.

\bibitem{ks}
C.~Kreuzer, K.~G. Siebert, Decay rates of adaptive finite elements with
  {D}\"orfler marking, Numer. Math. 117~(4) (2011) 679--716.

\bibitem{cn}
J.~M. Cascon, R.~H. Nochetto, {Quasioptimal cardinality of {AFEM} driven by
  nonresidual estimators}, IMA J. Numer. Anal. 32~(1) (2012) 1--29.

\bibitem{nonsymm}
M.~Feischl, T.~F\"uhrer, D.~Praetorius, Adaptive {FEM} with optimal convergence
  rates for a certain class of non-symmetric and possibly non-linear problems,
  ASC Report 43/2012, {\em Institute for Analysis and Scientific Computing,
  Vienna University of Technology}.

\bibitem{ch06b}
C.~Carstensen, R.~H.~W. Hoppe, Convergence analysis of an adaptive
  nonconforming finite element method, Numer. Math. 103~(2) (2006) 251--266.

\bibitem{rabus10}
H.~Rabus, A natural adaptive nonconforming {FEM} of quasi-optimal complexity,
  Comput. Methods Appl. Math. 10~(3) (2010) 315--325.

\bibitem{BeMao10}
R.~Becker, S.~Mao, Quasi-optimality of adaptive nonconforming finite element
  methods for the {S}tokes equations, SIAM J. Numer. Anal. 49~(3) (2011)
  970--991.

\bibitem{bms09}
R.~Becker, S.~Mao, Z.~Shi, A convergent nonconforming adaptive finite element
  method with quasi-optimal complexity, SIAM J. Numer. Anal. 47~(6) (2010)
  4639--4659.

\bibitem{cpr13}
C.~Carstensen, D.~Peterseim, H.~Rabus, Optimal adaptive nonconforming {FEM} for
  the {S}tokes problem, Numer. Math. 123~(2) (2013) 291--308.

\bibitem{HX07}
J.~Hu, J.~Xu, Convergence of adaptive conforming and nonconforming finite
  element methods for the perturbed {S}tokes equation, Research ReportSchool of
  Mathematical Sciences and Institute of Mathematics, Peking University,
  available at www.math.pku.edu.cn:8000/var/preprint/7297.pdf.

\bibitem{mzs10}
S.~Mao, X.~Zhao, Z.~Shi, Convergence of a standard adaptive nonconforming
  finite element method with optimal complexity, Appl. Numer. Math. 60 (2010)
  673--688.

\bibitem{ch06a}
C.~Carstensen, R.~H.~W. Hoppe, Error reduction and convergence for an adaptive
  mixed finite element method, Math. Comp. 75~(255) (2006) 1033--1042.

\bibitem{LCMHJX}
L.~Chen, M.~Holst, J.~Xu, Convergence and optimality of adaptive mixed finite
  element methods, Math.\ Comp. 78~(265) (2009) 35--53.

\bibitem{CR2012}
C.~Carstensen, H.~Rabus, The adaptive nonconforming {FEM} for the pure
  displacement problem in linear elasticity is optimal and robust, SIAM J.
  Numer. Anal. 50~(3) (2012) 1264--1283.

\bibitem{HuangXu}
H.~{Jian Guo}, X.~{Yi Feng}, Convergence and complexity of arbitrary order
  adaptive mixed element methods for the {P}oisson equation, Sci China Math
  55~(5) (2012) 1083--1098.

\bibitem{fkmp}
M.~Feischl, M.~Karkulik, J.~Melenk, D.~Praetorius, Quasi-optimal convergence
  rate for an adaptive boundary element method, SIAM J. Numer. Anal. 51 (2013)
  1327--1348.

\bibitem{gantumur}
G.~Tsogtgerel, Adaptive boundary element methods with convergence rates,
  Numerische Mathematik 124~(3) (2013) 471--516.

\bibitem{affkp}
M.~Aurada, M.~Feischl, T.~F{\"u}hrer, M.~Karkulik, D.~Praetorius, Efficiency
  and optimality of some weighted-residual error estimator for adaptive {2D}
  boundary element methods, J. Comput. Appl. Math. 255 (2014) 481--501.

\bibitem{ffkmp:part1}
M.~Feischl, T.~F\"uhrer, M.~Karkulik, J.~Melenk, D.~Praetorius, Quasi-optimal
  convergence rates for adaptive boundary element methods with data
  approximation, part {I}: {W}eakly-singular integral equation, Calcolo
  published online first (2013).

\bibitem{ffkmp:part2}
M.~Feischl, T.~F\"uhrer, M.~Karkulik, J.~Melenk, D.~Praetorius, Quasi-optimal
  convergence rates for adaptive boundary element methods with data
  approximation, part {II}: {H}ypersingular integral equation, ASC Report
  30/2013, {\em Institute for Analysis and Scientific Computing, Vienna
  University of Technology}.

\bibitem{mns03}
P.~Morin, R.~H. Nochetto, K.~G. Siebert, Local problems on stars: a posteriori
  error estimators, convergence, and performance, Math. Comp. 72~(243) (2003)
  1067--1097.

\bibitem{dirichlet2d}
M.~Feischl, M.~Page, D.~Praetorius, Convergence and quasi-optimality of
  adaptive {FEM} with inhomogeneous {D}irichlet data, J. Comput. Appl. Math.
  255 (2014) 481--501.

\bibitem{dirichlet3d}
M.~Aurada, M.~Feischl, J.~Kemetm\"uller, M.~Page, D.~Praetorius, Each
  {$H^{1/2}$}-stable projection yields convergence and quasi-optimality of
  adaptive {FEM} with inhomogeneous {Dirichlet} data in {$\R^d$}, ESAIM Math.\
  Model.\ Numer.\ Anal. 47 (2013) 1207--1235.

\bibitem{bv}
I.~Babuska, M.~Vogelius, Feedback and adaptive finite element solution of
  one-dimensional boundary value problems, Numer. Math. 44~(1) (1984) 75--102.

\bibitem{bm87}
I.~Babuska, A.~Miller, A feedback finite element method with a posteriori error
  estimation. {I}. {T}he finite element method and some basic properties of the
  a posteriori error estimator, Comput. Methods Appl. Mech. Engrg. 61~(1)
  (1987) 1--40.

\bibitem{msv}
P.~Morin, K.~G. Siebert, A.~Veeser, A basic convergence result for conforming
  adaptive finite elements, Math. Models Methods Appl. Sci. 18~(5) (2008)
  707--737.

\bibitem{siebert}
K.~Siebert, A convergence proof for adaptive finite elements without lower
  bound, IMA J. Numer. Anal. 31~(3) (2011) 947--970.

\bibitem{nc2010}
S.~Nicaise, S.~Cochez-Dhondt, Adaptive finite element methods for elliptic
  problems: abstract framework and applications, ESAIM Math. Model. Numer.
  Anal. 44~(3) (2010) 485--508.

\bibitem{primer}
R.~H. Nochetto, A.~Veeser, Primer of adaptive finite element methods, in:
  Multiscale and adaptivity: modeling, numerics and applications, Vol. 2040 of
  Lecture Notes in Math., Springer, Heidelberg, 2012, pp. 125--225.

\bibitem{afem}
R.~H. Nochetto, K.~G. Siebert, A.~Veeser, Theory of adaptive finite element
  methods: an introduction, in: Multiscale, nonlinear and adaptive
  approximation, Springer, Berlin, 2009, pp. 409--542.

\bibitem{cs96}
C.~Carstensen, E.~P. Stephan, Adaptive boundary element methods for some first
  kind integral equations, SIAM J. Numer. Anal. 33~(6) (1996) 2166--2183.

\bibitem{cms}
C.~Carstensen, M.~Maischak, E.~P. Stephan, A posteriori error estimate and
  {$h$}-adaptive algorithm on surfaces for {S}ymm's integral equation, Numer.
  Math. 90~(2) (2001) 197--213.

\bibitem{mn2005}
K.~Mekchay, R.~H. Nochetto, Convergence of adaptive finite element methods for
  general second order linear elliptic {PDE}s, SIAM J. Numer. Anal. 43~(5)
  (2005) 1803--1827.

\bibitem{braess}
D.~Braess, Finite elements, 2nd Edition, Cambridge University Press, Cambridge,
  2001.

\bibitem{SR}
S.~C. Brenner, L.~R. Scott, The mathematical theory of finite element methods,
  3rd Edition, Vol.~15 of Texts in Applied Mathematics, Springer, New York,
  2008.

\bibitem{brezzi-fortin}
F.~Brezzi, M.~Fortin, Mixed and hybrid finite element methods, Vol.~15 of
  Springer Series in Computational Mathematics, Springer-Verlag, New York,
  1991.

\bibitem{Gi_Ra}
V.~Girault, P.-A. Raviart, Finite element methods for {N}avier-{S}tokes
  equations, Vol.~5 of Springer Series in Computational Mathematics,
  Springer-Verlag, Berlin, 1986.

\bibitem{CR73}
M.~Crouzeix, P.-A. Raviart, Conforming and nonconforming finite element methods
  for solving the stationary {S}tokes equations. {I}, Rev. Francaise Automat.
  Informat. Recherche Operationnelle Ser. Rouge 7~(R-3) (1973) 33--75.

\bibitem{CKP11}
C.~Carstensen, D.~Kim, E.-J. Park, A priori and a posteriori
  pseudostress-velocity mixed finite element error analysis for the {S}tokes
  problem, SIAM J. Numer. Anal. 49~(6) (2011) 2501--2523.

\bibitem{bdk}
L.~Belenki, L.~Diening, C.~Kreuzer, Optimality of an adaptive finite element
  method for the {$p$}-{L}aplacian equation, IMA J. Numer. Anal. 32~(2) (2012)
  484--510.

\bibitem{gmz12}
E.~M. Garau, P.~Morin, C.~Zuppa, Quasi-optimal convergence rate of an {AFEM}
  for quasi-linear problems of monotone type, Numer. Math. Theory Methods Appl.
  5~(2) (2012) 131--156.

\bibitem{zz}
O.~C. Zienkiewicz, J.~Z. Zhu, A simple error estimator and adaptive procedure
  for practical engineering analysis, Internat. J. Numer. Methods Engrg. 24~(2)
  (1987) 337--357.

\bibitem{stevenson08}
R.~Stevenson, The completion of locally refined simplicial partitions created
  by bisection, Math. Comp. 77~(261) (2008) 227--241.

\bibitem{kpp}
M.~Karkulik, D.~Pavlicek, D.~Praetorius, On {2D} newest vertex bisection:
  Optimality of mesh-closure and {$H^1$}-stability of {$L_2$}-projection,
  accepted for publication in Constr. Approx. 38 (2013) 213--234.

\bibitem{bn}
A.~Bonito, R.~H. Nochetto, Quasi-optimal convergence rate of an adaptive
  discontinuous {G}alerkin method, SIAM J. Numer. Anal. 48~(2) (2010) 734--771.

\bibitem{ccrgb}
C.~Carstensen, An adaptive mesh-refining algorithm allowing for an
  {$H^1$}-stable {$L^2$}-projection onto {C}ourant finite element spaces,
  Constr. Approx. 20~(4) (2004) 549--564.

\bibitem{pavbakk}
D.~Pavlicek,
  \href{http://www.asc.tuwien.ac.at/$\sim$dirk/download/thesis/bsc/pavlicek2011.pdf}{Optimalit\"at
  adaptiver {FEM}, {B}achelor thesis (in {G}erman)}, Institute for Analysis and
  Scientific Computing, Vienna University of Technology (2010).
\newline\urlprefix\url{http://www.asc.tuwien.ac.at/$\sim$dirk/download/thesis/bsc/pavlicek2011.pdf}

\bibitem{br79}
I.~Babuska, W.~C. Rheinboldt, Analysis of optimal finite-element meshes in
  {${\R}^{1}$}, Math. Comp. 33~(146) (1979) 435--463.

\bibitem{cs95}
C.~Carstensen, E.~P. Stephan, A posteriori error estimates for boundary element
  methods, Math. Comp. 64~(210) (1995) 483--500.

\bibitem{cc97}
C.~Carstensen, An a posteriori error estimate for a first-kind integral
  equation, Math. Comp. 66~(217) (1997) 139--155.

\bibitem{v94}
R.~Verf{\"u}rth, A posteriori error estimation and adaptive mesh-refinement
  techniques, in: Proceedings of the {F}ifth {I}nternational {C}ongress on
  {C}omputational and {A}pplied {M}athematics ({L}euven, 1992), Vol.~50, 1994,
  pp. 67--83.

\bibitem{cc96}
C.~Carstensen, Efficiency of a posteriori {BEM}-error estimates for first-kind
  integral equations on quasi-uniform meshes, Math. Comp. 65~(213) (1996)
  69--84.

\bibitem{fop}
S.~Ferraz-Leite, C.~Ortner, D.~Praetorius, Convergence of simple adaptive
  {G}alerkin schemes based on $h-h/2$ error estimators, Numer. Math. 116 (2010)
  291--316.

\bibitem{bch07}
D.~Braess, C.~Carstensen, R.~H.~W. Hoppe, Convergence analysis of a conforming
  adaptive finite element method for an obstacle problem, Numer. Math. 107~(3)
  (2007) 455--471.

\bibitem{bch09}
D.~Braess, C.~Carstensen, R.~H.~W. Hoppe, Error reduction in adaptive finite
  element approximations of elliptic obstacle problems, J. Comput. Math.
  27~(2-3) (2009) 148--169.

\bibitem{cr11}
C.~Carstensen, H.~Rabus, An optimal adaptive mixed finite element method, Math.
  Comp. 80~(274) (2011) 649--667.

\bibitem{estconv}
M.~Aurada, S.~Ferraz-Leite, D.~Praetorius, Estimator reduction and convergence
  of adaptive {BEM}, Appl. Numer. Math. 62~(6) (2012) 787--801.

\bibitem{afp}
M.~Aurada, M.~Feischl, D.~Praetorius, Convergence of some adaptive {FEM}-{BEM}
  coupling for elliptic but possibly nonlinear interface problems, ESAIM Math.
  Model. Numer. Anal. 46~(5) (2012) 1147--1173.

\bibitem{page}
M.~Page,
  \href{http://www.asc.tuwien.ac.at/$\sim$dirk/download/thesis/msc/page2010.pdf}{Sch\"atzerreduktion
  und {K}onvergenz adaptiver {FEM} f\"ur {H}indernisprobleme, {M}aster thesis
  (in {G}erman)}, Institute for Analysis and Scientific Computing, Vienna
  University of Technology (2010).
\newline\urlprefix\url{http://www.asc.tuwien.ac.at/$\sim$dirk/download/thesis/msc/page2010.pdf}

\bibitem{pp}
M.~Page, D.~Praetorius, Convergence of adaptive {FEM} for some elliptic
  obstacle problem, Appl. Anal. 92~(3) (2013) 595--615.

\bibitem{p1afem}
S.~Funken, D.~Praetorius, P.~Wissgott, Efficient implementation of adaptive
  {P}1-{FEM} in {M}atlab, Comput. Methods Appl. Math. 11~(4) (2011) 460--490.

\bibitem{DDPV}
E.~Dari, R.~G. Duran, C.~Padra, V.~Vampa, A posteriori error estimators for
  nonconforming finite element methods, RAIRO Model. Math. Anal. Numer. 30
  (1996) 385--400.

\bibitem{CBJ}
S.~Bartels, C.~Carstensen, S.~Jansche, A posteriori error estimates for
  nonconforming finite element methods, Numer. Math. 92~(2) (2002) 233--256.

\bibitem{CGS}
C.~Carstensen, D.~Gallistl, M.~Schedensack, Discrete {R}eliability for
  {C}rouzeix--{R}aviart {FEM}s, SIAM J. Numer. Anal. 51~(5) (2013) 2935--2955.

\bibitem{Alonso1996}
A.~Alonso, Error estimators for a mixed method, Numer. Math. 74~(4) (1996)
  385--395.

\bibitem{Carstensen1997}
C.~Carstensen, A~posteriori error estimate for the mixed finite element method,
  Math. Comp. 66 (1997) 465--476.

\bibitem{hw97}
R.~H.~W. Hoppe, B.~Wohlmuth, Adaptive multilevel techniques for mixed finite
  element discretizations of elliptic boundary value problems, SIAM J. Numer.
  Anal. 34~(4) (1997) 1658--1681.

\bibitem{cb02}
C.~Carstensen, S.~Bartels, Each averaging technique yields reliable a
  posteriori error control in {FEM} on unstructured grids. {I}. {L}ow order
  conforming, nonconforming, and mixed {FEM}, Math. Comp. 71~(239) (2002)
  945--969.

\bibitem{hw}
G.~C. Hsiao, W.~L. Wendland, Boundary integral equations, Vol. 164 of Applied
  Mathematical Sciences, Springer-Verlag, Berlin, 2008.

\bibitem{mclean}
W.~McLean, Strongly elliptic systems and boundary integral equations, Cambridge
  University Press, Cambridge, 2000.

\bibitem{ss}
S.~A. Sauter, C.~Schwab, Boundary element methods, Vol.~39 of Springer Series
  in Computational Mathematics, Springer-Verlag, Berlin, 2011.

\bibitem{afembem}
M.~Aurada, M.~Feischl, T.~F{\"u}hrer, J.~Melenk, D.~Praetorius, Inverse
  estimates for elliptic boundary integral operators and their application to
  the adaptive coupling of {FEM} and {BEM}, ASC Report 07/2012, {\em Institute
  for Analysis and Scientific Computing, Vienna University of Technology}.

\bibitem{cmps}
C.~Carstensen, M.~Maischak, D.~Praetorius, E.~P. Stephan, Residual-based a
  posteriori error estimate for hypersingular equation on surfaces, Numer.
  Math. 97~(3) (2004) 397--425.

\bibitem{BMN}
E.~B{\"a}nsch, P.~Morin, R.~H. Nochetto, An adaptive {U}zawa {FEM} for the
  {S}tokes problem: convergence without the inf-sup condition, SIAM J. Numer.
  Anal. 40~(4) (2002) 1207--1229.

\bibitem{YS}
Y.~Kondratyuk, R.~Stevenson, An optimal adaptive finite element method for the
  {S}tokes problem, SIAM J. Numer. Anal. 46~(2) (2008) 747--775.

\bibitem{DDP}
E.~Dari, R.~Dur{\'a}n, C.~Padra, Error estimators for nonconforming finite
  element approximations of the {S}tokes problem, Math. Comp. 64~(211) (1995)
  1017--1033.

\bibitem{CGSpseudo}
C.~Carstensen, D.~Gallistl, M.~Schedensack, Quasi-optimal adaptive pseudostress
  approximation of the {S}tokes equations, SIAM J. Numer. Anal. 51~(3) (2013)
  1715--1734.

\bibitem{BeMao09}
R.~Becker, S.~Mao, Convergence and quasi-optimal complexity of a simple
  adaptive finite element method, ESAIM Math. Model. Numer. Anal. 43~(6) (2009)
  1203--1219.

\bibitem{fpp}
M.~Feischl, M.~Page, D.~Praetorius, Convergence of adaptive {FEM} for elliptic
  obstacle problems with inhomogeneous {D}irichlet data, Int. J. Numer. Anal.
  Model. 11 (2014) 229--253.

\bibitem{rod94}
R.~Rodr{\'{\i}}guez, Some remarks on {Z}ienkiewicz-{Z}hu estimator, Numer.
  Methods Partial Differential Equations 10~(5) (1994) 625--635.

\bibitem{cb02p}
S.~Bartels, C.~Carstensen, Each averaging technique yields reliable a
  posteriori error control in {FEM} on unstructured grids. {II}. {H}igher order
  {FEM}, Math. Comp. 71~(239) (2002) 971--994.

\bibitem{scottzhang}
L.~R. Scott, S.~Zhang, Finite element interpolation of nonsmooth functions
  satisfying boundary conditions, Math. Comp. 54~(190) (1990) 483--493.

\bibitem{zeidler}
E.~Zeidler, Nonlinear functional analysis and its applications. {II}/{B},
  Springer-Verlag, New York, 1990.

\bibitem{cupr}
R.~F. Curtain, A.~J. Pritchard, Functional analysis in modern applied
  mathematics, Academic Press, London, 1977.

\bibitem{evansPDE}
L.~C. Evans, Partial differential equations, 2nd Edition, Vol.~19 of Graduate
  Studies in Mathematics, American Mathematical Society, Providence, RI, 2010.

\bibitem{dxz}
X.~Dai, J.~Xu, A.~Zhou, Convergence and optimal complexity of adaptive finite
  element eigenvalue computations, Numer. Math. 110~(3) (2008) 313--355.

\bibitem{cg12}
C.~Carstensen, J.~Gedicke, An oscillation-free adaptive {FEM} for symmetric
  eigenvalue problems, Numer. Math. 118~(3) (2011) 401--427.

\bibitem{sfbook}
G.~Strang, G.~Fix, An analysis of the finite element method, 2nd Edition,
  Wellesley-Cambridge Press, Wellesley, MA, 2008.

\bibitem{cg13}
C.~Carstensen, J.~Gedicke, An adaptive finite element eigenvalue solver of
  asymptotic quasi-optimal computational complexity, SIAM J. Numer. Anal.
  50~(3) (2012) 1029--1057.

\bibitem{Gallistl2013}
D.~Gallistl, Adaptive finite element method for clustered eigenvalues, in
  preparation (2013).

\bibitem{ccgs13}
C.~Carstensen, D.~Gallistl, M.~Schedensack, Adaptive nonconforming
  {C}rouzeix-{R}aviart {FEM} for eigenvalue problems, Math. Comp. (2013)
  accepted for publication.

\bibitem{sv}
R.~Sacchi, A.~Veeser, Locally efficient and reliable a posteriori error
  estimators for {D}irichlet problems, Math. Models Methods Appl. Sci. 16~(3)
  (2006) 319--346.

\bibitem{bcd}
S.~Bartels, C.~Carstensen, G.~Dolzmann, Inhomogeneous {D}irichlet conditions in
  a priori and a posteriori finite element error analysis, Numer. Math. 99~(1)
  (2004) 1--24.

\end{thebibliography}
\end{document}